\documentclass[10pt]{amsart}

\usepackage{amsmath, amssymb, amscd,mathrsfs, amsthm}
\usepackage{verbatim}
\usepackage{bbm}
\usepackage[mathscr]{eucal}
\usepackage[all]{xypic}
\usepackage{tocvsec2}
\usepackage{xr-hyper}
\usepackage[hmargin=3cm,vmargin=3.5cm]{geometry}
\usepackage{color}

\usepackage{rotating}

\usepackage{enumitem}
\setlist[enumerate]{leftmargin=*}
\setlist[itemize]{labelindent=\parindent, leftmargin=*}

\usepackage{hyperref}

\usepackage{marginnote,letltxmacro}

\LetLtxMacro\mn\marginnote

\numberwithin{equation}{section}

\allowdisplaybreaks[4]

\theoremstyle{plain}
\newtheorem{thm}{Theorem}[section]
\newtheorem{lem}[thm]{Lemma}
\newtheorem{prop}[thm]{Proposition}
\newtheorem{cor}[thm]{Corollary}
\theoremstyle{definition}
\newtheorem{defn}[thm]{Definition}
\theoremstyle{remark}
\newtheorem{rem}[thm]{Remark}

\newtheorem{example}[thm]{Example}

\theoremstyle{plain}

\theoremstyle{plain}
\newtheorem{mtheorem}{Theorem}

\theoremstyle{definition}

\makeatletter
\def\varddots{\mathinner{\mkern1mu
    \raise\p@\hbox{.}\mkern2mu\raise4\p@\hbox{.}\mkern2mu
    \raise7\p@\vbox{\kern7\p@\hbox{.}}\mkern1mu}}
\makeatother

\def\Gal{\operatorname{Gal}}
\def\Ind{\operatorname{Ind}}
\def\Im{\operatorname{Im}}
\def\pr{\operatorname{pr}}
\def\Re{\operatorname{Re}}
\def\Res{\operatorname{Res}}

\def\vol{\operatorname{vol}}
\def\tr{\operatorname{tr}}
\def\Hom{\operatorname{Hom}}

\def\Aut{\operatorname{Aut}}
\def\End{\operatorname{End}}

\def\disc{\mathrm{disc}}

\def\Frob{\operatorname{Frob}}

\def\rank{\operatorname{rank}}

\def\CH{\operatorname{CH}}

\def\Ad{\operatorname{Ad}}

\def\sgn{\operatorname{sgn}}
\def\sign{\operatorname{sign}}

\def\JL{\operatorname{JL}}

\newcommand\res{\operatorname{res}}
\newcommand\BC{\operatorname{BC}}

\def\ccc{\mathrm{C}}
\def\sss{\mathrm{S}}
\def\GL{\mathrm{GL}}
\def\GU{\mathrm{GU}}
\def\GSU{\mathrm{GSU}}
\def\U{\mathrm{U}}
\def\SU{\mathrm{SU}}
\newcommand\PGU{\operatorname{PGU}}
\def\G{\mathrm{G}}
\def\Sp{\mathrm{Sp}}
\def\GSp{\mathrm{GSp}}
\def\Mp{\mathrm{Mp}}
\def\O{\mathrm{O}}
\def\SO{\mathrm{SO}}
\def\GO{\mathrm{GO}}
\newcommand\GSO{\mathrm{GSO}}
\def\SL{\mathrm{SL}}

\def\Sym{\mathrm{Sym}}

\newcommand\HDS{\mathrm{HDS}}

\def\Sh{\mathrm{Sh}}
\def\fin{\mathrm{fin}}
\def\M{\mathrm{M}}
\def\N{\mathrm{N}}
\def\et{\mathrm{et}}
\def\id{\mathrm{id}}
\def\ad{\mathrm{ad}}

\newcommand{\PB}{\mathrm{PB}}

\def\fg{\mathfrak{g}}

\def\ii{\mathfrak{i}}

\def\fh{\mathfrak{h}}

\def\ff{\mathfrak{f}}
\newcommand\fF{\mathfrak{F}}
\def\fk{\mathfrak{k}}
\def\fl{\mathfrak{l}}

\def\fp{\mathfrak{p}}

\def\fq{\mathfrak{q}}

\def\ft{\mathfrak{t}}
\def\fu{\mathfrak{u}}
\def\so{\mathfrak{so}}

\newcommand\Bc{\mathcal{B}}

\def\FF{\mathcal{F}}
\def\cF{\mathcal{F}}
\def\Gc{\mathcal{G}}
\def\cG{\mathcal{G}}
\def\cH{\mathcal{H}}

\def\II{\mathcal{I}}

\def\cK{\mathcal{K}}

\def\K{\mathcal{K}}

\def\cL{\mathcal{L}}
\def\cM{\mathcal{M}}

\def\cO{\mathcal{O}}

\def\SS{\mathcal{S}}
\def\cS{\mathcal{S}}

\def\Vc{\mathcal{V}}
\def\cV{\mathcal{V}}
\def\Wc{\mathcal{W}}

\def\A{\mathbb{A}}
\def\C{\mathbb{C}}

\def\H{\mathbb{H}}
\def\I{\mathbb{I}}

\def\Q{\mathbb{Q}}
\def\Qbar{\overline{\Q}}

\def\R{\mathbb{R}}
\def\Ss{\mathbb{S}}

\def\X{\mathbb{X}}
\def\Y{\mathbb{Y}}
\def\Z{\mathbb{Z}}

\def\V{\mathbb{V}}

\def\1{\mathbf{1}}

\def\a{\mathbf{a}}
\def\b{\mathbf{b}}

\def\e{\mathbf{e}}

\def\g{\mathbf{g}}
\def\h{\mathbf{h}}
\def\i{\mathbf{i}}
\def\j{\mathbf{j}}

\def\m{\mathbf{m}}
\def\n{\mathbf{n}}
\def\v{\mathbf{v}}
\def\w{\mathbf{w}}
\def\x{\mathbf{x}}
\def\y{\mathbf{y}}
\newcommand\z{\mathbf{z}}
\newcommand\ZZ{\mathbf{Z}}

\def\0{\mathbf{0}}

\def\GG{\mathbf{G}}
\def\J{\mathbf{J}}

\def\VV{\mathbf{V}}
\def\WW{\mathbf{W}}

\def\Hs{\mathscr{H}}
\def\Ps{\mathscr{P}}

\def\uk{\underline{k}}
\def\ul{\underline{\ell}}

\def\Gs{\mathsf{G}}

\def\ba{\boldsymbol{\alpha}}

\def\vvv{\boldsymbol{v}}
\def\fff{\boldsymbol{f}}
\def\bo{\boldsymbol{\omega}}

\def\Chi{\boldsymbol{\chi}}

\def\rar{\rightarrow}

\def\(({( \! (}
\def\)){) \! )}
\def\llangle{\langle \hspace{-1mm} \langle}
\def\rrangle{\rangle \hspace{-1mm} \rangle}

\SelectTips{cm}{}
\newdir^{ (}{{}*!/-5pt/@^{(}}

\newcommand{\mat}[4]{\begin{pmatrix} #1 & #2 \\ #3 & #4 \end{pmatrix}}
\newcommand{\smat}[4]{\left( \begin{smallmatrix} #1 & #2 \\ #3 & #4 \end{smallmatrix} \right)}

\makeatletter
\def\pmodx#1{\allowbreak\ifinner\mkern8mu\else\mkern18mu\fi
 ({\fam\z@ mod}^\times\,#1)}
\makeatother

\newcounter{defcounter}
\newenvironment{intro-equation}{%
\addtocounter{equation}{-1}
\refstepcounter{defcounter}

\begin{equation}}
{\end{equation}}

\DeclareSymbolFontAlphabet{\mathrsfs}{rsfs}

\newcommand{\bigboxplus}{\mathop{\mathchoice%
{\raise-0.35em\hbox{\huge $\boxplus$}}%
{\raise-0.15em\hbox{\Large $\boxplus$}}{\hbox{\large $\boxtimes$}}{\boxtimes}}}

\newcommand{\Mot}{\mathrm{Mot}}

\newcommand{\HT}{\mathcal{HT}}

\newcommand\Nekovar{$\mathrm{Nekov\acute{a}\check{r}}$}

\newcommand{\pit}{\tilde{\pi}}

\begin{document}

\title[]{Hodge classes and the Jacquet-Langlands correspondence}
\author{Atsushi Ichino }
\author{Kartik Prasanna}

\begin{abstract}
We prove that the Jacquet-Langlands correspondence for cohomological automorphic forms on quaternionic Shimura varieties 
is realized by a Hodge class. 
Conditional on Kottwitz's conjecture for Shimura varieties attached to unitary similitude groups, we also show that the image of this Hodge class in $\ell$-adic cohomology is Galois invariant for all $\ell$.
\end{abstract}

\maketitle

\tableofcontents

\newcommand{\cZ}{\mathcal{Z}}
\newcommand{\cl}{\operatorname{cl}}
\newcommand{\sA}{\mathrsfs{A}}

\section{Introduction}

This article is motivated by the following question:  is Langlands functoriality in the case 
of cohomological automorphic forms 
on Shimura varieties 
induced by algebraic cycle classes?  
When the forms in question contribute to $H^1$, this follows from
Faltings' theorem \cite{faltings-mordell} on the Tate conjecture for divisors on abelian varieties, 
 but for higher $H^i$ it seems completely open even in the simplest of cases. 
Since constructing algebraic cycle classes seems extremely difficult, one can ask for the next 
best thing, namely to construct the associated {\it absolute Hodge classes}
\cite{dmos}. 
We study this problem in the most classical example of functoriality, namely the Jacquet-Langlands correspondence
for $\GL_2$ and its inner forms. 

\subsection{The main theorem}
Let $F$ be a totally real field, $[F:\Q]=n$.
Denote by $\Sigma_\infty$  the set of infinite places of $F$, and 
for $v\in \Sigma_\infty$, let $\sigma_v: F \hookrightarrow \R\subset
\C $ denote
the corresponding embedding of $F$ in $\C$. Let 
$F^c\subset \Qbar \subset \C$ be the compositum of $\sigma_v (F)$ as $v$ varies over
$\Sigma_\infty$. Thus $F^c$ is the Galois closure of the image of
$\sigma(F)$ for any $\sigma \in \Sigma_\infty$.  

Let $\pi=\otimes_v \pi_v $ be an automorphic representation of
$\GL_2(\A_F)$ corresponding to a 
(cohomological) holomorphic Hilbert modular newform of weight $(\uk,r)$ where $\uk=(k_1,\ldots, k_n)$ and
$k_1\equiv k_2 \equiv \cdots \equiv k_n \equiv r \bmod 2$. 
For simplicity, we will assume that $\pi$ has trivial Nebentypus
character so that it is self-dual up to a (Tate) twist. (See \S
\ref{sec:non-self-dual} for the non-self-dual case.) 
Moreover, in the introduction alone, we assume that $\pi$ has parallel weight two and that 
 the Hecke eigenvalues $a_v (\pi)$ (suitably normalized) are rational; thus $\pi$ (at least conjecturally) corresponds to an 
elliptic curve $A/F$.  In any case, it is known that to such a $\pi$ and every rational prime $\ell$ one can attach 
a two-dimensional $\ell$-adic 
Galois representation $ \rho_{\pi,\ell} $
of the Galois group $\Gal (\Qbar/F)$. The representations
$\rho_{\pi,\ell}$ (for varying $\ell$) form a compatible system in the sense that 
for all finite primes $v$ of $F$ not dividing $\ell$ and the conductor of $\pi$, we have
\[
\tr \rho_{\pi,\ell} (\Frob_v) = a_v (\pi),
\]
where $\Frob_v$ denotes a {\it geometric} Frobenius element attached to $v$;
in particular, this trace is independent of $\ell$.

Let $B_1$ and $B_2$ be two (non-isomorphic) quaternion algebras over $F$ such that 
$\pi$ admits Jacquet-Langlands transfers to the algebraic groups $G_1
=\Res_{F/\Q} B_1^\times$ and $G_2=\Res_{F/\Q} B_2^\times$; we denote
the corresponding automorphic representations of $G_1(\A)$ and
$G_2(\A)$ by $\pi_1$ and $\pi_2$ respectively. 
We assume that 
the set of infinite places of $F$ where 
$B_1$ is split agrees with the set of infinite places where $B_2$ is
split, and denote this common set of infinite places by 
$\Sigma \subset \Sigma_\infty$. 
Let $F_\Sigma$ be the subfield of $\C$ given by:
\begin{align*}
F_\Sigma &:= \Qbar^{\{ \sigma \in \Gal(\Qbar/\Q) \, | \, \sigma \Sigma=\Sigma\}} =
           (F^c)^{\{\sigma \in \Gal(F^c/\Q) \, | \, \sigma \Sigma=\Sigma\}}.
\end{align*}
Then $F_\Sigma$ is also characterized as the subfield of $\Qbar$ generated (over $\Q$)
by the elements  
\[
\sum_{v\in \Sigma} \sigma_v(x), \quad x\in F
\]
and is called the reflex field of the pair $(F, \Sigma)$.

Let $X_{1}$ and $X_{2}$ denote the {\it quaternionic Shimura varieties} 
associated with $G_1$ and $G_2$. 
Then $X_1$ and $X_2$ are of dimension $d:=|\Sigma|$ and have canonical models over the {\it same} reflex field $F_\Sigma\subset \Qbar \subset \C$.
The Langlands-Kottwitz method can be used to study the $\ell$-adic cohomology of the varieties $X_{1}$ and $X_{2}$. 
Following the work of several authors  (\cite{langlands72}, \cite{brylinski-labesse}, \cite{carayol-hmf}, \cite{reimann}, \cite{nekovar}), we have the following theorem: for $i=1,2$, the $\pi_i$-isotypic part of $H^*_{\et} (X_{i,\Qbar}, \Q_\ell)$ is concentrated entirely in the middle degree $d$ and moreover is isomorphic to the {\it tensor induction}
\begin{equation}
\label{eqn:intro-ten-ind}
 {\bigotimes_{v\in \Sigma}}' \ \rho_{\pi,\ell}^v
\end{equation}
where $\rho_{\pi,\ell}^v$ denotes the representation of $\Gal (\Qbar/ \sigma_v (F))$ given by $g\mapsto  \rho_{\pi,\ell} (\sigma_v^{-1} g \sigma_v)$. 
As a consequence, for all rational primes $\ell$, we have isomorphisms  
\begin{equation}
\label{eqn:hd-ell-isom}
H^d (X_1, \Q_\ell)_{\pi_1} \simeq H^d (X_2, \Q_\ell)_{\pi_2}
\end{equation}
as representations of $\Gal(\Qbar/F_\Sigma)$. Here and henceforth we write $H^*(X,\Q_\ell)$ for the $\Q_\ell$-vector space 
$H^*_{\et} (X_{\Qbar}, \Q_\ell)$.

The isomorphisms \eqref{eqn:hd-ell-isom}
above may be viewed as giving a collection of 
Tate classes in
\[
H^{2d} (X_1\times X_2, \Q_\ell(d)),
\]
and it is natural to ask if there is a single algebraic cycle $\mathcal{Z} \in \CH^{d} (X_1 \times X_2)$ that gives rise to this collection of Tate classes. If $p_1$ and $p_2$ are
the two projections below,
\[
\xymatrix{
& X_1 \times X_2 \ar[ld]_{p_1} \ar[rd]^{p_2} & \\
X_1 & & X_2 
}
\]
the class of such a putative algebraic cycle $\cZ$ gives rise to a map
\[
\mathrm{cl} (\mathcal{Z})^*: H^d(X_1) \rightarrow H^d(X_2), \quad x\mapsto p_{2,*} (\cl(\cZ) \cup p_1^*(x))
\]
for any Weil cohomology theory, which induces isomorphisms
\begin{equation}
\label{eqn:hweil-isom}
H^d(X_1)_{\pi_1} \simeq H^d(X_2)_{\pi_2}.
\end{equation}
Moreover, these isomorphisms for different Weil cohomology theories will be compatible via the usual comparison theorems.

With this motivation, we state our main theorem.  We remark that our proof (of part (ii) of the theorem below) assumes the validity of Kottwitz's conjecture characterizing the Galois representations occurring in the cohomology of Shimura varieties in the special case of Shimura varieties attached to unitary similitude groups.  (See Remark \ref{rem:characterization-of-galrep} below for a more extensive discussion of the status of this conjecture.) 

\begin{mtheorem}
\label{thm:intro-main-full}
Suppose that there is at least one infinite place of $F$ at which $B_1$ and $B_2$ are ramified. \begin{enumerate}
\item There is a non-zero {\it Hodge} class
\[
\xi \in H^{2d} (X_1 \times X_2, \Q)_{\pi_1 \boxtimes \pi_2}
\]
such that 
the induced map

\begin{equation}
\label{eqn:intro-xi-Hodge}
\xi(d)^*: H^d (X_1,\Q)_{\pi_1} \rightarrow H^d (X_2,\Q)_{\pi_2}, \quad x\mapsto p_{2,*} (\xi (d) \cup p_1^*(x))
\end{equation}
is an isomorphism of $\Q$-Hodge structures. (i.e., is an isomorphism of $\Q$-vector spaces, that after extending scalars to $\C$, preserves the Hodge filtration.)

\item Assume Kottwitz's conjecture for Shimura varieties attached to unitary similitude groups. 
Then the Hodge class $\xi$ can be chosen such that for all rational primes $\ell$, the image $\xi_\ell(d)$ of
(the Tate twist) $\xi(d)$ in the $\ell$-adic {\'e}tale realization
\[
H^{2d} (X_1 \times X_2, \Q_\ell)_{\pi_1 \boxtimes \pi_2} (d)
\]
 is $\Gal(\Qbar/{F_\Sigma})$-invariant.  Consequently, 
the induced map 
\begin{equation}
\label{eqn:intro-xi-ell}
\xi(d)^*_\ell: H^d (X_1,\Q_\ell)_{\pi_1} \simeq H^d (X_2,\Q_\ell)_{\pi_2},\quad x\mapsto p_{2,*} (\xi (d) \cup p_1^*(x))
\end{equation}
is an isomorphism of $\Gal(\Qbar/{F_\Sigma})$-modules. (Here we view $\xi(d)$ as an {\'e}tale class via the Betti-{\'e}tale comparison theorems.)
\end{enumerate}

\end{mtheorem}

Our proof does not use the previously known isomorphisms \eqref{eqn:hd-ell-isom}. Rather, it provides an alternate 
verification of these isomorphisms which may be of independent interest. 
We note also that the isomorphism
\eqref{eqn:intro-xi-Hodge} of Hodge structures 
implies relations between periods of modular forms on $B_1^\times$ and $B_2^\times$. 
Such period relations have been studied previously by relating the periods
to the Fourier coefficients of half-integral weight modular
forms \cite{oda-book}, \cite{oda-hs}, \cite{mur-ram}. 
In principle, one could use the period relations to deduce an isomorphism of Hodge structures;
however, it seems very unlikely that such methods can show that this
isomorphism is also Galois equivariant.

\subsection{Outline of the proof} 
We now explain the strategy of the proof of Theorem \ref{thm:intro-main-full}.
 In fact, the proof in the general case is very similar to that for $F=\Q$, $n=d=1$, and 
so we first describe this case, even though formally speaking this case is excluded from the theorem
on account of the assumption that $B_1$ and $B_2$ be ramified at at
least one infinite place. To be precise, one should work with
intersection cohomology in this case, but for simplicity we just use
usual cohomology with the understanding that the proof given below is
only 
correct once generalized to the setting where $F$ is a totally real
field and there is some infinite place where $B_1$ and $B_2$ are both ramified.

 The basic idea of the proof 
is to embed $X_1 \times X_2$ in a larger Shimura variety $X$, construct a 
Hodge class $\xi$ on $X$ and then show that its pullback to $X_1 \times X_2$ has the right property.
The implementation of this idea is a bit involved and breaks up as follows.

\subsubsection{Unitary Shimura varieties} 
 We first replace $X_1$ and $X_2$ by closely related unitary Shimura varieties. 
Pick an imaginary quadratic field $E$ that embeds in both 
$B_1$ and $B_2$. 
Let $\VV_1 = B_1$ and $\VV_2 = B_2$, viewed as (right) $E$-vector spaces. These are equipped with 
natural hermitian forms that are of signature $(1,1)$ at the infinite place. 
The corresponding unitary similitude groups are given by 
\begin{equation}
\label{eqn:gubv12}
\GU_E(\VV_1) \simeq  (B_1^\times  \times E^\times)/ F^\times, \quad \GU_E(\VV_2) \simeq (B_2^\times  \times E^\times)/ F^\times.
\end{equation}
Let $\VV = \VV_1 \oplus \VV_2$. Thus $\VV$ has signature $(2,2)$ at the infinite place. Consider the maps of 
algebraic groups 
\begin{equation}
\label{eqn:guv-embed}
B_1^\times \times B_2^\times \rightarrow \PB_1^\times \times \PB_2^\times \leftarrow \G (\U_E(\VV_1)\times \U_E(\VV_2)) /E^\times \rightarrow \GU_E(\VV)/E^\times.
\end{equation}
These induce maps of the associated Shimura varieties
\[
X_1 \times X_2 \rightarrow \tilde{X}_1 \times \tilde{X}_2 \leftarrow Y \rightarrow X 
\]
(where $X$ is the Shimura variety associated with $\GU_E(\VV)/E^\times$, etc.),
which may be viewed as giving a correspondence on $(X_1 \times X_2) \times X$. This correspondence induces a map on cohomology:
\[
 \iota^*: H^*(X) \rightarrow H^*(X_1 \times X_2).
\]
As such, since the kernel of the map 
\[
\G (\U_E(\VV_1)\times \U_E(\VV_2)) /E^\times \rightarrow  \PB_1^\times \times \PB_2^\times 
 \]
 is isomorphic to
 \[
 \G (E^\times \times E^\times)/E^\times \simeq E^{(1)} \simeq E^\times/F^\times,
 \]
 one can introduce a character $\eta$ of $E^\times/F^\times$ in the construction of the correspondence; this gives a map
 \[
 \iota_\eta^*: H^*(X) \rightarrow H^*(X_1 \times X_2).
\]
that depends on the choice of $\eta$.

\subsubsection{Cohomological representations and Vogan-Zuckerman theory}
Since the cohomology of $X$ is given by automorphic forms \cite{borel-wallach},  it is natural to first look for a {\it non-tempered}
automorphic representation $\Pi$ of $\GU(\VV)$ (or say of $\U(\VV)$ for simplicity) which contributes
to $H^2(X)$ but only to the $(1,1)$-part. 
The paper of Vogan-Zuckerman \cite{vz} classifies cohomological representations; one finds 
that there 
is a unique non-trivial (non-tempered) representation $\Pi_\infty^1$ of $\U (\VV_\R) =\U (2,2)_\R$ with the property that 
\[
H^{1,1} (\fg, K; \Pi_\infty^1) \neq 0.
\]
The representation $\Pi_\infty^1$ can be realized as a cohomologically induced representation $A_\fq$ where 
$\fq$ is a $\theta$-stable parabolic subalgebra of $\fg$ with Levi component $\mathfrak{u}(1,1) \oplus \mathfrak{u}(1,1)$. 
In order to construct $\Pi$, it is first natural to look for an explicit construction of $\Pi_\infty^1$ which is 
what is accomplished in the next step.

\subsubsection{An exceptional isogeny: archimedean theta correspondence and Kudla-Millson theory} 
The representation $\Pi_\infty^1$ can be constructed as a theta lift of the trivial representation of $\U(1,1)$ with 
appropriate choices of splitting characters. However, for rather subtle reasons, this fact does not seem to be 
useful in our construction. Instead, 
 we use the fact that there
is an exceptional isogeny 
\begin{equation}
\label{eqn:isogen-intro}
\SU(2,2)_\R \rightarrow \SO(4,2)_\R.
\end{equation}
Ignoring for the moment the difference between $\U$ and $\SU$, and between $\O$ and $\SO$, we may view 
$\Pi_\infty^1$ as a representation of $\O(4,2)_\R$, and viewed this way, the representation 
$\Pi_\infty^1$ is in fact a theta lift from $\SL_2$.  This fact may appear somewhat
familiar to connoisseurs of Kudla-Millson theory. Indeed, Kudla-Millson theory studies certain explicit closed forms that are Poincare dual to 
geodesic cycles coming from embedded $\O(3,2)$s in $\O(4,2)$, 
and shows that the corresponding 
automorphic representations of $\O(4,2)$ (which contribute to $H^{1,1}$)
can be constructed as theta lifts of forms of weight $3$ on $\SL_2$. 

\subsubsection{Inner forms}
For our purposes, we need {\it inner form} versions both of the isogeny \eqref{eqn:isogen-intro}
and of the theta lift. Moreover, we need to work with similitude groups rather than isometry groups. 
First the theta lift: let $B$ be the 
quaternion algebra given by $B=B_1 \cdot B_2$ in the Brauer group of $\Q$. 
Since $B_1$ and $B_2$ are assumed to be non-isomorphic, $B$ is a non-split quaternion algebra. 
Then there is a theta lift 
\[
\Theta: \sA(\GU_B(W)) \longrightarrow \sA( \GU_B(\tilde{V})^0)
\]
where $\tilde{V}$ is a certain three dimensional $B$-vector space 
equipped with a $B$-skew-hermitian form, $W$ is a one dimensional
$B$-vector space equipped with a $B$-hermitian form and $\sA(G)$
denotes the space of automorphic forms on $G$.
(To be precise, the theta lift
depends on a choice of Schwartz function.) The groups 
$\U_B(W)$ and $\U_B(\tilde{V})$ are respectively the requisite inner forms of $\SL_2$ and $\O(4,2)$. 
As for the isogeny, we construct (in \S \ref{sec:global-ex-isom}) an explicit 
{\it isomorphism}
\begin{equation}
\label{eqn:intro-delta}
\delta: \PGU_E(\VV) \xrightarrow{\simeq} \PGU_B(\tilde{V})^0
\end{equation}
which is an inner form version of \eqref{eqn:isogen-intro} above for (projectivized) similitude groups. 

\subsubsection{The global theta lift: Schwartz forms}
With this preparation, we can describe the construction of a $(1,1)$-class on $X$. 
Let $h$ be a modular form of weight $3$ and central character $\xi_E$, the quadratic 
character associated with the extension $E/\Q$, chosen such that it admits a Jacquet-Langlands transfer
to $B^\times$. 
Let 
$\tilde{\tau}_h$ be the corresponding representation of $\GL_2 (\A)$. Let $\JL$ denote the 
Jacquet-Langlands correspondence. Consider the composite maps of automorphic forms
\[
\sA(\GL_2) \xrightarrow{\JL} \sA (B^\times) = \sA (\GU_B (W)) \xrightarrow{\Theta} \sA (\GU_B(\tilde{V})^0 ) 
\]
and
\[
 \sA (\PGU_B(\tilde{V})^0 ) \xrightarrow{\delta^{-1}} \sA (\PGU_E(\VV))
 \rightarrow \sA (\GU_E(\VV)).
\]
We show that $\Theta \circ \JL (\tilde{\tau}_h)$ has trivial central
character, and so may be viewed as an
automorphic representation of the group $\PGU_B(\tilde{V})^0 $. Thus
we can consider the composite
\[
\Pi := \delta^{-1} \circ \Theta \circ \JL (\tilde{\tau}_h),
\]
which we may view as an automorphic representation of the group $\GU_E (\VV) (\A)$.
This representation has the property that $\Pi_\infty \simeq \tilde{\Pi}_\infty^1$, where 
$\tilde{\Pi}_\infty^1$ denotes the unique representation of $\GU(2,2)_\R$ with trivial central character whose restriction to 
$\U(2,2)_\R$ is isomorphic to $\Pi_\infty^1$. Further, one can check that:
\[
\dim H^{p,q} (\fg, K; \Pi_\infty) = \begin{cases}
1 & \text{if } (p,q)=(1,1) \text{ or } (3,3); \\
2 & \text{if } (p,q)=(2,2).
\end{cases}
\]
Explicitly, we construct following the ideas of Kudla-Millson (and Funke-Millson in the higher weight case), 
a {\it Schwartz form} $\varphi_\infty$ (rather than a Schwartz function)
such that with
$\varphi=\varphi_{\fin} \otimes \varphi_\infty$ for any choice of
a Schwartz function $\varphi_{\fin}$, the theta lift 
\[
\theta_\varphi (\phi)
\]
may be viewed as giving a $(1,1)$-class on $X$, for $\phi$ in the
space of $\JL(\tilde{\tau}_h)$. To be precise, the construction only
depends on the restriction of $\JL(\tilde{\tau}_h)$ to the subgroup 
$\GL_2(\A)^+$ (consisting of elements in $\GL_2(\A)$ with positive determinant at infinity) and the vector $\phi$ must be chosen to lie in the {\it
  anti-holomorphic} component of this restriction. 

\subsubsection{Nonvanishing of the restriction}
Next we show that for suitable choice of $\eta$, $h$, $\varphi_{\fin}$ and
$\phi$, the $(1,1)$-form $\iota_\eta^*(\theta_\varphi (\phi))$ is
non-vanishing, when projected to the $\pi_1\boxtimes \pi_2$-isotypic
component. 
Let us now explain the main idea to prove this non-vanishing.
Let $\omega_{f_{B_1}} $ and $\omega_{f_{B_2}}$ denote holomorphic one-forms in
$H^1 (X_{B_1}, \C)_{\pi_1}$ and $H^1 (X_{B_2}, \C)_{\pi_2}$ respectively. The strategy is to compute the integral
\[
\int_{X_{B_1} \times X_{B_2}} \iota_{\eta}^* \theta_{\varphi}(\phi) \cdot (p_1^* \omega_{f_{B_1}} \wedge \overline{p_2^* \omega_{f_{B_2}}})
\]
and show it is non-zero. 
Using the isomorphism $\delta$ from \eqref{eqn:intro-delta}
and noting that the decomposition $\VV = \VV_1 \oplus \VV_2$ of $E$-hermitian spaces induces a decomposition $\tilde{V}=V \oplus V_0$ of $B$-skew-hermitian spaces such that 
\[
 \GU_B(V)^0 \simeq (B_1^\times \times B_2^\times)/F^\times, \quad
 \GU_B(V_0)^0 \simeq E^\times, 
\]
and 
\[
 \delta : \G(\U_E(\VV_1) \times \U_E (\VV_2)) / E^\times \xrightarrow{\simeq} \G(\U_B(V) \times \U_B(V_0))^0/F^\times, 
\]
we can reduce the integral to a period on the left-hand side of the seesaw diagram below, which
again involves quaternionic unitary groups:
\[
 \xymatrix{
 \GU_B(\tilde{V})  \ar@{-}[dr] \ar@{-}[d] & \G(\U_B(W) \times \U_B (W))
   \ar@{-}[dl] \ar@{-}[d]  \\
 \G(\U_B(V) \times \U_B(V_0)) & \GU_B(W)}.
\]
The seesaw then implies that the period can be computed on the right where it becomes 
a triple product period of the form 
\[
\int_{[\GU_B(W)]}\phi \cdot \overline{f_B \theta(\eta)},
\]
where $f_B = \theta(\overline{f_{B_1}} \boxtimes f_{B_2})$ is an automorphic form on $\GU_B(W) \simeq B^\times$
in the Jacquet-Langlands transfer $\pi_B$ of $\pi$.
We then show that 
$\eta$,  $h$, $\varphi_{\fin}$ and $\phi$ can be chosen 
(depending on the finite parts of $f_{B_1}$ and $f_{B_2}$)
so as to make this triple product integral non-zero. 
A similar argument also shows that 
\[
\int_{X_{B_1} \times X_{B_2}} \iota_{\eta}^* \theta_{\varphi}(\phi) \cdot (\overline{p_1^* \omega_{f_{B_1}} }\wedge p_2^* \omega_{f_{B_2}})
\]
is non-zero, and in fact that the induced map 
\[
\iota_{\eta}^* \theta_{\varphi}(\phi) : H^1 (X_{B_1}, \C) \rightarrow H^1 (X_{B_2}, \C)
\]
is an isomorphism. 

\subsubsection{Hodge classes}
As yet we do not know that $\theta_\varphi(\phi)$ is a Hodge class. In fact, strictly speaking it is not likely to be a rational cohomology class, but 
we show that it lies in the $\C$-span of the Hodge classes in $H^2(X)$. 
The key point here is that the (expected)
classification of automorphic representations implies that any 
automorphic representation that is nearly equivalent to $\Pi$ must have 
archimedean component lying in the (unique) $A$-packet containing $\tilde{\Pi}_\infty^1$. 
Moreover, this archimedean $A$-packet consists of two representations 
$\tilde{\Pi}_\infty^1, \tilde{\Pi}_\infty^2$ and the latter contributes only to $H^4(X)$ and not $H^2(X)$. 
From this, we deduce that $H^2(X,\C)[\Pi_\fin] $ is entirely of type
$(1,1)$.  (The notation $H^2(X,\C)[\Pi_\fin] $ stands for the 
subspace of $H^2(X,\C)$ on which the unramified Hecke algebra at some
finite level acts 
by the same Hecke eigenvalues as on $\Pi_\fin$.)

Suppose for the moment that $\Pi$ has coefficients in $\Q$. Then 
\[
H^2(X,\Q)[\Pi_\fin] \otimes_{\Q} \C = H^2(X,\C)[\Pi_\fin],
\]
hence $H^2(X,\Q)[\Pi_\fin]$ 
is a rational Hodge structure, pure of type $(1,1)$. 
Since $\theta_\varphi(\phi)$ lies in $H^2(X,\C)[\Pi_\fin]$, we see that it lies in the 
$\C$-span of $H^2(X,\Q)[\Pi_\fin] $ and in particular is a $\C$-linear
combination of Hodge classes $\xi$. 
We have already seen that $\theta_\varphi(\phi)\neq 0$ and moreover that its restriction to the 
$\pi_1 \boxtimes \pi_2$-component of 
$X_1 \times X_2$ is non-zero.
From this and a simple continuity argument, one deduces that there is a Hodge class $\xi \in H^2(X,\Q)[\Pi_{\fin}]$ such that the induced map
\[
\iota_\eta ( \xi (1))^* : H^1 (X_{B_1}, \Q)_{\pi_1} \rightarrow H^1 (X_{B_2}, \Q)_{\pi_2}
\]
given by 
\[
x \mapsto p_{2,*} (p_1^* (x) \cdot \iota_{\eta}^* \xi (1))
\]
is an isomorphism of rational Hodge structures. 

\subsubsection{Galois representations}
\label{sss:gal-rep-intro}
Next, we need to understand the Galois representation on $H^* (X)$ 
associated to the $A$-packet containing $\Pi$. 
Again, for simplicity let us suppose $F=\Q$, $d=1$, the general case being similar. 
Then the expected relation between the Galois representation and the
$A$-parameter can be deduced from Kottwitz's conjecture (see Remark \ref{rem:characterization-of-galrep} below).  In our case,  we have:
\begin{align*}
H^2(X,\Q_\ell)_{\Pi} &= \Q_\ell (-1),  \\
H^4(X, \Q_\ell)_{\Pi} &= \Q_\ell(-2) \oplus \Sym^2(\rho_{h,\ell}) , \\
H^6 (X,\Q_\ell)_{\Pi} &= \Q_\ell (-3),
\end{align*}
where $\rho_{h,\ell}$ is the two dimensional $\ell$-adic representation
attached to $h$. 
From this, one deduces that as a Galois module, $H^2
(X,\Q_\ell)[\Pi_\fin]\simeq \Q_\ell (-1)^{m} $ for some integer $m$. 
For every rational prime $\ell$, the action of $\Gal(\Qbar/\Q)$ on $\xi(1)$ is then trivial 
and thus $\xi(1)^*$ (viewed as acting on $\ell$-adic cohomology via the 
Betti-{\'e}tale comparison) is a Galois equivariant isomorphism. 

\subsubsection{Descending coefficients}
The argument above needs a bit more care, since $\Pi$ may not have
coefficients in $\Q$. Thus one needs some care to ensure that the 
Hodge class constructed has coefficients in $\Q$. This argument needed
to achieve this is 
explained in detail in \S 
\ref{sec:proof-of-main-thm}. 
Roughly, the point is to replace 
$H^2 (X,\Q)[\Pi_\fin]$ by 
$H^2 (X,\Q)[\I]$,  
 where $\I$ is the kernel of the action of the unramified Hecke algebra (with $\Q$-coefficients) on 
$\Pi_\fin$. Another possible source of extra coefficients is the character
$\eta$, and this needs to be handled separately.

\subsubsection{The general case}
This completes the outline of the proof of Theorem \ref{thm:intro-main-full} in the case 
$F=\Q$, $n=d=1$, $k=2$. The general case (assuming still that
$\uk=(2,\ldots,2)$) is only slightly more complicated.  
In general we have
\[
\U_E(\VV) (\R) \simeq \U(2,2)^d \times \U(4)^{n-d},
\]
where the $\U(2,2)$ factors correspond to the places in $\Sigma$ and
the $\U(4)$ factors to the infinite places not in $\Sigma$. 
At the infinite places in $\Sigma$, i.e., where $B_1$ and $B_2$ are both split, we 
just imitate the constructions above. However, we need to deal as well 
with the infinite places where $B_1$ and $B_2$ are both ramified. 
At such places the representation $\Pi_\infty$ is trivial and the 
local $A$-packet is a singleton, consisting of just the trivial representation. 
This is consistent with the fact that at such places $v$, we have 
\[ \U_E(\VV)_v\simeq \U(4), \quad \U_B(\tilde{V})_v \simeq \O(6), \quad \U_B(W)_v \simeq \SL_2
\]
and the theta lift of the weight $3$ {\it holomorphic} discrete series representation on $\SL_2$ is the 
trivial representation of $\O(6)$. 
The conclusion then is that $H^{2d} (X,\C)[\Pi_\fin]$ consists entirely 
of $(d,d)$-classes and one can find a Hodge class $\xi \in H^{2d}(X,\Q)[\Pi_{\fin}]$ such that the induced map
\[
 \xi (d)^* : H^d (X_{B_1}, \Q) \rightarrow H^d (X_{B_2}, \Q)
\]
given by 
\[
x \mapsto p_{2,*} (p_1^* (x) \cdot \iota_{\eta}^* \xi (d))
\]
is an isomorphism of rational Hodge structures, that is also Galois
invariant. 

\begin{rem}
We note the following conceptual reason why we work with the group $\U_B(\tilde{V})$ which at archimedean places is (almost) isomorphic to  a product $\O(4,2)^d \times \O(0,6)^{n-d}$. 
After all, in principle, one could also construct Kudla-Millson classes directly on the group $\U_B(V)$,  which at archimedean places looks like a product $\O(2,2)^d \times \O(0,4)^{n-d}$, by taking a lift of a form of parallel weight two. However, the issue is that on this smaller group, the Hodge classes are 
mixed up with other classes of the same degree, and therefore it is difficult to see that the Kudla-Millson class is in the $\C$-span of the Hodge classes, except in the ``trivial" situation when $B_1=B_2$; in that case the group $\U_B(V)$ is quasi-split and there are obvious ``diagonal" cycles in the correct degree. On the larger group however, the Hodge classes in degree $(d,d)$ can be separated out 
using Hecke operators; this is the crucial idea on which the proof rests.
\end{rem}

\begin{rem} 
The assumption that $B_1$ and $B_2$ are ramified at some infinite place is made for technical reasons; it ensures that the auxiliary Shimura variety $X$ used in the proof is compact. We
believe that with some extra work (e.g., working with intersection cohomology), this assumption could be
relaxed. 
\end{rem}

\begin{rem}

Our proof of Theorem \ref{thm:intro-main-full} requires the construction of the particular automorphic 
representation $\Pi$ on the unitary group $\PGU_E (\VV)$ and a precise characterization of the near equivalence class of this representation. 
We give two proofs of this characterization. The first proof uses the expected classification of {\it non-tempered} automorphic representations on unitary groups (associated to 
hermitian spaces over a CM field)
in terms of local and global $A$-packets, which is work in progress of 
Kaletha, Minguez, Shin and White \cite{kmsw}. 
The expected results from their work that we need are stated carefully in \S
\ref{sec:classification-global} and \S \ref{ss:local-A-packets}. But we also give another, more direct proof, of the characterization of this 
representation using the theta correspondence, that does not use  \cite{kmsw}.  While this latter proof is unconditional, we have retained the proof using the full classification, since 
it provides a conceptual justification for why the method works, and since it may be useful in other situations. 

\end{rem}

\begin{rem} 
\label{rem:characterization-of-galrep} 
As mentioned before, our proof of part (ii) of Theorem~\ref{thm:intro-main-full} is conditional on the truth of Kottwitz's conjecture describing the Galois representations occurring in the cohomology of Shimura varieties in terms of automorphic representations.  The main results on Galois representations that we need are stated in Propositions \ref{prop:galrep-char} and \ref{prop:galrep-char-ss}.  In \S \ref{subsec:Kottwitz}, we explain in some detail how these propositions follow from Kottwitz's conjecture \cite{kot}.

 While we do not prove any new results towards Kottwitz's conjecture in this paper, it is an area of active investigation and the results we rely on will hopefully be available in the near future. 
For the benefit of the reader,  we now explain what results towards this conjecture are currently available and what work still needs to be done.  In {\it loc.~cit.},  Kottwitz outlined a strategy to prove the conjecture via establishing a stable trace formula and comparing it to the Grothendieck-Lefschetz trace formula.  In the subsequent papers \cite{kot-jams}, \cite{kot-invent}, Kottwitz used this strategy to verify his conjecture for certain Shimura varieties of PEL type.  

The Shimura varieties that we use are of abelian-type but not PEL.  For abelian-type Shimura varieties,  a stable trace formula and the comparison with the Grothendieck-Lefschetz trace formula has recently been established by Kisin-Shin-Zhu \cite{ksz}.  However,  (as is explained in {\it loc.~cit.} \S 0.2 and \S 9.2) two additional pieces of work need to be done to complete the characterization of Galois representations:
\begin{enumerate}
\item First,  one needs an equality relating the stable distribution of \cite{ksz} to
the one in Kottwitz.  This relation is encoded in the expected formula (9.2.2.1) of \cite{ksz}, which the authors of \cite{ksz} 
are planning to investigate in a sequel to that paper. 

\item Second,  one needs the classification of automorphic representations on unitary similitude groups in terms of $A$-parameters.  The corresponding results for unitary groups are the subject of past and ongoing work of Kaletha-Minguez-Shin-White.  The extension of these results from unitary groups to unitary similitude groups is also expected to be within reach. 
\end{enumerate}

\end{rem}

\begin{rem}
At the request of one of the referees, we discuss the relation between this paper and the 
work of Bergeron-Millson-M{\oe}glin (e.g.~\cite{bmm-orth} and \cite{bmm-unit}), which proves many cases of the Hodge conjecture for certain orthogonal or unitary Shimura varieties.  
The strategy in those papers is to show that in a range of degrees, the space of 
Hodge classes on the varieties under consideration is spanned (for the most part) by the classes of Kudla-Millson cycles,  which are linear combinations of cycle classes of sub-Shimura varieties, and are thus algebraic.  (In some cases, for example $\U(2,2)$, they also need to use classes that are known to be algebraic due to the Lefschetz-$(1,1)$ theorem, but are not obviously in the span of the classes of Kudla-Millson cycles.) 
Our work is complementary to this,  and in a somewhat orthogonal direction, since in our setting, there are no obvious Kudla-Millson cycles in the degrees under consideration, nevertheless we construct interesting Hodge classes.  For example, the simplest interesting setting for us (beyond $(1,1)$-classes for $\U(1,1)$ which can be addressed using Lefschetz-$(1,1)$) is the case of $(2,2)$-classes for $\U(2,2) \times \U(2,2)$, which is not covered in {\it loc.~cit.} Our expectation is that these Hodge classes (that represent functoriality) cannot be obtained from cycle classes of sub-Shimura varieties by any functorial process, even if one throws in classes that are known to be algebraic by the Lefschetz-$(1,1)$ theorem. 

The reader may also be interested in the discussion in \S \ref{sec:ahc}.  
As pointed out there,  there are also situations where there are Kudla-Millson cycles in the degrees of interest, but they do not span the space of Hodge classes. Thus it seems that to understand whether these Hodge classes are algebraic requires studying algebraic cycles on Shimura varieties that do not arise from sub-Shimura varieties.  This is a topic that has not seen much systematic work so far.  

\end{rem}

\subsection{Extensions and generalizations}

In this section, we discuss some extensions and generalizations of the main result stated above.
 
\subsubsection{Local systems and normalizations}
While we have stated the main result for trivial coefficients, it works equally well 
for local systems. In the main text, this more general case is treated. 

We briefly mention the numerology in the case of general local
systems. Suppose that the form $\pi$ has weights $\uk=(k_1,\ldots,
k_n)$. Then (in the classical normalization) the Hodge structure of $H^*(X_i)_{\pi_i}$ is a tensor product
over the places $v$ in $\Sigma$ of a Hodge structure of type
\[
(k_v-1,0) + (0,k_v-1).
\]
Thus $H^*(X_1)_{\pi_1} \otimes H^*(X_2)_{\pi_2}$ is a tensor product
over the places $v$ in $\Sigma$ of a Hodge structure of type
\begin{equation}
\label{eqn:hs-of-tensor}
(2k_v-2,0) + 2 (k_v-1,k_v-1) + (0,2k_v-2).
\end{equation}
The Hodge class in $H^*(X_1)_{\pi_1} \otimes H^*(X_2)_{\pi_2}$  should come from the tensor product over the places
$v$ in $\Sigma$ of a class of type $(k_v-1,k_v-1)$. In our
construction, we pick an auxiliary form $\tilde{\tau}$ of weights
$\uk+\1 = (k_1+1, \ldots, k_n+1)$. Then $\JL( \tilde{\tau})$  corresponds to
a Hodge structure which is a tensor 
 product
over the places $v$ in $\Sigma$ of a Hodge structure of type
\[
(k_v,0) + (0,k_v).
\]
Its lift $\Pi$ to $\U_B(\tilde{V})$ contributes to different cohomological
degrees; so there is an associated Hodge diamond which is the tensor 
 product
over the places $v$ in $\Sigma$ of a Hodge diamond of the form:

\begin{equation}
\label{eqn:hs-of-unitary}
\xymatrix{
& (k_v+1,k_v+1) & \\
(2k_v,0) &2 (k_v,k_v) & (0,2k_v) \\
& (k_v-1,k_v-1) & 
}
\end{equation}

The Hodge class in $H^*(X)_\Pi $  comes from the tensor product over the places
$v$ in $\Sigma$ of the class of type $(k_v-1,k_v-1)$. The ``rest'' of
the Hodge structure at any place $v$ consists of Tate twists
of this $(k_v-1,k_v-1)$ class and $\Sym^2$ of the Hodge structure
attached to $\tilde{\tau}$. 

In the main text, we use the ``automorphic normalization'' instead of
the classical normalization. This amounts to twisting the Hodge
structures in \eqref{eqn:hs-of-tensor} and \eqref{eqn:hs-of-unitary} above by 
$(2-k_v,2-k_v)$. This twist is therefore not visible in parallel
weight $2$.

\subsubsection{The non-self-dual case}
\label{sec:non-self-dual}

The assumption that $\pi$ has trivial central character forces
  $\pi$ to be self-dual, and is just made for simplicity.
The non-self-dual case can also be treated similarly; we do not 
treat this in the paper, but we outline here the main differences. 

Let us denote the contragredient of $\pi$ by $\pi^\vee$ and let $\chi$ be the central character of $\pi$ so that 
$\pi \simeq \pi^\vee \otimes \chi$. 
The method described above extends to this case, except that we must 
choose $\eta$ such that $\eta|_{\A^\times_F} = \chi^{-1} $ to
compensate for the central character of $\pi$. 
We remark on one unusual feature. Namely, it seems that the method outlined here naturally produces a Hodge class $\xi \in H^{2d}( X, \Q(\chi^2))_\Pi $ such that:
\begin{enumerate}
\item The induced map 
\[
\xi(d)^*: H^d(X_{B_1}, \Q(\chi^{-1}))_{\pi_1^\vee} \rightarrow H^d(X_{B_2}, \Q(\chi))_{\pi_2}
\]
is an isomorphism of $\Q(\chi)=\Q(\chi^{-1})$-vector spaces and preserves the Hodge filtration (on tensoring with $\C$). 

\item The Galois module $H^{2d}(X,\Q_\ell (\chi^2))$ is isomorphic to (a sum of copies of) $\Q_\ell (-d) (\chi^2)$ and the induced map
\[
\xi(d) ^*: H^d(X_{B_1}, \Q_\ell(\chi^{-1}))_{\pi_1^\vee} \rightarrow H^d(X_{B_2}, \Q_\ell(\chi))_{\pi_2}
\]
satisfies the following Galois equivariance:
\[
\sigma (\xi^*(x)) = \chi^2 (\sigma) \cdot \xi^* (\sigma (x)).
\]
We note that $\Q(\chi^{-1})=\Q(\chi)$ and $\Q_\ell (\chi^{-1})=\Q_\ell (\chi)$. 
\end{enumerate}
The reason this is unusual is that one might expect to have a natural construction producing a {\it rational} Hodge class in 
$H^*(X_{B_1}, \Q(\chi))_{\pi_1^\vee} \otimes H^*(X_{B_2}, \Q(\chi))_{\pi_2}$, since 
after all the Galois representation $H^*(X_{B_1}, \Q_\ell (\chi))_{\pi_1^\vee} \otimes H^*(X_{B_2}, \Q_\ell (\chi))_{\pi_2}$ always
contains the {\it trivial} representation as a direct summand. Instead, our construction naturally produces a Hodge class (with coefficients in a number field) in 
$H^* (X_{B_1}, \Q(\chi))_{\pi_1} \otimes  H^*(X_{B_2}, \Q(\chi))_{\pi_2}$
and then one has to ``untwist'' it to produce the rational Hodge class
that one expects to exist.

\subsubsection{Absolute Hodge classes}
\label{sec:ahc}
The main theorem above is close to saying that the class 
$\xi$ is an {\it absolute Hodge class} in the sense of Deligne \cite{dmos}. However, what is missing is the 
de Rham piece of the story, i.e., in order to show that $\xi$ is absolutely Hodge, we 
would need to show in addition to the above that it is also de Rham rational and that for every embedding $\tau$ of $F_\Sigma$ in 
$\C$, the class $\tau(\xi)$ is a Hodge class, whose image in $\ell$-adic cohomology is 
Galois invariant for all $\ell$. It seems difficult to show this directly.
In a previous version of this paper, we expressed the hope that one might be able to deduce that $\xi$ is absolutely Hodge by 
showing that it satisfies a stronger property, namely that it is a {\it motivated cycle} in the sense of Andr\'{e}.
However, the strategy that we had in mind runs into a serious obstacle that we are unable to circumvent at the moment, 
so the problem 
of showing that $\xi$ is absolutely Hodge remains open. The obstacle is related to the following fact: there exist 
tempered $L$-packets $\Pi$ of representations on $\U_E(\VV)$ (with $\dim_E(\VV)=3$) which contribute to Hodge classes on the associated Shimura varieties, such that the rank of the 
$\Pi$-isotypic component of the space of algebraic cycles of group-theoretic origin (i.e., coming from embeddings $\U_E(\VV') \hookrightarrow \U_E(\VV)$, with $\dim_E(\VV')=2$) is non-zero, yet is strictly smaller than the dimension of $\Pi$-isotypic component of the space of Hodge classes. In particular, there exist Hodge classes on such varieties that are not represented by algebraic cycles coming from embedded unitary groups.

\subsubsection{Functoriality for unitary groups}
\label{sec:jl-for-unitary}
It would be very interesting to generalize the results of this paper to general unitary groups. 
The main obstacle to doing this seems to be understanding automorphic periods for the 
embedding
\begin{equation}
\label{eqn:abcd}
\U_E(\VV_1) \times \U_E (\VV_2) \hookrightarrow \U_E (\VV),
\end{equation}
where $\dim_E(\VV_1) = \dim_E(\VV_2) = n$ say, with tempered representations $\pi_1$, $\pi_2$ on $\U_E(\VV_1)$ and $\U_E(\VV_2)$ respectively, and 
a non-tempered representation $\pi$ on $\U_E (\VV)$. In the case treated in this paper, this is accomplished for $n=2$ by using 
exceptional isogenies to relate the unitary groups above to inner forms of orthogonal groups, and then using a seesaw to relate the 
requisite period integrals to triple product periods for $\GL_2$, which are well understood and fall within the purview of the Gan-Gross-Prasad (GGP) conjectures.
In the general case, these exceptional isogenies are not available. Thus, it seems important to formulate and prove analogs of the GGP conjecture in the setting of the 
\eqref{eqn:abcd} above.

\bigskip

\noindent {\bf Acknowledgements:}
We thank the referees for many useful suggestions that improved the article.
K.P.  also thanks the Institute for Advanced Study 
for its support in 2014-15 as a von Neumann fellow, when the work on this paper was initiated. 
During the preparation of this article, A.I. was partially supported by JSPS KAKENHI Grant Number 26287003,
and K.P. was partially supported by NSF grants DMS 1160720, DMS 1600494, DMS 2001293 and a grant from the 
Simons Foundation (\# 305784). 

\newpage

\section{Shimura varieties, local systems, and motives} 

\subsection{Realizations of motives}
\label{sec:realizations}

Some of our definitions below may be somewhat non-standard. 

\subsubsection{Hodge structures}

Let $L$ be a number field given with a fixed embedding in $\C$. An $L$-Hodge structure pure of weight $n$ will be an $L$-vector space $V$ equipped with a descending filtration $F^\cdot V_\C$ on $V_\C = V\otimes_L  \C$ such that for $p+q=n+1$, we have
\[
V_\C = F^p V_\C \oplus \overline{F^q V_\C}.
\]
For any pair $(p,q)$ with $p+q=n$, we set $V^{p,q} = F^p V_\C \cap \overline{F^q V_\C}$.

\subsubsection{Realizations of motives with coefficients}

Let $k$ and $L$ be number fields. Let $\Mot_k^L$ denote the category of motives over $k$ with coefficients in $L$. (For the moment it is 
not very important what equivalence relation we use on algebraic cycles.) We are particularly interested in certain 
realization functors on $\Mot_k^L$, assuming we are given embeddings $k\hookrightarrow \Qbar \subset \C$ and $L\subset \C$. 

\begin{itemize}
\item {\it The Betti realization.} The Betti realization $H_B (M)$ which is 
 an $L$-Hodge structure. 

\item {\it The $\ell$-adic realizations.} For each rational prime $\ell$, $H_\ell (M)$ is a free $L\otimes \Q_\ell$-module, equipped with a continuous ($L\otimes \Q_\ell$-linear) action of 
$G_k:=\Gal (\Qbar/k)$.

\end{itemize}

We also have a natural comparison isomorphism\[
H_B (M) \otimes_{\Q} \Q_\ell \simeq H_\ell (M)
\]
of (free) $L\otimes \Q_\ell$-modules. 

There are other realizations which will not concern us in this paper. 
Thus we define a category $\cM_{k}^L$ as follows. The objects in this category are collections 
\[
(V, V_\ell)
\]
as $\ell$ varies over the primes, where $V$ is an $L$-vector space equipped with an $L$-Hodge structure and $V_\ell$ is a free $L \otimes \Q_\ell$-module with a
continuous ($L\otimes \Q_\ell$-linear) action of $G_k$, along with isomorphisms
\[
i_\ell: V \otimes \Q_\ell \simeq V_\ell.
\]
A morphism between two such objects $(V,V_\ell, i_\ell)$ and $(V',V'_\ell, i'_\ell)$ is an $L$-linear map $j: V \rightarrow V'$ that is a morphism of $L$-Hodge structures such that  the $L \otimes \Q_\ell$-linear maps $j_\ell: V_\ell \rightarrow V'_\ell$ defined 
by the commutative diagram below 
\[
\xymatrix{
V \otimes \Q_\ell  \ar[r]^{j\otimes 1} \ar[d]^{i_\ell} & V' \otimes \Q_\ell \ar[d]^{i'_\ell} \\
V_\ell \ar[r]^{j_\ell} & V'_\ell
}
\]
are $G_k$-equivariant. 

If $L=\Q$, we omit the superscript and simply write $\cM_k$. Note that to any proper smooth variety $X$ over $k$, we can attach
objects 
\[
\cH^n (X) = \left(H^n (X (\C), \Q), H^n_{\et} (X_{\Qbar}, \Q_\ell), i_\ell \right)
\]
in the category $\cM_k$. 

If $L\subset L'\subset \C$, there is a natural functor $\cM_k^L \rightarrow \cM_k^{L'}$, sending $(V,V_\ell)$ to $(V\otimes_L L', V_\ell \otimes_L L' = V_\ell \otimes_{L\otimes \Q_\ell} (L'\otimes \Q_\ell)).$

\subsection{Shimura varieties and local systems}

\subsubsection{Shimura varieties}

We recall some basic facts about Shimura varieties \cite{deligne-shimura}.
Let $\Ss = \Res_{\C/\R} \mathbb{G}_m$ denote the Deligne torus.
As usual, a Shimura datum is a pair $(G,X)$ consisting of a reductive algebraic group $G$ over $\Q$ 
and a $G(\R)$-conjugacy class $X$ of homomorphisms $h:\Ss \rar G_\R$ satisfying the following conditions:
\begin{enumerate}
\item For $h$ in $X$, the Hodge structure on the Lie algebra $\fg$ of $G_\R$ given by ${\Ad} \circ h$ is of type $(0,0) + (-1,1) + (1,-1)$. (In particular, the restriction of such an $h$ to $\mathbb{G}_{m,\R} \subset \Ss$ has image in the center of $G_\R$.)
\item For $h$ in $X$, $({\Ad} \circ h)(i)$ is a Cartan involution on $G_{\R}^{\ad}$, where $G^{\ad}$ is the adjoint group of $G$.  
\item $G^\ad$ has no factor defined over $\Q$ whose real points form a compact group. 
\end{enumerate}
These conditions imply that $X$ has the natural structure of a disjoint union of Hermitian 
symmetric domains. 
The group $G(\R)$ acts on $X$ on the left by
\[
(g \cdot h)(z) = g\cdot h(z) \cdot g^{-1}.
\]

Let $\A$ and $\A_f$ denote respectively the ring of ad\`{e}les and finite ad\`{e}les of $\Q$.
Let $\K$ be an open compact subgroup of $G(\A_f)$. The Shimura variety associated to $(G,X,\K)$ is the quotient
\[
\Sh_\K (G,X) =  G(\Q) \backslash X \times G(\A_f)/ \K.
\]
For $\cK$ small enough, this has the natural structure of a smooth variety over $\C$. The inverse limit
\[
\Sh (G,X) = \underleftarrow{\lim}_\K \ \Sh_\K (G,X)
\]
is a pro-algebraic variety that has a canonical model over a number field $E(G,X)$, the reflex field of the Shimura datum $(G,X)$. In particular, each $\Sh_\K (G,X)$ has a canonical model over $E(G,X)$. For brevity of notation, we will often write simply
$\Sh_G$ or $\Sh_{G,\cK}$ since $X$ will be understood from context. 

We recall the definition of $E(G,X)$.
This field is defined to be the field of definition of the conjugacy class of cocharacters
\[
\mu_h: \mathbb{G}_{m,\C} \rightarrow \Ss_\C \rar G_\C,
\]
where the first map is $z\mapsto (z,1)$ and the second is the one induced by $h$. 

\begin{rem}
The field $E(G,X)$ is given as a subfield of $\C$, and as such 
has by definition a canonical embedding into $\C$. When not specified below, 
any embedding of $E(G,X)$ in $\C$ will always be this canonical embedding. Indeed, we will not 
have use for any other embedding. 
\end{rem}

{\it  All Shimura varieties occurring in this paper will be compact, so we will assume this to be the case in the rest of this chapter.}

\subsubsection{Local systems and cohomology}

Let $(\rho,V)$ be a finite-dimensional representation of $G$ defined over a number field $L\subset \C$.
We assume that $\rho$ factors through an action of $G/Z_s$, where $Z_s$ is the largest subtorus of the center of $G$ which is split over $\R$ but which has no subtorus split over $\Q$.
To the data $(G,X,\rho)$, we can associate the following:
\begin{enumerate}
\item A local system $\V$ of $L$-vector spaces on $\Sh_G$. 
\item For each prime $\ell$, an $\ell$-adic local system $\V_{\ell}$ (of $L\otimes \Q_\ell$-vector spaces) on $\Sh_G$. 
\end{enumerate}
Then $H^i (\Sh_{G,\cK} (\C) , \V)$ is an $L$-vector space (in fact, an $L$-Hodge structure) and there are natural isomorphisms of free $L\otimes \Q_\ell$-modules:
\begin{equation}
\label{eqn:betti-etale-comparison-K}
H^i (\Sh_{G,\cK} ( \C), \V) \otimes_{\Q} \Q_\ell \simeq H^i_{\et}  (\Sh_{G,\cK} \otimes_{E(G,X)} \Qbar , \V_{\ell})
\end{equation}
(see \cite[Expos\'{e} XI]{SGA4-3}).
Note that we are using the given embedding $E(G,X) \hookrightarrow \Qbar \subset \C$ on both sides of the isomorphism above. 
The Hecke algebra $\cH (G(\A_f), \cK)$ acts on both sides of \eqref{eqn:betti-etale-comparison-K} and the isomorphism is Hecke equivariant. 
Taking the direct limit over $\cK$, we get 
an isomorphism:
\[
H^i (\Sh_{G} ( \C), \V) \otimes_{\Q} \Q_\ell \simeq H^i_{\et}  (\Sh_{G} \otimes_{E(G,X)} \Qbar , \V_{\ell}).
\]

Let $\Pi$ be an irreducible cohomological automorphic representation of $G(\A)$. The $\Pi$-isotypic component of $H^i(\Sh_{G} ( \C), \V_\C)$ is defined to be 
\begin{align*}
H^i (\Sh_{G} ( \C), \V_\C)_\Pi & := \Hom_{\cH(G(\A_f),\cK)} (\Pi_f^{\cK}, H^i (\Sh_{G} ( \C), \V_\C )^{\cK} ) \\ & = \Hom_{\cH(G(\A_f),\cK)} (\Pi_f^{\cK}, H^i (\Sh_{G,\cK} ( \C), \V_\C ) ) 
\end{align*}
for $\cK$ small enough, this being independent of the choice of $\cK$. 
By Matsushima's formula \cite{borel-wallach},
\[
H^i (\Sh_{G,\cK} ( \C), \V_\C) \simeq \bigoplus_{\pi} m(\pi) H^i (\fg, K; \pi_\infty \otimes \V_\C) \otimes \pi_f^{\cK}
\]
where the sum is over automorphic representations $\pi= \pi_\infty \otimes \pi_f$ of $G(\A)$ and $m(\pi)$ is the multiplicity of $\pi$ in the discrete spectrum of $G$. It follows that
\[
H^i (\Sh_{G} ( \C), \V_\C)_\Pi \simeq \bigoplus_{\pi, \pi_f^{\cK} \simeq \Pi_f^{\cK}}  m(\pi) H^i (\fg, K; \pi_\infty \otimes \V_\C),
\]
where the sum is over those $\pi$ such that $\pi_f^{\cK} \simeq \Pi_f^{\cK}$ as $\cH(G(\A_f),\cK)$-modules.

\subsubsection{Pullback and pushforward}

Let $f:(G,X_1) \rightarrow (H,X_2)$ be a morphism of Shimura data.
We assume that the reflex fields of $(G,X_1)$ and $(H,X_2)$ are the same
subfield $E$ of $\C$. Let $\rho$ be a finite-dimensional representation of $H$ defined over a number field $L\subset \C$ (which we can also view as a representation of $G$ via the map $G\rightarrow H$) and denote by $\V$ the associated local systems on $\Sh_H$, $\Sh_G$. (Thus the local system on $\Sh_G$ is just obtained by pull back from $\Sh_H$.) Then there are functorial maps
\[
f^*: H^i(\Sh_H, \V) \rightarrow H^i(\Sh_G, \V),
\]
defined both in Betti and $\ell$-adic cohomology, which may be viewed as giving a morphism in the category $\cM_E^L$. Suppose in addition that:
\begin{enumerate} 
\item \label{GHsurj} $G\rightarrow H$ is surjective with kernel $Z$ contained in the center of $G$.
\item \label{Zcohtriv} $Z$ is cohomologically trivial so that $G(\A) \rightarrow H(\A)$ is surjective as well. 
\end{enumerate}
Then there is a bijection between automorphic representations $\Pi_H$ of $H(\A)$ and $\Pi_G$ of $G(\A)$ on which $Z$ acts trivially. 
Assuming that $L$ contains the (common) field of definition of $\Pi_H$ and $\Pi_G$,  the map 
$f^*$ induces an isomorphism
\[
H^i(\Sh_H, \V)_{\Pi_H}  \simeq H^i(\Sh_G, \V)_{\Pi_G}
\]
in $\cM^L_E$.

We will also need to consider the case when the map $G\rightarrow H$ satisfies \eqref{GHsurj} but not \eqref{Zcohtriv}. Typically, in such cases, we will be interested in maps in the opposite direction. Indeed, there is a natural pushforward map
\[
f_*: H^i(\Sh_{G}, \V) \rightarrow H^i(\Sh_{H}, \V).
\]
If $\cK_1$ and $\cK_2$ are (small enough) open compact subgroups of $G(\A_f)$ and $H(\A_f)$ respectively, the induced map 
\[
\Sh_{G,\cK_1} \rightarrow \Sh_{H,\cK_2}
\]
is finite {\'e}tale onto its image which is a union of components of
$\Sh_{H,\cK_2}$. Thus $f_*$ can be defined by taking the trace to the
image, and then extending by zero outside the image. 

\begin{rem}
To make the definition of $f_*$ independent of the choice of $\cK_1$ and 
$\cK_2$, we need to normalize it by multiplying by the factor $\vol
(\cK_1)/\vol(\cK_2)$
for some choice of Haar measures on $G(\A_f)$ and $H(\A_f)$. In our application, we will
implicitly make such a choice in \S \ref{ss:form-construction} and \S
\ref{sec:proof-of-main-thm}, but the
exact choice is unimportant. 
\end{rem}

\section{Quaternionic Shimura varieties and the main theorem}
\label{sec:qsv-mt}

\subsection{Quaternionic Shimura varieties}

Let $F$ be a totally real field and $\Sigma_\infty$ the set of infinite places of $F$. Let $B$ be a {\it non-split} quaternion algebra over $F$
and $\Sigma$ the set of infinite places of $F$ where $B$ is split.
Put $n = [F:\Q]$ and $d = |\Sigma|$.
We fix an isomorphism
\[
B\otimes_F \R \simeq \M_2(\R)^d \times \H^{n-d},
\]
which gives an identification
\[
G_B (\R) \simeq \GL_2(\R)^d \times (\H^\times)^{n-d},
\]
where $G_B:=\Res_{F/\Q} B^\times$.
Let $\tau_v$ denote the composite map 
\[
B \otimes_{F,\sigma_v} \R \simeq \M_2 (\R) \hookrightarrow \M_2 (\C)
\]
for $v\in \Sigma$ and 
\[
B \otimes_{F,\sigma_v} \R \simeq \H \hookrightarrow \M_2 (\C)
\]
for $v\in \Sigma_\infty\smallsetminus \Sigma$.
Then $\tau_v$ may be viewed as giving a two-dimensional complex representation of 
$G_B$. 
We identify $\C^\times$ with a subgroup of $\GL_2 (\R)$ via
\begin{equation}
\label{eqn:iota}
z=a+bi \mapsto \iota(z) := \begin{pmatrix} a& b \\ 
-b & a \end{pmatrix}.
\end{equation}
Let $X$ denote the $G_B(\R)$-conjugacy class of 
\[
h: \Ss \rightarrow G_{B,\R}, \quad h(z) = \left(\iota(z),\cdots, \iota(z), 1, \cdots ,1\right),
\]
so that $h_v(z):=h(z)_v=\iota(z)$ for $v \in \Sigma$ and $h_v (z) = 1$ for $v \in \Sigma_\infty \smallsetminus \Sigma$. We write either $\Sh_{G_B}$ (or 
for ease of notation, simply $\Sh_B$) for the associated Shimura variety. 
The variety $\Sh_B$ admits a canonical model over the reflex field $F_\Sigma$. 
The Hecke algebra $\cH (G_B (\A_f), \cK)$ acts on $\Sh_{B,\cK}$ via correspondences. Moreover, the 
inverse limit
\[
\Sh_B = \varprojlim_{\cK} \Sh_{B,\cK}
\]
admits a right $G_B( \A_f)$-action. We refer the reader to
\cite[\S 1]{periods1} for a more detailed discussion of the 
Shimura varieties $\Sh_B$. 

\subsection{Local systems}
Let $\pi$ be an automorphic representation of $\GL_2 (\A_F)$ attached to a holomorphic 
Hilbert modular newform of weight $(\uk,r)$, where $\uk=(k_1, \ldots, k_n)$ is a collection of integers of the same parity and $r$ is an integer with 
$k_i\equiv r \bmod 2$.
(We will often denote $k_i$ by $k_v$ if $v$ is the $i$th place in the ordering and write $\uk = (k_v)_{v \in \Sigma_\infty}$.)
We suppose that $\pi$ admits a Jacquet-Langlands transfer $\pi_B$ to $G_B(\A)$. We also assume that  $k_i\ge 2$ for all $i$. This implies that 
the representation $\pi_{B,\infty}$ is cohomological, namely $\pi_B$ contributes to the cohomology of a 
local system on $\Sh_B$. The local system is attached to the representation $\tau_{\uk,r}^\vee$ of $G_B(\C)$, where 
\begin{align*}
\tau_{\uk,r} & := \bigotimes_{v} (\sigma_v \circ \nu)^{(r-k_v+2)/2}\Sym^{k_v-2} (\tau_v ) \\
 & = \bigotimes_v ({\det} \circ \tau_v )^{(r-k_v+2)/2} \Sym^{k_v-2} (\tau_v).
\end{align*}
Here $\nu$ denotes the reduced norm on $B$.
If $k_i$ is even for all $i$, then by \cite{waldspurger-rationality} (Proposition I.3 and \S II.2), the restriction of the representation $\tau_{\uk,r}$ to the group $G_B$ is defined over (any field $L$ containing) 
 $\Q(\uk)$, where $\Q(\uk)$ is the fixed field of the subgroup
\[
 \{ \sigma \in \Aut(\C/\Q) \, | \, \sigma \uk = \uk \}
\]
with $\sigma \uk = (k_{\sigma^{-1} \circ v})_{v \in \Sigma_\infty}$.
More precisely, this representation contains an $L$-structure invariant by $G_B$ and that is unique up to homothety.

\begin{rem}
 In this paper, we are only concerned with the case when $\pi$ has trivial central character up to twisting by a power of the reduced norm. 
 This implies that the weights $k_i$ must all be even. 
 Then, by twisting $\pi$ by a power of the norm character, we may assume that $r=0$. For simplicity, we will thus make this assumption for the rest of the paper and drop $r$ from the notation. Thus we will just write $\tau_{\uk}$ below. (For the more general case of non-trivial central characters, see \S \ref{sec:non-self-dual}.) 
 \end{rem} 

Let $L=\Q(\pi)$ be the field of rationality of $\pi$, as defined in \cite{waldspurger-rationality}, \S 1.8. By {\it loc.~cit.} Corollary I.8.3 and Lemma I.2.3, this field contains $\Q(\uk)$ and also agrees with the field generated by (all but finitely many, in particular the unramified) Hecke eigenvalues of $\pi$. 
Thus we may view $\tau_{\uk}$ as being defined over $L$, and then 
we get an associated local system of $L$-vector spaces $\V_{\uk}(L)$ on $\Sh_B (\C)$ and 
for every finite prime $\ell$, an \'{e}tale $L\otimes \Q_\ell$-sheaf $\V_{\uk}(L)_{\ell}$ on $\Sh_B$. (See also \cite{carayol-hmf}, \S 2.1.) 
Let 
\begin{align*}
V^{\cK}_B (L) &:=H^* (\Sh_{B, \cK}(\C), \V_{\uk}(L)), \\ V^{\cK}_B(\C)&:=H^* (\Sh_{B, \cK}(\C), \V_{\uk}(\C) ), \\
  V_{B}^{\cK}(L)_\ell&:= H^*_{\et} (\Sh_{B, \cK} \otimes_{F_\Sigma} \Qbar, \V_{\uk,\ell}),
\end{align*}
so that there are canonical isomorphisms
\[
V^{\cK}_B (L)  \otimes_L \C \simeq V^{\cK}_B(\C)
\]
and 
\[
V^{\cK}_B (L)  \otimes_{\Q} \Q_\ell \simeq  V_{B}^{\cK}(L)_\ell.
\]

We fix an isomorphism
\[
B \otimes \A_F^S \simeq \M_2(\A_F^S),
\]
where $\A_F^S$ denotes the ad{\`e}les of $F$ outside a finite set of places $S$ containing $\Sigma_\infty$ and all finite places where $B$ is ramified.
This gives an isomorphism
\begin{equation}
\label{eqn:identification-at-finite-places}
B^\times (\A_F^S) \simeq \GL_2(\A_F^S).
\end{equation}
We assume that $\pi$ transfers to $B^\times(\A_F)$, i.e., there exists an
automorphic representation $\pi_B= \pi_{B,\infty} \otimes \pi_B^f$ of $B^\times(\A_F)$ (necessarily unique by strong multiplicity one) such that 
$\pi_B^S \simeq \pi^S$ via the identification \eqref{eqn:identification-at-finite-places} above.

For the cohomology with complex coefficients, we can define the $\pi_B$-isotypic component by 
\[
V_{B,\pi_B} (\C):= \Hom_{\cH_{\C} (G_B(\A_f), \cK)} ((\pi_B^f)^{\cK}, V_B^\cK (\C)),
\]
for $\cK$ small enough. This 
is concentrated in degree $2d$ 
and is independent of the choice of $\cK$. To work over the field of rationality, we note that 
by \cite[Lemma 1.2.2 and \S II.1]{waldspurger-rationality}, the Hecke module $(\pi_B^f)^{\cK}$ is also defined over $L$. 
More precisely, it contains an $L$-structure $(\pi_B^f)^{\cK} (L)$ that is invariant by the Hecke algebra with $\Q$-coefficients, $\cH_{\Q} (G_B(\A_f), \cK)$,
and that is unique up to homothety. 
This allows us to define the $\pi_B$-isotypic components
\begin{align*}
V_{B,\pi_B} & := \Hom_{\cH_{\Q} (G_B(\A_f), \cK)} ((\pi_B^f)^{\cK}(L), V_B^\cK(L)), \\ V_{B,\pi_B, \ell} & := \Hom_{\cH_{\Q} (G_B(\A_f), \cK)} ((\pi_B^f)^{\cK} (L),  V_{B}^\cK(L)_\ell),
\end{align*}
for $\cK$ small enough, these being independent of the choice of $\cK$. 
Moreover, there are canonical isomorphisms of free $L\otimes \Q_\ell$-modules:
\[
V_{B,\pi_B} \otimes_{\Q} \Q_\ell \simeq  V_{B,\pi_B, \ell}.
\]

\subsection{The main theorem in the general case}

We can now state the main theorem in the case of general local
systems. 
Let $B_1$ and $B_2$ be two quaternion algebras that are split at the same set of archimedean places $\Sigma \subset \Sigma_\infty$,
such that $\pi$ transfers to both $B_1^\times(\A_F)$ and $B_2^\times (\A_F)$. For ease of notation we write the transfers 
as $\pi_1$ and $\pi_2$ instead of $\pi_{B_1}$ and $\pi_{B_2}$ respectively. 
 
\begin{thm} \label{mainthm:generalcase}
Suppose that there is at least one infinite place of $F$ at which $B_1$ and $B_2$ are ramified. 
\begin{enumerate}
\item 
Let $L$ be the coefficient field of $\pi$. 
Then there is an isomorphism of $L$-Hodge structures 
\[
\iota: V_{B_1,\pi_1} \simeq V_{B_2,\pi_2}.
\]

\item 
Assume Kottwitz's conjecture for Shimura varieties attached to unitary similitude groups.
Then the isomorphism $\iota$ of part (i) can be chosen such that for all finite primes $\ell$, the maps $\iota_\ell$, defined by
requiring the diagram
\begin{equation}
\label{eqn:compat-hd-general}
\xymatrix{
V_{B_1,\pi_1} \otimes_{\Q} \Q_\ell \ar[r]^{\iota\otimes 1} \ar[d]^{\simeq}  & V_{B_2,\pi_2}\otimes_\Q \Q_\ell \ar[d]^{\simeq} \\
V_{B_1,\pi_1,\ell} \ar[r]^{\iota_\ell} & V_{B_2,\pi_2,\ell}
}
\end{equation}
be commutative, are $\Gal(\Qbar/ F_\Sigma)$-isomorphisms.

\end{enumerate}
\end{thm}

This theorem will be proved in \S \ref{sec:proof-of-main-thm}.

\section{Unitary and quaternionic unitary Shimura varieties} 
\label{sec:unitary-quat-unitary}

The proof of the main theorem will require working with several different auxiliary Shimura varieties, some that are 
associated with unitary groups and some with quaternionic unitary groups. In this section, we introduce the main actors 
and the relations between them. Many of the claims below will only be justified in the following section; however 
we believe it is more transparent to introduce all the different groups up front, and relegate the details of various isomorphisms and maps to \S \ref{sec:global-ex-isom}. The reader may want to read these sections in parallel. 

\subsection{Unitary and quaternionic unitary groups}
\label{ss:unitary-quat-unitary}

Let $F$ be a totally real field.
Let $B_1$ and $B_2$ be two quaternion algebras that are split at the same set of infinite places of $F$.
Let $E$ be a CM extension of $F$ that embeds in both $B_1$ and $B_2$. We fix such embeddings
$E \hookrightarrow B_1$, $E \hookrightarrow B_2$
and write 
\[
B_1 = E + E\j_1, \quad B_2 = E+ E\j_2
\]
for some trace zero elements $\j_1 \in B^\times_1$, $\j_2 \in B^\times_2$.
We write $\pr_i$ for the projection $B_i \rightarrow E$ onto the ``first coordinate" and $*_i$ for the main involution on $B_i$. 
Then $\VV_i:=B_i$ is a right Hermitian $E$-space, the form being given by:
$$ (x,y)_i = \pr_i (x^{*_i} y).$$
If $x=a+\j_i b$, $y= c+\j_i d$, then 
$$ (x,y)_i = (a + \j_i b, c+\j_i d)_i =  a^\rho c - J_i  b^\rho d,$$
where $\rho$ is the non-trivial Galois automorphism of $E/F$
and $J_i = \j_i^2 \in F^\times$. 
This form satisfies the relations
$$ (x\alpha, y\beta)_i = \alpha^\rho (x,y)_i \beta$$
for $\alpha,\beta \in E$
and
$$ (x,y)_i = (y,x)_i^\rho.$$
 Then 
\[
\cG_1 := \GU_E (\VV_1) \simeq (B_1^\times \times E^\times)/F^\times, \quad \cG_2:=\GU_E (\VV_2) \simeq (B_2^\times \times E^\times)/F^\times,
\]
where the (inverses of these) isomorphisms are given by: $(\beta, \alpha) \mapsto (x \mapsto \beta x \alpha^{-1})$.
Let 
\[
\cG = \G (\U_E (\VV_1) \times \U_E (\VV_2))/E^\times = \G \left( (B_1^\times \times E^\times)/F^\times \times (B_2^\times \times E^\times)/F^\times\right)/E^\times,
\]
where $E^\times$ embeds as $\alpha \mapsto ([1,\alpha], [1,\alpha])$. 
We define groups $\tilde{\cG}$ and $\cG_0$ that are closely related to $\cG$ as follows:
\[
\tilde{\cG} =  \G (\U_E (\VV_1) \times \U_E (\VV_2)), \quad 
\cG_0 = B_1^\times/F^\times \times B_2^\times/F^\times.
\]

Let $\VV = \VV_1 \oplus \VV_2$, which is a four dimensional $E$-hermitian space.  Also let $\tilde{\VV} =\wedge^2 (\VV)$. In \S \ref{sec:constr-gl-ex-isom} we will show that $\tilde{\VV}$ is naturally equipped with the structure of a right $B$-space where 
$B:=B_1 \cdot B_2$ is the quaternion algebra over $F$ whose class in the Brauer group of $F$ equals the product of the classes of $B_1$ and $B_2$.  (Note that $B$ is split at all the infinite places of $F$.) When we want to think of $\tilde{\VV}$ as a $B$-space, we will write instead $\tilde{V}$ for it. Moreover, we show that 
$\tilde{V}$ is equipped with a $B$-{\it skew}-hermitian form such that there is a canonical isomorphism
\[
\GU_E(\VV)/E^\times=\PGU_E(\VV) \simeq \PGU_B(\tilde{V})^0= \GU_B(\tilde{V})^0/F^\times.
\]
There is also a canonical decomposition $\tilde{V} = V^\sharp \oplus V^\sharp_0$ of $B$-skew-hermitian spaces. 
Let
\begin{align*}
\tilde{\mathrsfs{G} } &= \GU_E(\VV), \\ \mathrsfs{G} &=
                       \GU_E(\VV)/E^\times, \\
\mathrsfs{G}_B&=\GU_B(\tilde{V})^0/F^\times, \\ \tilde{\mathrsfs{G}}_B&=\GU_B(\tilde{V})^0, \\
   \cG_B &= \G (\U_B(V^\sharp) \times \U_B (V_0^\sharp))^0/F^\times \\  \tilde{\cG}_B &= \G (\U_B(V^\sharp) \times \U_B (V_0^\sharp))^0.
\end{align*}
We regard these as algebraic groups over $\Q$ by restriction of scalars.
We then have the following diagram, which we also write out in gory
detail below. (Here the dual notation in the right most column
indicates also the notation
used (locally) in \S \ref{sec:construction-nonvanishing}.)
\[
\xymatrix{
 \tilde{\mathrsfs{G} } \ar[r] &  \mathrsfs{G}  \ar[r]^{\simeq} & \mathrsfs{G}_B &\tilde{\mathrsfs{G}}_B =\tilde{G} \ar[l] \\
   \tilde{\cG} \ar[r] \ar[u]^{i} \ar[dr]_{\pr} & \cG \ar[u]^{i} \ar[d]^{\pr} \ar[r]^{\simeq} & \cG_B  \ar[u]^{i} \ar[dl]^{\pr} & \tilde{\cG}_B = \GG\ar[l] \ar[u]^{i} \ar[d]^{p} \ar[dll]^{\pr}\\
& \cG_0 & & (B_1^\times \times B_2^\times)/F^\times=G  \ar[ll]^{q} \\
& & G_{B_1} \times G_{B_2} \ar[lu] \ar[ru]
}
\]

\newpage

\[
\begin{turn}{90}
\xymatrix@R+1pc@C-1pc{
 \GU_E(\VV) \ar[r] & \PGU_E(\VV) = \GU_E(\VV)/E^\times \ar[r]^{\simeq} & \PGU_B (\tilde{V})^0 =\GU_B(\tilde{V})^0/F^\times& \ar[l] \GU_B(\tilde{V})^0 \\
 \G (\U_E (\VV_1) \times \U_E (\VV_2)) \ar[u]^{i} \ar[r] &  \G (\U_E (\VV_1) \times \U_E (\VV_2))/E^\times \ar[u]^{i} \ar[r]^{\simeq} &  \G (\U_B (V^\sharp) \times \U_E (V^\sharp_0))^0/F^\times \ar[u]^{i} & \G (\U_B (V^\sharp) \times \U_E (V^\sharp_0))^0 \ar[u]^{i} \ar[l] \\
 \G((B_1^\times \times E^\times)/F^\times \times (B_2^\times \times E^\times)/F^\times) \ar[u]^*[@]{\simeq} \ar[r] \ar[rdd]^{\pr} &  \G((B_1^\times \times E^\times)/F^\times \times (B_2^\times \times E^\times)/F^\times)/E^\times \ar[u]^*[@]{\simeq} \ar[r]^{ \qquad \enspace \simeq}_{\qquad \enspace \xi} \ar[dd]^{\pr} &  \G ((B_1^\times \times B_2^\times)/F^\times \times E^\times)/F^\times \ar[u]^*[@]{\simeq} \ar[ldd]^{\pr} &  \G ((B_1^\times \times B_2^\times)/F^\times \times E^\times) \ar[u]^*[@]{\simeq} \ar[l] \ar[lldd]^{\pr} \ar[dd]_{p} \\
 \\
 & \PB_1^\times \times \PB_2^\times & & (B_1^\times \times B_2^\times)/F^\times \ar[ll]^{q} \\
 & & B_1^\times \times B_2^\times \ar[lu] \ar[ru]
}
\end{turn}
\]
\newpage
\noindent
Here the maps $\pr$, $p$ and $q$ are the obvious projection maps.
We write down formulas for some of the maps as well:
\begin{align*}
F^\times \subset (B_i^\times \times E^\times), & \quad t\mapsto (t,t),  \\
E^\times \subset  \G((B_1^\times \times E^\times)/F^\times \times (B_2^\times \times E^\times)/F^\times), & \quad \alpha \mapsto ([1,\alpha], [1,\alpha] ),\\
F^\times \subset (B_1^\times \times B_2^\times), & \quad t\mapsto (t,t^{-1}),  \\
F^\times \subset \G ((B_1^\times \times B_2^\times)/F^\times \times E^\times), & \quad t\mapsto ([t,1],t)=([1,t],t),  \\
\xi \left( [ [b_1,\alpha_1], [b_2,\alpha_2]] \right)&=[ [b_1, b_2], \nu(b_1) \alpha_1^{-1} \alpha_2)],
\end{align*}
where the map $\xi$ is given in the diagram and $\nu$ denotes the reduced norm.

\subsection{Shimura data}
All the groups in the diagram have associated Shimura varieties, defined such that the maps in the diagram induce morphisms of Shimura data. 
It suffices to describe the Shimura datum for $\tilde{\cG}$ and $\tilde{\cG}_B$, since the Shimura data for all the other groups are defined by composing with the maps above . For $\tilde{\cG}$, this is given by
\[
h_v(z) = ([\iota(z),1], [\iota(z),1])
\]
at the infinite places $v \in \Sigma$ and 
\[
h_v(z) = ([1,1], [1,1])
\]
at the other infinite places. For $\tilde{\cG}_B$, this is given by
\[
h_v(z) = ([\iota(z),\iota(z)], z\bar{z})
\]
at the infinite places $v \in \Sigma$ and 
\[
h_v(z) = ([1,1], 1)
\]
at the other infinite places. 
In \S \ref{ss:shimura-data-for-all-groups}, we will write out the Shimura data more explicitly for some of the other groups in the diagram. 

\subsection{Components}
\label{sec:components}
Later (in \S \ref{sec:proof-of-main-thm})  we will need to use the structure of the components of $\Sh_{\tilde{\cG}_B}$. Let $G_1:=\tilde{\cG}_B$ and $G_2:=(B_1^\times \times B_2^\times)/F^\times$.
We may consider the canonical sequences
\[
1 \rightarrow G_i^{\mathrm{der}} \rightarrow G_i \rightarrow T_i \rightarrow 1
\]
where $G_i^{\mathrm{der}}$ denotes the derived group and $T_i$ the maximal commutative quotient of $G_i$. The map $p$ induces a 
map of exact sequences as below.

\[
\xymatrix{
1 \ar[r] & (B_1^{(1)} \times B_2^{(1)})/{\{\pm 1\}} \ar[r] \ar[d] & \G\left((B_1^\times \times B_2^\times)/F^\times \times E^\times\right) \ar[r]^{([\nu,\nu],\id)} \ar[d]^{p} & \G \left((F^\times \times F^\times)/F^\times \times E^\times\right) \ar[r] \ar[d] & 1 \\ 
1 \ar[r] & (B_1^{(1)} \times B_2^{(1)})/{\{\pm 1\}} \ar[r] & (B_1^\times \times B_2^\times)/F^\times \ar[r]^{[\nu,\nu]} & (F^\times \times F^\times)/F^\times \ar[r] &  1 
}
\]
The set of components of $\Sh_{\tilde{\cG}_B}$ is in bijection with the Shimura variety attached to $(T_1,h_1)$ where
\[
T_1= \G \left((F^\times \times F^\times)/F^\times \times E^\times\right) , \quad h_1 (z) =([z\bar{z},z\bar{z}], z\bar{z}).
\] 
Now 
\[
Z(G_1) \simeq \{ (t,\alpha) \in F^\times \times E^\times: \ t^2 = \N (\alpha)\},
\]
the inverse of this isomorphism being given by $(t,\alpha) \mapsto ([t,1],\alpha) = ([1,t],\alpha) $. 
The natural map $Z(G_1) \rightarrow T_1$ is given by 
\[
(t,\alpha) \mapsto ([t^2,1], \alpha) = ([t,t],\alpha)= ([1,t^2],\alpha)
\]
and induces an isomorphism
\[
Z(G_1)/\langle (-1,1) \rangle \simeq T_1.
\]
Note that any finite order character $\eta$ of $T_1(\Q)\backslash T_1(\A)$ gives rise to a class in $H^0 (\Sh_{\tilde{\cG}_B}, \Q(\eta))$,
where $\Q(\eta)$ is the field generated by the values of $\eta$. We will denote this class $c_\eta$. Of particular interest to us are the characters obtained as follows: we fix a finite order character $\eta$ of $E^{(1)} \backslash \A^{(1)}_E$  and define a character of $T_1$ by
\[
\eta ([t_1,t_2],\alpha) = \eta ((t_1 t_2)^{-1} \alpha).
\]
The pull back of this character to $Z(G_1)$  is given by
\[
\eta (t,\alpha) = \eta (t^{-1} \alpha). 
\]

\subsection{Automorphic forms and cohomology of local systems}
\label{ss:aut-coh-ls}
Recall that we have the following relation between unitary and quaternionic unitary groups given by the top line of the diagram above:
\[
\tilde{\mathrsfs{G}} = \GU_E(\VV) \rightarrow \GU_E(\VV)/E^\times =\mathrsfs{G} \simeq \mathrsfs{G}_B = \GU_B(\tilde{V})^0/F^\times \leftarrow \GU_B(\tilde{V})^0 =\tilde{\mathrsfs{G}}_B.
\]
Since $E^\times$ and $F^\times$ are cohomologically trivial, the maps $\tilde{\mathrsfs{G}}_B (\A) \rightarrow \mathrsfs{G}_B (\A)$ and $\tilde{\mathrsfs{G}}  (\A) \rightarrow \mathrsfs{G}(\A)$ are surjective. Hence the isomorphism in the middle induces a natural bijection between 
automorphic representations of $\tilde{\mathrsfs{G}}_B$ with trivial central character and those of $\tilde{\mathrsfs{G}}$ with trivial central character. 
Thus if $\Pi$ is an automorphic representation of any of the groups at the ends with trivial central character, it may be viewed as 
an automorphic representation of any of the other groups above; we denote all such representations by the same symbol $\Pi$. 
Moreover, if $(\rho,\V_\rho)$ is a finite dimensional representation (again of one of the groups at the end, but trivial on the center) defined over a field $L$ containing the reflex field $F_\Sigma$ and $\Pi$ is defined over $L$, then 
we get a local system also denoted by $\V_\rho$ on each of the associated Shimura varieties and 
there are natural isomorphisms
\[
H^i (\Sh_{\tilde{\mathrsfs{G}}}, \V_\rho)_\Pi \simeq H^i (\Sh_{\mathrsfs{G}}, \V_\rho)_\Pi \simeq H^i (\Sh_{\mathrsfs{G}_B}, \V_\rho)_\Pi 
\simeq H^i (\Sh_{\tilde{\mathrsfs{G}}_B}, \V_\rho)_\Pi
\]
in the category $\cM_{F_\Sigma}^L$. Note that in general, the field of definition of $\Pi$ contains the field of rationality, but it is not clear if these are equal. See \cite{buzzard-gee} and \cite{shin-templier} for a discussion of these issues, which are not so important for us, since we will need instead a version of the above isomorphism for eigenvectors of the unramified Hecke algebra at finite level. 

Let $\cK$ be an open compact subgroup of $G(\A_f)$ for $G$ each of the end groups such that the images of the two $\cK$ under the quotient map are identified by the isomorphism in the middle; we write $\cK$ for this image, which is an open compact subgroup of $(G/Z)(\A_f) = G(\A_f)/Z(\A_f)$. (There will be no confusion since we will always identify what the ambient group is.)
Let $S$ be a finite set of rational primes such that for all the groups $G = \tilde{\mathrsfs{G}}, \mathrsfs{G}, \mathrsfs{G}_B, \tilde{\mathrsfs{G}}_B$ above and for all $p \notin S$:
\begin{itemize}
 \item $G_p$ is unramified over $\Q_p$;
 \item $\cK_p$ is a hyperspecial maximal compact subgroup of $G_p$;
 \item $\Pi_p$ has a non-zero $\cK_p$-fixed vector.
\end{itemize}
Let $\mathrsfs{H}_G^S = \mathrsfs{H}(G(\A^S), \cK^S)$ be the Hecke algebra of compactly supported $\cK^S$-bi-invariant functions on $G(\A^S)$, where $\A^S = {\prod'_{p \notin S}} \Q_p$ and $\cK^S = \prod_{p \notin S} \cK_p$.
Then $\mathrsfs{H}_G^S$ acts on $H^i(\Sh_{G,\cK}, \V_\rho)$.
Put $\Pi^S = \bigotimes'_{p \notin S} \Pi_p$ and
\[
 H^i(\Sh_{G,\cK}, \V_\rho)[\Pi^S] = \{ x \in H^i(\Sh_{G,\cK}, \V_\rho) \, | \, Tx = \chi(T) x \text{ for all } T \in \mathrsfs{H}_G^S \},
\]
where $\chi$ is the character of $\mathrsfs{H}_G^S$ associated to $\Pi^S$. Let $L$ be a number field such that
$(\rho,\V_\rho)$ is defined over $L$ and such that $L$ contains the values of $\chi$. Then there are canonical isomorphisms
\[
H^i (\Sh_{\tilde{\mathrsfs{G}},\cK}, \V_\rho)[\Pi^S] \simeq H^i (\Sh_{\mathrsfs{G},\cK}, \V_\rho)[\Pi^S] \simeq H^i (\Sh_{\mathrsfs{G}_B,\cK}, \V_\rho)[\Pi^S] 
\simeq H^i (\Sh_{\tilde{\mathrsfs{G}}_B,\cK}, \V_\rho)[\Pi^S]
\]
in the category $\cM_{F_\Sigma}^L$.

\section{The global exceptional isomorphism}
\label{sec:global-ex-isom}

In this section, we construct the global exceptional 
isomorphism between a (projectivized) unitary group attached to 
a hermitian (or skew-hermitian) space $\VV$ and the identity component of 
a (projectivized) quaternionic unitary group attached to a 
quaternionic {\it skew-hermitian} space $\tilde{V}$. Moreover, 
we study the restriction of this isomorphism to certain natural 
subgroups corresponding to the decomposition of $\VV$ into the direct sum of 
two subspaces.

\subsection{Hermitian spaces and unitary groups}

Let $E/F$ be a quadratic extension of number fields and $\rho$ the non-trivial Galois automorphism of $E/F$.
Write $E = F + F\i$ for some trace zero element $\i \in E^\times$.
Let $\VV$ be an $E$-hermitian space. 
Thus $\VV$ is equipped with a non-degenerate form
\[
(\cdot,\cdot)_\VV: \VV \times \VV \rightarrow E
\]
satisfying
\[
(v\alpha ,w\beta )_\VV = \alpha^\rho (v,w)_\VV \beta, \quad (v,w)_\VV=(w,v)_\VV^\rho. 
\]
We denote by $\GU_E(\VV)$ the unitary similitude group of $\VV$:
\[
 \GU_E(\VV) = \{ g \in \GL_E(\VV) \, | \, (gv,gw)_\VV = \nu(g) \cdot (v,w)_\VV \text{ for all } v, w \in \VV \},
\]
where $\nu : \GU_E(\VV) \rightarrow F^\times$ is the similitude character.

\begin{prop}
\label{prop:norm-sim}
Let $\VV$ be a hermitian space over $E$ of dimension $n$ and let $g\in \GU_E(\VV)$. Then
\[
\N (\det(g)) = \nu(g)^n,
\]
where $\N$ denotes the norm map $E^\times \rightarrow F^\times$. 
\end{prop}

\begin{proof}
This is obviously well known but the proof will serve to establish some notation. 
Let $\VV^*$ be the $E$-linear dual of $\VV$.
First, the form $(\cdot,\cdot)_\VV$ induces an $E$-conjugate linear isomorphism
\[
\varphi: \VV \simeq \VV^*,
\]
given by
\[
\varphi(x) (y) = ( x,y )_\VV.
\]
Then for a positive integer $r$,
\begin{equation}
\label{eqn:wedge-d-phi}
{\wedge^r} \varphi: \wedge^r \VV \rightarrow \wedge^r \VV^*
\end{equation}
is also $E$-conjugate linear. Let $\iota$ be the $E$-linear isomorphism
\begin{equation}
\label{eqn:iota-defn}
\iota:  \wedge^r (\VV^*) \simeq (\wedge^r \VV)^*,
\end{equation}
 induced by the multilinear map
\begin{equation}
\label{eqn:multilinear-pairing}
(\VV^*)^r \times \VV^r \rightarrow E, \quad (\lambda_1, \ldots, \lambda_r, \v_1, \ldots, \v_r) \mapsto \det ( \lambda_i (\v_j)).
\end{equation}
Now any $g\in \GL_E(\VV)$ acts on $\VV^*$ via $g\lambda (\v) = \lambda (g^{-1}\v)$ and $\iota$ is equivariant for this action since 
\eqref{eqn:multilinear-pairing} is equivariant for the diagonal action of $\GL_E (\VV)$. The composite 
\[
\iota \circ \wedge^r \varphi: \wedge^r \VV \rightarrow (\wedge^r \VV)^*
\]
is an $E$-conjugate linear isomorphism and may be viewed as giving a hermitian form on $\wedge^r \VV$, 
denoted by $(\cdot, \cdot)_{\wedge^r \VV}$. (That this form is conjugate symmetric follows for instance by computing it in matrix form in terms of the matrix 
of the form on $\VV$ with respect to an orthogonal basis. If the matrix of the original form is the diagonal matrix with entries $a_1, \ldots, a_n$, then  the entries $a_i$ lie in $F$ and the form on 
$\wedge^r \VV$ is represented by the diagonal matrix whose entries are products of the form $a_{i_1} \cdots a_{i_r}$ with $1\le i_1 < \cdots < i_r \le n$.)

Now suppose $g\in \GU_E (\VV)$. Then 
\[
\varphi (g\v) (\w) = (g\v,\w)_\VV =\nu(g) \cdot (\v, g^{-1} \w)_\VV = \nu(g) \cdot (g \varphi)(\v) (\w),
\]
so that 
\[
\varphi \circ g = \nu(g) \cdot g \circ \varphi
\]
and 
\[
\wedge^r \varphi \circ g = \nu(g)^r \cdot g \circ \wedge^r \varphi.
\]
Thus for $\x,\y \in \wedge^r \VV$, we have 
\[
(g \x, g \y)_{\wedge^r \VV} = \nu(g)^r (\x,\y)_{\wedge^r \VV}. 
\]
Now take $r=n$. Then $g$ acts on $\wedge^n \VV$ as the scalar $\det(g)$, so that $(g\x, g\y)_{\wedge^n \VV} = \N (\det(g)) (\x,y)_{\wedge^n \VV}$, from 
which it follows that $\N (\det(g)) = \nu(g)^n$. 
\end{proof}

\subsection{Construction of the (global) exceptional isomorphism} 
\label{sec:constr-gl-ex-isom}

Let $\VV$ be a four dimensional (right) $E$-hermitian space.
Such a $\VV$ is classified by a collection of its determinant $\delta_v \in F_v^\times/\N E_v^\times$ for all places $v$, which equals its discriminant (see \cite[\S 2.1.1]{periods1} for our convention) since $\dim \VV =4$, 
together with its signature at ramified archimedean places.
(For a split place, $\delta_v$ is always trivial.
For a ramified archimedean place, $\delta_v$ is trivial if the signature is either $(4,0)$, $(2,2)$ or $(0,4)$ and is non-trivial if the signature is $(3,1)$ or $(1,3)$.)
Let $B$ be the unique quaternion algebra over $F$ which is ramified exactly at those places 
$v$ of $F$ at which $\delta_v$ is non-trivial.
Let $*$ be the main involution on $B$.
Then we will construct 
\begin{itemize}
\item a three dimensional right $B$-space $\tilde{V}$,
\item a skew-hermitian $B$-form $\langle \cdot, \cdot \rangle$ on $\tilde{V}$, i.e., a non-degenerate sesquilinear form $\langle \cdot, \cdot \rangle: \tilde{V} \times \tilde{V} \rightarrow B$ satisfying
\[
 \langle v\alpha ,w\beta \rangle = \alpha^* \langle v,w \rangle \beta, \quad \langle v,w \rangle = -\langle w,v \rangle^*,
\]
\end{itemize}
such that there is a natural isogeny
\[
\GSU_E (\VV) \rightarrow \GU_B(\tilde{V})^0,
\]
as well as a natural isomorphism
\[
\PGU_E( \VV) \simeq \PGU_B (\tilde{V})^0.
\]
Here we denote by $\GU_B(\tilde{V})^0$ the identity component of the unitary similitude group of $\tilde{V}$:
\[
 \GU_B(\tilde{V}) = \{ g \in \GL_B(\tilde{V}) \, | \, \langle gv,gw \rangle = \nu(g) \cdot \langle v,w \rangle \text{ for all } v, w \in \tilde{V} \},
\]
where $\nu : \GU_B(\tilde{V}) \rightarrow F^\times$ is the similitude character, and put
\begin{align*}
 \GSU_E (\VV) & = \{ g\in \GU_E(\VV) \, | \, \det (g)=\nu(g)^2 \}, \\
 \PGU_E(\VV) & = \GU_E(\VV) / E^\times, \\
 \PGU_B(\tilde{V})^0 & = \GU_B(\tilde{V})^0 / F^\times.
\end{align*}

Let $\tilde{\VV} = \wedge^2 \VV$.  
This is a right $E$-space and we will extend the $E$-action to 
a right $B$-action. To do so, we must construct an element 
$L\in \End_F (\tilde{\VV})$ which is conjugate linear for the $E$-action:
\[
L(x \alpha) = (Lx) \alpha^\rho
\]
for $x\in \tilde{\VV}$, $\alpha \in E$. 

The map $L$ will be a composite of three maps:
\begin{enumerate}
\item   The map 
\[
\wedge^2 \varphi: \wedge^2 \VV \rightarrow \wedge^2 (\VV^*)
\]
obtained by specializing \eqref{eqn:wedge-d-phi} to $r=2$, which is an $E$-conjugate linear isomorphism. 
\item The map
\[
\iota: \wedge^2 (\VV^*) \simeq (\wedge^2 \VV)^*
\]
obtained by specializing \eqref{eqn:iota-defn} to $r=2$, which is an $E$-linear isomorphism. 
\item Here we use that $\dim \VV =4$. 
Fix an isomorphism
\[
d: \wedge^4 \VV \simeq E.
\]
This is well defined up to scaling. 
The natural map
\begin{equation}
\label{eqn:pairing-on-wedge-2}
\wedge^2 \VV \times \wedge^2 \VV \rightarrow \wedge^4 \VV \simeq E
\end{equation}
is {\it symmetric} and 
induces an $E$-linear isomorphism 
\[
\psi: \wedge^2 \VV \simeq (\wedge^2 \VV)^*.
\]

\end{enumerate}

Let 
\[
L= \psi^{-1} \circ \iota \circ \wedge^2 \varphi. 
\]
Clearly $L$ depends on the choice of $d$. If $d$ is scaled by $\alpha$, then $\psi$ is scaled by $\alpha$ as well and $L$ is scaled by $\alpha^{-1} $.
However $L^2$ changes to 
\[
(\alpha^{-1} L)(\alpha^{-1} L) =  (\alpha^\rho \alpha)^{-1} L^2 = \N (\alpha)^{-1} L^2.
\]
Thus $L^2$ is well defined up to norms from $E^\times$ to $F^\times$. 
In fact, $L^2$ turns out to be a scalar operator. To identify this scalar, we recall 
the following invariant attached to a hermitian space $\VV$ of dimension $n$ and an 
isomorphism $d: \wedge^n \VV \simeq E$. 

\begin{defn}
Let $\VV$ be a hermitian space of dimension $n$ with form $H$ and let 
$d: \wedge^n \VV \simeq E$ be an isomorphism. The form $H$ induces
a map
\[
\VV^n \times \VV^n \rightarrow E, \quad (\v_1, \ldots, \v_n, \w_1, \ldots, \w_n) \mapsto \det [ H(\v_i, \w_j ) ],
\]
which factors through $\wedge^n \VV \times \wedge^n \VV$
and gives a hermitian form
\[
h: \wedge^n \VV \times \wedge^n \VV \rightarrow E.
\]
Let $v \in \wedge^n \VV$ be such that $d(v)=1$. Then define
\[
\vol (H,d) = h (v,v).
\]
Note that $h$ is a hermitian form, so $\vol(H,d)$ lies in $F^\times$ and its class in $F^\times/\N E^\times$ 
equals the class of the determinant of $H$. 
\end{defn}

\begin{prop}
\label{prop:volume-V-d}
The map $L^2$ is multiplication by $\vol ((\cdot,\cdot)_\VV,d)$.  
\end{prop}

\begin{proof}
We will pick a suitable basis and compute. 
Since $L^2$ and $\vol ((\cdot,\cdot)_\VV,d)$ scale in exactly the same way 
as a function of $d$, we can choose any convenient $d$ as well. 
Let $\v_1, \ldots, \v_4$ be 
a basis of $\VV$ with respect to which the form $(\cdot, \cdot)_\VV$ is diagonal with entries 
$a_1, \ldots, a_4 \in F$. Let $\e_1,\ldots, \e_4$ be the dual basis of $\VV^*$. 
Then 
\[
\varphi (\v_i ) = a_i \e_i
\]
and 
\[
\wedge^2 (\varphi) (\v_i \wedge \v_j) = a_i a_j \e_i \wedge \e_j.
\]
For $1\le i < j \le 4$, let $\v_{ij}$ denote the element $\v_i \wedge \v_j \in \wedge^2 \VV$. 
This collection gives a basis of $\wedge^2 \VV$. We let 
$\{ \e_{ij} \} \subset (\wedge^2 \VV)^*$ be the dual basis. 
Then 
\[
\iota (\e_i \wedge \e_j) = \e_{ij}.
\]
For any pair $(i,j)$ as above let $(i',j')$ be the unique pair of elements such that $\{i,j,i',j'\} =\{ 1,2,3,4\}$ and such that 
$i' < j'$. Define $\sign(i,j) =\pm 1$ by 
\[
\v_{ij} \wedge \v_{i'j'}=\sign(i,j) \v_1 \wedge \v_2 \wedge \v_3 \wedge \v_4.
\]
Now {\it choose} $d$ such that 
\begin{equation}
\label{eqn:choice-of-d}
d (\v_1 \wedge \v_2 \wedge \v_3 \wedge \v_4) = -1.
\end{equation}
(This choice may seem surprising but it is made so as to agree with some conventions in \cite{periods1}.) 
Then 
\[
\psi^{-1} (\e_{ij}) =- \sign (i,j) \cdot \v_{i'j'}.
\]
(The $-$ sign here occurs because of the choice made in \eqref{eqn:choice-of-d}.)
Now we can write down $L$ explicitly in the basis $\v_{ij}$. It is given by:
\begin{align*}
\v_{12} & \mapsto -a_1 a_2 \v_{34}, \\
\v_{13} & \mapsto a_1 a_3 \v_{24}, \\
\v_{14} & \mapsto -a_1 a_4 \v_{23}, \\
\v_{23} & \mapsto -a_2 a_3 \v_{14}, \\
\v_{24} & \mapsto a_2 a_4 \v_{13}, \\
\v_{34} & \mapsto -a_3 a_4 \v_{12}.
\end{align*}
The proposition follows from this explicit description. 
\end{proof}

Now let us define a quaternion algebra $B$ as follows. Let 
\[
J:=\vol( (\cdot,\cdot)_\VV,d)
\] 
and define $B$ by
\[
B := E + E\j, \quad \j^2 = J, \quad \alpha \j = \j \alpha^\rho
\]
for all $\alpha \in E$. Then we can define a right action of $B$ on $\tilde{\VV}$ by 
\[
x \cdot \j = L(x).
\]
We will denote this space by $\tilde{V}$ when we want to regard it as a $B$-space rather than an $E$-space.

As in the proof of Proposition \ref{prop:norm-sim}, the composite map $\iota \circ \wedge^2 \varphi$ is a conjugate linear isomorphism
\[
\wedge^2 \VV \simeq (\wedge^2 \VV)^*
\]
that gives rise to a hermitian form $(\cdot, \cdot)_{\tilde{\VV}}$ on $\tilde{\VV} =\wedge^2 \VV$. 
 Multiplying this form by the trace zero element $\i$ gives 
a {\it skew-hermitian} form on $\tilde{\VV}$, which we denote simply by $(\cdot, \cdot)$.

\begin{lem}
The form $(\cdot, \cdot)$ on $\tilde{\VV}$ satisfies: for all $x,y \in \tilde{\VV}$,
\begin{enumerate}
\item $(x\j , y) = (y\j, x)$.
\item $(x\j,y\j)^\rho = - J (x,y) $.
\end{enumerate}
\end{lem}

\begin{proof} Firstly,
\begin{align*}
(x\j,y) &= \i \cdot [\iota \circ \wedge^2 (\varphi) (x\j) ](y) \\ 
&= \i \cdot [\psi \circ L \circ L(x)](y) \\
&= J \i \cdot [\psi (x) ](y).
\end{align*}
Since \eqref{eqn:pairing-on-wedge-2} is symmetric, we have 
\[
[\psi (x)] (y) = [\psi (y)](x),
\]
from which it follows that $(x\j,y) =(y\j,x)$. 

Secondly, we have
\begin{align*}
(x\j, y\j) &=  J \i \cdot [\psi (x) ](y\j) \\
&= J \i \cdot [\psi(y\j)](x) \\
&= J \i \cdot[\psi \circ \psi^{-1} \circ \iota \circ \wedge^2 (\varphi) (y)](x) \\
&= J \cdot (y,x),
\end{align*}
so that 
\[
(x\j,y\j)^\rho = J\cdot (y,x)^\rho = -J (x,y).
\]

\end{proof}

\begin{rem}
In fact (ii) above follows from (i). Indeed, assuming (i), we have 
\begin{align*}
(x\j, y\j) &= (y\j \cdot \j , x) = J(y,x) = -J(x,y)^\rho.
\end{align*}

\end{rem}

Now we can define a $B$-skew-hermitian form on $\tilde{V}$ by
\begin{equation}
\label{eqn:B-form-defn}
\langle x,y \rangle = (x,y) - \frac{1}{J} \cdot \j \cdot (x\j,y).
\end{equation}

\begin{prop}
\label{p:isom-PSU}
The map
\[
\GL_E (\VV) \rightarrow \GL_E(\tilde{\VV}), \quad g \mapsto \wedge^2 g
\]
induces an
isogeny 
\[
 \tilde{\xi} : \GSU_E (\VV) \longrightarrow \GU_B (\tilde{V})^0
\]
with kernel $\{ \pm 1 \}$.
\end{prop}

\begin{proof}
Let $g \in \GU_E (\VV)$.  
We first compute the commutator of $L$ and $\wedge^2 g$.  Recall that 
\[
L= \psi^{-1} \circ \iota \circ \wedge^2 \varphi.
\]
Now for any $g\in \GL_E (\VV)$, we have 
\begin{equation}
\label{eqn:iota-wedge2g}
\iota \circ \wedge^2 g = \wedge^2 g \circ \iota
\end{equation}
and 
\begin{equation}
\label{eqn:psi-wedge2g}
\psi \circ \wedge^2 g = \det(g) \wedge^2 g \circ \psi.
\end{equation}
If further $g \in \GU_E(\VV)$, then 
\[
(gx,gy)_\VV =\nu(g) (x,y)_\VV
\]
which implies that $(gx,y)_\VV = \nu(g) (x, g^{-1} y)_\VV$ and 
\[
\varphi \circ g = \nu(g) g \circ \varphi.
\]
Thus 
\begin{equation}
\label{eqn:wedge2phi-wedge2g}
\wedge^2 \varphi \circ \wedge^2 g = \nu(g)^2 \wedge^2 g \circ \wedge^2 \varphi.
\end{equation}
It follows from \eqref{eqn:iota-wedge2g}, \eqref{eqn:psi-wedge2g} and \eqref{eqn:wedge2phi-wedge2g} that 
\begin{equation}
\label{eqn:L-wedge2g}
L\circ \wedge^2 g=\nu(g)^2 \det(g)^{-1} 
 \wedge^2 g \circ L  
 \end{equation}
for $g\in \GU_E(\VV)$.  
Thus for $g\in \GSU_E(\VV)$, the endomorphism $\wedge^2 g$ lies in $\GU_B (\tilde{V})$. 
This gives a map
\[
 \GSU_E(\VV) \rightarrow \GU_B( \tilde{V})
 \]
whose kernel is easily checked to be $\{ \pm 1\}$.
On the other hand, we have
\begin{align*}
 \dim \GSU_E(\VV) & = \dim \GU_E(\VV) - 1 = \dim \U_E(\VV) = 4^2 = 16, \\
 \dim \GU_B(\tilde{V}) & = \dim \U_B(\tilde{V}) + 1 = 3 (2 \cdot 3-1) + 1 = 16.
\end{align*}
Since $\GSU_E(\VV)$ is connected, this shows that the image of the above map is $\GU_B(\tilde{V})^0$.
\end{proof}

\begin{prop}
\label{p:isom-PGU}
There is a natural isomorphism 
\[
\xi: \PGU_E(\VV) \overset{\simeq}{\longrightarrow} \PGU_B(\tilde{V})^0,
\]
where
\[
 \PGU_E(\VV) = \GU_E(\VV) / E^\times, \qquad
 \PGU_B(\tilde{V})^0 = \GU_B(\tilde{V})^0 / F^\times.
\]
\end{prop}

\begin{proof}
Let $g \in \GU_E(\VV)$.
Put $f = \wedge^2 g$ and $\alpha = \nu(g)^2 / \det g$.
By \eqref{eqn:L-wedge2g}, we have 
\[
 Lf = \alpha fL.
\]
By Proposition \ref{prop:norm-sim}, we have $\N(\det g) = \nu(g)^4$, so $\N (\alpha)=1$ and we can choose $\beta \in E^\times$ (unique up to multiplication by $F^\times$) such that $\alpha = \beta/\beta^\rho$.
Then 
\[
 L \beta f = \beta^\rho L f = \beta^\rho \alpha fL = \beta fL, 
\]
so that $\beta f \in \GU_B(\tilde{V})$.
The assignment $g \mapsto \beta f$ gives a homomorphism
\[
 \GU_E(\VV) \longrightarrow \GU_B(\tilde{V})/ F^\times.
\]
It is easy to see that its kernel is the center of $\GU_E(\VV)$.
Indeed, if $\wedge^2 g$ is a scalar multiplication on $\tilde{\VV}$, then since $\dim \VV > 2$, $g$ has to be semisimple and hence is a scalar multiplication on $\VV$.
On the other hand, as in the proof of Proposition \ref{p:isom-PSU},
we have $\dim \PGU_E(\VV) = \dim \PGU_B(\tilde{V})$.
Since $\GU_E(\VV)$ is connected, this shows that the image of the above map is the identity component of $\GU_B(\tilde{V})/F^\times$.
\end{proof}

\begin{rem}
The reader may note that the notation is mildly confusing here. Namely, $\beta f$ is the map given by $\beta f (x) := f(x) \beta$, since the action of $E$ is on the right.  \end{rem}

\begin{rem}
For $g\in \GSU_E(\VV)$, we have $\alpha=1$, so we may take $\beta=1$ as well. This implies that the maps 
constructed in the previous two propositions fit into the commutative diagram below, where the vertical maps 
are the natural homomorphisms:
\[
\xymatrix
{
\GSU_E(\VV) \ar[d] \ar[r]^{\tilde{\xi}} & \GU_B(\tilde{V})^0 \ar[dd]\\
\GU_E (\VV)  \ar[d]& \\
\PGU_E(\VV) \ar[r]^{\xi} & \PGU_B (\tilde{V})^0.
}
\]
\end{rem}

\subsection{Subgroups}

In this section we discuss the effect of the isogeny/isomorphism of the previous section on certain natural 
subgroups of the unitary group obtained from a decomposition of the hermitian space into a sum of two hermitian spaces.

\subsubsection{The sum of two $2$-dimensional spaces}
\label{ss:sum-of-2-dim}

We first discuss the case when $\VV$ is an orthogonal direct sum of the form $\VV=\VV_1\oplus \VV_2$ with $\dim_E \VV_1 = \dim_E \VV_2 = 2$.
Then
\begin{equation}
\label{eqn:decomposition-of-VV}
\tilde{\VV} = \wedge^2 (\VV_1 \oplus \VV_2) = \wedge^2 (\VV_1) \oplus \wedge^2 (\VV_2) \oplus (\VV_1 \otimes \VV_2),
\end{equation}
where we identify $\VV_1 \otimes \VV_2$ with its image under the natural map to $\wedge^2 \VV$ (which sends 
$v\otimes w$ to $v\wedge w$).

We may assume that the basis $(\v_1, \v_2, \v_3, \v_4)$ of 
$\VV$ is chosen such that $(\v_1,\v_2)$ forms a basis of $\VV_1$ 
and $(\v_3, \v_4)$ a basis of $\VV_2$. From the explicit formulas for $L$, it is clear that $L$ preserves the subspaces
\[
\VV_0^\sharp= \wedge^2 (\VV_1) \oplus \wedge^2 (\VV_2) \quad \text{ and} \quad \VV^\sharp= \VV_1 \otimes \VV_2,
\]
so these are $B$-spaces that we denote by $V_0^\sharp$ and $V^\sharp$ respectively. 
Since the collection $(\v_{ij}$, $1\le i < j \le 4)$ forms an orthogonal basis for the form $(\cdot,\cdot)$, the 
decomposition $\tilde{\VV} = \VV^\sharp \oplus \VV_0^\sharp$ is one of skew-hermitian $E$-spaces.  Moreover, the formula \eqref{eqn:B-form-defn} shows that 
the decomposition $\tilde{V} = V^\sharp \oplus V_0^\sharp$ is one of skew-hermitian $B$-spaces. 

\begin{prop}
Let $H$ be the subgroup of $\GSU_E (\VV)$ given by 
\[
H= \GSU_E (\VV) \cap \G (\U_E (\VV_1) \times \U_E (\VV_2)). 
\]
Then $\tilde{\xi}$ restricts to an isogeny 
\begin{equation}
\label{eqn:xit-rest}
H \rightarrow \G(\U_B(V^\sharp) \times \U_B (V_0^\sharp))^0,
\end{equation}
with kernel $\{ \pm 1 \}$. 
\end{prop}

\begin{proof}
Let $g_1 \in \GU_E(\VV_1)$, $g_2 \in \GU_E( \VV_2)$ be such that 
$g=(g_1,g_2) \in H$. Then $\wedge^2 g$ acts as right multiplication by $\det(g_i)$ on 
$\wedge^2 \VV_i$ and by $g_1 \otimes g_2$ on $\VV_1 \otimes \VV_2$.
Since $H$ is connected, this shows that $\tilde{\xi}$ maps $H$ into the subgroup $\G(\U_B(V^\sharp) \times \U_B (V_0^\sharp))^0$ of 
$\GU_B (\tilde{V})^0$.
On the other hand, we have
\begin{align*}
 \dim H & = \dim \G (\U_E (\VV_1) \times \U_E (\VV_2)) - 1 \\
 & = \dim \U_E (\VV_1) + \dim \U_E (\VV_2) \\
 & = 2^2 + 2^2 = 8, \\
 \dim \G(\U_B(V^\sharp) \times \U_B (V_0^\sharp))
 & = \dim \U_B(V^\sharp) + \dim \U_B (V_0^\sharp) + 1 \\
 & = 2(2 \cdot 2-1) + 1(2 \cdot 1-1) + 1 = 8.
\end{align*}
Hence the image of $H$ under $\tilde{\xi}$ is $\G(\U_B(V^\sharp) \times \U_B (V_0^\sharp))^0$.
\end{proof}

Likewise, one has an analogous result for subgroups in the context of Proposition \ref{p:isom-PGU}.

\begin{prop}
The map $\xi$ restricts to an isomorphism
\begin{equation}
\label{eqn:rest-of-xi}
\G (\U_E(\VV_1) \times \U_E (\VV_2)) / E^\times \simeq \G (\U_B( V^\sharp) \times \U_B (V_0^\sharp))^0/F^\times.
\end{equation}
\end{prop}

\begin{proof}
Let $g=(g_1,g_2) \in \G (\U_E(\VV_1) \times \U_E (\VV_2))$. As in the previous proposition, the map $\wedge^2 g$ clearly preserves the decomposition 
$\tilde{\VV} = \VV^\sharp \oplus \VV_0^\sharp$, hence so does $\beta \cdot \wedge^2 g$. Since $\beta \cdot \wedge^2 g$ lies in $\GU_B (\tilde{V})$,  
it must in fact lie in $\G (\U_B( V^\sharp ) \times \U_B (V_0^\sharp))$. This gives the map \eqref{eqn:rest-of-xi}, which must be injective since 
$\xi$ is injective. From dimension considerations, it must be an isomorphism.
\end{proof}

It is useful to write down explicitly the maps in \eqref{eqn:xit-rest} and \eqref{eqn:rest-of-xi}. First for 
$i=1,2$, we define a quaternion algebra $B_i$ such that 
$\VV_i$ is naturally a right $B_i$-module and is equipped with a 
$B_i$-hermitian form whose projection to $E$ recovers the hermitian form.
Let us outline for $i=1$, the case $i=2$ being exactly similar. The hermitian form on 
$\VV_1$ gives an $E$-conjugate linear isomorphism
\[
\varphi_1: \VV_1 \simeq \VV_1^*,
\]
where $\VV_1^*$ as usual is the $E$-linear dual of $\VV_1$. Let us fix an isomorphism
\[
d_1: \wedge^2 \VV_1 \simeq E.
\]
For definiteness, we let $d_1 (\v_1 \wedge \v_2) = 1$. This gives a bilinear pairing
\[
\VV_1 \times \VV_1 \rightarrow E, \quad (x,y) \mapsto d_1 (x\wedge y)
\]
and thus gives an $E$-linear isomorphism 
\[
\psi_1: \VV_1 \simeq \VV_1^*, \quad \psi_1 (x) (y) = d_1 (x\wedge y).
\]
Define $L_1 \in \End_F(\VV_1)$ by 
\[
L_1 = \psi_1^{-1} \circ \varphi_1.
\]
Explicitly, we see that $L_1$ acts on $\VV_1$ by:
\begin{align}
\label{eqn:L11} \v_1 & \mapsto -a_1 \v_2, \\
\label{eqn:L12} \v_2 & \mapsto a_2 \v_1,
\end{align}
so that $L_1^2 = -a_1a_2 = -\vol (\VV_1, d_1)$. 
Let 
\[
J_1 = -\vol(\VV_1, d_1)
\]
and define $B_1$ by 
\[
B_1 = E+ E\j_1, \quad \j_1^2 = J_1, \quad \alpha \j_1 = \j_1 \alpha^\rho
\]
for all $\alpha \in E$. Then the right $E$-action on $\VV_1$ extends to a right $B_1$-action 
defined by 
\[
x \cdot \j_1 = L_1 (x).
\]
When we want to think of $\VV_1$ as a $B_1$-space we simply denote it $V_1$. 

\begin{lem}
For $x,y \in \VV_1$, we have
\begin{enumerate}
\item $(x\j_1, y) = - (y\j_1, x)$.
\item $(x\j_1, y\j_1)^\rho = -J_1 (x,y) $. 
\end{enumerate}
\end{lem}

\begin{proof}
We have for $x,y \in \VV_1$, 
\[
(x\j_1, y) = \varphi_1 (x\j_1)  (y) = \psi_1 L_1 (x\j_1) (y) = \psi_1 L_1^2 (x) (y) = J_1 \psi_1 (x) (y) =J_1 d_1 (x\wedge y),
\]
and likewise $(y\j_1, x) = J_1 d_1 (y \wedge x)$. It follows from this that $(x\j_1, y) = - (y\j_1,x)$ which proves (i). Now (ii) 
follows from (i) since
\[
(x\j_1, y\j_1) = -(y\j_1 \cdot \j_1, x) = -J_1(y,x) = -J_1 (x,y)^\rho.
\]
\end{proof}

Using the lemma, we find that
\[
\langle x,y \rangle := (x,y) - \frac{1}{J_1} \j_1 (x\j_1, y)
\]
defines a $B_1$-{\it hermitian} form on $V_1$ such that $\pr \circ \langle \cdot,\cdot \rangle = (\cdot, \cdot)$. Thus there 
is a natural embedding
\[
\GU_{B_1}(V_1) \hookrightarrow \GU_E(\VV_1)
\]
which we can make explicit as follows. Pick a $B_1$-basis $\x_1$ for $V_1$. Then 
$V_1 = \x_1 B_1$ and for $g\in \GU_{B_1}(V_1)$, let $\beta_g$ be defined by
\[
g\x_1 = \x_1 \beta_g.
\]
The assignment $g\mapsto \beta_g$ gives an identification $\GU_{B_1} (V_1) \simeq B_1^\times$. 
Indeed, this map is injective; it is also surjective since for any $\beta\in B_1^\times$ and $\alpha,\alpha' \in B_1$, we have
\[
\langle \x_1 \beta \alpha ,\x_1 \beta \alpha' \rangle = (\beta\alpha)^\rho \langle \x_1,\x_1 \rangle \beta \alpha' = \alpha^\rho \beta^\rho \langle \x_1,\x_1 \rangle   \beta \alpha' = \nu (\beta)  \alpha^\rho \langle \x_1,\x_1 \rangle 
\alpha' = \nu(\beta) \langle \x_1 \alpha, \x_1 \alpha' \rangle.\]
Here we have used that $\langle \x_1,\x_1 \rangle$ lies in $F$, the form $\langle \cdot, \cdot \rangle$ being $B_1$-hermitian. 
Then 
\[
\GU_E (\VV_1) \simeq (B_1^\times \times E^\times)/F^\times
\]
where the $B_1^\times$ corresponds to the $\GU_{B_1} (V_1)$ action described above, the $E^\times$-action is given by $\alpha \mapsto $ (right)-multiplication by $\alpha^{-1}$ for $\alpha \in E^\times$ and the embedding of 
$F^\times$ in $B_1^\times \times E^\times$ is just the diagonal embedding $t\mapsto (t,t)$. All of the above discussion carries over verbatim to the case $i=2$, so that after picking a 
$B_2$-basis $\x_2$ of $V_2$, the
map $\GU_{B_2} (V_2) \hookrightarrow \GU_E (\VV_2)$ is identified with the embedding
\[
B_2^\times \hookrightarrow (B_2^\times \times E^\times)/F^\times.
\]

Next, we explicate the groups on the right of the map \eqref{eqn:xit-rest}. 
First we note that 
\[
\GU_B(V_0^\sharp)^0 \simeq E^\times.
\]
Indeed, let $\x$ be any non-zero vector in $\wedge^2 \VV_1$, so that $V_0^\sharp =\x B$. Then for $\alpha \in E^\times$, the map
\[
\x \beta \mapsto \x \alpha \beta
\]
gives an element of $\GU_B (V_0^\sharp)$. By dimension considerations, this gives an isomorphism $E^\times \simeq \GU_B(V_0^\sharp)^0$.
More precisely, it is easy to see that for any $g\in \GU_B (V_0^\sharp)$, we have $g \cdot \x = \x \gamma_g$ for a unique $\gamma_g \in B^\times$ and the assignment 
$g\mapsto \gamma_g$ gives an isomorphism of $\GU_B( V_0^\sharp)$ with the semi-direct product 
$E^\times \rtimes  \langle \j \rangle$,
where $\j$ acts on $E^\times$ by conjugation in $B^\times$. 
Note that the induced isomorphism $\GU_B( V_0^\sharp)^0 \simeq E^\times$ is independent of the choice of $\x \in \wedge^2 \VV_1$.

As for $V^\sharp$, we have the following proposition.
\begin{prop}
There is an isomorphism
\[
\GU_B (V^\sharp)^0 \simeq (B_1^\times \times B_2^\times)/F^\times,
\]
depending on the choice of a basis vector for $V_1$ (as $B_1$-space) and for $V_2$ (as $B_2$-space).
\end{prop}

\begin{proof}
First note that the restriction of the hermitian form $(\cdot, \cdot)_{\tilde{\VV}}$ to 
$\VV_1 \otimes \VV_2$ is just the tensor product of the hermitian forms on $\VV_1$ and $\VV_2$. 
Also, from \eqref{eqn:L11} and \eqref{eqn:L12}, right multiplication by $\j_1$ and $\j_2$ on 
$\VV_1$ and $\VV_2$ is given explicitly by 
\[
 \v_1 \j_1 = -a_1 \v_2, \quad  \v_2 \j_1 = a_2 \v_1, \quad  \v_3 \j_2 = -a_3 \v_4, \quad
 \v_4 \j_2 = a_4 \v_3. 
\]
From this and the explicit formula for the action of $L=\j$ on $\tilde{\VV}$, we see that 
(right)-multiplication by $\j$ on $\VV_1 \otimes \VV_2$ is the same as (right)-multiplication by 
$\j_1 \otimes \j_2$.

Choose a $B_1$-basis $\x_1$ for $V_1$ and a $B_2$-basis $\x_2$ for $V_2$.
Then there is an action of $B_1^\times \times B_2^\times$ on $\VV_1 \otimes_E \VV_2$ which 
on pure tensors is given by
\[
(\beta_1, \beta_2) \cdot (\x_1 \alpha_1 \otimes \x_2 \alpha_2) = \x_1 \beta_1 \alpha_1 \otimes \x_2 \beta_2 \alpha_2, 
\]
for any $\alpha_1 \in B_1$, $\alpha_2 \in B_2$. This action is clearly $E$-linear and also $\j$-linear since $\j$ acts 
as right multiplication by $\j_1 \otimes \j_2$, hence is in fact $B$-linear. 

Now, 
\begin{align*}
(\x_1 \beta_1 \alpha_1 \otimes \x_2 \beta_2 \alpha_2, \x_1 \beta_1 \alpha'_1 \otimes \x_2 \beta_2 \alpha'_2) &= \i \cdot (\x_1 \beta_1 \alpha_1 ,  \x_1 \beta_1 \alpha'_1)_{\VV_1} (\x_2 \beta_2 \alpha_2,\x_2 \beta_2 \alpha'_2)_{\VV_2} \\
&= \i \cdot \nu(\beta_1) \nu (\beta_2) \cdot (\x_1  \alpha_1 ,  \x_1  \alpha'_1)_{\VV_1} (\x_2 \alpha_2,\x_2  \alpha'_2)_{\VV_2} \\
&= \nu(\beta_1) \nu (\beta_2) (\x_1 \alpha_1 \otimes \x_2 \alpha_2, \x_1 \alpha'_1 \otimes \x_2 \alpha'_2),
\end{align*}
which shows that $(\beta_1, \beta_2)$ gives an element in $\GU_E (\VV^\sharp)$. Since the action of 
$(\beta_1, \beta_2)$ commutes with $\j$, we see from the formula \eqref{eqn:B-form-defn} that 
it in fact defines an element of $\GU_B (V^\sharp)$. This gives a map 
\[
B_1^\times \times B_2^\times \rightarrow \GU_B(V^\sharp)
\]
whose kernel is the diagonal $F^\times $ in $B_1^\times \times B_2^\times$, embedded as 
$t \mapsto (t,t^{-1})$. By dimension considerations, we see that this gives an isomorphism 
$(B_1^\times \times B_2^\times)/F^\times \simeq \GU_B( V^\sharp)^0$.
\end{proof}

Now we can write down the map \eqref{eqn:xit-rest} explicitly. Let $h=(h_1,h_2) \in H$ with $h_1 \in \GU_E (\VV_1)$ and $h_2 \in \GU_E (\VV_2)$. 
Then
\[
\nu(h_1) = \nu(h_2)
\]
and 
\[
\det(h_1) \det(h_2) = \nu(h_1)^2 = \nu (h_2)^2.
\]
Moreover, $\N (\det h_i) = \nu (h_i)^2$ by Proposition \ref{prop:norm-sim}, so that $\det(h_2) = \det(h_1)^\rho$. 
Now $\wedge^2 h$ acts as (right)-multiplication by $\det(h_1)$ on $\wedge^2 \VV_1$ and by $\det(h_2)$ on $\wedge^2 \VV_2 = \wedge^2 \VV_1 \cdot \j$, 
and thus acts on $V_0^\sharp$ as the element $\det(h_1) \in E^\times = \GU_B(V_0^\sharp)^0$. 

Fix a basis vector $\x_1$ for $V_1$ and $\x_2$ for $V_2$ as above. This gives identifications
\[
 \GU_E (\VV_1) = (B_1^\times \times E^\times)/F^\times, \quad \GU_E (\VV_2) = (B_2^\times \times E^\times)/F^\times.
\]
Let $g_1=[b_1,\alpha_1] \in  \GU_E (\VV_1)$ and $g_2=[b_2,\alpha_2] \in  \GU_E (\VV_2)$. Then 
\[
\nu(g_1) = \nu(b_1) \N(\alpha_1)^{-1},  \quad \det(g_1)= \nu(b_1) \cdot \alpha_1^{-2},
\]
and
\[
\nu(g_2) = \nu(b_2) \N(\alpha_2)^{-1},  \quad \det(g_2)= \nu(b_2) \cdot \alpha_2^{-2}.
\]
Let $g=(g_1, g_2) \in \G (\U_E (\VV_1) \times \U_E (\VV_2))$ viewed as an element in $\GU_E(\VV)$.
Then 
\[
\nu(g) = \nu(g_1) =\nu(g_2) = \nu(b_1) \N(\alpha_1)^{-1} = \nu(b_2) \N(\alpha_2)^{-1}
\]
while
\[
\det (g) =  \nu(b_1) \cdot \alpha_1^{-2} \cdot  \nu(b_2) \cdot \alpha_2^{-2}.
\]
Thus
\begin{align*}
g \in \GSU_E(\VV) & \iff \det(g) = \nu(g)^2 \\
& \iff  \nu(b_1) \cdot \alpha_1^{-2} \cdot  \nu(b_2) \cdot \alpha_2^{-2} = \nu(b_1) \N(\alpha_1)^{-1} \cdot \nu(b_2) \N(\alpha_2)^{-1} \\
& \iff \alpha_1 \alpha_2 = \alpha_1^\rho \alpha_2^\rho \\
& \iff \alpha_1 \alpha_2 \in F^\times. 
\end{align*}
We conclude that
\[
H = \left\{ \left( [b_1,\alpha_1], [b_2,\alpha_2] \right) \, | \,  \nu(b_1) \N(\alpha_1)^{-1} = \nu(b_2) \N (\alpha_2)^{-1} , \ \alpha_1 \alpha_2 \in F^\times \right\}.
\]
The action of $h= \left( [b_1,\alpha_1], [b_2,\alpha_2] \right)$ on $\VV_1 \otimes_E \VV_2$ is then given by
\[
\x_1 \beta_1 \otimes \x_2 \beta_2 \mapsto \x_1 b_1 \beta_1 \alpha_1^{-1} \otimes \x_2 b_2 \beta_2 \alpha_2^{-1} = \x_1 b_1 \beta_1 \alpha_1^{-1} \alpha_2^{-1} \otimes \x_2 b_2 \beta_2 =\x_1 b_1 (\alpha_1 \alpha_2)^{-1} \beta_1 \otimes \x_2 b_2 \beta_2,
\]
so that the map \eqref{eqn:xit-rest} is given explicitly by:
\[
h \mapsto  \left([b_1 (\alpha_1 \alpha_2)^{-1}, b_2],  \nu(b_1) \alpha_1^{-2} \right) \in \G ( (B_1^\times \times B_2^\times)/F^\times) \times E^\times).
\]

Likewise, we can make \eqref{eqn:rest-of-xi} explicit. Let $g =(g_1,g_2) \in \G (\U_E(\VV_1) \times \U_E (\VV_2)) $, with $g_1 =[b_1,\alpha_1]$ and $g_2=[b_2,\alpha_2]$. 
Then 
\[
\frac{\nu(g)^2}{\det (g)} = \frac{\nu (b_1) \N (\alpha_1)^{-1} \nu(b_2) \N (\alpha_2)^{-1} }{\nu(b_1) \nu(b_2) \alpha_1^{-2} \alpha_2^{-2}} = \frac{\alpha_1 \alpha_2}{(\alpha_1\alpha_2)^\rho},
\]
so we may take $\beta=\alpha_1 \alpha_2$ in the definition of $\xi(g)$.
Then the map \eqref{eqn:rest-of-xi} is given by 
\[
g \mapsto \left( [b_1, b_2], \det(g_1) \alpha_1 \alpha_2 \right).
\]

\begin{example}
\label{eg:skew-herm-periods1}
This is the case that is of most interest to us. Instead of starting with the spaces 
$\VV$ or $\VV_1$ or $\VV_2$, we start with two quaternion algebras
$B_1$ and $B_2$ over $F$ containing $E$. Suppose that 
\[
B_1 = E + E \j_1, \quad B_2 = E+ E \j_2,
\]
with $\j_1^2 = J_1$ and $\j_2^2 = J_2$. 
Let $\VV_i = B_i$ considered as a (right)-$E$-hermitian space with the same form as in \cite[\S 2.2]{periods1}, i.e.,
\[
(a+\j_i b,c+\j_i d)_i = a^\rho c - J_i b^\rho d.
\]
We specialize the setup above to the case:
\[
\VV = \VV_1 \oplus \VV_2,
\]
with hermitian form $(\cdot, \cdot)_\VV$ given by the direct sum of $(\cdot, \cdot)_1$ and $(\cdot,\cdot)_2$. 
In the basis 
\[
\v_1=(1,0), \  \v_2=(\j_1,0), \ \v_3=(0,1), \ \v_4=(0,\j_2),
\]
this form is diagonal with matrix 
\[
\begin{pmatrix} a_1 & & & \\ & a_2 & & \\ & & a_3 & \\ & & & a_4 \end{pmatrix}  =
\begin{pmatrix} 1 & & & \\ & -J_1 & & \\ & & 1 & \\ & & & -J_2 \end{pmatrix}.
\]

Let us pick $d:\wedge^4 \VV \simeq E$ such that 
\[
d (\v_1 \wedge \v_2 \wedge \v_3 \wedge \v_4) = -1
\]
as in the proof of Proposition \ref{prop:volume-V-d}. 
Then $J=\vol( (\cdot,\cdot)_\VV,d)$ is equal to $J_1 J_2$.

With respect to the bases $(\v_1 \wedge \v_2, \v_3 \wedge \v_4 )$ of  $\VV_0^\sharp$ and 
$(\v_1 \otimes \v_3, \v_2 \otimes \v_3 , \v_1 \otimes \v_4, \v_2 \otimes \v_4)$ of $\VV^\sharp$, the 
matrices of these hermitian forms are given by
\[
\begin{pmatrix} -J_1 & \\ & - J_2 \end{pmatrix} \quad \text{and} \quad 
\begin{pmatrix} 1 & & & \\ & -J_1 & & \\ & & -J_2 & \\ & & & J \end{pmatrix}
\]
respectively. Thus $\VV^\sharp$ is the tensor product of 
$\VV_1$ and $\VV_2$ as {\it hermitian} spaces. The action of 
$\j$ on $\VV^\sharp$ can be read off from the formulas in the proof of Proposition \ref{prop:volume-V-d} and is given by:
\begin{align*}
(\v_1 \otimes \v_3)\cdot \j &= \v_2 \otimes \v_4, \\
(\v_2 \otimes \v_3) \cdot \j &= J_1 \cdot \v_1 \otimes \v_4, \\
(\v_1 \otimes \v_4) \cdot \j & = J_2 \cdot \v_2 \otimes \v_3, \\
(\v_2 \otimes \v_4) \cdot \j &= J \cdot \v_1 \otimes \v_3.
\end{align*}
This shows that $V^\sharp$ with its $B$-action and $B$-skew-hermitian form is 
exactly the same as the space $V$ occurring in \cite[\S 2.2]{periods1}. 
\end{example}

\subsubsection{The sum of a $3$-dimensional and a $1$-dimensional space}

In this section (which is not used in this paper), 
we suppose that $\VV$ is an orthogonal direct sum of the form $\VV= \VV_3 \oplus \VV_4$ with $\dim_E \VV_3 = 3$ and $\dim_E \VV_4 =1$.
Then there is an inclusion
\[
\G(\U_E(\VV_3) \times \U_E(\VV_4)) \hookrightarrow \GU_E (\VV)
\]
and so we can ask for a description of how this relates to the maps $\tilde{\xi}$, $\xi$. 

Note that
\[
\wedge^2 \VV = \wedge^2 \VV_3 \oplus (\VV_3 \otimes \VV_4)
\]
as a sum of (skew)-hermitian spaces.
We may assume that the basis vectors $\v_i$ are chosen such that 
$(\v_1, \v_2, \v_3)$ forms a basis for $\VV_3$ while $\v_4$ is a basis for $\VV_4$. 
The explicit formula for $L$ shows that $L$ interchanges $\wedge^2 \VV_3$ and 
$\VV_3 \otimes \VV_4$. Thus letting $\WW:=\VV_3 \otimes \VV_4$, we have 
\begin{equation}
\label{eqn:extendingW}
\tilde{V} = \WW \otimes _E B,
\end{equation}
at least as $B$-spaces. The formula \eqref{eqn:B-form-defn} shows that the 
restriction of the  $B$-skew-hermitian form $\langle \cdot, \cdot \rangle$ to $\WW$ is the same as 
the restriction of the $E$-skew-hermitian form $(\cdot, \cdot)$ (from $\tilde{\VV}$) to $\WW$, from which it follows that 
the form $\langle \cdot, \cdot \rangle$ on $\tilde{V}$ is just the $B$-linear extension of $(\cdot ,\cdot)$ via the isomorphism 
\eqref{eqn:extendingW}. Thus there is a canonical inclusion
\[
\GU_E(\WW) \rightarrow \GU_B (\tilde{V}),
\]
which must land inside $\GU_B(\tilde{V})^0$, since $\GU_E(\WW)$ is connected. 

\begin{prop}
The map $\xi$ restricts to an isomorphism
\[
\G (\U_E (\VV_3) \times \U_E (\VV_4))/E^\times \simeq \GU_E(\WW)/F^\times.
\]
\end{prop}

\begin{proof}
It is clear that $\xi$ restricts to an injective map between the given source and target. That it is an isomorphism follows from dimension considerations and since the target is connected. 
\end{proof}

\section{The local exceptional isomorphism}
\label{sec:local-exc}

In this section, we study the local exceptional isomorphism between the (projectivized) unitary group attached to the $4$-dimensional hermitian space $\VV$ and the identity component of the (projectivized) quaternionic unitary group attached to the $3$-dimensional quaternionic skew-hermitian space $\tilde{V}$.
For later use, we need to consider the localizations at almost all (finite) places and at real places.

\subsection{Setup}

Let $F$ be a local field of characteristic zero and $E$ an \'etale quadratic algebra over $F$.
We denote by $\rho$ the non-trivial automorphism of $E$ over $F$ and by $\N = \N_{E/F}$ the norm map from $E$ to $F$.
Fix a trace zero element $\i \in E^\times$ and put $u = \i^2 \in F^\times$.

We recall the construction in \S \ref{sec:global-ex-isom}.
Let $\VV$ be a $4$-dimensional $E$-space equipped with a hermitian form $(\cdot,\cdot)_\VV$.
(We considered a right $E$-space $\VV$ earlier but regard it as a left $E$-space by setting $\alpha v = v \alpha$ for $\alpha \in E$ and $v \in \VV$.)
Fix an $E$-linear isomorphism $d:\wedge^4 \VV \rightarrow E$.
Then we have a $6$-dimensional $E$-space $\tilde{\VV} = \wedge^2 \VV$ equipped with a skew-hermitian form
\[
 (x_1 \wedge x_2, y_1 \wedge y_2) = \i \cdot \det 
 \begin{pmatrix}
  (x_1, y_1)_\VV & (x_1, y_2)_\VV \\
  (x_2, y_1)_\VV & (x_2, y_2)_\VV
 \end{pmatrix}
\]
and a conjugate $E$-linear automorphism
\[
 L = \psi^{-1} \circ \iota \circ \wedge^2 \varphi, 
\]
where 
\begin{itemize}
 \item $\varphi: \VV \rightarrow \VV^*$ is the conjugate $E$-linear isomorphism induced by $(\cdot,\cdot)_\VV$;
 \item $\iota: \wedge^2(\VV^*) \rightarrow (\wedge^2 \VV)^*$ is the natural $E$-linear isomorphism;
 \item $\psi: \wedge^2 \VV \rightarrow (\wedge^2 \VV)^*$ is the $E$-linear isomorphism relative to $d$.
\end{itemize}
Note that $L^2$ is the scalar multiplication by some $J \in F^\times$.
This gives rise to a quaternion $F$-algebra $B = E + E \j$ with a trace zero element $\j \in B^\times$ such that $\j^2 = J$ and a $3$-dimensional right $B$-space $\tilde{V} = \tilde{\VV}$ equipped with a skew-hermitian form
\[
 \langle x, y \rangle = (x, y) - \j^{-1} \cdot (L(x), y).
\]
Moreover, we have a natural isomorphism
\[
 \xi : \PGU_E(\VV) \overset{\simeq}\longrightarrow \PGU_B(\tilde{V})^0
\]
by Proposition \ref{p:isom-PGU}.
Since $B$ and $\tilde{V}$ do not depend on the choice of $d$, we can make a convenient choice in the following computation.

\subsection{The split case}

In this section, we assume that $\VV$ is split.
Fix a basis $\v_1, \dots, \v_4$ of $\VV$ such that
\[
 (\v_i, \v_j)_\VV =  
 \begin{cases}
  1 & \text{if $(i,j) = (1,3), (2,4), (3,1),(4,2)$,} \\
  0 & \text{otherwise.}
 \end{cases}
\]
Let $\e_1, \dots, \e_4$ be its dual basis of $\VV^*$.
We take an isomorphism $d: \wedge^4 \VV \rightarrow E$ such that
\[
 d(\v_1 \wedge \v_2 \wedge \v_3 \wedge \v_4) = 1.
\]
Let $T$ be the maximal torus of $\GU_E(\VV)$ consisting of elements $t$ such that
\[
 t \v_1 = t_1 \v_1, \qquad  
 t \v_2 = t_2 \v_2, \qquad  
 t \v_3 = \nu (t_1^{\rho})^{-1} \v_3, \qquad
 t \v_4 = \nu (t_2^{\rho})^{-1} \v_4
\]
for some $t_i \in E^\times$ and $\nu \in F^\times$.
We identify $T$ with $(E^\times)^2 \times F^\times$ via the map $t \mapsto (t_1, t_2, \nu)$.
Then the center of $\GU_E(\VV)$ is equal to $E^\times$ embedded into $T$ by $z \mapsto (z, z, \N(z))$.

We take a basis $\{ \v_{ij} \, | \, 1 \le i < j \le 4 \}$ of $\tilde{\VV}$ given by $\v_{ij} = \v_i \wedge \v_j$.
Let $\{ \e_{ij} \, | \, 1 \le i < j \le 4 \}$ be its dual basis of $\tilde{\VV}^*$.
Since
\[
 \varphi(\v_1) = \e_3, \qquad
 \varphi(\v_2) = \e_4, \qquad
 \varphi(\v_3) = \e_1, \qquad
 \varphi(\v_4) = \e_2,
\]
we have 
\begin{align*}
 (\iota \circ \wedge^2 \varphi)(\v_{12}) & = \e_{34}, &  
 (\iota \circ \wedge^2 \varphi)(\v_{34}) & = \e_{12}, \\
 (\iota \circ \wedge^2 \varphi)(\v_{13}) & = - \e_{13}, &
 (\iota \circ \wedge^2 \varphi)(\v_{24}) & = -\e_{24}, \\
 (\iota \circ \wedge^2 \varphi)(\v_{14}) & = - \e_{23}, &
 (\iota \circ \wedge^2 \varphi)(\v_{23}) & = -\e_{14}.
\end{align*}
Also, we have
\begin{align*}
 \psi(\v_{12}) & = \e_{34}, &  
 \psi(\v_{34}) & = \e_{12}, \\
 \psi(\v_{13}) & = -\e_{24}, &
 \psi(\v_{24}) & = -\e_{13}, \\
 \psi(\v_{14}) & = \e_{23}, &
 \psi(\v_{23}) & = \e_{14}. 
\end{align*}
Hence we have
\begin{align*}
 L(\v_{12}) & = \v_{12}, & 
 L(\v_{34}) & = \v_{34}, \\
 L(\v_{13}) & = \v_{24}, &
 L(\v_{24}) & = \v_{13}, \\
 L(\v_{14}) & = - \v_{14}, &
 L(\v_{23}) & = - \v_{23}.
\end{align*}
In particular, $J = 1$.
Moreover, $(\v_{ij}, \v_{i'j'})$ and $\langle \v_{ij}, \v_{i'j'} \rangle$ are given by the following tables:
\begin{align*}
 &
\begin{array}{|c||c|c|c|c|c|c|} \hline
 (\cdot, \cdot) & \v_{12} & \v_{34} & \v_{13} & \v_{24} & \v_{14} & \v_{23} \\ \hline \hline
 \v_{12} & 0 & \i & 0 & 0 & 0 & 0 \\ \hline
 \v_{34} & \i & 0 & 0 & 0 & 0 & 0 \\ \hline
 \v_{13} & 0 & 0 & -\i & 0 & 0 & 0 \\ \hline
 \v_{24} & 0 & 0 & 0 & -\i & 0 & 0 \\ \hline
 \v_{14} & 0 & 0 & 0 & 0 & 0 & -\i \\ \hline
 \v_{23} & 0 & 0 & 0 & 0 & -\i & 0 \\ \hline
\end{array} 
 \\
 &
\begin{array}{|c||c|c|c|c|c|c|} \hline
 \langle \cdot, \cdot \rangle & \v_{12} & \v_{34} & \v_{13} & \v_{24} & \v_{14} & \v_{23} \\ \hline \hline
 \v_{12} & 0 & \i + \i\j & 0 & 0 & 0 & 0 \\ \hline
 \v_{34} & \i + \i \j & 0 & 0 & 0 & 0 & 0 \\ \hline
 \v_{13} & 0 & 0 & -\i & -\i\j & 0 & 0 \\ \hline
 \v_{24} & 0 & 0 & -\i\j & -\i & 0 & 0 \\ \hline
 \v_{14} & 0 & 0 & 0 & 0 & 0 & -\i+\i\j \\ \hline
 \v_{23} & 0 & 0 & 0 & 0 & -\i+\i\j & 0 \\ \hline
\end{array} 
\end{align*}

We take a basis $\tilde{v}_1, \tilde{v}_2, \tilde{v}_3$ of $\tilde{V}$ over $B$ given by
\[
 \tilde{v}_1 = \v_{12} + \v_{14}, \qquad
 \tilde{v}_2 = \v_{34} + \v_{23}, \qquad
 \tilde{v}_3 = \v_{13},
\]
so that
\[
 (\langle \tilde{v}_i, \tilde{v}_j \rangle) =
 \begin{pmatrix}
  0 & 2 \i\j & 0 \\
  2 \i\j & 0 & 0 \\
  0 & 0 & -\i
 \end{pmatrix}.
\]
Let $\ii:B \rightarrow \M_2(F)$ be an isomorphism defined by 
\[
 \ii(a + b \i + c \j + d \i\j) = 
 \begin{pmatrix}
  a+c & b-d \\
  (b+d)u & a-c
 \end{pmatrix}.
\]
Put 
\[
 e = \frac{1}{2}(1+\j), \qquad
 e' = \frac{1}{2}(\i - \i\j), \qquad
 e'' = \frac{1}{2u}(\i + \i\j), \qquad
 e^* = \frac{1}{2}(1-\j),
\]
so that
\[
 \ii(e) = \mat{1}{0}{0}{0}, \qquad
 \ii(e') = \mat{0}{1}{0}{0}, \qquad
 \ii(e'') = \mat{0}{0}{1}{0}, \qquad
 \ii(e^*) = \mat{0}{0}{0}{1}.
\]
Put $\tilde{V}^\dagger = \tilde{V} e$.
Then, by Morita theory (see \cite[\S C.2]{periods1} for details), $\tilde{V}^\dagger$ is a $6$-dimensional $F$-space equipped with a symmetric bilinear form $\langle \cdot, \cdot \rangle^\dagger$ determined by
\[
 \frac{1}{2} \langle x, y \rangle = \langle x, y \rangle^\dagger \cdot e''
\]
for $x, y \in \tilde{V}^\dagger$ such that the restriction to $\tilde{V}^\dagger$ induces an isomorphism $\GU_B(\tilde{V})^0 \simeq \GSO(\tilde{V}^\dagger)$ (see \cite[Lemma C.2.1]{periods1}).
We take a basis $v_1, \dots, v_6$ of $\tilde{V}^\dagger$ over $F$ given by
\begin{align*}
 v_1 & = \tilde{v}_1 e  = \v_{12}, & 
 v_2 & = \tilde{v}_1 e'' = \frac{\i}{u} \cdot \v_{14}, \\
 v_3 & = \tilde{v}_2 e = \v_{34}, & 
 v_4 & = \tilde{v}_2 e'' = \frac{\i}{u} \cdot \v_{23}, \\
 v_5 & = \tilde{v}_3 e = \frac{1}{2} (\v_{13} + \v_{24}), &
 v_6 & = \tilde{v}_3 e'' = \frac{\i}{2u} (\v_{13} - \v_{24}),
\end{align*}
so that 
\[
 (\langle v_i, v_j \rangle^\dagger) = 
 \begin{pmatrix}
  0 & 0 & u & 0 & 0 & 0 \\
  0 & 0 & 0 & 1 & 0 & 0 \\
  u & 0 & 0 & 0 & 0 & 0 \\
  0 & 1 & 0 & 0 & 0 & 0 \\ 
  0 & 0 & 0 & 0 & -\frac{u}{2} & 0 \\
  0 & 0 & 0 & 0 & 0 & \frac{1}{2}
 \end{pmatrix}.
\]
Let $\tilde{T}$ be the maximal torus of $\GSO(\tilde{V}^\dagger)$ consisting of elements $\tilde{t}$ such that
\begin{align*}
 \tilde{t} v_1 & = \tilde{t}_1 v_1, &
 \tilde{t} v_2 & = \tilde{t}_2 v_2, \\
 \tilde{t} v_3 & = \tilde{\nu} \tilde{t}_1^{-1} v_3, &
 \tilde{t} v_4 & = \tilde{\nu} \tilde{t}_2^{-1} v_4, \\
 \tilde{t} v_5 & = a v_5 + bu v_6, &
 \tilde{t} v_6 & = b v_5 + a v_6 
\end{align*}
for some $\tilde{t}_i \in F^\times$ and $a, b \in F$ such that $\tilde{\nu} = a^2 - b^2 u \ne 0$.
We identify $\tilde{T}$ with $(F^\times)^2 \times E^\times$ via the map $\tilde{t} \mapsto (\tilde{t}_1, \tilde{t}_2, a + b\i)$.
Then the center of $\GSO(\tilde{V}^\dagger)$ is equal to $F^\times$ embedded into $\tilde{T}$ by $z \mapsto (z, z, z)$.

\begin{lem}
\label{l:isom-tori}
The isomorphism $\xi$ restricts to an isomorphism
\[
 T/E^\times \overset{\simeq}\longrightarrow \tilde{T}/F^\times
\]
given by
\[
 (t_1, t_2, \nu) \longmapsto (\N(t_1 t_2), \nu \N(t_1), \nu t_1 t_2^\rho).
\]
\end{lem}

\begin{proof}
Let $t = (t_1, t_2, \nu) \in T \simeq (E^\times)^2 \times F^\times$.
Since
\[
 \nu(t) = \nu, \qquad 
 \det t = \frac{\nu^2 t_1 t_2}{t_1^\rho t_2^\rho}, \qquad
 \frac{\nu(t)^2}{\det t} = \frac{t_1^\rho t_2^\rho}{t_1 t_2},
\]
the image of $t$ under the homomorphism $\GU_E(\VV) \rightarrow \PGU_B(\tilde{V})$ in the proof of Proposition \ref{p:isom-PGU} is equal to the image of 
\[
 \tilde{t} = t_1^\rho t_2^\rho \cdot \wedge^2 t
\]
in $\PGU_B(\tilde{V})$.
Put 
\[
 \tilde{t}_1 = \N(t_1 t_2), \qquad
 \tilde{t}_2 = \nu \N(t_1), \qquad
 a + b\i = \nu t_1 t_2^\rho, \qquad
 \tilde{\nu} = a^2 - b^2 u = \nu^2 \N(t_1 t_2).
\]
Then
\begin{align*}
 \tilde{t} v_1 & = \tilde{t} \v_{12} = \N(t_1 t_2) \v_{12} = \tilde{t}_1 v_1, \\
 \tilde{t} v_2 & = \frac{\i}{u} \cdot \tilde{t} \v_{14} = \frac{\i}{u} \cdot \nu \N(t_1) \v_{14} = \tilde{t}_2 v_2, \\
 \tilde{t} v_3 & = \tilde{t} \v_{34} = \nu^2 \v_{34} = \tilde{\nu} \tilde{t}_1^{-1} v_3, \\
 \tilde{t} v_4 & = \frac{\i}{u} \cdot \tilde{t} \v_{23} = \frac{\i}{u} \cdot \nu \N(t_2) \v_{23} =  \tilde{\nu} \tilde{t}_2^{-1} v_4, \\
 \tilde{t} v_5 
 & = \frac{1}{2} (\tilde{t} \v_{13} + \tilde{t} \v_{24})
 = \frac{1}{2} (\nu t_1 t_2^\rho \v_{13} + \nu t_1^\rho t_2 \v_{24}) \\
 & = \frac{a}{2} (\v_{13} + \v_{24}) + \frac{b\i}{2} (\v_{13} - \v_{24})
 = a v_5 + bu v_6, \\
 \tilde{t} v_6
 & = \frac{\i}{2u} (\tilde{t} \v_{13} - \tilde{t} \v_{24})
 = \frac{\i}{2u} (\nu t_1 t_2^\rho \v_{13} - \nu t_1^\rho t_2 \v_{24}) \\
 & = \frac{b}{2} (\v_{13} + \v_{24}) + \frac{a\i}{2u} (\v_{13} - \v_{24}) 
 = b v_5 + a v_6.
\end{align*}
Hence the assertion follows.
\end{proof}

\subsection{The real case}
\label{ss:local-ex-isom-real}

In this section, we assume that $F = \R$ and $E = \C$.
Write
\[
 \i = u_0 \cdot i
\]
with $u_0 \in \R^\times$, so that $u = -u_0^2$.
We further assume that the signature of $\VV$ is either $(2,2)$ or $(4,0)$.
Fix a basis $\v_1, \dots, \v_4$ of $\VV$ such that
\[
 (\v_i, \v_j)_\VV =
 \begin{cases}
  1 & \text{if $(i,j) = (1,1), (2,2)$,} \\
  \zeta & \text{if $(i,j) = (3,3), (4,4)$,} \\
  0 & \text{otherwise,}
 \end{cases}
\]
where $\zeta = \pm 1$.
Let $\e_1, \dots, \e_4$ be its dual basis of $\VV^*$.
We take an isomorphism $d: \wedge^4 \VV \rightarrow E$ such that
\[
 d(\v_1 \wedge \v_2 \wedge \v_3 \wedge \v_4) = 1.
\]
Let $T$ be the maximal torus of $\GU_E(\VV)$ consisting of elements $t$ such that
\[
 t \v_1 = r z_1 \v_1, \qquad
 t \v_2 = r z_2 \v_2, \qquad
 t \v_3 = r z_3 \v_3, \qquad
 t \v_4 = r z_4 \v_4
\]
for some $r \in \R_+^\times$ and $z_i \in \C^1$.
We identify $T$ with $(\C^1)^4 \times \R_+^\times$ via the map $t \mapsto (z_1, z_2, z_3, z_4, r)$.
Then the center of $\GU_E(\VV)$ is equal to
\[
 \{ (z,z,z,z,r) \, | \, z \in \C^1, \, r \in \R^\times_+ \} \simeq \C^\times.
\]

We take a basis $\{ \v_{ij} \, | \, 1 \le i < j \le 4 \}$ of $\tilde{\VV}$ given by $\v_{ij} = \v_i \wedge \v_j$.
Let $\{ \e_{ij} \, | \, 1 \le i < j \le 4 \}$ be its dual basis of $\tilde{\VV}^*$.
Since
\[
 \varphi(\v_1) = \e_1, \qquad
 \varphi(\v_2) = \e_2, \qquad
 \varphi(\v_3) = \zeta \e_3, \qquad
 \varphi(\v_4) = \zeta \e_4,
\]
we have 
\begin{align*}
 (\iota \circ \wedge^2 \varphi)(\v_{12}) & = \e_{12}, &  
 (\iota \circ \wedge^2 \varphi)(\v_{34}) & = \e_{34}, \\
 (\iota \circ \wedge^2 \varphi)(\v_{13}) & = \zeta \e_{13}, &
 (\iota \circ \wedge^2 \varphi)(\v_{24}) & = \zeta \e_{24}, \\
 (\iota \circ \wedge^2 \varphi)(\v_{14}) & = \zeta \e_{14}, &
 (\iota \circ \wedge^2 \varphi)(\v_{23}) & = \zeta \e_{23}.
\end{align*}
Also, we have
\begin{align*}
 \psi(\v_{12}) & = \e_{34}, &  
 \psi(\v_{34}) & = \e_{12}, \\
 \psi(\v_{13}) & = -\e_{24}, &
 \psi(\v_{24}) & = -\e_{13}, \\
 \psi(\v_{14}) & = \e_{23}, &
 \psi(\v_{23}) & = \e_{14}.
\end{align*}
Hence we have
\begin{align*}
 L(\v_{12}) & = \v_{34}, & 
 L(\v_{34}) & = \v_{12}, \\
 L(\v_{13}) & = -\zeta \v_{24}, &
 L(\v_{24}) & = -\zeta \v_{13}, \\
 L(\v_{14}) & = \zeta \v_{23}, &
 L(\v_{23}) & = \zeta \v_{14}.
\end{align*}
In particular, $J = 1$.
Moreover, $(\v_{ij}, \v_{i'j'})$ and $\langle \v_{ij}, \v_{i'j'} \rangle$ are given by the following tables:
\begin{align*}
 &
\begin{array}{|c||c|c|c|c|c|c|} \hline
 (\cdot, \cdot) & \v_{12} & \v_{34} & \v_{13} & \v_{24} & \v_{14} & \v_{23} \\ \hline \hline
 \v_{12} & \i & 0 & 0 & 0 & 0 & 0 \\ \hline
 \v_{34} & 0 & \i & 0 & 0 & 0 & 0 \\ \hline
 \v_{13} & 0 & 0 & \zeta \i & 0 & 0 & 0 \\ \hline
 \v_{24} & 0 & 0 & 0 & \zeta \i & 0 & 0 \\ \hline
 \v_{14} & 0 & 0 & 0 & 0 & \zeta\i & 0 \\ \hline
 \v_{23} & 0 & 0 & 0 & 0 & 0 & \zeta\i \\ \hline
\end{array} 
 \\
 &
\begin{array}{|c||c|c|c|c|c|c|} \hline
 \langle \cdot, \cdot \rangle & \v_{12} & \v_{34} & \v_{13} & \v_{24} & \v_{14} & \v_{23} \\ \hline \hline
 \v_{12} & \i & \i\j & 0 & 0 & 0 & 0 \\ \hline
 \v_{34} & \i \j & \i & 0 & 0 & 0 & 0 \\ \hline
 \v_{13} & 0 & 0 & \zeta \i & -\i\j & 0 & 0 \\ \hline
 \v_{24} & 0 & 0 & -\i\j & \zeta \i & 0 & 0 \\ \hline
 \v_{14} & 0 & 0 & 0 & 0 & \zeta\i & \i\j \\ \hline
 \v_{23} & 0 & 0 & 0 & 0 & \i\j & \zeta\i \\ \hline
\end{array} 
\end{align*}

We take a basis $\tilde{v}_1, \tilde{v}_2, \tilde{v}_3$ of $\tilde{V}$ over $B$ given by
\[
 \tilde{v}_1 = \v_{13}, \qquad
 \tilde{v}_2 = \v_{14}, \qquad
 \tilde{v}_3 = \v_{12},
\]
so that
\[
 (\langle \tilde{v}_i, \tilde{v}_j \rangle) =
 \begin{pmatrix}
  \zeta \i & 0 & 0 \\
  0 & \zeta \i & 0 \\
  0 & 0 & \i
 \end{pmatrix}.
\]
Let $\ii:B \rightarrow \M_2(F)$ be an isomorphism defined by 
\[
 \ii(a + b \i + c \j + d \i\j) = 
 \begin{pmatrix}
  a+c & b-d \\
  (b+d)u & a-c
 \end{pmatrix}.
\]
Put 
\[
 e = \frac{1}{2}(1+\j), \qquad
 e' = \frac{1}{2}(\i - \i\j), \qquad
 e'' = \frac{1}{2u}(\i + \i\j), \qquad
 e^* = \frac{1}{2}(1-\j),
\]
so that
\[
 \ii(e) = \mat{1}{0}{0}{0}, \qquad
 \ii(e') = \mat{0}{1}{0}{0}, \qquad
 \ii(e'') = \mat{0}{0}{1}{0}, \qquad
 \ii(e^*) = \mat{0}{0}{0}{1}.
\]
Put $\tilde{V}^\dagger = \tilde{V} e$.
Then, by Morita theory (see \cite[\S C.2]{periods1} for details), $\tilde{V}^\dagger$ is a $6$-dimensional $F$-space equipped with a symmetric bilinear form $\langle \cdot, \cdot \rangle^\dagger$ determined by
\[
 \frac{1}{2} \langle x, y \rangle = \langle x, y \rangle^\dagger \cdot e''
\]
for $x, y \in \tilde{V}^\dagger$ such that the restriction to $\tilde{V}^\dagger$ induces an isomorphism $\GU_B(\tilde{V})^0 \simeq \GSO(\tilde{V}^\dagger)$ (see \cite[Lemma C.2.1]{periods1}).
We take a basis $v_1, \dots, v_6$ of $\tilde{V}^\dagger$ over $F$ given by
\begin{align*}
 v_1 & = \frac{\sqrt{2}}{u_0} \cdot \tilde{v}_1 e = \frac{1}{\sqrt{2} u_0} (\v_{13} - \zeta \v_{24}), &  
 v_2 & = \sqrt{2} \cdot \tilde{v}_1 e'' = -\frac{i}{\sqrt{2} u_0} (\v_{13} + \zeta \v_{24}), \\
 v_3 & = \frac{\sqrt{2}}{u_0} \cdot \tilde{v}_2 e = \frac{1}{\sqrt{2} u_0} (\v_{14} + \zeta \v_{23}), & 
 v_4 & = \sqrt{2} \cdot \tilde{v}_2 e'' = -\frac{i}{\sqrt{2} u_0} (\v_{14} - \zeta \v_{23}), \\
 v_5 & = \frac{\sqrt{2}}{u_0} \cdot \tilde{v}_3 e = \frac{1}{\sqrt{2} u_0} (\v_{12} + \v_{34}), &
 v_6 & = \sqrt{2} \cdot \tilde{v}_3 e'' = -\frac{i}{\sqrt{2} u_0} (\v_{12} - \v_{34}),
\end{align*}
so that 
\[
 (\langle v_i, v_j \rangle^\dagger) = 
 \begin{pmatrix}
  -\zeta & 0 & 0 & 0 & 0 & 0 \\
  0 & -\zeta & 0 & 0 & 0 & 0 \\
  0 & 0 & -\zeta & 0 & 0 & 0 \\
  0 & 0 & 0 & -\zeta & 0 & 0 \\ 
  0 & 0 & 0 & 0 & -1 & 0 \\
  0 & 0 & 0 & 0 & 0 & -1
 \end{pmatrix}.
\]
Let $\tilde{T}$ be the maximal torus of $\GSO(\tilde{V}^\dagger)$ consisting of elements $\tilde{t}$ such that
\begin{align*}
 \tilde{t} v_1 & = \tilde{r}(a_1 v_1 - b_1 v_2), &
 \tilde{t} v_2 & = \tilde{r}(b_1 v_1 + a_1 v_2), \\
 \tilde{t} v_3 & = \tilde{r}(a_2 v_3 - b_2 v_4), &
 \tilde{t} v_4 & = \tilde{r}(b_2 v_3 + a_2 v_4), \\
 \tilde{t} v_5 & = \tilde{r}(a_3 v_5 - b_3 v_6), &
 \tilde{t} v_6 & = \tilde{r}(b_3 v_5 + a_3 v_6)
\end{align*}
for some $\tilde{r} \in \R^\times_+$ and $a_i, b_i \in \R$ such that $\tilde{z}_i = a_i + b_i i \in \C^1$.
We identify $\tilde{T}$ with $(\C^1)^3 \times \R^\times_+$ via the map $\tilde{t} \mapsto (\tilde{z}_1, \tilde{z}_2, \tilde{z}_3, \tilde{r})$.
Then the center of $\GSO(\tilde{V}^\dagger)$ is equal to
\[
 \{ (\tilde{z},\tilde{z},\tilde{z},\tilde{r}) \, | \, \tilde{z} = \pm 1, \, \tilde{r} \in \R^\times_+ \} \simeq \R^\times.
\]

\begin{lem}
\label{l:isom-tori-real}
The isomorphism $\xi$ restricts to an isomorphism
\[
 T/\C^\times \overset{\simeq}\longrightarrow \tilde{T}/\R^\times
\]
given by
\[
 (z_1, z_2, z_3, z_4, r) \longmapsto \left(
 \frac{y_1y_3}{y_2y_4},  \frac{y_1y_4}{y_2y_3}, \frac{y_1y_2}{y_3y_4}, r^2 \right), 
\]
where we choose $y_i \in \C^1$ such that $z_i = y_i^2$.
\end{lem}

\begin{proof}
Let $t = (z_1, z_2, z_3, z_4, r) \in T \simeq (\C^1)^4 \times \R_+^\times$ and choose $y_i \in \C^1$ such that $z_i = y_i^2$.
Since
\[
 \nu(t) = r^2, \qquad 
 \det t = r^4 z_1 z_2 z_3 z_4, \qquad
 \frac{\nu(t)^2}{\det t} = \frac{1}{z_1 z_2 z_3 z_4},
\]
the image of $t$ under the homomorphism $\GU_E(\VV) \rightarrow \PGU_B(\tilde{V})$ in the proof of Proposition \ref{p:isom-PGU} is equal to the image of 
\[
 \tilde{t} = \frac{1}{y_1 y_2 y_3 y_4} \cdot \wedge^2 t
\]
in $\PGU_B(\tilde{V})$.
Put $v_i' = \sqrt{2} u_0 \cdot v_i$ and write
\[
 \frac{y_1y_3}{y_2y_4} = a_1 + b_1 i, \qquad
 \frac{y_1y_4}{y_2y_3} = a_2 + b_2 i, \qquad
 \frac{y_1y_2}{y_3y_4} = a_3 + b_3 i
\]
with $a_i, b_i \in \R$.
Then
\begin{align*}
 \frac{1}{r^2} \cdot \tilde{t} v_1' & = \frac{1}{r^2}(\tilde{t} \v_{13} - \zeta \tilde{t} \v_{24}) = \frac{y_1y_3}{y_2y_4} \cdot \v_{13} - \frac{y_2y_4}{y_1y_3} \cdot \zeta \v_{24} \\
 & = a_1 (\v_{13} - \zeta \v_{24}) + b_1 i (\v_{13} + \zeta \v_{24}) = a_1 v'_1 - b_1 v'_2, \\
 \frac{1}{r^2} \cdot \tilde{t} v_2' & = - \frac{i}{r^2}(\tilde{t} \v_{13} + \zeta \tilde{t} \v_{24}) = - \frac{y_1y_3}{y_2y_4} \cdot i \v_{13} - \frac{y_2y_4}{y_1y_3} \cdot \zeta i \v_{24} \\
 & = b_1(\v_{13} - \zeta \v_{24}) - a_1 i (\v_{13} + \zeta \v_{24}) = b_1 v'_1 + a_1 v'_2, \\
 \frac{1}{r^2} \cdot \tilde{t} v_3' & = \frac{1}{r^2}(\tilde{t} \v_{14} + \zeta \tilde{t} \v_{23}) = \frac{y_1y_4}{y_2y_3} \cdot \v_{14} + \frac{y_2y_3}{y_1y_4} \cdot \zeta \v_{23} \\
 & = a_2(\v_{14}+\zeta \v_{23}) + b_2 i (\v_{14} - \zeta \v_{23}) = a_2 v'_3  - b_2 v'_4, \\
 \frac{1}{r^2} \cdot \tilde{t} v_4' & = -\frac{i}{r^2}(\tilde{t} \v_{14} - \zeta \tilde{t} \v_{23}) = -\frac{y_1y_4}{y_2y_3} \cdot i \v_{14} + \frac{y_2y_3}{y_1y_4} \cdot \zeta i \v_{23} \\
 & = b_2 (\v_{14} + \zeta \v_{23}) - a_2 i(\v_{14}- \zeta \v_{23}) = b_2 v'_3 + a_2 v'_4, \\
 \frac{1}{r^2} \cdot \tilde{t} v_5' & = \frac{1}{r^2}(\tilde{t} \v_{12} + \tilde{t} \v_{34}) = \frac{y_1y_2}{y_3y_4} \cdot \v_{12} + \frac{y_3y_4}{y_1y_2} \cdot \v_{34} \\
 & = a_3 (\v_{12} + \v_{34}) + b_3 i (\v_{12} - \v_{34}) = a_3 v'_5 - b_3 v'_6, \\
 \frac{1}{r^2} \cdot \tilde{t} v_6' & = -\frac{i}{r^2}(\tilde{t} \v_{12} - \tilde{t} \v_{34}) = - \frac{y_1y_2}{y_3y_4} \cdot i \v_{12} + \frac{y_3y_4}{y_1y_2} \cdot i \v_{34} \\
 & = b_3 (\v_{12} + \v_{34}) - a_3 i (\v_{12} - \v_{34}) = b_3 v'_5 + a_3 v'_6.
\end{align*}
Hence the assertion follows.
\end{proof}

Let $X^*(T/\C^\times)$ and $X^*(\tilde{T}/\R^\times)$ be the weight lattices of $T/\C^\times$ and $\tilde{T}/\R^\times$, respectively. 
Then we have 
\begin{align*}
 X^*(T/\C^\times) & \simeq \{ (k_1, k_2, k_3, k_4) \in \Z^4 \, | \, k_1+k_2+k_3+k_4=0 \}, \\
 X^*(\tilde{T}/\R^\times) & \simeq \{ (l_1, l_2, l_3) \in \Z^3 \, | \, l_1 + l_2 + l_3 \equiv 0 \bmod 2 \},
\end{align*}
where $(k_1, k_2, k_3, k_4)$ and $(l_1, l_2, l_3)$ on the right-hand sides correspond to the characters
\[
 (z_1,z_2,z_3,z_4,r) \longmapsto z_1^{k_1} z_2^{k_2} z_3^{k_3} z_4^{k_4} \quad \text{and} \quad (\tilde{z}_1, \tilde{z}_2, \tilde{z}_3, \tilde{r}) \longmapsto \tilde{z}_1^{l_1} \tilde{z}_2^{l_2} \tilde{z}_3^{l_3},
\]
respectively.
As an immediate consequence of Lemma \ref{l:isom-tori-real}, we have:

\begin{cor}
\label{c:isom-weights}
The isomorphism $\xi$ induces an isomorphism 
\[
 X^*(\tilde{T}/\R^\times) \overset{\simeq}\longrightarrow X^*(T/\C^\times)
\]
given by 
\[
 (l_1,l_2,l_3) \longmapsto \left( \frac{l_1+l_2+l_3}{2}, \frac{-l_1-l_2+l_3}{2}, \frac{l_1-l_2-l_3}{2}, \frac{-l_1+l_2-l_3}{2} \right)
\]
under the above identifications.
\end{cor}

\section{Cohomological representations}

In this section we recall various facts about cohomological representations for 
real groups, with the goal of constructing cohomology classes on the Shimura variety 
attached to the group $\tilde{\mathrsfs{G}}_B = \GU_B(\tilde{V})$ of the previous section.
 Since 
 \[
 \tilde{\mathrsfs{G}}_B (\R) \simeq \prod_{v\in \Sigma} \GSO(4,2) \times \prod_{v \not \in \Sigma} \GSO(0,6),
 \]
 we will be particularly interested in the orthogonal groups $\O(4,2)$ and $\O(0,6)$.

\subsection{Cohomological representations}
\label{ss:coh-rep}

Let $G$ be a connected real reductive group and $K$ a maximal compact subgroup of $G$.
We assume that $\rank G = \rank K$ and that $G/K$ is a Hermitian symmetric domain.
Let $\fg_0$ and $\fk_0$ be the Lie algebras of $G$ and $K$, respectively.
Let $\theta$ be the Cartan involution of $G$ associated to $K$.
Then we have a Cartan decomposition 
\[
 \fg_0 = \fk_0 \oplus \fp_0,
\]
where $\fp_0$ is the $(-1)$-eigenspace of $\theta$.
Let $J$ be the complex structure on $\fp_0$, i.e., the automorphism of $\fp_0$ given by the multiplication by $i$ on the tangent space of $G/K$ at the origin.
Fix a Cartan subalgebra $\ft_0$ of $\fk_0$.
Let $\fg$, $\fk$, $\fp$, $\ft$ be the complexifications of $\fg_0$, $\fk_0$, $\fp_0$, $\ft_0$, respectively.
Let $\fp^\pm$ be the $(\pm i)$-eigenspace of $J$ in $\fp$, so that 
\[
 \fg = \fk \oplus \fp^+ \oplus \fp^-.
\]
For any subspace $\ff$ of $\fg$ stable under the adjoint action of $\ft$, we denote by $\Delta(\ff)$ the set of roots of $\ft$ in $\ff$.

We consider an irreducible unitary $(\fg, K)$-module $\pi$ such that the relative Lie algebra cohomology 
\[
 H^*(\fg, K; \pi \otimes F)
\]
is non-zero for some irreducible finite-dimensional representation $F$ of $G$.
Such $(\fg, K)$-modules are called cohomological and classified by Vogan-Zuckerman \cite{vz}.
We also consider each piece of the Hodge decomposition
\[
 H^i(\fg, K; \pi \otimes F) = \bigoplus_{p+q=i} H^{p,q}(\fg, K; \pi \otimes F).
\]

Let $\fq$ be a $\theta$-stable parabolic subalgebra of $\fg$, i.e.~$\fq$ is the sum of non-negative eigenspaces of $\ad(x)$ for some $x \in i \ft_0$.
Then we have a Levi decomposition
\[
 \fq = \fl \oplus \fu,
\]
where $\fl$ is the centralizer of $x$ and $\fu$ is the unipotent radical of $\fq$.
Note that $\fl$ is the complexification of $\fl_0 = \fq \cap \fg_0$ and contains $\ft$.
Fix a positive system $\Delta^+(\fl \cap \fk)$ of $\Delta(\fl \cap \fk)$ and choose a positive system $\Delta^+(\fl)$ of $\Delta(\fl)$ containing $\Delta^+(\fl \cap \fk)$.
Then 
\[
 \Delta^+(\fk) = \Delta^+(\fl \cap \fk) \cup \Delta(\fu \cap \fk)
 \quad \text{and} \quad 
 \Delta^+(\fg) = \Delta^+(\fl) \cup \Delta(\fu)
\]
are positive systems of $\Delta(\fk)$ and $\Delta(\fg)$, respectively.
Put
\[
 \rho = \frac{1}{2} \sum_{\alpha \in \Delta^+(\fg)} \alpha,
 \qquad
 \rho(\fu \cap \fp) = \frac{1}{2} \sum_{\alpha \in \Delta(\fu \cap \fp)} \alpha.
\]
Let $L$ be the centralizer of $x$ in $G$, so that its Lie algebra is $\fl_0$.
Let $\lambda \in \fl^*$ be the differential of a unitary character of $L$ such that $\langle \alpha, \lambda|_\ft \rangle \ge 0$ for all $\alpha \in \Delta(\fu)$.
Then, by \cite[Theorem 5.3]{vz} (see also \cite{kv}), there exists a unique irreducible unitary $(\fg, K)$-module $A_\fq(\lambda)$ such that
\begin{itemize}
 \item $A_\fq(\lambda)$ has infinitesimal character $\lambda|_\ft + \rho$;
 \item $A_\fq(\lambda)$ contains the $K$-type with highest weight $\lambda|_\ft + 2 \rho(\fu \cap \fp)$;
 \item any $K$-type contained in $A_\fq(\lambda)$ has highest weight of the form
 \[
  \lambda|_\ft + 2 \rho(\fu \cap \fp) + \sum_{\alpha \in \Delta(\fu \cap \fp)} n_\alpha \alpha
 \]
 for some non-negative integers $n_\alpha$.
\end{itemize}
Let $F$ be an irreducible finite-dimensional representation of $G$ with highest weight $\gamma$.
Then
\[
 H^i(\fg, K; A_\fq(\lambda) \otimes F^*) 
 \simeq \Hom_K(\wedge^i \fp, A_\fq(\lambda) \otimes F^*)
\]
if $\gamma = \lambda|_\ft$ and
\[
 H^i(\fg, K; A_\fq(\lambda) \otimes F^*) = 0
\]
otherwise (see \cite{borel-wallach}).
Now suppose that $\gamma = \lambda|_\ft$.
Then, by \cite[Proposition 6.19]{vz}, we have
\[
 H^{i+R^+, i+R^-}(\fg, K; A_\fq(\lambda) \otimes F^*)
 \simeq \Hom_{L \cap K}(\wedge^{2i}(\fl \cap \fp), \C),
\]
where $R^\pm = \dim(\fu \cap \fp^\pm)$, and 
\[
 H^{p,q}(\fg, K; A_\fq(\lambda) \otimes F^*) = 0
\]
if $p-q \ne R^+ - R^-$.

\subsection{Local theta lifts}
\label{ss:local-theta}

Let the notation be as in \S \ref{sec:kudlamillson} below.
In particular, $G \simeq \O(p,q)$ and $G' \simeq \SL_2(\R)$.
Let $G^0$ be the topological identity component of $G$.
Let $\omega$ be the Weil representation of $G \times G'$ (relative to the character $x \mapsto e^{2 \pi i x}$ of $\R$).
For any irreducible $(\fg',K')$-module $\pi$, the maximal $\pi^\vee$-isotypic quotient of $\omega$ is of the form
\[
 \Theta(\pi) \boxtimes \pi^\vee
\]
for some admissible $(\fg,K)$-module $\Theta(\pi)$.
If $\Theta(\pi)$ is non-zero, then it has a unique irreducible quotient $\theta(\pi)$ by the Howe duality \cite{howe-jams}.

Now suppose that $\pi$ is a holomorphic discrete series representation of $G'$ of weight $k + 1$ (i.e.~with Harish-Chandra parameter $k$), where $k$ is an integer with 
\[
 k \ge 2.
\]
For our applications, we consider the theta lifts $\theta(\pi)$ and $\theta(\pi^\vee)$ when 
\[
 (p,q) = (4,2) \text{ or } (0,6).
\]

\subsubsection{The case $(p,q) = (4,2)$}

In this case, by the result of J.-S.~Li \cite[Theorem 6.2]{li-duke}, we have
\[
 \theta(\pi)|_{G^0} = A_{\fq_0}(\lambda_0), \qquad
 \theta(\pi^\vee)|_{G^0} = A_{\fq_1}(\lambda_1), 
\]
where $\fq_i = \fl_i \oplus \fu_i$ is the $\theta$-stable parabolic subalgebra of $\fg$ with
\begin{align*}
 \fl_0 & \simeq \so(4) \oplus \so(2), &
 \fu_0 & = \C X_{-\varepsilon_1 + \varepsilon_3} \oplus \C X_{-\varepsilon_2 + \varepsilon_3} \oplus \C X_{\varepsilon_1 + \varepsilon_3} \oplus \C X_{\varepsilon_2 + \varepsilon_3}, \\
 \fl_1 & \simeq \so(2) \oplus \so(2,2), &
 \fu_1 & = \C X_{\varepsilon_1 - \varepsilon_2} \oplus \C X_{\varepsilon_1 - \varepsilon_3} \oplus \C X_{\varepsilon_1 + \varepsilon_2} \oplus \C X_{\varepsilon_1 + \varepsilon_3}
\end{align*}
(where $X_{\pm \varepsilon_i \pm \varepsilon_j}$ is a root vector for $\pm \varepsilon_i \pm \varepsilon_j$), and
\[
 \lambda_0 = (0,0,k-2), \qquad
 \lambda_1 = (k-2,0,0).
\]
Since $\fl_0 = \fk$ and $\fu_0 = \fp^+$, $A_{\fq_0}(\lambda_0)$ is a holomorphic discrete series representation of $G^0$.
Also, we have
\[
 \fu_1 \cap \fp^+ = \C X_{\varepsilon_1 + \varepsilon_3}, \qquad
 \fu_1 \cap \fp^- = \C X_{\varepsilon_1 - \varepsilon_3},
\]
so that $2 \rho(\fu_1 \cap \fp) = 2 \varepsilon_1$.
Hence the minimal $K^0$-type of $A_{\fq_1}(\lambda_1)$ has highest weight 
\[
 (k,0,0).
\]
Moreover, since $\fl_1 \cap \fk \simeq \so(2) \oplus \so(2) \oplus \so(2)$ and
\[
 \fl_1 \cap \fp = \C X_{\varepsilon_2 - \varepsilon_3} \oplus \C X_{-\varepsilon_2 + \varepsilon_3} \oplus\C X_{\varepsilon_2 + \varepsilon_3} \oplus \C X_{-\varepsilon_2 - \varepsilon_3}, 
\]
we have 
\[
 \dim H^{i,j}(\fg, K^0; A_{\fq_1}(\lambda_1) \otimes F) = 
 \begin{cases}
  1 & \text{if $(i,j) = (1,1), (3,3)$,} \\
  2 & \text{if $(i,j) = (2,2)$,} \\
  0 & \text{otherwise,}
 \end{cases}
\]
where $F$ is the irreducible finite-dimensional representation of $G^0$ with highest weight $\lambda_1$.
Note that $F$ is self-dual.

\subsubsection{The case $(p,q) = (0,6)$}

In this case, $\theta(\pi)|_{G^0}$ is the irreducible finite-dimensional representation of $G^0$ with highest weight
\[
 \lambda = (k-2,0,0)
\]
and $\theta(\pi^\vee)$ is zero (see e.g.~\cite[Proposition 6.5]{adams-theta-R}).

\section{Kudla-Millson theory}
\label{sec:kudlamillson}

In the previous section, we studied certain cohomological representations 
for $G=\O(p,q)$ with $(p,q)=(4,2)$ or $(0,6)$. In this section, 
we recall the explicit construction of $(\fg,K)$-cohomology classes attached to these representations 
using the Weil representation \`{a} la Kudla-Millson. While the original 
papers of Kudla and Millson considered the case of the trivial local system,
 the case of more general local systems was discussed in Funke-Millson. 
 We also study the restriction of these explicit $(\fg, K)$-cohomology classes 
 to the subgroup $\O(2,2) \times \O(2,0)$ and $\O(0,4) \times \O(0,2)$ of 
 $\O(4,2)$ and $\O(0,6)$ respectively. 

\subsection{Groups and Lie algebras}
\label{ss:groups}

Let $V$ be an $m$-dimensional quadratic space over $\R$ of signature $(p,q)$, where $p$ and $q$ are non-negative integers such that $p+q=m$.
Namely, $V$ is equipped with a non-degenerate symmetric bilinear form
$\langle \cdot, \cdot \rangle : V \times V \rightarrow \R$ and an orthogonal basis $\{ e_i \, | \, 1 \le i \le m \}$ such that
\[
 \langle e_i, e_i \rangle = 
 \begin{cases}
  +1 & \text{if $1 \le i \le p$,} \\
  -1 & \text{if $p+1 \le i \le m$.}
 \end{cases}
\]
We assume that $p$ and $q$ are even.
Let $G = \O(V) \simeq \O(p,q)$ be the orthogonal group of $V$.
Put 
\[
 V_+ = \R e_1 + \dots + \R e_p, \qquad
 V_- = \R e_{p+1} + \dots + \R e_m,
\]
so that $V = V_+ \oplus V_-$.
We define a Cartan involution $\theta$ of $G$ by
\[
 \theta(g) = I_V \cdot g \cdot I_V,
\]
where $I_V = \id_{V_+} \oplus (-\id_{V_-})$.
Let $K$ be the maximal compact subgroup of $G$ with respect to $\theta$.
Then $K = \O(V_+) \times \O(V_-) \simeq \O(p) \times \O(q)$.
We define a maximal torus $T$ of $G$ by 
\[
 T = \SO(V_1) \times \dots \times \SO(V_r) \simeq \SO(2)^r,
\]
where $V_i = \R e_{2i-1} + \R e_{2i}$ and $r = \frac{m}{2}$.

Let $\fg_0$ be the Lie algebra of $G$.
Then we have a $G$-equivariant isomorphism $\rho: \wedge^2 V \rightarrow \fg_0$ given by 
\[
 \rho(u \wedge v)(w) = \langle u, w \rangle v - \langle v, w \rangle u.
\]
We take a basis $\{ X_{ij} \, | \, 1 \le i < j \le m \}$ of $\fg_0$ given by $X_{ij} = \rho(e_i \wedge e_j)$.
Let $\{ \omega_{ij} \, | \, 1 \le i < j \le m \}$ be its dual basis of $\fg_0^*$.
We have a Cartan decomposition
\[
 \fg_0 = \fk_0 \oplus \fp_0, 
\]
where $\fk_0$ and $\fp_0$ are the $(+1)$-eigenspace and $(-1)$-eigenspace of $\theta$, respectively.
Note that $\fk_0$ is the Lie algebra of $K$.
Via the isomorphism $\rho$, we have
\[
 \fk_0 \simeq \wedge^2 V_+ \oplus \wedge^2 V_-, \qquad
 \fp_0 \simeq V_+ \otimes V_-.
\]
Let $\ft_0$ be the Lie algebra of $T$.
Let $\fg$, $\fk$, $\fp$, $\ft$ be the complexifications of $\fg_0$, $\fk_0$, $\fp_0$, $\ft_0$, respectively. 
If $q=2$, then we have a complex structure $\id_{V_+} \otimes J_{V_-}$ on $\fp_0$ and hence a decomposition
\[
 \fg = \fk \oplus \fp^+ \oplus \fp^-,
\]
where $J_{V_-}$ is defined by 
\[
 J_{V_-}(e_{m-1}) = -e_m, \qquad 
 J_{V_-}(e_m) = e_{m-1}, 
\]
and $\fp^+$ and $\fp^-$ are the $(+i)$-eigenspace and $(-i)$-eigenspace of $\id_{V_+} \otimes J_{V_-}$ in $\fp$, respectively.

Let $W$ be a $2$-dimensional symplectic space over $\R$.
Namely, $W$ is equipped with a non-degenerate skew-symmetric bilinear form $\langle \cdot, \cdot \rangle : W \times W \rightarrow \R$ and a basis $\{ e, f \}$ such that
\[
 \langle e, e \rangle = \langle f, f \rangle = 0, \qquad
 \langle e, f \rangle = 1.
\]
Let $G' = \Sp(W) \simeq \SL_2(\R)$ be the symplectic group of $W$.
Let $K' \simeq \U(1)$ be the standard maximal compact subgroup of $G'$, where $\U(1)$ is embedded into $\SL_2(\R)$ by $a + bi \mapsto \smat{a}{b}{-b}{a}$.
Let $\fg_0'$ be the Lie algebra of $G'$ and $\fg'$ its complexification.

\subsection{Finite-dimensional representations of $G$}
\label{ss:fin-dim-rep-SO}

Let $\{ \varepsilon_i \, | \, 1 \le i \le r \}$ be the basis of $\ft_0^*$ given by $\varepsilon_i(t) = t_i$ for 
\[
 t = \left( \mat{}{t_1}{-t_1}{}, \dots, \mat{}{t_r}{-t_r}{} \right).
\]
We identify $\ft^*$ with $\C^r$ via this basis.
Under this identification, the weight lattice is given by $\Z^r$.  
We take
\[
 \{ \varepsilon_i - \varepsilon_{i+1} \, | \, 1 \le i < r \} \cup \{ \varepsilon_{r-1} + \varepsilon_r \}
\]
as a set of simple roots.
Then $\lambda = (\lambda_1,\dots,\lambda_r) \in \Z^r$ is dominant if and only if
\[
 \lambda_1 \ge \dots \ge \lambda_{r-1} \ge |\lambda_r|.
\]

Now assume that $r \ge 2$.
We only consider dominant weights $\lambda$ of the form
\[
 \lambda = (\ell,0,\dots,0)
\]
for some non-negative integer $\ell$.
In particular, we have $\Ss_\lambda V = \Sym^\ell V$, where $\Ss_\lambda$ is the Schur functor associated to $\lambda$.
Put $\sss^\ell V = \Sym^\ell V \otimes \C$ and equip it with a non-degenerate $G$-invariant bilinear pairing $\langle \cdot, \cdot \rangle : \sss^\ell V \times \sss^\ell V \rightarrow \C$ given by 
\[
 \langle v_1 \cdots v_\ell, w_1 \cdots w_\ell \rangle = \sum_{\sigma \in \mathfrak{S}_\ell} \langle v_1, w_{\sigma(1)} \rangle \cdots \langle v_\ell, w_{\sigma(\ell)} \rangle,
\]
where $\mathfrak{S}_\ell$ is the symmetric group of degree $\ell$.
We denote by $\Hs^\ell V$ the kernel of the contraction $\sss^{\ell} V \rightarrow \sss^{\ell-2} V$ given by 
\[
 v_1 \cdots v_\ell \longmapsto
 \sum_{i<j} \langle v_i, v_j \rangle \cdot 
 v_1 \cdots \hat{v}_i \cdots \hat{v}_j \cdots v_\ell.
\]
Then, by \cite[\S 19.5]{fulton-harris}, $\Hs^\ell V$ is the irreducible finite-dimensional representation of $G$ with highest weight $\lambda$, whose highest weight vector is given by
\[
 (e_1+ie_2)^\ell.
\]
Also, the pairing $\langle \cdot, \cdot \rangle$ induces a $G$-equivariant orthogonal projection
\begin{equation}
\label{eq:orth-proj-HlV}
 \sss^\ell V \longrightarrow \Hs^\ell V.
\end{equation}

\subsection{Weil representations}

We recall the Schr\"odinger model $\cS(V)$ of the Weil representation $\omega$ of $G \times G'$ (relative to the character $x \mapsto e^{2 \pi i x}$ of $\R$), where $\cS(V)$ is the space of Schwartz functions on $V$.
By \cite[\S 5]{kudla-splitting}, the action of $G$ is given by
\[
 \omega(g)\varphi(x) = \varphi(g^{-1} x), 
\]
and the action of $G'$ is given by
\begin{align*}
 \omega \mat{a}{}{}{a^{-1}} \varphi(x) & = a^{\frac{m}{2}} \varphi(ax), \\
 \omega \mat{1}{b}{}{1} \varphi(x) & = \varphi(x) e^{\pi i b \langle x, x \rangle}, \\
 \omega \mat{}{-1}{1}{} \varphi(x) & = i^{\frac{q-p}{2}} \int_V \varphi(y) e^{-2 \pi i \langle x, y \rangle} \, dy.
\end{align*}
Let $x_1, \dots, x_m$ be the coordinates on $V$ with respect to the basis $\{ e_1, \dots, e_m \}$.
We denote by $S(V)$ the subspace of $\cS(V)$ consisting of functions of the form $p \cdot \varphi_0$, where $p$ is a polynomial function and $\varphi_0$ is the Gaussian defined by
\[
 \varphi_0(x_1, \dots, x_m) = e^{- \pi (x_1^2 + \dots + x_m^2)}.
\]

We also recall the Fock model $\Ps(\C^m)$ of the Weil representation $\omega$ of $\fg \times \fg'$ (relative to $\lambda = 2 \pi i$), where $\Ps(\C^m)$ is the space of polynomial functions on $\C^m$.
We refer the reader to \cite[\S 7]{km-ihes}, \cite[Appendix A]{funke-millson} for details.
Let $z_1, \dots, z_m$ be the coordinates on $\C^m$.
Then, by \cite[Lemma A.3]{funke-millson}, we have a $\fg \times \fg'$-equivariant isomorphism $\iota : S(V) \simeq \Ps(\C^m)$ such that $\iota(\varphi_0) = 1$ and such that
\begin{equation}
\label{eq:weil-fock-isom}
\begin{aligned}
 \iota \left( x_\alpha - \frac{1}{2 \pi} \frac{\partial}{\partial x_\alpha} \right) \iota^{-1} & = \frac{1}{2 \pi i} z_\alpha, &
 \iota \left( x_\alpha + \frac{1}{2 \pi} \frac{\partial}{\partial x_\alpha} \right) \iota^{-1} & = 2i \frac{\partial}{\partial z_\alpha}, \\
 \iota \left( x_\mu - \frac{1}{2 \pi} \frac{\partial}{\partial x_\mu} \right) \iota^{-1} & = - \frac{1}{2 \pi i} z_\mu, &
 \iota \left( x_\mu + \frac{1}{2 \pi} \frac{\partial}{\partial x_\mu} \right) \iota^{-1} & = - 2i \frac{\partial}{\partial z_\mu}
\end{aligned}
\end{equation}
for $1 \le \alpha \le p$ and $p+1 \le \mu \le m$.

\subsection{Schwartz forms}
\label{ss:schwartzforms}

We now recall the Schwartz forms constructed by Kudla-Millson \cite{km1} in the case of trivial coefficients and Funke-Millson \cite{funke-millson} in general.
Let $\ell$ be a non-negative integer.
Recall that the signature of $V$ is $(p,q)$.
If $p \ge 1$, then as in \cite[\S 6.2]{funke-millson}, we define 
\[
 \varphi_{q,\ell} \in \Ps(\C^m) \otimes \wedge^q \fp^* \otimes \sss^{\ell} V
\]
by
\[
 \varphi_{q,\ell} = \left( \frac{1}{4 \pi i} \right)^{\ell+q}
 \sum_{\alpha} \sum_{\beta} 
 z_\alpha z_\beta \otimes \omega_\alpha \otimes e_\beta, 
\]
where the sums run over $\alpha = (\alpha_1, \dots, \alpha_q) \in \{ 1, \dots, p \}^q$ and $\beta = (\beta_1, \dots, \beta_\ell) \in \{ 1, \dots, p \}^\ell$, and 
\begin{align*}
 z_\alpha & = z_{\alpha_1} \cdots z_{\alpha_q}, &
 z_\beta & = z_{\beta_1} \cdots z_{\beta_\ell}, \\
 \omega_\alpha & = \omega_{\alpha_1 p+1} \wedge \cdots \wedge \omega_{\alpha_q p+q}, &
 e_\beta & = e_{\beta_1} \cdots e_{\beta_\ell}.
\end{align*}
(Note that we scale the Schwartz form given in \cite[\S 6.2]{funke-millson} by $2^{-\frac{q}{2}}$ and take its image under the projection $V^{\otimes \ell} \otimes \C \rightarrow \sss^\ell V$.)
Then we define 
\[
 \varphi'_{q, \ell} \in \Ps(\C^m) \otimes \wedge^q \fp^* \otimes \Hs^\ell V
\]
as the image of $\varphi_{q,\ell}$ under the $G$-equivariant projection $\sss^\ell V \rightarrow \Hs^\ell V$ as in \eqref{eq:orth-proj-HlV}.
Via the isomorphism $\iota$, we also regard $\varphi'_{q, \ell}$ as an element in $S(V) \otimes \wedge^q \fp^* \otimes \Hs^\ell V$.
By \cite[Theorem 5.6]{funke-millson}, $\varphi'_{q, \ell}$ is invariant under the diagonal action of $K$ and 
\[
 (\omega(t) \otimes 1 \otimes 1) \varphi'_{q, \ell} = t^{\ell + \frac{m}{2}} \varphi'_{q, \ell}
\]
for $t \in K' \simeq \U(1)$.
By \cite[Theorem 5.7]{funke-millson}, $\varphi'_{q, \ell}$ defines a closed differential form on $G/K$.
Namely, $d \varphi'_{q, \ell} = 0$, where
\[
 d : (\Ps(\C^m) \otimes \wedge^q \fp^* \otimes \Hs^\ell V)^K
 \longrightarrow (\Ps(\C^m) \otimes \wedge^{q+1} \fp^* \otimes \Hs^\ell V)^K
\]
is the differential as in \cite[\S 5.1]{funke-millson}.

Similarly, if $p=0$, then we define 
\[
 \varphi_\ell \in \Ps(\C^m) \otimes \sss^{\ell} V
\]
by 
\[
 \varphi_\ell = \left( \frac{1}{4 \pi i} \right)^{\ell}
 \sum_{\beta} z_\beta \otimes e_\beta, 
\]
where the sum runs over $\beta = (\beta_1, \dots, \beta_\ell) \in \{ 1, \dots, m \}^\ell$.
Then we define 
\[
 \varphi_\ell' \in \Ps(\C^m) \otimes \Hs^{\ell} V
\]
as the image of $\varphi_\ell$ under the $G$-equivariant projection $\sss^\ell V \rightarrow \Hs^\ell V$ as in \eqref{eq:orth-proj-HlV}.
Via the isomorphism $\iota$, we also regard $\varphi'_\ell$ as an element in $S(V) \otimes \Hs^\ell V$.
Then $\varphi'_\ell$ is invariant under the diagonal action of $G$ and
\[
 (\omega(t) \otimes 1) \varphi'_\ell = t^{-\ell-\frac{m}{2}} \varphi'_\ell
\]
for $t \in K' \simeq \U(1)$.

\subsection{Restrictions and contractions}
\label{ss:Res-C}

For our applications, we consider a $6$-dimensional quadratic space $\tilde{V}$ over $\R$ of signature $(p,q)$, where 
\[
 (p,q) = (4,2) \text{ or } (0,6).
\]
Let $\tilde{G} = \O(\tilde{V})$ be the orthogonal group of $\tilde{V}$.
As in \S \ref{ss:groups}, we take a basis $\{e_1, \dots, e_6\}$ of $\tilde{V}$ and define a Cartan involution $\theta$ of $\tilde{G}$.
Let $\tilde{K}$ be the maximal compact subgroup of $\tilde{G}$ with respect to $\theta$.
Let $\tilde\fg = \tilde\fk \oplus \tilde\fp$ be the complexified Lie algebra of $\tilde{G}$, where $\tilde\fk$ and $\tilde\fp$ are $(+1)$-eigenspace and $(-1)$-eigenspace of $\theta$, respectively.

Put 
\[
 V = \R e_1 + \R e_2 + \R e_5 + \R e_6, \qquad 
 V_0 = \R e_3 + \R e_4,
\]
so that $\tilde{V} = V \oplus V_0$.
Let $G = \O(V)$ be the orthogonal group of $V$ and regard it as a subgroup of $\tilde{G}$.
The natural inclusion $V \hookrightarrow \tilde{V}$ induces a commutative diagram
\[
 \xymatrix{
 \sss^\ell V \ar[d] \ar[r] &
 \sss^\ell \tilde{V} \ar[d] \\
 \sss^{\ell-2} V \ar[r] &
 \sss^{\ell-2} \tilde{V}
 }
\]
(where the horizontal maps are the inclusions and the vertical maps are the contractions) and hence a $G$-equivariant inclusion
\[
 \Hs^\ell V \hookrightarrow \Hs^\ell \tilde{V}. 
\]
Also, the natural projection $\tilde{V} \twoheadrightarrow V$ induces a projection $\sss^\ell \tilde{V} \twoheadrightarrow \sss^\ell V$.
Composing this with the inclusion $\Hs^\ell \tilde{V} \hookrightarrow \sss^\ell \tilde{V}$ and the projection $\sss^\ell V \twoheadrightarrow \Hs^\ell V$, we obtain a $G$-equivariant projection
\begin{equation}
\label{eq:proj-Hs-ell}
 \Hs^\ell \tilde{V} \twoheadrightarrow \Hs^\ell V
\end{equation}
which restricts to the identity on $\Hs^\ell V$.

\subsubsection{The case $(p,q)=(4,2)$}
\label{sss:Res-C-1}

In this case, $\tilde\fp^*$ is equipped with a basis $\{\omega_{ij} \, | \, 1 \le i \le 4, \, 5 \le j \le 6 \}$ as in \S \ref{ss:groups}.
Let $\tilde{\fp}^* \twoheadrightarrow \fp^*$ be the $K$-equivariant projection induced by the natural inclusion $\fp \hookrightarrow \tilde{\fp}$.
This together with \eqref{eq:proj-Hs-ell} gives rise to a $K \times G'$-equivariant map
\[
 \Res:
 S(\tilde{V}) \otimes \wedge^2 \tilde{\fp}^* \otimes \Hs^\ell \tilde{V}
 \longrightarrow S(\tilde{V}) \otimes \wedge^2 \fp^* \otimes \Hs^\ell V.
\]
Here $K$ acts diagonally on all three factors, while $G'$ acts only on the first factor.
Choose an isomorphism $\wedge^4 \fp^* \simeq \C$ so that $\omega_{15} \wedge \omega_{25} \wedge \omega_{16} \wedge \omega_{26} \mapsto 1$.
This induces a non-degenerate $K$-invariant bilinear pairing $\cdot \wedge \cdot : \wedge^2 \fp^* \times \wedge^2 \fp^* \rightarrow \wedge^4 \fp^* \simeq \C$.
For $\bo \in \wedge^2 \fp^*$ and $\vvv \in \sss^\ell V$, we define a contraction
\[
 \ccc_{\bo,\vvv} : S(\tilde{V}) \otimes \wedge^2 \fp^* \otimes \sss^\ell V \longrightarrow S(\tilde{V})
\]
by $\ccc_{\bo,\vvv} = 1 \otimes ( \cdot \wedge \bo) \otimes \langle \cdot, \vvv \rangle$.

Let $\varphi'_{2,\ell} \in S(\tilde{V}) \otimes \wedge^2 \tilde{\fp}^* \otimes \Hs^\ell \tilde{V}$ be the Schwartz form as in \S \ref{ss:schwartzforms}.
We shall compute $\ccc_{\bo,\vvv}(\Res(\varphi'_{2,\ell}))$ for $\bo$ and $\vvv$ given as follows.
Put
\begin{align*}
 \omega^{++} & = \omega_{15} + i \omega_{25} + i \omega_{16} - \omega_{26}, \\
 \omega^{+-} & = \omega_{15} + i \omega_{25} - i \omega_{16} + \omega_{26}, \\
 \omega^{-+} & = \omega_{15} - i \omega_{25} + i \omega_{16} + \omega_{26}, \\
 \omega^{--} & = \omega_{15} - i \omega_{25} - i \omega_{16} - \omega_{26}.
\end{align*}
Then for $t = (t_1, t_2) \in T \simeq \U(1)^2$ (where we identify $\U(1)$ with $\SO(2)$ by $a + b i \mapsto \smat{a}{b}{-b}{a}$), we have
\[
 t \cdot \omega^{\epsilon_1\epsilon_2} = t_1^{\epsilon_1} t_2^{\epsilon_2} \omega^{\epsilon_1\epsilon_2}.
\]
In particular, 
\[
 (\fp^+)^* = \C \omega^{++} + \C \omega^{-+}, \qquad 
 (\fp^-)^* = \C \omega^{+-} + \C \omega^{--}.
\]
Hence we obtain a basis of $\wedge^2 \fp^*$ given by
\begin{align*}
 \omega^{++} \wedge \omega^{-+} & \in \wedge^2 (\fp^+)^*, &
 \omega^{++} \wedge \omega^{--}, \, \omega^{-+} \wedge \omega^{+-}
 & \in (\fp^+)^* \wedge (\fp^-)^*, \\
 \omega^{+-} \wedge \omega^{--} & \in \wedge^2 (\fp^-)^*, &
 \omega^{++} \wedge \omega^{+-}, \, \omega^{-+} \wedge \omega^{--}
 & \in (\fp^+)^* \wedge (\fp^-)^*.
\end{align*}
Note that in the context of the introduction, the above elements in $\wedge^2 (\fp^+)^*$, $(\fp^+)^* \wedge (\fp^-)^*$, $\wedge^2 (\fp^-)^*$ correspond to those in 
\[
 H^{2,0}(X_1 \times X_2), \qquad
 H^{1,1}(X_1 \times X_2), \qquad
 H^{0,2}(X_1 \times X_2),
\]
respectively, where $X_1$ and $X_2$ are quaternionic Shimura varieties.
Also, the elements in $(\fp^+)^* \wedge (\fp^-)^*$ in the first row correspond to those in 
\[
 H^{1,1}(X_1) \otimes H^{0,0}(X_2), \qquad
 H^{0,0}(X_1) \otimes H^{1,1}(X_2), 
\]
respectively, whereas the elements in $(\fp^+)^* \wedge (\fp^-)^*$ in the second row (which are the most relevant for us) correspond to those in 
\[
 H^{1,0}(X_1) \otimes H^{0,1}(X_2), \qquad
 H^{0,1}(X_1) \otimes H^{1,0}(X_2), 
\]
respectively.
From the representation-theoretic viewpoint, the former corresponds to the contribution of the trivial representation, whereas the latter corresponds to the contribution of holomorphic and anti-holomorphic vectors in the discrete series representation.
Put
\begin{align*}
 \bo^+ & = - \frac{1}{2i} \cdot \omega^{++} \wedge \omega^{+-} = (\omega_{15} + i \omega_{25}) \wedge (\omega_{16} + i \omega_{26}), \\
 \bo^- & = - \frac{1}{2i} \cdot \omega^{-+} \wedge \omega^{--} = (\omega_{15} - i \omega_{25}) \wedge (\omega_{16} - i \omega_{26}),
\end{align*}
and 
\[
 \vvv^+ = \frac{1}{\ell!} \cdot (e_1 + i e_2)^\ell, \qquad
 \vvv^- = \frac{1}{\ell!} \cdot (e_1 - i e_2)^\ell.
\]

\begin{prop}
\label{p:schwartz-rc1}
For $\epsilon = \pm$, we have
\[
 \ccc_{\bo^\epsilon,\vvv^\epsilon}(\Res(\varphi'_{2,\ell}))(x) = (x_1 + \epsilon ix_2)^{\ell+2} \cdot \varphi_0(x).
\]
\end{prop}

\begin{proof}
Consider the diagram
\[
 \xymatrix{
 \sss^\ell \tilde{V} \ar[d]_-p \ar[r]^-r &
 \sss^\ell V \ar[d]^-q \\
 \Hs^\ell \tilde{V} \ar[r]_-s &
 \Hs^\ell V
 },
\]
where $p,q,r,s$ are the projections.
By \cite[\S 19.5]{fulton-harris}, we have 
\begin{align*}
 \sss^\ell \tilde{V} & \simeq \tau_{(\ell,0,0)} \oplus \tau_{(\ell - 2,0,0)} \oplus \cdots \oplus \tau_{(\ell-2k,0,0)}, & 
 \sss^\ell V & \simeq \tau_{(\ell,0)} \oplus \tau_{(\ell - 2,0)} \oplus \cdots \oplus \tau_{(\ell-2k,0)}, \\
 \Hs^\ell \tilde{V} & \simeq \tau_{(\ell,0,0)}, &
 \Hs^\ell V & \simeq \tau_{(\ell,0)},
\end{align*}
where $\tau_\lambda$ denotes the irreducible representation with highest weight $\lambda$ and $k = [\frac{\ell}{2}]$.
Also, by \cite[\S 25.3]{fulton-harris}, we have
\[
 \tau_{(j,0,0)}|_{G} \simeq \tau_{(j,0)} \oplus \tau^{\oplus 2}_{(j-1,0)} \oplus \cdots \oplus \tau_{(0,0)}^{\oplus j+1}.
\]
Hence if $p(x) = 0$ for $x \in \sss^\ell \tilde{V}$, then $(q \circ r)(x) = 0$.
This implies that the above diagram is commutative.
Since $q$ is the orthogonal projection and $\vvv^\epsilon \in \Hs^\ell V$, we have
\[
 \langle (s \circ p)(x), \vvv^\epsilon \rangle = \langle (q \circ r)(x), \vvv^\epsilon \rangle = \langle r(x), \vvv^\epsilon \rangle
\]
and hence 
\[
 \ccc_{\bo^\epsilon,\vvv^\epsilon}(\Res(\varphi'_{2,\ell})) = \ccc_{\bo^\epsilon,\vvv^\epsilon}(\widetilde{\Res}(\varphi_{2,\ell})),
\]
where 
\[
 \widetilde{\Res}:
 S(\tilde{V}) \otimes \wedge^2 \tilde{\fp}^* \otimes \sss^\ell \tilde{V}
 \longrightarrow 
 S(\tilde{V}) \otimes \wedge^2 \fp^* \otimes \sss^\ell V
\]
is the natural projection.
By definition, we have
\[
 \ccc_{\bo^\epsilon,\vvv^\epsilon}(\widetilde{\Res}(\varphi_{2,\ell})) = 
 \left( \frac{1}{4 \pi i} \right)^{\ell+2}
 \sum_{\alpha} \sum_{\beta} z_\alpha z_\beta (\omega_\alpha \wedge \bo^\epsilon)
 \langle e_\beta, \vvv^\epsilon \rangle
\]
in $\Ps(\C^6)$, where the sums run over $\alpha = (\alpha_1, \alpha_2) \in \{ 1, 2 \}^2$ and $\beta = (\beta_1, \dots, \beta_\ell) \in \{ 1, 2 \}^\ell$.
It is easy to see that
\[
 \sum_{\alpha} z_\alpha (\omega_\alpha \wedge \bo^\epsilon) = (z_1 + \epsilon i z_2)^2,
 \qquad
 \sum_{\beta} z_\beta \langle e_\beta, \vvv^\epsilon \rangle = (z_1 + \epsilon iz_2)^\ell,
\]
so that 
\[
 \ccc_{\bo^\epsilon,\vvv^\epsilon}(\widetilde{\Res}(\varphi_{2,\ell})) = 
 \left( \frac{1}{4 \pi i} \right)^{\ell+2}
 (z_1 + \epsilon iz_2)^{\ell+2}.
\]
This combined with \eqref{eq:weil-fock-isom} gives the desired formula.
\end{proof}

\subsubsection{The case $(p,q)=(0,6)$}
\label{sss:Res-C-2}

In this case, we have $\tilde{G} = \tilde{K}$ and $\tilde{\fp}^* = \{ 0 \}$.
As above, \eqref{eq:proj-Hs-ell} gives rise to a $G \times G'$-equivariant map
\[
 \Res: S(\tilde{V}) \otimes \Hs^\ell \tilde{V}
 \longrightarrow S(\tilde{V}) \otimes \Hs^\ell V.
\]
Here $G$ acts diagonally on both factors, while $G'$ acts only on the first factor.
For $\vvv \in \sss^\ell V$, we define a contraction
\[
 \ccc_{\vvv} : S(\tilde{V}) \otimes \sss^\ell V \longrightarrow S(\tilde{V})
\]
by $\ccc_{\vvv} = 1 \otimes \langle \cdot, \vvv \rangle$.

Let $\varphi'_\ell \in S(\tilde{V}) \otimes \Hs^\ell \tilde{V}$ be the Schwartz form as in \S \ref{ss:schwartzforms}.
Then, as in Proposition \ref{p:schwartz-rc1}, we have:

\begin{prop}
\label{p:schwartz-rc2}
For $\epsilon = \pm$, we have
\[
 \ccc_{\vvv^\epsilon}(\Res(\varphi'_\ell))(x) = (-1)^\ell \cdot (x_1 + \epsilon ix_2)^\ell \cdot \varphi_0(x),
\]
where $\vvv^\epsilon = \frac{1}{\ell!} \cdot (e_1 + \epsilon i e_2)^\ell$.
\end{prop}

\section{Theta lifting}

In this section, we study global theta lifts for some quaternionic dual pairs. The material in this section will be needed 
 in \S \ref{sec:construction-nonvanishing} to globalize the construction of the Kudla-Millson cohomology classes from the previous section to the group $\tilde{\mathrsfs{G}}_B=\GU_B(\tilde{V})^0$ and 
 to show the non-vanishing of their restriction to a suitable subgroup. Moreover, we will also use it in \S \ref{sec:arthur-galois-hodge} to 
 study their associated Galois representations and to show that they lie in the $\C$-span of the Hodge classes. 

\subsection{Setup}
\label{ss:theta-setup}

Let $F$ be a number field and $\A = \A_F$ the ring of ad{\`e}les of $F$.
Let $B$ be a quaternion \emph{division} algebra over $F$ and $*$ the main involution on $B$.
Let $E$ be a quadratic extension of $F$ which embeds into $B$.
Fix a trace zero element $\i \in E^\times$ and write $\N = \N_{E/F}$ for the norm map from $E$ to $F$.
Let $\xi_E$ be the quadratic character of $\A^\times/F^\times$ associated to $E/F$ by class field theory.

Let $V$ be an $m$-dimensional right $B$-space equipped with a skew-hermitian form $\langle \cdot, \cdot \rangle : V \times V \rightarrow B$.
Let $W = B$ be a $1$-dimensional left $B$-space equipped with a hermitian form $\langle \cdot, \cdot \rangle : W \times W \rightarrow B$ given by
\[
 \langle x, y \rangle = x \cdot y^*.
\]
Then $\GU(W) \simeq B^\times$.
Put
\begin{align*}
 G & = \GU(V)^0, & G_1 & = \U(V)^0, \\
 H & = \GU(W), & H_1 & = \U(W),
\end{align*}
and
\[
 R = \{ (g,h) \in G \times H \, | \, \nu(g) = \nu(h) \},
\]
where $\nu$ denotes the similitude character.
Let $Z_G \simeq F^\times$ and $Z_H \simeq F^\times$ be the centers of $\GU(V)$ and $\GU(W)$, respectively.
Put
\[
 (\A^\times)^+ = \nu(G(\A)) \cap \nu(H(\A))
\]
and 
\begin{align*}
 G(\A)^+ & = \{ g \in G(\A) \, | \, \nu(g) \in (\A^\times)^+ \}, & 
 G(F)^+ & = G(F) \cap G(\A)^+, \\
 H(\A)^+ & = \{ h \in H(\A) \, | \, \nu(h) \in (\A^\times)^+ \}, & 
 H(F)^+ & = H(F) \cap H(\A)^+.
\end{align*}

Let $\V = V \otimes_B W$ be a $4m$-dimensional $F$-space equipped with a symplectic form
\[
 \llangle \cdot, \cdot \rrangle = \frac{1}{2} \tr_{B/F} \left( \langle \cdot, \cdot \rangle \otimes \langle \cdot, \cdot \rangle^* \right).
\]
Let $\Mp(\V)_\A$ be the metaplectic group:
\[
 1 \longrightarrow \C^1 \longrightarrow \Mp(\V)_\A \longrightarrow \Sp(\V)(\A) \longrightarrow 1.
\]
Fix a complete polarization $\V = \X \oplus \Y$.
Then we can realize the Weil representation $\omega_\psi$ of $\Mp(\V)_\A$ (relative to a non-trivial additive character $\psi$ of $\A/F$) on the Schwartz space $\SS(\X(\A))$.
Assume that there exists a homomorphism $\tilde{\imath} : R(\A) \rightarrow \Mp(\V)_\A$ such that the diagram
\[
 \xymatrix{
 R(F) \ar@{^{(}->}[r] \ar@{->}[d] &
 R(\A) \ar@{->}[d]^-{\tilde{\imath}} \\
 \Sp(\V)(F) \ar@{->}[r]^-i &
 \Mp(\V)_\A
 }
\]
is commutative, where $i$ is the canonical splitting.
Then for any $\varphi \in \SS(\X(\A))$, we may form a theta function on $R(\A)$:
\[
 \Theta_\varphi(g,h) = \sum_{x \in \X} \omega_\psi(\tilde{\imath}(g,h)) \varphi(x).
\]

\subsection{Theta lifts from $E^\times$ to $B^\times$}
\label{ss:theta-1}

Let $V = B$ be a $1$-dimensional right $B$-space equipped with a skew-hermitian form
\[
 \langle x, y \rangle = x^* \cdot \kappa \i \cdot y
\]
for some $\kappa \in F^\times$.
Then $\GU(V)^0 \simeq E^\times$, so that $(\A^\times)^+ = \N(\A_E^\times)$ and $G(\A)^+ = G(\A)$.
In Appendix \ref{sec:weil-hodge}, we define a splitting $\tilde{\imath} : R(\A) \rightarrow \Mp(\V)_\A$ as above.
Let $\eta$ be a character of $\A_E^\times/E^\times$.
We regard $\eta$ as an automorphic character of $G(\A)$.
For $\varphi \in \SS(\X(\A))$ and $h \in H(\A)^+$, put
\[
 \theta_\varphi(\eta)(h) = \int_{G_1(F) \backslash G_1(\A)} \Theta_\varphi(g_1 g, h) \eta(g_1 g) \, dg_1,
\]
where we choose $g \in G(\A)$ such that $\nu(g) = \nu(h)$.
Since $F^\times \cap \N(\A_E^\times) = \N(E^\times)$ and hence $\nu(H(F)^+) = \nu(G(F))$, this integral defines an automorphic form $\theta_\varphi(\eta)$ on $H(\A)^+$.
Since $[H(\A) : H(F) H(\A)^+] = [\A^\times : F^\times \N(\A_E^\times)] = 2$, we may extend $\theta_\varphi(\eta)$ to an automorphic form on $H(\A)$ by the natural embedding
\[
 H(F)^+ \backslash H(\A)^+ \hookrightarrow H(F) \backslash H(\A)
\]
and extension by zero.
Let $\theta(\eta)$ be the automorphic representation of $H(\A)$ generated by $\theta_\varphi(\eta)$ for all $\varphi \in \SS(\X(\A))$.

\begin{lem}
\label{l:theta-1}
Assume that:
\begin{itemize}
 \item $B_v$ is split for all archimedean places $v$ of $F$;
 \item $\eta_v$ does not factor through the norm map for any place $v$ of $F$ such that $B_v$ is ramified.
\end{itemize}
Then we have
\[
 \theta(\eta) = \pi(\eta)_B,
\]
where $\pi(\eta)$ is the automorphic induction of $\eta$ from $\GL_1(\A_E)$ to $\GL_2(\A)$ and $\pi(\eta)_B$ is its Jacquet-Langlands transfer to $B^\times(\A)$.
\end{lem}

\begin{proof}
Suppose that $\theta(\eta)$ is non-zero.
Let $v$ be a place of $F$.
If $B_v$ is split, then by \S \ref{ss:compatibility-hk-duke}, the splitting $\tilde{\imath} : R(F_v) \rightarrow \Mp(\V_v)$ agrees with the standard one for symplectic-orthogonal dual pairs.
Hence it follows from the local theta correspondence for unramified representations that for any irreducible component $\pi$ of $\theta(\eta)$, we have $\pi_v \simeq \pi(\eta_v)$ for almost all $v$, where $\pi(\eta_v)$ is the automorphic induction of $\eta_v$ from $\GL_1(E_v)$ to $\GL_2(F_v)$.
By the strong multiplicity one theorem, $\theta(\eta)$ is irreducible and $\theta(\eta) = \pi(\eta)_B$.

Thus it remains to show that $\theta(\eta)$ is non-zero.
Let $\VV$ and $\WW$ be the $1$-dimensional hermitian $E$-space and the $2$-dimensional skew-hermitian $E$-space, respectively, as in \S \ref{ss:doubling-U(WW)}.
Then $\GU(V)^0 = \GU(\VV)$ and $\GU(W) \hookrightarrow \GU(\WW)$.
By \S \ref{ss:compatibility-periods2}, the splitting $\tilde{\imath} : R(\A) \rightarrow \Mp(\V)_\A$ agrees with the restriction of the standard one $\G(\U(\VV) \times \U(\WW))(\A) \rightarrow \Mp(\V)_\A$ for unitary dual pairs.
Hence it suffices to show that the global theta lift of $\chi := \eta|_{\A_E^1}$ (regarded as an automorphic character of $\U(\VV)(\A)$) to $\U(\WW)(\A)$ is non-zero.
By assumption, we have $\chi \ne 1$, so that the standard $L$-function $L(s, \chi)$ is holomorphic and non-zero at $s=1$.
This together with the Rallis inner product formula \cite{ichino-mathZ, gqt, yamana} implies that the non-vanishing of the global theta lift $\theta(\chi)$ to $\U(\WW)(\A)$ is equivalent to the non-vanishing of the local theta lift $\theta(\chi_v)$ to $\U(\WW_v)$ for all $v$.
If $B_v$ is split, then $\theta(\chi_v)$ is non-zero since the dual pair $(\U(\VV_v), \U(\WW_v))$ is in the stable range \cite{li-invent}.
Suppose that $B_v$ is ramified, so that $v$ is non-archimedean.
Let $r^+(\chi_v)$ and $r^-(\chi_v)$ be the first occurrence indices:
\begin{align*}
 r^+(\chi_v) & = \min \{ r \, | \, \text{the theta lift of $\chi_v$ to $\U(\mathbf{H}_v^{\oplus r})$ is non-zero} \}, \\ 
 r^-(\chi_v) & = \min \{ r \, | \, \text{the theta lift of $\chi_v$ to $\U(\WW_v \oplus \mathbf{H}_v^{\oplus r-1})$ is non-zero} \},
\end{align*}
where $\mathbf{H}_v$ is the hyperbolic plane over $E_v$.
Since $\chi_v \ne 1$ by assumption, we have $r^+(\chi_v) = 1$.
On the other hand, we have
\[
 r^+(\chi_v) + r^-(\chi_v) = 2
\]
by the conservation relation \cite{sun-zhu}.
Hence we have $r^-(\chi_v) = 1$, so that $\theta(\chi_v)$ is non-zero.
This completes the proof.
\end{proof}

\subsection{Theta lifts from $B_1^\times \times B_2^\times$ to $B^\times$}
\label{ss:theta-2}

Let $V = B_1 \otimes_E B_2$ be the $2$-dimensional skew-hermitian right $B$-space as in \cite[\S 2.2]{periods1}, where $B_1$ and $B_2$ are quaternion algebras over $F$ such that $E$ embeds into $B_1$ and $B_2$, and such that $B_1 \cdot B_2 = B$ in the Brauer group.
Then $\GU(V)^0 \simeq (B_1^\times \times B_2^\times) / F^\times$, so that $(\A^\times)^+$ consists of elements $a \in \A^\times$ with $a_v > 0$ for all infinite places $v$ such that $B_{1,v}$ or $B_{2,v}$ or $B_v$ is ramified.
In \cite[Appendix C]{periods1}, we have defined a splitting $\tilde{\imath} : R(\A) \rightarrow \Mp(\V)_\A$ as above.
Let $\sigma_1$ and $\sigma_2$ be irreducible unitary cuspidal automorphic representations of $B_1^\times(\A)$ and $B_2^\times(\A)$, respectively.
We assume that they have the same central character, so that we may regard $\sigma_1 \boxtimes \sigma_2$ as an automorphic representation of $G(\A)$.
For $\varphi \in \SS(\X(\A))$, $f \in \sigma_1 \boxtimes \sigma_2$, and $h \in H(\A)^+$, put
\[
 \theta_\varphi(f)(h) = \int_{G_1(F) \backslash G_1(\A)} \Theta_\varphi(g_1 g, h) f(g_1 g) \, dg_1,
\]
where we choose $g \in G(\A)^+$ such that $\nu(g) = \nu(h)$.
Since $\nu(H(F)^+) = F^\times \cap (\A^\times)^+ = \nu(G(F)^+)$ by Eichler's norm theorem, this integral defines an automorphic form $\theta_\varphi(f)$ on $H(\A)^+$.
Since $H(F) H(\A)^+ = H(\A)$, we may extend $\theta_\varphi(f)$ to an automorphic form on $H(\A)$.
Let $\theta(\sigma_1 \boxtimes \sigma_2)$ be the automorphic representation of $H(\A)$ generated by $\theta_\varphi(f)$ for all $\varphi \in \SS(\X(\A))$ and $f \in \sigma_1 \boxtimes \sigma_2$.

\begin{lem}
\label{l:theta-2}
Let $\pi_i$ be the Jacquet-Langlands transfer of $\sigma_i$ to $\GL_2(\A)$.
\begin{enumerate}
\item If $\pi_1 \ne \pi_2$, then $\theta(\sigma_1 \boxtimes \sigma_2) = 0$.
\item If $\pi_1 = \pi_2$, then $\theta(\sigma_1 \boxtimes \sigma_2)$ is the Jacquet-Langlands transfer of $\pi_i$ to $B^\times(\A)$ (which exists).
\end{enumerate}
\end{lem}

\begin{proof}
Let $\sigma$ be an irreducible unitary cuspidal automorphic representation of $B^\times(\A)$ and $\pi$ its Jacquet-Langlands transfer to $\GL_2(\A)$.
For $\varphi \in \SS(\X(\A))$, $f \in \sigma_1 \boxtimes \sigma_2$, and $f' \in \bar{\sigma}$ (where $\bar{\sigma}$ is the complex conjugate of $\sigma$), we have a seesaw identity
\[
 \int_{Z_H(\A) H(F) \backslash H(\A)} \theta_{\varphi}(f)(h) \cdot f'(h) \, dh
 = \int_{Z_G(\A) G(F) \backslash G(\A)} f(g) \cdot \theta_{\varphi}(f')(g) \, dg,
\]
where $\theta_\varphi(f')$ is the theta lift of $f'$ to $G(\A)$ as in \cite[\S 4]{periods1}.
Since the theta lift of $\bar{\sigma}$ to $G(\A)$ is $\bar{\pi}_{B_1} \boxtimes \bar{\pi}_{B_2}$ by \cite[Proposition 4.2.3]{periods1}, 
where $\pi_{B_i}$ is the Jacquet-Langlands transfer of $\pi$ to $B_i^\times(\A)$ (if it exists), this integral vanishes unless $\sigma_i = \pi_{B_i}$.
In particular, (i) follows. 
Moreover, if $\sigma_i = \pi_{B_i}$, then we can find $\varphi$, $f$, and $f'$ such that the integral is non-zero.
This implies that 
\[
 \theta(\sigma_1 \boxtimes \sigma_2) = \sigma, 
\]
so that (ii) follows.
\end{proof}

\subsection{Theta lifts from $B^\times$ to an inner form of $\operatorname{GSO}(4,2)$}
\label{ss:theta-3}

Let $V$ be the $3$-dimensional skew-hermitian right $B$-space as in \S \ref{sec:constr-gl-ex-isom}.
Then $(\A^\times)^+ = \N(\A_E^\times)$ and $G(\A)^+ = G(\A)$.
In Appendix \ref{sec:weil-hodge}, we define a splitting $\tilde{\imath} : R(\A) \rightarrow \Mp(\V)_\A$ as above.
Let $\tau$ be an irreducible unitary automorphic representation of $H(\A)^+$.
For $\varphi \in \SS(\X(\A))$, $\phi \in \tau$, and $g \in G(\A)$, put
\[
 \theta_\varphi(\phi)(g) = \int_{H_1(F) \backslash H_1(\A)} \Theta_\varphi(g, h_1 h) \phi(h_1 h) \, dh_1,
\]
where we choose $h \in H(\A)^+$ such that $\nu(h) = \nu(g)$.
This integral defines an automorphic form $\theta_\varphi(\phi)$ on $G(\A)$.
Let $\theta(\tau)$ be the automorphic representation of $G(\A)$ generated by $\theta_\varphi(\phi)$ for all $\varphi \in \SS(\X(\A))$ and $\phi \in \tau$.

In the rest of this section, we assume that $\theta(\tau)$ is non-zero and cuspidal.
Note that $\theta(\tau)$ is automatically cuspidal if $\U(V)$ is anisotropic.
Then:

\begin{lem}
\label{lem:irred-theta-similitude}
The global theta lift $\theta(\tau)$ is irreducible and
\[
 \theta(\tau) \simeq \otimes_v \theta(\tau_v),
\]
where $\theta(\tau_v)$ is the local theta lift of $\tau_v$ (see the proof for its definition).
\end{lem}

\begin{proof}
As in \cite[Corollary 7.1.3]{kr94}, the assertion follows from the Howe duality, which we describe below.
For this, we fix a place $v$ of $F$ and suppress the subscript $v$ from the notation.
Also, we work with the category of Harish-Chandra modules if $F$ is archimedean.

Consider the compact induction
\[
 \Omega = \operatorname{ind}^{G\times H^+}_R \omega,
\]
where $\omega$ is the Weil representation of $R$ (relative to a fixed non-trivial character of $F$ and a fixed splitting over $R$).
For any irreducible representation $\tau$ of $H^+$, the maximal $\tau^\vee$-isotypic quotient of $\Omega$ is of the form
\[
 \Theta(\tau) \boxtimes \tau^\vee
\]
for some representation $\Theta(\tau)$ of $G$.
Then the Howe duality asserts that
\begin{enumerate}
 \item $\Theta(\tau)$ is of finite length;
 \item $\Theta(\tau)$ is zero or has a unique irreducible quotient $\theta(\tau)$;
 \item for any irreducible representations $\tau$ and $\tau'$ of $H^+$ which occur as quotients of $\Omega$, we have
\[
 \theta(\tau) \simeq \theta(\tau') \Longrightarrow \tau \simeq \tau'. 
\]
\end{enumerate}
This can be deduced from the Howe duality \cite{howe-jams, waldspurger-howe-duality, gan-takeda, gan-sun} for $(\U(V), \U(W))$ as follows.

We first show that the Howe duality for $(\U(V)^0, \U(W))$ follows from the Howe duality for $(\U(V), \U(W))$.
If $B$ is ramified, then there is nothing to prove since $\U(V)^0(F) = \U(V)(F)$.
If $B$ is split, then we have $\U(V) \simeq \O(V^\dagger)$ and $\U(W) \simeq \Sp(W^\dagger)$, where $V^\dagger$ and $W^\dagger$ are the $6$-dimensional quadratic $F$-space and the $2$-dimensional symplectic $F$-space, respectively, associated to $V$ and $W$ by Morita theory.
For brevity, we write $\Gs = \O(V^\dagger)$ and $\Gs^0 = \SO(V^\dagger)$.
Let $\sigma_0$ be an irreducible representation of $\Gs^0$.
Then $\sigma_0$ is an irreducible component of $\sigma|_{\Gs^0}$ for some irreducible representation $\sigma$ of $\Gs$.
Note that $\sigma$ is not necessarily uniquely determined.
Namely, $\sigma_0$ is also an irreducible component of $(\sigma \otimes \sgn)|_{\Gs^0}$, where $\sgn$ denotes the unique non-trivial character of $\Gs$ trivial on $\Gs^0$.
Fix $\varepsilon \in \Gs \smallsetminus \Gs^0$ and put $\sigma_0^\varepsilon(g) = \sigma_0(\varepsilon g \varepsilon^{-1})$ for $g \in \Gs^0$.
We have
\[
 \sigma_0 \simeq \sigma_0^\varepsilon \Longleftrightarrow \sigma \not\simeq \sigma \otimes \sgn, 
\]
and
\begin{itemize}
\item if $\sigma \not\simeq \sigma \otimes \sgn$, then 
\[
 \sigma|_{\Gs^0} = \sigma_0, \qquad
 \Ind^{\Gs}_{\Gs^0} \sigma_0 = \sigma \oplus (\sigma \otimes \sgn);
\]
\item if $\sigma \simeq \sigma \otimes \sgn$, then 
\[
 \sigma|_{\Gs^0} = \sigma_0 \oplus \sigma_0^\varepsilon, \qquad 
 \Ind^{\Gs}_{\Gs^0} \sigma_0 = \sigma.
\]
\end{itemize}
Then, by the conservation relation \cite{sun-zhu}, we have
\begin{itemize}
\item if $\sigma \not\simeq \sigma \otimes \sgn$, then at most one of $\sigma$ and $\sigma \otimes \sgn$ occurs as a quotient of $\omega$;
\item if $\sigma \simeq \sigma \otimes \sgn$, then $\sigma$ does not occur as a quotient of $\omega$.
\end{itemize}
This reduces the Howe duality for $(\U(V)^0, \U(W))$ to the Howe duality for $(\U(V), \U(W))$.

Finally, as in \cite{roberts}, \cite[\S 3]{gan-tantono} (noting that the projections $R \rightarrow G$ and $R\rightarrow H^+$ are surjective), the Howe duality for $(\GU(V)^0, \GU(W)^+)$ follows from the Howe duality for $(\U(V)^0, \U(W))$.
This completes the proof.
\end{proof}

Now we explicate the local theta lift $\theta(\tau_v)$ in the unramified case.
Fix a non-archimedean place $v$ of $F$ such that:
\begin{itemize}
 \item $F_v$ is of odd residual characteristic;
 \item $E_v$ is unramified over $F_v$;
 \item $B_v$ is split over $F_v$;
 \item $V_v^\dagger$ has a self-dual $\mathcal{O}_{F_v}$-lattice;
 \item $\psi_v$ is of order zero;
 \item $\tau_v$ is unramified.
\end{itemize}
Here $V_v^\dagger$ is the $6$-dimensional quadratic $F_v$-space associated to $V_v$.
For the moment, we suppress the subscript $v$ from the notation.
We may take a trace zero element $\i \in E^\times$ such that $u = \i^2 \in \mathcal{O}_F^\times$ and identify $G$ with the group
\[
 \{ g \in \GL_6(F) \, | \, {}^t g \mathcal{Q} g = \nu(g) \cdot \mathcal{Q}, \, \det g = \nu(g)^3 \},
\]
where
\[
 \mathcal{Q} = 
 \begin{pmatrix}
  & & & & & 1 \\
  & & & & 1 & \\
  & & 1 & & & \\
  & & & -u & & \\
  & 1 & & & & \\
  1 & & & & & 
 \end{pmatrix}.
\]
Let $B_G$ be a Borel subgroup of $G$ and $T$ a maximal torus of $G$ given by
\[
 B_G = \left\{ 
 \begin{pmatrix}
  * & * & * & * & * & * \\
  & * & * & * & * & * \\
  & & * & * & * & * \\
  & & * & * & * & * \\
  & & & & * & * \\
  & & & & & *
 \end{pmatrix}
 \right\}
 \quad \supset \quad
 T = \left\{ 
 \begin{pmatrix}
  * & & & & & \\
  & * & & & & \\
  & & * & * & & \\
  & & * & * & & \\
  & & & & * & \\
  & & & & & *
 \end{pmatrix}
 \right\}.
\]
Then we have an isomorphism $T \simeq (F^\times)^2 \times E^\times$ defined by
\[
 (t_1, t_2, a + b\i) \longmapsto
 \begin{pmatrix}
  t_1 & & & & & \\
  & t_2 & & & & \\
  & & a & bu & & \\
  & & b & a & & \\
  & & & & \nu t_2^{-1} & \\
  & & & & & \nu t_1^{-1}
 \end{pmatrix},
\]
where $\nu = a^2 - b^2 u$.
Also, we may identify $H$ with $\GL_2(F)$, so that
\[
 H^+ = \{ h \in \GL_2(F) \, | \, \det h \in \N(E^\times) \}.
\]
Let $B_H$ be the Borel subgroup of $H$ consisting of upper triangular matrices.
Recall that $\tau$ is an irreducible unramified representation of $H^+$.
Then $\tau$ is an irreducible component of 
\[
 \Ind^H_{B_H}(\chi_1 \otimes \chi_2)|_{H^+}
\]
for some unramified characters $\chi_1, \chi_2$ of $F^\times$.
Note that $\chi_1, \chi_2$ are not necessarily uniquely determined.
Namely, $\tau$ is also an irreducible component of $\Ind^H_{B_H}(\chi_1 \xi_E \otimes \chi_2 \xi_E)|_{H^+}$.
Then:

\begin{lem}
\label{l:theta-unram}
The local theta lift $\theta(\tau)$ is an irreducible component of
\[
 \Ind^G_{B_G}(\chi_1 \chi_2^{-1} \xi_E \otimes |\cdot| \otimes (\chi_2 |\cdot|^{-\frac{1}{2}}) \circ \N ).
\]
\end{lem}

\begin{proof}
By \S \ref{ss:compatibility-hk-duke}, the splitting $\tilde{\imath} : R \rightarrow \Mp(\V)$ agrees with the standard one for symplectic-orthogonal dual pairs.
Hence the assertion follows from the standard unramified computation.
We omit the details.
\end{proof}

Suppose again that $F$ is a number field.
We further assume that $B_v$ is split for all archimedean places $v$ of $F$ and that $\U(V)$ is anisotropic over $F$.
Then we show that the near equivalence class of $\theta(\tau)$ consists of automorphic representations $\theta(\tau')$, where $\tau'$ runs over automorphic representations in the near equivalence class of $\tau$.
Namely, we have:

\begin{prop}
\label{p:near-equiv-theta}
Let $\pi$ be an irreducible unitary automorphic representation of $G(\A)$ such that $\pi_v \simeq \theta(\tau_v)$ for almost all $v$.
Then there exists an irreducible automorphic representation $\tau'$ of $H(\A)^+$ such that $\tau_v' \simeq \tau_v$ for almost all $v$ and such that 
\[
 \pi = \theta(\tau').
\]
\end{prop}

To prove this proposition, we consider the theta lift in the opposite direction.
For $\varphi \in \SS(\X(\A))$, $f \in \pi$, and $h \in H(\A)^+$, put
\[
 \theta_\varphi(\bar{f})(h) = \int_{G_1(F) \backslash G_1(\A)} \Theta_\varphi(g_1 g, h) \overline{f(g_1 g)} \, dg_1,
\]
where we choose $g \in G(\A)$ such that $\nu(g) = \nu(h)$.
This integral defines an automorphic form $\theta_\varphi(\bar{f})$ on $H(\A)^+$.
Let $\theta(\bar{\pi})$ be the automorphic representation of $H(\A)^+$ generated by $\theta_\varphi(\bar{f})$ for all $\varphi \in \SS(\X(\A))$ and $f \in \pi$.

\begin{lem}
\label{l:theta-nonzero}
We have
\[
 \theta(\bar{\pi}) \ne 0.
\]
\end{lem}

Now Proposition \ref{p:near-equiv-theta} is an immediate consequence of Lemma \ref{l:theta-nonzero}.
Indeed, as in Lemma \ref{lem:irred-theta-similitude}, it follows from the Howe duality that $\theta(\bar{\pi})$ is irreducible and $\theta(\bar{\pi}) \simeq \otimes_v \theta(\bar{\pi}_v)$, where $\theta(\bar{\pi}_v)$ is the local theta lift of $\bar{\pi}_v$.
Since $\pi_v \simeq \theta(\tau_v)$ for almost all $v$ by assumption, we have
\[
 \theta(\bar{\pi}_v) \simeq \bar{\tau}_v
\]
for almost all $v$.
Hence $\tau' = \overline{\theta(\bar{\pi})}$ satisfies the desired condition.

Lemma \ref{l:theta-nonzero} can be deduced from the Rallis inner product formula as follows.

\subsection{Proof of Lemma \ref{l:theta-nonzero}}

\subsubsection{Doubled spaces}

Let $V^\square = V \oplus V$ and $\V^\square = \V \oplus \V = V^\square \otimes_B W$ be the doubled spaces as in \S \ref{ss:doubling-U(V)}.
Let $\iota : \U(V) \times \U(V) \hookrightarrow \U(V^\square)$ be the natural embedding.
Let $\V^\square = \X^\square \oplus \Y^\square = \V^{\bigtriangledown} \oplus \V^{\triangle}$ be complete polarizations defined by
\begin{align*}
 \X^\square & = \X \oplus \X, & 
 \V^{\bigtriangledown} & = V^{\bigtriangledown} \otimes_B W, &
 V^{\bigtriangledown} & = \{ (v,-v) \, | \, v \in V \}, \\
 \Y^\square & = \Y \oplus \Y, &
 \V^{\triangle} & = V^{\triangle} \otimes_B W, &
 V^{\triangle} & = \{ (v,v) \, | \, v \in V \}.
\end{align*}
Let $\omega_\psi^\square$ be the Weil representation of $\U(V^\square)(\A) \times \U(W)(\A)$ relative to $\psi$ realized on the Schwartz space $\SS(\V^\bigtriangledown(\A))$ as in \cite[\S 5]{kudla-splitting}.
Then for any $\varphi \in \SS(\V^\bigtriangledown(\A))$, we may form a theta function on $\U(V^\square)(\A) \times \U(W)(\A)$:
\[
 \Theta_\varphi^\square(g,h) = \sum_{x \in \V^\bigtriangledown(F)} \omega_\psi^\square(g,h) \varphi(x).
\]
We define a partial Fourier transform
\[
 \mathrsfs{F} : \SS(\X^\square(\A)) = \SS(\X(\A)) \otimes \SS(\X(\A)) \longrightarrow \SS(\V^\bigtriangledown(\A))
\]
as in \cite[\S 4.1.1]{periods1}.
It follows from the definition of $\tilde{\imath}$, combined with the analog of \cite[Proposition 2.2]{hks} for $(\U(V), \U(W))$, that $\mathrsfs{F}$ induces an isomorphism
\[
 (\omega_\psi \circ \tilde{\imath}) \otimes (\bar{\omega}_\psi \circ \tilde{\imath}) \simeq \omega_\psi^\square \circ (\iota \otimes \id)
\]
as representations of $G_1(\A) \times G_1(\A) \times H_1(\A)$.
Hence we have
\begin{equation}
\label{eq:theta-function-decomp} 
 \Theta_{\varphi}^\square(\iota(g_1,g_2), h) = 
 \Theta_{\varphi_1}(g_1, h) \overline{\Theta_{\varphi_2}(g_2, h)}
\end{equation}
for $\varphi = \mathrsfs{F}(\varphi_1 \otimes \bar{\varphi}_2)$ with $\varphi_1, \varphi_2 \in \SS(\X(\A))$, $g_1, g_2 \in G_1(\A)$, and $h \in H_1(\A)$.

\subsubsection{Degenerate principal series representations}

Write
\[
 \U(V^\square) = 
 \left\{ g \in \GL_6(B) \, \left| \,
 g \begin{pmatrix} & \1_3 \\ -\1_3 & \end{pmatrix} {}^t g^*
 = \begin{pmatrix} & \1_3 \\ -\1_3 & \end{pmatrix}
 \right. \right\}
\]
as in \S \ref{ss:doubling-U(V)} and put
\[
 G_1^\square = \U(V^\square)^0.
\]
Let $P$ be the Siegel parabolic subgroup of $G_1^\square$ stabilizing $V^\triangle$:
\[
 P = \left\{ \left.
 \begin{pmatrix}
  a & * \\
  & ({}^t a^*)^{-1}
 \end{pmatrix}
 \, \right| \, a \in \GL_3(B) \right\}.
\]
For $s \in \C$, let $\II(s)$ be the degenerate principal series representation of $G_1^\square(\A)$ consisting of smooth functions $\FF$ on $G_1^\square(\A)$ which satisfy
\[
 \FF
 \left( 
 \begin{pmatrix}
  a & * \\
  & ({}^t a^*)^{-1}
 \end{pmatrix}
 g \right) = |\nu(a)|^{s+\frac{5}{2}} \cdot \FF(g).
\]
For a holomorphic section $\FF = \FF(\cdot, s)$ of $\II(s)$, we define an Eisenstein series $E(s,\FF)$ on $G_1^\square(\A)$ by (the meromorphic continuation of)
\[
 E(g, s,\FF) = \sum_{\gamma \in P(F) \backslash G_1^\square(F)} \FF(\gamma g, s).
\]

For each place $v$ of $F$, let $\II_v(s)$ be the degenerate principal series representation of $G_{1,v}^\square$ given similarly as above.
We define an intertwining operator
\[
 M_v(s) : \II_v(s) \longrightarrow \II_v(-s)
\]
by (the meromorphic continuation of)
\[
 (M_v(s) \FF)(g,-s) = \int_{U_v} \FF \left(
 \begin{pmatrix}
  & -\1_3 \\
  \1_3 & 
 \end{pmatrix}
 ug, s \right) du, 
\]
where $U$ is the unipotent radical of $P$.
Let 
\[
 M^*_v(s) = \frac{b_v(s)}{a_v(s)} \cdot M_v(s)
\]
be the normalized intertwining operator, where 
\begin{align*}
 a_v(s) & = \zeta_v(2s) \zeta_v(2s-2) \zeta_v(2s-4), \\
 b_v(s) & = \zeta_v(2s+1) \zeta_v(2s+3) \zeta_v(2s+5).
\end{align*}
By \cite[Proposition 4.11(2)]{yamana:sw}, $M^*_v(s)$ is holomorphic for $\Re s \ge 0$.

Recall the Weil representation $\omega^\square_{\psi_v}$ of $\U(V_v^\square) \times \U(W_v)$ relative to $\psi_v$ realized on the Schwartz space $\SS(\V_v^\bigtriangledown)$.
For $\varphi \in \SS(\V_v^\bigtriangledown)$, we define $\FF_\varphi \in \II_v(-\frac{3}{2})$ by 
\[
 \FF_\varphi(g) = \omega^\square_{\psi_v}(g) \varphi(0).
\]
We denote by $\mathcal{R}(W_v)$ the subspace of $\II_v(-\frac{3}{2})$ spanned by $\FF_\varphi$ for all $\varphi \in \SS(\V_v^\bigtriangledown)$.

\begin{lem}
\label{eq:image_of_int_op}
We have
\[
 \Im M^*_v(\tfrac{3}{2}) = \mathcal{R}(W_v).
\]
\end{lem}

\begin{proof}
If $B_v$ is ramified (so that $v$ is non-archimedean by assumption and $\U(V_v^\square) = \U(V_v^\square)^0$), then the assertion is proved in \cite[Theorem 1.3]{yamana:dps}.
Assume that $B_v$ is split. 
Let $\tilde{W}_v$ be the unique $4$-dimensional hermitian left $B_v$-space and define the subspace $\mathcal{R}(\tilde{W}_v)$ of $\II_v(\frac{3}{2})$ similarly as above.
Then we have 
\[
 \mathcal{R}(\tilde{W}_v) = \II_v(\tfrac{3}{2}), \qquad
 M^*_v(\tfrac{3}{2}) \mathcal{R}(\tilde{W}_v) = \mathcal{R}(W_v)
\]
by \cite[Theorem 1.6]{yamana:dps}, \cite[Appendix A]{loke}, \cite[Proposition 4.11(3)]{yamana:sw}.
We remark that, in these references, the results are stated for the degenerate principal series representation of $\U(V_v^\square)$, but the above equalities can be deduced by restriction to $\U(V_v^\square)^0$.
This completes the proof.
\end{proof}

\subsubsection{The doubling method}

We denote by $\Res^{G}_{G_1}(\pi)$ the restriction of $\pi$ to $G_1(\A)$ as functions.
Fix an irreducible component $\sigma$ of $\Res^{G}_{G_1}(\pi)$.
Note that $\sigma_v$ is the irreducible unramified component of
\begin{equation}
\label{eq:sigma-unram}
 \Ind^{G_{1,v}}_{B_{G_{1,v}}}(\chi_{1,v} \chi_{2,v}^{-1} \xi_{E_v} \otimes |\cdot|_v \otimes 1) 
\end{equation}
for almost all $v$, where $B_{G_{1,v}}$ is a Borel subgroup of $G_{1,v}$ with Levi component $F_v^\times \times F_v^\times \times E^1_v$.
Let $\langle \cdot, \cdot \rangle$ be the Petersson inner product on $\sigma$ given by
\[
 \langle f_1, f_2 \rangle = \int_{G_1(F) \backslash G_1(\A)} f_1(g) \overline{f_2(g)} \, dg,
\]
where $dg$ is the Tamagawa measure on $G_1(\A)$.
Fix decompositions $\langle \cdot, \cdot \rangle = \prod_v \langle \cdot, \cdot \rangle_v$ and $dg = \prod_v dg_v$, where $\langle \cdot, \cdot \rangle_v$ is an invariant hermitian inner product on $\sigma_v$ and $dg_v$ is a Haar measure on $G_{1,v}$.

We now consider the doubling zeta integral of Piatetski-Shapiro and Rallis \cite{psr,li-crelle,lapid-rallis,kakuhama} given by
\[
 Z(s,\FF,f_1,f_2) = \int_{G_1(F) \backslash G_1(\A)} \int_{G_1(F) \backslash G_1(\A)} E(\iota(g_1,g_2), s, \FF) \overline{f_1(g_1)} f_2(g_2) \, dg_1 \, dg_2
\]
for a holomorphic section $\FF = \otimes_v \FF_v$ of $\II(s)$ and $f_1 = \otimes_v f_{1,v}, f_2 = \otimes_v f_{2,v} \in \sigma$.
Recalling \eqref{eq:sigma-unram}, we have
\begin{equation}
\label{eq:doubling-zeta-main-id}
 Z(s,\FF,f_1,f_2) = \frac{L^S(s+\frac{1}{2}, \operatorname{ad} \tau \times \xi_E) \zeta^S(s+\frac{3}{2}) \zeta^S(s+\frac{1}{2}) \zeta^S(s-\frac{1}{2})}{\zeta^S(2s+1) \zeta^S(2s+3) \zeta^S(2s+5)} \cdot \prod_{v \in S} Z_v(s,\FF_v,f_{1,v},f_{2,v}),
\end{equation}
where $S$ is a sufficiently large finite set of places of $F$ and $Z_v(s,\FF_v,f_{1,v},f_{2,v})$ is the local zeta integral given by 
\[
 Z_v(s,\FF_v,f_{1,v},f_{2,v}) = \int_{G_{1,v}} \FF_v(\iota(g_v, 1), s) \overline{\langle \sigma_v(g_v) f_{1,v}, f_{2,v} \rangle_v} \, dg_v.
\]
Moreover, as in \cite[Theorem 3.2.2]{kr90}, \cite[Proposition 7.2.1]{kr94}, we can prove the following.

\begin{lem}
\label{l:local_zeta_nonvanish}
There exist a holomorphic section $\FF_v$ of $\II_v(s)$ and $f_{1,v}, f_{2,v} \in \sigma_v$ such that $Z_v(s,\FF_v,f_{1,v},f_{2,v})$ is holomorphic and non-zero at $s=\frac{3}{2}$.
\end{lem}

\subsubsection{The Rallis inner product formula}

By \eqref{eq:doubling-zeta-main-id} and Lemma \ref{l:local_zeta_nonvanish}, there exist a holomorphic section $\FF = \otimes_v \FF_v$ of $\II(s)$ and $f_1 = \otimes_v f_{1,v}, f_2 = \otimes_v f_{2,v} \in \sigma$ such that
\[
 \Res_{s=\frac{3}{2}} Z(s,\FF,f_1,f_2) \ne 0.
\]
In fact, $E(s,\FF)$ has a simple pole at $s=\frac{3}{2}$ by \cite[Theorem 3.1]{yamana:sw} and its residue can be described as follows.
By Lemma \ref{eq:image_of_int_op}, we have $M_v^*(\frac{3}{2}) \FF_v = \FF_{\varphi_v}$ for some $\varphi_v \in \SS(\V_v^\bigtriangledown)$.
Put $\varphi = \otimes_v \varphi_v \in \SS(\V^\bigtriangledown(\A))$.
We define an automorphic form $I(\varphi)$ on $G_1^\square(\A)$ by 
\[
 I(g,\varphi) = \int_{H_1(F) \backslash H_1(\A)} \Theta_{\varphi}^\square(g,h) \, dh.
\]
Then, by the Siegel-Weil formula \cite[Theorem 7.11]{yamana:sw}, we have
\[
 \Res_{s=\frac{3}{2}} E(s,\FF) = I(\varphi)  
\]
up to a non-zero constant.
Hence we have
\[
 \int_{G_1(F) \backslash G_1(\A)}  \int_{G_1(F) \backslash G_1(\A)}
 I(\iota(g_1,g_2), \varphi) \overline{f_1(g_1)} f_2(g_2) \, dg_1 \, dg_2 \ne 0.
\]
We may further assume that $\varphi = \mathrsfs{F}(\varphi_1 \otimes \bar{\varphi}_2)$ for some $\varphi_1, \varphi_2 \in \SS(\X(\A))$.
Then, by \eqref{eq:theta-function-decomp}, the left-hand side is equal to
\begin{multline*}
 \int_{G_1(F) \backslash G_1(\A)} \int_{G_1(F) \backslash G_1(\A)} \int_{H_1(F) \backslash H_1(\A)} \Theta_{\varphi_1}(g_1, h) \overline{\Theta_{\varphi_2}(g_2, h)} \overline{f_1(g_1)} f_2(g_2) \, dh \, dg_1 \, dg_2 \\
 = \int_{H_1(F) \backslash H_1(\A)} \theta_{\varphi_1}(\bar{\tilde{f}}_1)(h) \overline{\theta_{\varphi_2}(\bar{\tilde{f}}_2)(h)} \, dh,
\end{multline*}
where we choose $\tilde{f}_i \in \pi$ such that $\tilde{f}_i|_{G_1(\A)} = f_i$.
Hence we have $\theta_{\varphi_i}(\bar{\tilde{f}}_i) \ne 0$.
This completes the proof of Lemma \ref{l:theta-nonzero} and hence of Proposition \ref{p:near-equiv-theta}.

\section{Construction of the cohomology class and non-vanishing of its restriction}
\label{sec:construction-nonvanishing}

\subsection*{Notation}

For any reductive algebraic group $G$ over a number field $F$, we denote by $\mathrsfs{A}(G)$ the space of automorphic forms on $G(\A)$.

\subsection{Groups}

Let $F$ be a totally real number field. 
We denote by $\A = \A_F$ and $\A_f = \A_{F,f}$ the rings of ad{\`e}les and finite ad{\`e}les of $F$, respectively.
Let $B$ be a quaternion division algebra over $F$.
We assume that $B_v$ is split for all real places $v$ of $F$.
Let $E$ be a totally imaginary quadratic extension of $F$ which embeds into $B$.
We write $E = F + F \i$ and $B = E + E\j$ for some trace zero elements $\i \in E^\times$ and $\j \in B^\times$.
Put $u = \i^2 \in F^\times$ and $J = \j^2 \in F^\times$.
Let $\tilde{V} = V^\sharp \oplus V_0^\sharp$ be the $3$-dimensional skew-hermitian right $B$-space as in \S \ref{sec:constr-gl-ex-isom}, where $V^\sharp$ and $V_0^\sharp$ are the $2$- and $1$-dimensional subspaces as in \S \ref{ss:sum-of-2-dim}, respectively.
To ease notation, we write $V = V^\sharp$ and $V_0 = V_0^\sharp$.
Recall from Example \ref{eg:skew-herm-periods1} and \cite[\S 2.2]{periods1} that we may write $V = B_1 \otimes_E B_2$ for some quaternion algebras $B_1$ and $B_2$ over $F$ such that $B_1 \cdot B_2 = B$ in the Brauer group and such that $E$ embeds into $B_1$ and $B_2$.
In particular, $B_1$ and $B_2$ act on $V$ by left multiplication.
Put
\[
 \tilde{G} = \GU(\tilde{V})^0, \qquad
 G = \GU(V)^0 \simeq (B_1^\times \times B_2^\times)/F^\times, \qquad
 G_0 = \GU(V_0)^0 \simeq E^\times.
\]
Let $\tilde{Z} \simeq F^\times$ and $Z \simeq F^\times$ be the centers of $\tilde{G}$ and $G$, respectively.
We define a subgroup $\GG$ of $G \times G_0$ by
\[
 \GG = \G(\U(V) \times \U(V_0))^0 = \{ (g, \alpha) \, | \, \nu(g) = \N(\alpha) \},
\]
where $\nu$ is the similitude character and $\N = \N_{E/F}$ is the norm map.
We also regard $\GG$ as a subgroup of $\tilde{G}$ via the natural embedding.
Let $\ZZ \subset Z \times G_0$ be the center of $\GG$:
\[
 \ZZ \simeq \{ (z, \alpha) \, | \, z^2 = \N(\alpha) \}.
\]
Then we have a natural embedding $\tilde{Z} \hookrightarrow \ZZ$ and an exact sequence
\[
 1 \longrightarrow \tilde{Z} \longrightarrow \ZZ \overset{\mathbf{p}}\longrightarrow E^1 \longrightarrow 1,
\]
where $\mathbf{p}(z,\alpha) = z^{-1} \alpha$.

Let $W$ be the $1$-dimensional hermitian left $B$-space as in \S \ref{ss:theta-setup}.
Put
\[
 H = \GU(W) \simeq B^\times.
\]
Let $Z_H \simeq F^\times$ be the center of $H$.

\subsection{Weil representations}

Let $\tilde{\V} = \V \oplus \V_0$ be the $12$-dimensional symplectic $F$-space given by
\[
 \tilde{\V} = \tilde{V} \otimes_B W, \qquad
 \V = V \otimes_B W, \qquad
 \V_0 = V_0 \otimes_B W.
\]
As in \S \ref{ss:weil-hodge-setup}, we take complete polarizations
\[
 \tilde{\V} = \tilde{\X} \oplus \tilde{\Y}, \qquad
 \V = \X \oplus \Y, \qquad
 \V_0 = \X_0 \oplus \Y_0
\]
such that 
\[
 \tilde{\X} = \X \oplus \X_0, \qquad
 \tilde{\Y} = \Y \oplus \Y_0.
\]
By Appendix \ref{sec:weil-hodge} and \cite[Appendix C]{periods1}, we may define Weil representations $\omega$ (relative to the standard additive character $\psi$ of $\A/F$) of
\[
 \G(\U(\tilde{V}) \times \U(W))^0(\A), \qquad 
 \G(\U(V) \times \U(W))^0(\A), \qquad 
 \G(\U(V_0) \times \U(W))^0(\A)
\]
on 
\[
 \SS(\tilde{\X}(\A)), \qquad
 \SS(\X(\A)), \qquad
 \SS(\X_0(\A)),
\]
respectively, satisfying various compatibilities.

\subsection{Real groups}
\label{ss:notation-at-real-places}

Let $\Sigma_\infty$ be the set of real places of $F$ and $\Sigma$ the subset of $v \in \Sigma_\infty$ such that $B_{1,v}$ and $B_{2,v}$ are split.
We assume that $\Sigma \ne \Sigma_\infty$.
Put $d = |\Sigma|$.
For any $v \in \Sigma_\infty$, we may write $J = t_v^2$ for some $t_v \in F_v^\times$ since $B_v$ is split.
We define an isomorphism $\ii_v:B_v \rightarrow \M_2(F_v)$ of quaternion $F_v$-algebras by
\[
 \ii_v(a + b \i + c \j + d \i \j) =
 \begin{pmatrix}
  a + ct_v & b - dt_v \\
  (b + dt_v)u  & a - ct_v
 \end{pmatrix}.
\]
Put
\[
 e_v = \frac{1}{2} + \frac{t_v}{2J} \j, \qquad
 e_v' = \frac{1}{2} \i - \frac{t_v}{2J} \i \j, \qquad
 e_v'' =  \frac{1}{2u} \i + \frac{t_v}{2uJ} \i \j, \qquad
 e_v^* = \frac{1}{2} - \frac{t_v}{2J} \j,
\]
so that
\[
 \ii_v(e_v) = \mat{1}{0}{0}{0}, \qquad
 \ii_v(e_v') = \mat{0}{1}{0}{0}, \qquad
 \ii_v(e_v'') = \mat{0}{0}{1}{0}, \qquad
 \ii_v(e_v^*) = \mat{0}{0}{0}{1}.
\]
Note that
\[
 \begin{bmatrix}
  e_v \cdot x \\
  e_v' \cdot x
 \end{bmatrix}
 = \ii_v(x) \cdot 
 \begin{bmatrix}
  e_v \\
  e_v'
 \end{bmatrix}
\]
for $x \in B_v$.

Let $v \in \Sigma_\infty$.
Let $\tilde{V}_v^\dagger = V_v^\dagger \oplus V_{0,v}^\dagger$ be the $6$-dimensional quadratic $F_v$-space as in \cite[\S C.2]{periods1} associated to the $B_v$-space $\tilde{V}_v = V_v \oplus V_{0,v}$.
By \S \ref{ss:local-ex-isom-real}, the signature of $\tilde{V}_v^\dagger$ is equal to 
\[
\begin{cases}
 (4,2) & \text{if $v \in \Sigma$;} \\
 (0,6) & \text{if $v \in \Sigma_\infty \smallsetminus \Sigma$.}
\end{cases} 
\]
As in \S \ref{ss:groups}, we take a basis of $\tilde{V}_v^\dagger$ so that we have identifications
\[
 \tilde{G}_v = \GSO(4,2), \qquad
 G_v = \GSO(2,2), \qquad
 G_{0,v} = \GSO(2,0)
\]
if $v \in \Sigma$ and
\[
 \tilde{G}_v = \GSO(0,6), \qquad
 G_v = \GSO(0,4), \qquad
 G_{0,v} = \GSO(0,2)
\]
if $v \in \Sigma_\infty \smallsetminus \Sigma$.
Here
\[
 \GSO(p,q) = \{ g \in \GL_{p+q}(\R) \, | \, {}^t g I_{p,q} g = \nu(g) \cdot I_{p,q}, \, \det g = \nu(g)^{\frac{p+q}{2}} \}
\]
with 
\[
 I_{p,q} = 
 \begin{pmatrix}
  \1_p & \\
  & -\1_q
 \end{pmatrix}
\]
if $p+q$ is even.
Let $\tilde{\fg}_v = \tilde{\fk}_v \oplus \tilde{\fp}_v$ and $\fg_v = \fk_v\oplus \fp_v$ be the complexified Lie algebra of $\tilde{G}_v$ and $G_v$, respectively, where $\tilde{\fk}_v$ and $\fk_v$ (resp.~$\tilde{\fp}_v$ and $\fp_v$) are the $(+1)$-eigenspaces (resp.~the $(-1)$-eigenspaces) of the Cartan involutions as in \S \ref{ss:groups}.
Put
\[
 \tilde{\fp} = \prod_{v \in \Sigma} \tilde{\fp}_v, \qquad
 \fp = \prod_{v \in \Sigma} \fp_v.
\]
Let
\[
 \iota_v:\C^\times \times \C^\times \simeq E_v^\times \times E_v^\times \longrightarrow (B_{1,v}^\times \times B_{2,v}^\times)/F_v^\times \simeq G_v
\]
be a map induced by the isomorphism $E_v \simeq \C$ given by $a+b\i \mapsto a + b |u|_{F_v}^{\frac{1}{2}} i$ for $a, b \in F_v = \R$ and the fixed embeddings $\iota_1:E \hookrightarrow B_1$ and $\iota_2:E \hookrightarrow B_2$.
We explicate $\iota_v$ below.
Recall from Example \ref{eg:skew-herm-periods1} that $V = \e_1 B + \e_2 B$ is equipped with a skew-hermitian form
\[
 \langle \e_1 x_1 + \e_2 x_2, \e_1 y_1 + \e_2 y_2 \rangle = x_1^* \cdot \i \cdot y_1 - x_2^* \cdot J_1 \i \cdot y_2,
\]
where $\e_1 = 1 \otimes 1$ and $\e_2 = \j_1 \otimes 1$.
We take a basis
\begin{align*}
 e_{1,v} & = \sqrt{2} \cdot |uJ_1|_{F_v}^{-\frac{1}{2}} \cdot \e_2 e_v, & 
 e_{2,v} & = \sqrt{2} \cdot |J_1|_{F_v}^{-\frac{1}{2}} \cdot \e_2 e''_v, \\
 e_{3,v} & = \sqrt{2} \cdot |u|_{F_v}^{-\frac{1}{2}} \cdot \e_1 e_v, &
 e_{4,v} & = \sqrt{2} \cdot \e_1 e''_v
\end{align*}
of $V_v^\dagger$, so that
\[
 (\langle e_{i,v}, e_{j,v} \rangle^\dagger) =
 \begin{cases}
  I_{2,2} & \text{if $v \in \Sigma$;} \\
  I_{0,4} & \text{if $v \in \Sigma_\infty \smallsetminus \Sigma$.}
 \end{cases}
\]
Since 
\begin{align*}
 \iota_1(\i) \e_1 & = \e_1 \i, &
 \iota_1(\i) \e_2 & = - \e_2 \i, \\
 \iota_2(\i) \e_1 & = \e_1 \i, &
 \iota_2(\i) \e_2 & = \e_2 \i,
\end{align*}
and 
\[
 \i e_v = u e_v'', \qquad
 \i e_v'' = e_v,
\]
we have
\begin{equation}
\label{eq:i-actions}
\begin{aligned}
 \iota_1(i) e_{1,v} & = e_{2,v}, \qquad &
 \iota_1(i) e_{2,v} & = - e_{1,v}, \qquad &
 \iota_1(i) e_{3,v} & = - e_{4,v}, \qquad &
 \iota_1(i) e_{4,v} & = e_{3,v}, \\
 \iota_2(i) e_{1,v} & = - e_{2,v}, \qquad &
 \iota_2(i) e_{2,v} & = e_{1,v}, \qquad &
 \iota_2(i) e_{3,v} & = - e_{4,v}, \qquad &
 \iota_2(i) e_{4,v} & = e_{3,v},
\end{aligned}
\end{equation}
where $i = |u|_{F_v}^{-\frac{1}{2}} \cdot \i$.
Hence 
\[
 \iota_v(a_1 + b_1 i, a_2 + b_2 i) = 
 \begin{pmatrix}
  a_1 & -b_1 & & \\
  b_1 & a_1 & & \\
  & & a_1 & b_1 \\
  & & -b_1 & a_1
 \end{pmatrix}
 \begin{pmatrix}
  a_2 & b_2 & & \\
  -b_2 & a_2 & & \\
  & & a_2 & b_2 \\
  & & -b_2 & a_2
 \end{pmatrix}.
\]

Also, let $W_v^\dagger$ be the $2$-dimensional symplectic $F_v$-space as in \cite[\S C.2]{periods1} associated to the $B_v$-space $W_v$.
Using a basis $e_v, e_v'$ of $W_v^\dagger$, we identify $H_v$ with $\GL_2(\R)$.
We embed $\C^\times$ into $H_v = \GL_2(\R)$ by
\[
 a+bi \longmapsto \mat{a}{b}{-b}{a}.
\]
Since $\tilde{\V}_v = \tilde{V}_v^\dagger \otimes_{F_v} W_v^\dagger$, etc., we have identifications
\[
 \tilde{\X}_v = \tilde{V}_v^\dagger, \qquad
 \X_v = V_v^\dagger, \qquad
 \X_{0,v} = V_{0,v}^\dagger.
\]
For any $v \in \Sigma_\infty$ and any non-negative integer $\ell$, we put $\sss^\ell \tilde{V}_v = \Sym^\ell \tilde{V}_v^\dagger \otimes_{F_v} \bar{F}_v$ and denote by $\Hs^\ell \tilde{V}_v$ the kernel of the contraction $\sss^\ell \tilde{V}_v \rightarrow \sss^{\ell-2} \tilde{V}_v$ (see \S \ref{ss:fin-dim-rep-SO}).
We define $\sss^\ell V_v$ and $\Hs^\ell V_v$ similarly.

\subsection{Construction}
\label{ss:form-construction}

For $v \in \Sigma_\infty$, let $k_v \ge 2$ be a positive even integer and put $\ell_v = k_v-2$.
Put $\ul = (\ell_v)_{v \in \Sigma_\infty}$ and
\[
 \Hs^{\ul} \tilde{V} = \bigotimes_{v \in \Sigma_\infty} \Hs^{\ell_v} \tilde{V}_v, \qquad
 \Hs^{\ul} V = \bigotimes_{v \in \Sigma_\infty} \Hs^{\ell_v} V_v.
\]
We consider a Schwartz form
\[
 \tilde{\varphi} = \otimes_v \tilde{\varphi}_v \in \SS(\tilde{\X}(\A)) \otimes \wedge^{2d} \tilde{\fp}^* \otimes \Hs^{\ul} \tilde{V}
\]
such that
\[
 \tilde{\varphi}_v = 
 \begin{cases}
  \varphi'_{2, \ell_v} & \text{if $v \in \Sigma$;} \\
  \varphi'_{\ell_v} & \text{if $v \in \Sigma_\infty \smallsetminus \Sigma$}
 \end{cases}
\]
(see \S \ref{ss:schwartzforms}; note that $\bigotimes_{v \in \Sigma} \wedge^2 \tilde{\fp}_v^* \subset \wedge^{2d} \tilde{\fp}^*$).
Then we have a theta form 
\[
 \Theta_{\tilde{\varphi}}(\tilde{g},h) = \sum_{x \in \tilde{\X}(F)} (\omega(\tilde{g},h) \otimes 1 \otimes 1) \tilde{\varphi}(x)
\]
on $\G(\U(\tilde{V}) \times \U(W))^0(\A)$, where we regard $\tilde{\varphi}$ as a $\wedge^{2d} \tilde{\fp}^* \otimes \Hs^{\ul} \tilde{V}$-valued function on $\tilde{\X}(\A)$.
Let $\tau$ be an irreducible unitary automorphic representation of $H(\A)^+$ with central character $\xi_E$ such that:
\begin{itemize}
 \item $\tau_v$ is the anti-holomorphic discrete series representation of $\GL_2(\R)^+$ of weight $-k_v-1$ if $v \in \Sigma$;
 \item $\tau_v$ is the holomorphic discrete series representation of $\GL_2(\R)^+$ of weight $k_v+1$ if $v \in \Sigma_\infty \smallsetminus \Sigma$,
\end{itemize}
where 
\begin{align*}
 H(\A)^+ & = \{ h \in H(\A) \, | \, \nu(h) \in \N(\A_E^\times) \}, \\
 \GL_2(\R)^+ & = \{ h \in \GL_2(\R) \, | \, \det h > 0 \}.
\end{align*}
Let $\phi = \otimes_v \phi_v \in \tau$ be a non-zero vector such that
\[
 \tau_v(z) \phi_v =
 \begin{cases}
  z^{-k_v-1} \cdot \phi_v & \text{if $v \in \Sigma$;} \\
  z^{k_v+1} \cdot \phi_v & \text{if $v \in \Sigma_\infty \smallsetminus \Sigma$}
 \end{cases}
\]
for $z \in \C^1$, where we embed $\C^\times$ into $\GL_2(\R)$ as in \S \ref{ss:notation-at-real-places}.
We define a theta lift
\[
 \theta_{\tilde{\varphi}}(\phi) \in \mathrsfs{A}(\tilde{G}) \otimes \wedge^{2d} \tilde{\fp}^* \otimes \Hs^{\ul} \tilde{V}
\]
by
\[
 \theta_{\tilde{\varphi}}(\phi)(\tilde{g}) = \int_{\U(W)(F) \backslash \U(W)(\A)} \Theta_{\tilde{\varphi}}(\tilde{g}, h_1 h) \phi(h_1 h) \, dh_1
\]
for $\tilde{g} \in \tilde{G}(\A)$, where we choose $h \in H(\A)^+$ such that $\nu(h) = \nu(\tilde{g})$ but the integral is independent of the choice of $h$.
By Proposition \ref{prop:spl-hodge}, $\theta_{\tilde{\varphi}}(\phi)$ has trivial central character.

Next we take the image $\tilde{\Xi} := \res(\theta_{\tilde{\varphi}}(\phi))$ of $\theta_{\tilde{\varphi}}(\phi)$ under the map
\[
 \res : \mathrsfs{A}(\tilde{G}) \otimes \wedge^{2d} \tilde{\fp}^* \otimes \Hs^{\ul} \tilde{V} \longrightarrow \mathrsfs{A}(\GG) \otimes \wedge^{2d} \fp^* \otimes \Hs^{\ul} V
\]
induced by the restriction $\mathrsfs{A}(\tilde{G}) \rightarrow \mathrsfs{A}(\GG)$ and the projections $\wedge^{2d} \tilde{\fp}^* \rightarrow \wedge^{2d} \fp^*$ and $\Hs^{\ul} \tilde{V} \rightarrow \Hs^{\ul} V$ (see \S \ref{ss:Res-C}).
For any character $\eta$ of $\A_E^\times/E^\times$ such that $\eta|_{\A^\times} = 1$, we define the $\eta$-component
\[
 \tilde{\Xi}_\eta \in \mathrsfs{A}(\GG) \otimes \wedge^{2d} \fp^* \otimes \Hs^{\ul} V
\]
of $\tilde{\Xi}$ by 
\[
 \tilde{\Xi}_\eta(\g) = \int_{\tilde{Z}(\A) \ZZ(F) \backslash \ZZ(\A)}
 \tilde{\Xi}(\z \g) \cdot (\eta \circ \mathbf{p})(\z) \, d \z,
\]
where the Haar measure $d \z$ is normalized so that $\vol(\tilde{Z}(\A) \ZZ(F) \backslash \ZZ(\A))=1$.
Furthermore, we define its push-forward
\[
 \pr_*(\tilde{\Xi}_\eta) \in \mathrsfs{A}(G) \otimes \wedge^{2d} \fp^* \otimes \Hs^{\ul} V
\]
by the first projection $\pr:\GG(\A) \rightarrow G(\A)$ as follows.
Let $G(\A)^+$ be the image of $\pr$, i.e.,
\[
 G(\A)^+ = \{g \in G(\A) \, | \, \nu(g) \in \N(\A_E^\times) \}.
\]
Note that $Z(\A) \subset G(\A)^+$ and $[G(\A):G(F)G(\A)^+] = [\A^\times : F^\times \N(\A_E^\times)] = 2$.
For $g \in G(\A)^+$, choose $\alpha_g \in \A_E^\times$ such that $\nu(g) = \N(\alpha_g)$ and put 
\[
 \pr_*(\tilde{\Xi}_\eta)(g) = \tilde{\Xi}_\eta(g, \alpha_g) \cdot \eta(\alpha_g),
\]
which is independent of the choice of $\alpha_g$.
Then we extend $\pr_*(\tilde{\Xi}_\eta)$ to a $\wedge^{2d} \fp^* \otimes \Hs^{\ul} V$-valued automorphic form on $G(\A)$ by the natural embedding
\[
 G(F)^+ \backslash G(\A)^+ \hookrightarrow G(F) \backslash G(\A) 
\]
and extension by zero, where $G(F)^+ = G(F) \cap G(\A)^+$.
Note that $\pr_*(\tilde{\Xi}_\eta)$ has trivial central character.

Finally, for any open compact subgroup $\cK$ of $Z(\A_f) \backslash G(\A_f)$, we define the $\cK$-invariant projection
\[
 \Xi_\cK \in \mathrsfs{A}(G) \otimes \wedge^{2d} \fp^* \otimes \Hs^{\ul} V
\]
of $\Xi := \pr_*(\tilde{\Xi}_\eta)$ by
\[
 \Xi_\cK(g) = \int_\cK \Xi(gk) \, dk,
\]
where the Haar measure $dk$ is normalized so that $\vol(\cK)=1$.

\subsection{Non-vanishing}

Let $\pi$ be an irreducible unitary cuspidal automorphic representation of $\GL_2(\A)$ with trivial central character such that:
\begin{itemize}
 \item $\pi_v$ is the discrete series representation of $\GL_2(\R)$ of even weight $k_v$ if $v \in \Sigma_\infty$.
\end{itemize}
We assume that $\pi$ has the Jacquet-Langlands transfers $\pi_{B_1}$ and $\pi_{B_2}$ to $B_1^\times(\A)$ and $B_2^\times(\A)$, respectively. 
We regard $\pi_{B_1} \boxtimes \pi_{B_2}$ as an irreducible unitary automorphic representation of $G(\A)$ with trivial central character.

For $\epsilon = (\epsilon_v)_{v \in \Sigma_\infty}$ with $\epsilon_v = \pm$, let
\[
 f_1^\epsilon = \Big( \bigotimes_{v \in \Sigma_\infty} f_{1,v}^{\epsilon_v} \Big) \otimes \Big( \bigotimes_{v \notin \Sigma_\infty} f_{1,v} \Big) \in \pi_{B_1}, \qquad
 f_2^\epsilon = \Big( \bigotimes_{v \in \Sigma_\infty} f_{2,v}^{\epsilon_v} \Big) \otimes \Big( \bigotimes_{v \notin \Sigma_\infty} f_{2,v} \Big) \in \pi_{B_2}
\]
be non-zero vectors such that:
\begin{itemize}
\item if $v \in \Sigma$, then 
\begin{equation}
\label{eq:K-type-f_i-spl}
 \pi_{B_1,v}(z) f_{1,v}^{\epsilon_v} =  z^{\epsilon_v k_v} \cdot f_{1,v}^{\epsilon_v}, \qquad
 \pi_{B_2,v}(z) f_{2,v}^{\epsilon_v} = z^{-\epsilon_v k_v} \cdot f_{2,v}^{\epsilon_v}
\end{equation}
for $z \in \C^1$ (such $f_{i,v}^{\epsilon_v}$ is unique up to scalars);
\item if $v \in \Sigma_\infty \smallsetminus \Sigma$, then 
\begin{equation}
\label{eq:K-type-f_i-ram}
 \pi_{B_1,v}(z) f_{1,v}^{\epsilon_v} = z^{\epsilon_v(k_v-2)} \cdot f_{1,v}^{\epsilon_v}, \qquad
 \pi_{B_2,v}(z) f_{2,v}^{\epsilon_v} = z^{-\epsilon_v(k_v-2)} \cdot f_{2,v}^{\epsilon_v}
\end{equation}
for $z \in \C^1$ (such $f_{i,v}^{\epsilon_v}$ is unique up to scalars);
\item if $v \notin \Sigma_\infty$, then $f_{i,v}$ does not depend on $\epsilon$. 
\end{itemize}
Here, for $v \in \Sigma_\infty$, we embed $\C^\times$ into $B_{i,v}^\times$ via the isomorphism $\C \simeq E_v$ as in \S \ref{ss:notation-at-real-places} and the fixed embedding $E \hookrightarrow B_i$.
We regard $f^\epsilon := f_1^\epsilon \boxtimes f_2^\epsilon$ as an automorphic form on $G(\A)$ with trivial central character.
Put
\[
 \fff^\epsilon = f^\epsilon \otimes \bo^\epsilon \otimes \vvv^\epsilon \in \mathrsfs{A}(G) \otimes \wedge^{2d} \fp^* \otimes \Hs^{\ul} V 
\]
with 
\[
 \bo^\epsilon = \bigotimes_{v \in \Sigma} \bo_v^{\epsilon_v}, \qquad
 \vvv^\epsilon = \bigotimes_{v \in \Sigma_\infty} \vvv_v^{\epsilon_v},
\]
where $\ul = (\ell_v)_{v \in \Sigma_\infty}$ with $\ell_v = k_v-2$, and $\bo_v^{\epsilon_v} \in \wedge^2 \fp_v^*$ and $\vvv_v^{\epsilon_v} \in \Hs^{\ell_v} V_v$ are as in \S \ref{ss:Res-C}.

Finally, let $(\cdot, \cdot)$ be the non-degenerate bilinear pairing on $\wedge^{2d} \fp^* \otimes \Hs^{\ul} V$ induced by
\begin{itemize}
 \item the bilinear pairing $\cdot \wedge \cdot : \wedge^2 \fp_v^* \times \wedge^2 \fp_v^* \rightarrow \wedge^4 \fp_v^* \simeq \C$ as in \S \ref{ss:Res-C};
 \item the bilinear pairing $\langle \cdot, \cdot \rangle : \sss^{\ell_v} V_v \times \sss^{\ell_v} V_v \rightarrow \C$ as in \S \ref{ss:fin-dim-rep-SO}.
\end{itemize}

\begin{prop}
\label{p:form-nonvanishing}
Suppose that $f_1^\epsilon$ and $f_2^\epsilon$ as above are given.
Let $\cK = \prod_v \cK_v$ be an open compact subgroup of $Z(\A_f) \backslash G(\A_f)$ such that $f_{1,v} \boxtimes f_{2,v}$ is $\cK_v$-fixed for all $v \notin \Sigma_\infty$.
Assume further that there exists a finite place $v_0$ of $F$ such that 
\begin{enumerate}
 \item \label{item:E1} $E_{v_0}/F_{v_0}$ is \emph{ramified};
 \item \label{item:E2} $B_{1,v_0}$ and $B_{2,v_0}$ are split;
 \item \label{item:E3} $\cK_{v_0}$ is a hyperspecial maximal compact subgroup of $Z_{v_0} \backslash G_{v_0}$.
\end{enumerate}
Then there exist $\tilde{\varphi}, \tau, \phi, \eta$ as in \S \ref{ss:form-construction} such that
\[
 (\Xi_\cK, \fff^\epsilon) := \int_{Z(\A) G(F) \backslash G(\A)} (\Xi_\cK(g), \fff^\epsilon(g)) \, dg \ne 0
\]
for all $\epsilon$, where $\Xi = \pr_*(\tilde{\Xi}_\eta)$ with $\tilde{\Xi} = \res(\theta_{\tilde{\varphi}}(\phi))$ and $\fff^\epsilon = f^\epsilon \otimes \bo^\epsilon \otimes \vvv^\epsilon$ with $f^\epsilon = f_1^\epsilon \boxtimes f_2^\epsilon$.
\end{prop}

The rest of this section is devoted to the proof of Proposition \ref{p:form-nonvanishing}.

\subsection{Reduction to triple product integrals}

Let $\tilde{\varphi} = \otimes_v \tilde{\varphi}_v$ be a Schwartz form as in \S \ref{ss:form-construction}.
For any finite place $v$ of $F$, we assume that $\tilde{\varphi}_v \in \SS(\tilde{\X}_v)$ is a Schwartz function of the form
\[
 \tilde{\varphi}_v = \varphi_v \otimes \varphi_{0,v}
\]
for some $\varphi_v \in \SS(\X_v)$ and $\varphi_{0,v} \in \SS(\X_{0,v})$.
For any real place $v$ of $F$, we define Schwartz functions $\varphi_v^{\epsilon_v} \in \SS(\X_v)$ and $\varphi_{0,v} \in \SS(\X_{0,v})$ by 
\begin{align}
\label{eq:varphi-infty}
 \varphi_v^{\epsilon_v}(x_1,x_2,x_5,x_6) & =
 \begin{cases}
  (x_1 + \epsilon_v i x_2)^{k_v} \cdot e^{-\pi(x_1^2 + x_2^2 + x_5^2 + x_6^2)} & \text{if $v \in \Sigma$;} \\
  (x_1 + \epsilon_v i x_2)^{k_v-2} \cdot e^{-\pi(x_1^2 + x_2^2 + x_5^2 + x_6^2)} & \text{if $v \in \Sigma_\infty \smallsetminus \Sigma$,}
 \end{cases} \\
\label{eq:varphi0-infty}
 \varphi_{0,v}(x_3,x_4) & = e^{-\pi(x_3^2+x_4^2)},
\end{align}
where $x_1, \dots, x_6$ are the coordinates on $\tilde{\X}_v = \tilde{V}_v^\dagger$ as in \S \ref{ss:Res-C}.
Put
\[
 \varphi^\epsilon = \Big( \bigotimes_{v \in \Sigma_\infty} \varphi_v^{\epsilon_v} \Big) \otimes \Big( \bigotimes_{v \notin \Sigma_\infty} \varphi_v \Big) \in \SS(\X(\A)), \qquad
 \varphi_0 = \bigotimes_v \varphi_{0,v} \in \SS(\X_0(\A)),
\]
so that $\varphi^\epsilon \otimes \varphi_0 \in \SS(\tilde{\X}(\A))$.

\begin{lem}
\label{l:vector/scalar-valued}
We have
\begin{equation}
\label{eq:vector/scalar-valued}
 (\Xi_\cK, \fff^\epsilon) = \int_{\tilde{Z}(\A) \GG(F) \backslash \GG(\A)}
 \theta_{\varphi^\epsilon \otimes \varphi_0}(\phi)(\g) \cdot (f^\epsilon \boxtimes \eta)(\g) \, d\g,
\end{equation}
where $\theta_{\varphi^\epsilon \otimes \varphi_0}(\phi)$ is the theta lift as defined in \S \ref{ss:theta-3} and $f^\epsilon \boxtimes \eta$ is regarded as an automorphic form on $\GG(\A)$.
\end{lem}

\begin{proof}
If $v \in \Sigma$, then by Proposition \ref{p:schwartz-rc1}, we have
\[
 \ccc_{v, \bo_v^{\epsilon_v}, \vvv_v^{\epsilon_v}}(\Res_v(\tilde{\varphi}_v)) = \varphi_v^{\epsilon_v} \otimes \varphi_{0,v},
\]
where 
\begin{align*}
 \Res_v & : S(\tilde{\X}_v) \otimes \wedge^2 \tilde{\fp}_v^* \otimes \Hs^\ell \tilde{V}_v \longrightarrow S(\tilde{\X}_v) \otimes \wedge^2 \fp^*_v \otimes \Hs^\ell V_v, \\
 \ccc_{v, \bo_v^{\epsilon_v}, \vvv_v^{\epsilon_v}} & : S(\tilde{\X}_v) \otimes \wedge^2 \fp_v^* \otimes \sss^\ell V_v \longrightarrow S(\tilde{\X}_v)
\end{align*}
are the restriction and the contraction as in \S \ref{sss:Res-C-1}.
Also, if $v \in \Sigma_\infty \smallsetminus \Sigma$, then by Proposition \ref{p:schwartz-rc2}, we have
\[
 \ccc_{v, \vvv_v^{\epsilon_v}}(\Res_v(\tilde{\varphi}_v)) = \varphi_v^{\epsilon_v} \otimes \varphi_{0,v},
\]
where 
\begin{align*}
 \Res_v: S(\tilde{\X}_v) \otimes \Hs^\ell \tilde{V}_v & \longrightarrow S(\tilde{\X}_v) \otimes \Hs^\ell V_v, \\
 \ccc_{v, \vvv_v^{\epsilon_v}} : S(\tilde{\X}_v) \otimes \sss^\ell V_v & \longrightarrow S(\tilde{\X}_v)
\end{align*}
are the restriction and the contraction as in \S \ref{sss:Res-C-2}.
This implies that
\[
 (\Res(\Theta_{\tilde{\varphi}}(\tilde{g},h)), \bo^\epsilon \otimes \vvv^\epsilon) = \Theta_{\varphi^\epsilon \otimes \varphi_0}(\tilde{g},h), 
\]
where
\[
 \Res : \wedge^{2d} \tilde{\fp}^* \otimes \Hs^{\ul} \tilde{V} \longrightarrow \wedge^{2d} \fp^* \otimes \Hs^{\ul} V
\]
is the projection.
Hence we have
\[
 (\Res(\theta_{\tilde{\varphi}}(\phi)(\tilde{g})), \bo^\epsilon \otimes \vvv^\epsilon) = \theta_{\varphi^\epsilon \otimes \varphi_0}(\phi)(\tilde{g}), 
\]
so that the right-hand side of \eqref{eq:vector/scalar-valued} is equal to
\[
 \int_{\tilde{Z}(\A) \GG(F) \backslash \GG(\A)} (\tilde{\Xi}(\g), \bo^\epsilon \otimes \vvv^\epsilon) \cdot (f^\epsilon \boxtimes \eta)(\g) \, d\g.
\]
This integral is equal to 
\begin{align*}
 & \int_{\ZZ(\A) \GG(F) \backslash \GG(\A)} \int_{\tilde{Z}(\A) \ZZ(F) \backslash \ZZ(\A)} (\tilde{\Xi}(\z \g), \bo^\epsilon \otimes \vvv^\epsilon) \cdot (f^\epsilon \boxtimes \eta)(\z \g) \, d \z \, d\g \\
 & = \int_{\ZZ(\A) \GG(F) \backslash \GG(\A)} (\tilde{\Xi}_\eta(\g), \bo^\epsilon \otimes \vvv^\epsilon) \cdot (f^\epsilon \boxtimes \eta)(\g) \, d\g \\
 & = \int_{Z(\A) G(F)^+ \backslash G(\A)^+} (\pr_*(\tilde{\Xi}_\eta)(g), \bo^\epsilon \otimes \vvv^\epsilon) \cdot f^\epsilon(g) \, dg \\
 & = \int_{Z(\A) G(F) \backslash G(\A)} (\pr_*(\tilde{\Xi}_\eta)(g), \bo^\epsilon \otimes \vvv^\epsilon) \cdot f^\epsilon(g) \, dg \\
 & = \int_{Z(\A) G(F) \backslash G(\A)} (\Xi(g), \fff^\epsilon(g)) \, dg \\
 & = \int_{Z(\A) G(F) \backslash G(\A)} (\Xi_\cK(g), \fff^\epsilon(g)) \, dg,
\end{align*}
noting that $\pr_*(\tilde{\Xi}_\eta)$ is supported in $G(F) G(\A)^+$ and $\fff^\epsilon$ is $\cK$-fixed.
\end{proof}

We now consider the seesaw diagram
\[
 \xymatrix{
  \GU(\tilde{V})^0 \ar@{-}[dr] \ar@{-}[d] &
  \G(\U(W) \times \U(W)) \ar@{-}[dl] \ar@{-}[d] \\
  \G(\U(V) \times \U(V_0))^0 & \GU(W)}.
\]
Then the seesaw identity (combined with Lemma \ref{l:vector/scalar-valued}) says that 
\begin{equation}
\label{eq:seesaw-id}
\begin{aligned}
 (\Xi_\cK, \fff^\epsilon) 
 & = \int_{\tilde{Z}(\A) \GG(F) \backslash \GG(\A)}
 \theta_{\varphi^\epsilon \otimes \varphi_0}(\phi)(\g) \cdot (f^\epsilon \boxtimes \eta)(\g) \, d\g \\
 & = \int_{Z_H(\A) H(F)^+ \backslash H(\A)^+}
 \theta_{\varphi^\epsilon}(f^\epsilon)(h) \cdot \theta_{\varphi_0}(\eta)(h) \cdot \phi(h) \, dh,
\end{aligned}
\end{equation}
where $H(F)^+ = H(F) \cap H(\A)^+$, and $\theta_{\varphi^\epsilon}(f^\epsilon)$ and $\theta_{\varphi_0}(\eta)$ are the theta lifts as defined in \S \ref{ss:theta-2} and \S \ref{ss:theta-1}, respectively.
Hence, to prove Proposition \ref{p:form-nonvanishing}, it suffices to find $\varphi^\epsilon, \varphi_0, \eta, \tau, \phi$ such that the right-hand side of \eqref{eq:seesaw-id} is non-zero for all $\epsilon$.

\subsection{Choosing $\varphi^\epsilon$}

Let $\pi_B$ be the Jacquet-Langlands transfer of $\pi$ to $B^\times(\A)$ (which exists since $\pi_{B_1}$ and $\pi_{B_2}$ exist by assumption and $B_1 \cdot B_2 = B$ in the Brauer group).
Note that:
\begin{itemize}
 \item $\pi_B$ has trivial central character;
 \item $\pi_{B,v}$ is the discrete series representation of $\GL_2(\R)$ of weight $k_v$ if $v \in \Sigma_\infty$.
\end{itemize}
By Lemma \ref{l:theta-2}, we have a non-zero equivariant map
\[
 \theta: \SS(\X(\A)) \otimes (\pi_{B_1} \boxtimes \pi_{B_2}) \longrightarrow \pi_B
\]
given by $\varphi \otimes f \mapsto \theta_\varphi(f)$.

\begin{lem}
\label{l:choose-varphi-1}
Let $\varphi^\epsilon = \left( \bigotimes_{v \in \Sigma_\infty} \varphi_v^{\epsilon_v} \right) \otimes \left( \bigotimes_{v \notin \Sigma_\infty} \varphi_v \right) \in \SS(\X(\A))$ be a Schwartz function such that $\varphi_v^{\epsilon_v}$ is as in \eqref{eq:varphi-infty} for all $v \in \Sigma_\infty$.
Then 
\begin{equation}
\label{eq:theta-lift-K-type1}
 \pi_{B,v}(z) \theta_{\varphi^\epsilon}(f^\epsilon) =
 \begin{cases}
  z^{k_v} \cdot \theta_{\varphi^\epsilon}(f^\epsilon) & \text{if $v \in \Sigma$;} \\
  z^{-k_v} \cdot \theta_{\varphi^\epsilon}(f^\epsilon) & \text{if $v \in \Sigma_\infty \smallsetminus \Sigma$}
 \end{cases}
\end{equation}
for $z \in \C^1$.
Moreover, $\theta_{\varphi^\epsilon}(f^\epsilon)$ is non-zero for some such $\varphi^\epsilon$.
\end{lem}

\begin{proof}
We have
\begin{equation}
\label{eq:varphi-K-type}
 \omega_v(t, z) \varphi_v^{\epsilon_v} = 
 \begin{cases}
 t_1^{\epsilon_v k_v} \cdot z^{k_v} \cdot \varphi_v^{\epsilon_v} & \text{if $v \in \Sigma$;} \\
 t_1^{\epsilon_v (k_v-2)} \cdot z^{-k_v} \cdot \varphi_v^{\epsilon_v} & \text{if $v \in \Sigma_\infty \smallsetminus \Sigma$}
 \end{cases}
\end{equation}
for $t = (t_1, t_2) \in \U(V_v)^0$ with $t_i \in \SO(2) \simeq \C^1$ and $z \in \U(W_v)$ with $z \in \C^1$.
This proves \eqref{eq:theta-lift-K-type1}.

By the Howe duality for $(\GU(V_v)^0, \GU(W_v)^+)$ (see the proof of Lemma \ref{lem:irred-theta-similitude}), we have a decomposition
\[
 \theta = \bigotimes_v \theta_v, 
\]
where 
\[
 \theta_v :\SS(\X_v) \otimes (\pi_{B_1, v} \boxtimes \pi_{B_2, v}) \longrightarrow \pi_{B,v}
\]
is the unique (up to scalars) non-zero $\G(\U(V_v) \times \U(W_v))^0$-equivariant map.
Since $\pi_{B_1, v} \boxtimes \pi_{B_2, v}$ is irreducible, we may choose $\varphi_v$ so that $\theta_v(\varphi_v \otimes (f_{1,v} \boxtimes f_{2,v})) \ne 0$ for $v \notin \Sigma_\infty$.

It remains to show that $\theta_v(\varphi_v^{\epsilon_v} \otimes (f_{1,v}^{\epsilon_v} \boxtimes f_{2,v}^{\epsilon_v})) \ne 0$ for $v \in \Sigma_\infty$, where $\varphi_v^{\epsilon_v}$ is as in \eqref{eq:varphi-infty}.
Let 
\[
 {}^t\theta_v :\SS(\X_v) \otimes \pi_{B,v}^\vee \longrightarrow (\pi_{B_1, v} \boxtimes \pi_{B_2, v})^\vee
\]
be the $\G(\U(V_v) \times \U(W_v))^0$-equivariant map induced by $\theta_v$.
Let $w_v \in \pi_{B,v}^\vee$ be the unique (up to scalars) non-zero vector such that
\begin{equation}
\label{eq:K-type-w_v}
 \pi_{B,v}^\vee(z) w_v =
 \begin{cases}
  z^{-k_v} \cdot w_v & \text{if $v \in \Sigma$;} \\
  z^{k_v} \cdot w_v & \text{if $v \in \Sigma_\infty \smallsetminus \Sigma$}
 \end{cases}
\end{equation}
for $z \in \C^1$.
Then, by \eqref{eq:varphi-K-type}, ${}^t\theta_v(\varphi_v^{\epsilon_v} \otimes w_v)$ is a scalar multiple of the unique (up to scalars) non-zero vector $\fF_v^{\epsilon_v} \in (\pi_{B_1, v} \boxtimes \pi_{B_2, v})^\vee$ such that
\begin{equation}
\label{eq:K-type-F_v}
 (\pi_{B_1, v} \boxtimes \pi_{B_2, v})^\vee(t) \fF_v^{\epsilon_v} = 
 \begin{cases}
 t_1^{\epsilon_v k_v} \cdot \fF_v^{\epsilon_v} & \text{if $v \in \Sigma$;} \\
 t_1^{\epsilon_v (k_v-2)} \cdot \fF_v^{\epsilon_v} & \text{if $v \in \Sigma_\infty \smallsetminus \Sigma$}
 \end{cases}
\end{equation}
for $t = (t_1, t_2) \in \U(V_v)^0$ with $t_i \in \SO(2) \simeq \C^1$.
Since
\[
 \langle \theta_v(\varphi_v^{\epsilon_v} \otimes (f_{1,v}^{\epsilon_v} \boxtimes f_{2,v}^{\epsilon_v})), w_v \rangle = 
 \langle f_{1,v}^{\epsilon_v} \boxtimes f_{2,v}^{\epsilon_v}, {}^t\theta_v(\varphi_v^{\epsilon_v} \otimes w_v) \rangle
\]
and $\langle f_{1,v}^{\epsilon_v} \boxtimes f_{2,v}^{\epsilon_v}, \fF_v^{\epsilon_v} \rangle \ne 0$ by \eqref{eq:K-type-f_i-spl}, \eqref{eq:K-type-f_i-ram}, and \eqref{eq:K-type-F_v},
where $\langle \cdot, \cdot \rangle$ denotes the natural pairing, it suffices to show that ${}^t\theta_v(\varphi_v^{\epsilon_v} \otimes w_v) \ne 0$.

For this, we realize ${}^t\theta_v$ explicitly as follows.
Recall that we write $J = t_v^2$ for some $t_v \in F^\times_v$ in \S \ref{ss:notation-at-real-places}.
We define an isomorphism $\ii'_v: B_{1,v} \rightarrow B_{2,v}$ of quaternion $F_v$-algebras by
\[
 \ii'_v(a + bi + cj_1 + dij_1) = a + bi + cj_2 + dij_2,
\]
where 
\[
 i = |u|_{F_v}^{-\frac{1}{2}} \cdot \i, \qquad 
 j_1 = |J_1|_{F_v}^{-\frac{1}{2}} \cdot \j_1, \qquad
 j_2 = \zeta_v t_v^{-1} \cdot |J_1|_{F_v}^{\frac{1}{2}} \cdot \j_2
\]
with
\[
 \zeta_v = 
 \begin{cases}
  +1 & \text{if $v \in \Sigma$;} \\
  -1 & \text{if $v \in \Sigma_\infty \smallsetminus \Sigma$.}
 \end{cases}
\]
Since $\j_1 \e_2 = J_1 \e_1$, $\j_2 \e_2 = \e_1 \j$, and $\j e_v'' = -t_v e_v''$, we have
\[
 j_1 e_{2,v} = \zeta_v e_{4,v}, \qquad 
 j_2 e_{2,v} = -\zeta_v e_{4,v}.
\]
From this and \eqref{eq:i-actions}, we deduce that
\[
 x^* e_{2,v} = \ii'_v(x) e_{2,v}
\]
for all $x \in B_{1,v}$, where $*$ is the main involution on $B_{1,v}$.
In particular, if we define a subgroup $\varDelta_v$ of $G_v = (B_{1,v}^\times \times B_{2,v}^\times)/F_v^\times$ by 
\[
 \varDelta_v = \{ ((x^*)^{-1}, \ii'_v(x)) \, | \, x \in B_{1,v}^\times \} / F_v^\times,
\]
then $e_{2,v}$ is $\varDelta_v$-fixed.
We now realize $\pi_{B,v}^\vee$ on the Whittaker model $\Wc(\pi_{B,v}^\vee)$ with respect to the character $\smat{1}{x}{}{1} \mapsto e^{-2 \pi i \zeta_v x}$ and define a map
\[
 \tilde{\Bc}_v: \SS(\X_v) \otimes \Wc(\pi_{B,v}^\vee) \longrightarrow \C
\]
by
\[
 \tilde{\Bc}_v(\Phi \otimes W) = \int_{N_v \backslash \SL_2(F_v)} \omega_v(h) \Phi(\sqrt{2} e_{2,v}) W(h) \, dh, 
\]
where $N_v$ is the group of unipotent upper triangular matrices in $H_v = \GL_2(F_v)$ and the integral is absolutely convergent by \cite[Lemme 5]{wald-periods85}.
For $\fF \in (\pi_{B_1, v} \boxtimes \pi_{B_2, v})^\vee$, put
\[
 \Bc_v(\fF) = \tilde{\Bc}_v(\tilde{\fF}),
\]
where we choose $\tilde{\fF} \in \SS(\X_v) \otimes \Wc(\pi_{B,v}^\vee)$ such that ${}^t\theta_v(\tilde{\fF}) = \fF$.
Then, by \cite[Lemme 6]{wald-periods85}, this does not depend on the choice of $\tilde{\fF}$ and defines a $\varDelta_v$-invariant map $\Bc_v: (\pi_{B_1, v} \boxtimes \pi_{B_2, v})^\vee \rightarrow \C$, so that 
\[
 \tilde{\Bc}_v = \Bc_v \circ {}^t \theta_v.
\]
Note that the representation $x \mapsto \pi_{B_1, v}((x^*)^{-1})$ is isomorphic to $\pi_{B_1,v}^\vee$.
Thus it suffices to show that $\tilde{\Bc}_v(\varphi_v^{\epsilon_v} \otimes w_v) \ne 0$.
By \eqref{eq:K-type-w_v}, we may normalize $w_v$ so that 
\[
 w_v \! \mat{a}{}{}{a^{-1}} = a^{k_v} e^{-2 \pi a^2}.
\]
If $v \in \Sigma$, then
\begin{align*}
 \tilde{\Bc}_v(\varphi_v^{\epsilon_v} \otimes w_v) 
 & = \int_0^\infty a^2 \varphi_v^{\epsilon_v}(a \sqrt{2} e_{2,v}) \cdot
 w_v \! \mat{a}{}{}{a^{-1}} \cdot a^{-2} \, d^\times a \\
 & = (\epsilon_v i \sqrt{2})^{k_v} \cdot
 \int_0^\infty a^{2k_v} e^{-4 \pi a^2} \, d^\times a \\
 & = (\epsilon_v i \sqrt{2})^{k_v} \cdot (4 \pi)^{-k_v} \cdot 2^{-1} \cdot
 \int_0^\infty a^{k_v} e^{-a} \, d^\times a \\
 & = (\epsilon_v i \sqrt{2})^{k_v} \cdot (4 \pi)^{-k_v} \cdot 2^{-1} \cdot
 \Gamma(k_v),
\end{align*}
where $d^\times a = da/a$.
Similarly, if $v \in \Sigma_\infty \smallsetminus \Sigma$, then
\[
 \tilde{\Bc}_v(\varphi_v^{\epsilon_v} \otimes w_v) 
 = (\epsilon_v i \sqrt{2})^{k_v-2} \cdot (4 \pi)^{-k_v+1} \cdot 2^{-1} \cdot
 \Gamma(k_v-1).
\]
This completes the proof.
\end{proof}

By Lemma \ref{l:choose-varphi-1}, we may choose $\varphi^\epsilon$ so that $\theta_{\varphi^\epsilon}(f^\epsilon)$ is non-zero.
Moreover, by replacing $f^\epsilon$ by its scalar multiple if necessary, we may assume that $\theta_{\varphi^\epsilon}(f^\epsilon)$ does not depend on $\epsilon$.

\begin{lem}
\label{l:choose-varphi-2}
There exists an element $(g_0, h_0) \in \G(\U(V_{v_0}) \times \U(W_{v_0}))^0$ such that the restriction of $\theta_{\omega(g_0,h_0) \varphi^\epsilon}(f^\epsilon)$ to $H(\A)^+$ is non-zero.
\end{lem}

\begin{proof}
Since $H(F) H(\A)^+$ is the kernel of $\xi_E \circ \nu$ and $E_{v_0}$ is a ramified quadratic extension of $F_{v_0}$, we have
\[
 H(\A) = H(F) H(\A)^+ \bigsqcup H(F) H(\A)^+ h_0
\]
for some $h_0 \in H_{v_0}$ such that $\nu(h_0) \in \cO_{F_{v_0}}^\times \smallsetminus \N(\cO_{E_{v_0}}^\times)$.
Then there exists an element $g_0 \in G_{v_0}$ such that $\nu(g_0) = \nu(h_0)$ and such that the image of $g_0$ in $Z_{v_0} \backslash G_{v_0}$ belongs to $\cK_{v_0}$.
Since $f^\epsilon$ is $\cK_{v_0}$-fixed, we have
\begin{align*}
 \theta_{\varphi^\epsilon}(f^\epsilon)(h h_0)
 & = \int_{\U(V)^0(F) \backslash \U(V)^0(\A)} \Theta_{\varphi^\epsilon}(g_1 g g_0, h h_0) f^\epsilon(g_1 g g_0) \, dg_1 \\
 & = \int_{\U(V)^0(F) \backslash \U(V)^0(\A)} \Theta_{\varphi^\epsilon}(g_1 g g_0, h h_0) f^\epsilon(g_1 g) \, dg_1 \\
 & = \theta_{\omega(g_0,h_0) \varphi^\epsilon}(f^\epsilon)(h)
\end{align*}
for $h \in H(\A)^+$, where we choose $g \in G(\A)$ such that $\nu(g) = \nu(h)$.
Hence $\theta_{\varphi^\epsilon}(f^\epsilon)(h)$ or $\theta_{\omega(g_0,h_0) \varphi^\epsilon}(f^\epsilon)(h)$ is non-zero for some $h \in H(\A)^+$.
\end{proof}

By Lemma \ref{l:choose-varphi-2}, we may assume that the restriction of $\theta_{\varphi^\epsilon}(f^\epsilon)$ to $H(\A)^+$ is non-zero.

\subsection{Choosing $\eta$ and $\varphi_0$}

We choose $\eta$ satisfying the conditions of the following lemma.

\begin{lem}
\label{l:choose-eta}
There exists a character $\eta$ of $\A_E^\times/E^\times$ such that:
\begin{itemize}
 \item $\eta|_{\A^\times} = 1$;
 \item $\eta_v = 1$ for all real places $v$ of $F$;
 \item $\eta_v$ does not factor through the norm map if $B_v$ is ramified.
\end{itemize}
\end{lem}

\begin{proof}
By Hilbert 90, the map $x \mapsto x/x^\rho$ induces an isomorphism $E^\times / F^\times \simeq E^1$.
Hence it suffices to find a character $\chi$ of $\A_E^1/E^1$ such that:
\begin{itemize}
 \item $\chi_v = 1$ for all real places $v$ of $F$;
 \item $\chi_v^2 \ne 1$ if $B_v$ is ramified.
\end{itemize}
Since $E_v$ is non-split if either $v$ is real or $B_v$ is ramified, it remains to show the following: if $S$ is a finite set of places of $F$ such that $E_v$ is non-split for all $v \in S$ and $\chi_S$ is a character of $E_S^1 = \prod_{v \in S} E_v^1$, then there exists a character $\chi$ of $\A_E^1/E^1$ such that $\chi|_{E^1_S} = \chi_S$.
But this assertion follows from the fact that $E_S^1$ is compact and hence the image of the natural continuous injective homomorphism
\[
 E_S^1 \longrightarrow \A_E^1 \longrightarrow \A_E^1/E^1
\]
is closed.
\end{proof}

Let $\pi(\eta)$ be the automorphic induction of $\eta$ to $\GL_2(\A)$ and $\pi(\eta)_B$ its Jacquet-Langlands transfer to $B^\times(\A)$ (which exists since $\eta_v$ does not factor through the norm map if $B_v$ is ramified).
Note that:
\begin{itemize}
 \item $\pi(\eta)_B$ has central character $\xi_E$;
 \item $\pi(\eta)_{B,v}$ is the limit of discrete series representation of $\GL_2(\R)$ of weight $1$ if $v \in \Sigma_\infty$.
\end{itemize}
By Lemma \ref{l:theta-1}, we have a non-zero equivariant map
\[
 \theta: \SS(\X_0(\A)) \longrightarrow \pi(\eta)_B
\]
given by $\varphi_0 \mapsto \theta_{\varphi_0}(\eta)$.

\begin{lem}
\label{l:choose-varphi0-1}
Let $\varphi_0 = \bigotimes_v \varphi_{0,v} \in \SS(\X_0(\A))$ be a Schwartz function such that $\varphi_{0,v}$ is as in \eqref{eq:varphi0-infty} for all $v \in \Sigma_\infty$.
Then 
\begin{equation}
\label{eq:theta-lift-K-type2}
 \pi(\eta)_{B,v}(z) \theta_{\varphi_0}(\eta) = 
 \begin{cases}
  z \cdot \theta_{\varphi_0}(\eta) & \text{if $v \in \Sigma$;} \\
  z^{-1} \cdot \theta_{\varphi_0}(\eta) & \text{if $v \in \Sigma_\infty \smallsetminus \Sigma$}
 \end{cases}
\end{equation}
for $z \in \C^1$.
Moreover, $\theta_{\varphi_0}(\eta)$ is non-zero for some such $\varphi_0$.
\end{lem}

\begin{proof}
We have
\[
 \omega_v(t, z) \varphi_{0,v} = 
 \begin{cases}
 z \cdot \varphi_{0,v} & \text{if $v \in \Sigma$;} \\
 z^{-1} \cdot \varphi_{0,v} & \text{if $v \in \Sigma_\infty \smallsetminus \Sigma$}
 \end{cases}
\]
for $t \in \U(V_{0,v})^0$ and $z \in \U(W_v)$ with $z \in \C^1$.
This proves \eqref{eq:theta-lift-K-type2}.

As explained in the proof of Lemma \ref{l:theta-1}, we may regard $\theta_{\varphi_0}(\eta)$ as the theta lift of $\eta$ (regarded as an automorphic character of $\GU(\VV)(\A)$) to $\GU(\WW)(\A)$, where $\VV$ and $\WW$ are the $1$-dimensional hermitian $E$-space and the $2$-dimensional skew-hermitian $E$-space, respectively, as in \S \ref{ss:doubling-U(WW)}.
Hence, by the Howe duality for $(\GU(\VV_v), \GU(\WW_v)^+)$, we have a decomposition
\[
 \theta = \bigotimes_v \theta_v, 
\]
where 
\[
 \theta_v : \SS(\X_{0,v}) \longrightarrow \pi(\eta)_{B,v}
\]
is the unique (up to scalars) non-zero $\G(\U(\VV_v) \times \U(\WW_v))$-equivariant map.
Here $(g, [h,\alpha]) \in \G(\U(\VV_v) \times \U(\WW_v))$ with $h \in B_v^\times$ and $\alpha \in E_v^\times$ acts as $\omega_v(g,[h,\alpha]) \otimes \eta_v(g)$ on the left-hand side and as $\pi(\eta)_{B,v}(h) \otimes \eta_v(\alpha)^{-1}$ on the right-hand side.
For $v \notin \Sigma_\infty$, we may choose $\varphi_{0,v}$ so that $\theta_v(\varphi_{0,v}) \ne 0$.

It remains to show that $\theta_v(\varphi_{0,v}) \ne 0$ for $v \in \Sigma_\infty$, where $\varphi_{0,v}$ is as in \eqref{eq:varphi0-infty}.
Since $\eta_v=1$, $\theta_v$ can be realized by
\[
 \theta_v(\Phi) = \cF_\Phi, \qquad
 \cF_\Phi(h) = \omega_v(g,h) \Phi(0)
\]
for $\Phi \in \SS(\X_{0,v})$ and $h \in H_v^+$, where we choose $g \in G_{0,v}$ such that $\nu(g) = \nu(h)$ and regard $\pi(\eta)_{B,v}$ as a subrepresentation of some unitary principal series representation.
Then, noting that $\varphi_{0,v}(0) = 1$, we have $\theta_v(\varphi_{0,v})  \ne 0$.
\end{proof}

By Lemma \ref{l:choose-varphi0-1}, we may choose $\varphi_0$ so that $\theta_{\varphi_0}(\eta)$ is non-zero.
Since $\theta_{\varphi_0}(\eta)$ is supported in $H(F) H(\A)^+$ by definition, its restriction to $H(\A)^+$ is also non-zero.

\subsection{Choosing $\tau$ and $\phi$}

\begin{lem}
\label{l:product-K-type}
Let $\psi$ be the restriction of $\theta_{\varphi^\epsilon}(f^\epsilon) \cdot \theta_{\varphi_0}(\eta)$ to $H(\A)^+$.
For $v \in \Sigma_\infty$, let $\sigma_v$ be the representation of $\GL_2(\R)^+$ generated by $\psi$.
Then 
\begin{equation}
\label{eq:product-K-type}
 \sigma_v(z) \psi = 
 \begin{cases}
  z^{k_v+1} \cdot \psi & \text{if $v \in \Sigma$;} \\
  z^{-k_v-1} \cdot \psi & \text{if $v \in \Sigma_\infty \smallsetminus \Sigma$}
 \end{cases}
\end{equation}
for $z \in \C^1$.
Moreover, if $\psi$ is non-zero, then 
\begin{itemize}
 \item $\sigma_v$ is the holomorphic discrete series representation of $\GL_2(\R)^+$ of weight $k_v+1$ if $v \in \Sigma$;
 \item $\sigma_v$ is the anti-holomorphic discrete series representation of $\GL_2(\R)^+$ of weight $-k_v-1$ if $v \in \Sigma_\infty \smallsetminus \Sigma$.
\end{itemize}
\end{lem}

\begin{proof}
We only consider the case $v \in  \Sigma$; the other case is similar.
Let $\psi'$ and $\psi''$ be the restrictions of $\theta_{\varphi^\epsilon}(f^\epsilon)$ and $\theta_{\varphi_0}(\eta)$ to $H(\A)^+$, respectively.
Since $\psi = \psi' \cdot \psi''$, \eqref{eq:product-K-type} follows from \eqref{eq:theta-lift-K-type1} and \eqref{eq:theta-lift-K-type2}.

Let $\sigma'_v$ and $\sigma''_v$ be the representations of $\GL_2(\R)^+$ generated by $\psi'$ and $\psi''$, respectively.
Then $\sigma'_v \simeq \HDS_{k_v}$ and $\sigma''_v \simeq \HDS_1$, where for any positive integer $k$, $\HDS_k$ denotes the (limit of) holomorphic discrete series representation of $\GL_2(\R)^+$ of weight $k$ with central character trivial on $\R^\times_+$.
Since $\psi = \psi' \cdot \psi''$, $\sigma_v$ is a subquotient of $\sigma'_v \otimes \sigma_v''$.
However, we have
\[
 \HDS_{k_v} \otimes \HDS_1 \simeq \bigoplus_{i=0}^\infty \HDS_{k_v+1+2i}
\]
by \cite[Theorem 8.1]{repka}.
Hence, if $\psi$ is non-zero, then \eqref{eq:product-K-type} forces $\sigma_v \simeq \HDS_{k_v+1}$.
This completes the proof.
\end{proof}

\begin{lem}
\label{l:choose-varphi0-2}
There exists an element $(g_0, h_0) \in \G(\U(V_0) \times \U(W))^0(\A_f)$ such that the restriction of $\theta_{\varphi^\epsilon}(f^\epsilon) \cdot \theta_{\omega(g_0, h_0)\varphi_0}(\eta)$ to $H(\A)^+$ is non-zero.
\end{lem}

\begin{proof}
Let $\psi'$ and $\psi''$ be the restrictions of $\theta_{\varphi^\epsilon}(f^\epsilon)$ and $\theta_{\varphi_0}(\eta)$ to $H(\A)^+$, respectively.
Choose an open compact subgroup $\cK_H^+$ of $H(\A_f)^+$ so that $\psi'$ and $\psi''$ are $\cK_H^+$-fixed.
Since $Z_H(\A) H(F)^+ \backslash H(\A)^+$ is compact, we have a finite decomposition
\[
 H(\A)^+ = \bigsqcup_i H(F)^+ H(F_\infty)^+ h_i \cK_H^+
\]
for some $h_i \in H(\A_f)^+$, where $F_\infty = F \otimes_\Q \R$.
This gives rise to a natural identification
\[
 H(F)^+ \backslash H(\A)^+ / K_{H, \infty} \cK_H^+ = \bigsqcup_i \Gamma_i \backslash \fh^n, 
\]
where $K_{H, \infty} = \prod_{v \in \Sigma_\infty} \R^\times \cdot \SO(2)$ is a maximal compact modulo center subgroup of $H(F_\infty)^+$, $\fh$ is the upper half plane, $n = [F:\Q]$, and $\Gamma_i = H(F)^+ \cap h_i \cK_H h_i^{-1}$.
Hence the restrictions of $\psi'$ and $\psi''$ to $H(F_\infty)^+ h_i$ descend to analytic functions $\Psi_i'$ and $\Psi_i''$ on $\fh^n$ (regarded as a real analytic manifold), respectively, satisfying some equivariance properties relative to the action of $\Gamma_i$ on $\fh^n$.
Since $\psi'$ and $\psi''$ are non-zero, so are $\Psi'_i$ and $\Psi''_j$ for some $i$ and $j$.
Then the product $\Psi'_i \cdot \Psi''_j$ is also non-zero.
Namely, if we put $h_0 = h_i^{-1} h_j \in H(\A_f)^+$, then 
\[
 \psi'(h) \cdot \psi''(h h_0) \ne 0
\]
for some $h \in H(F_\infty)^+ h_i$.
Choose $g_0 \in G_0(\A_f)$ such that $\nu(g_0) = \nu(h_0)$.
Then 
\begin{align*}
 \theta_{\varphi_0}(\eta)(h h_0)
 & = \int_{\U(V_0)^0(F) \backslash \U(V_0)^0(\A)} \Theta_{\varphi_0}(g_1 g g_0, h h_0) \eta(g_1 g g_0) \, dg_1 \\
 & = \eta(g_0) \cdot \int_{\U(V_0)^0(F) \backslash \U(V_0)^0(\A)} \Theta_{\varphi_0}(g_1 g g_0, h h_0) \eta(g_1 g) \, dg_1 \\
 & = \eta(g_0) \cdot \theta_{\omega(g_0, h_0)\varphi_0}(\eta)(h) 
\end{align*}
for $h \in H(\A)^+$, where we choose $g \in G_0(\A)$ such that $\nu(g) = \nu(h)$.
Hence 
\[
 \theta_{\varphi^\epsilon}(f^\epsilon)(h) \cdot \theta_{\omega(g_0, h_0)\varphi_0}(\eta)(h)
 = \eta(g_0)^{-1} \cdot \psi'(h) \cdot \psi''(h h_0) \ne 0
\]
for some $h \in H(\A)^+$.
\end{proof}

By Lemma \ref{l:choose-varphi0-2}, we may assume that the restriction of $\theta_{\varphi^\epsilon}(f^\epsilon) \cdot \theta_{\varphi_0}(\eta)$ to $H(\A)^+$ is non-zero.
Then, noting that $Z_H(\A) H(F)^+ \backslash H(\A)^+$ is compact, we deduce from the spectral decomposition together with Lemma \ref{l:product-K-type} that
\[
 \int_{Z_H(\A) H(F)^+ \backslash H(\A)^+} \theta_{\varphi^\epsilon}(f^\epsilon)(h) \cdot \theta_{\varphi_0}(\eta)(h) \cdot \phi(h) \, dh \ne 0
\]
for some non-zero vector $\phi$ in some irreducible automorphic representation $\tau$ of $H(\A)^+$ as in \S \ref{ss:form-construction}.
This completes the proof of Proposition \ref{p:form-nonvanishing}.

\section{Arthur packets, Galois representations, and Hodge classes}
\label{sec:arthur-galois-hodge}

\subsection{Classification}
\label{sec:classification-global}

Let $F$ be a totally real number field and $E$ a totally imaginary quadratic extension of $F$.
Let $\VV$ be an $n$-dimensional hermitian $E$-space and $G = \U(\VV)$ the unitary group of $\VV$.
In this section, we recall the classification of automorphic representations of $G(\A_F)$, which has been established by Mok \cite{mok} in the quasi-split case, following Arthur's book \cite{arthur}, and has been extended to the general case by Kaletha-Minguez-Shin-White \cite{kmsw}.

More precisely, let $L^2_{\disc}(G)$ be the discrete spectrum of the unitary representation of $G(\A_F)$ on the Hilbert space $L^2(G(F) \backslash G(\A_F))$.
Then the decomposition of $L^2_{\disc}(G)$ into near equivalence classes is described as follows.
We say that an irreducible cuspidal automorphic representation $\pi$ of $\GL_m(\A_E)$ is conjugate-self-dual if $\pi^\rho \simeq \pi^\vee$, where $\pi^\rho$ and $\pi^\vee$ are the Galois conjugate and the contragredient of $\pi$, respectively.
In this case, exactly one of the Asai $L$-functions $L(s, \pi, \mathrm{As}^+)$ and $L(s, \pi, \mathrm{As}^-)$ has a pole at $s=1$ (see \cite[\S 7]{ggp} for the definition of the Asai representations $\mathrm{As}^\pm$).
For $\epsilon = \pm$, we say that $\pi$ has sign $\epsilon$ if $L(s, \pi, \mathrm{As}^\epsilon)$ has a pole at $s=1$.
We also say that $\pi$ is conjugate-orthogonal (resp.~conjugate-symplectic) if it is conjugate-self-dual with sign $+$ (resp.~$-$).
Consider a formal unordered finite direct sum
\[
 \psi = \bigoplus_i \pi_i \boxtimes \Sym^{d_i-1},
\]
where
\begin{itemize}
 \item $\pi_i$ is an irreducible cuspidal automorphic representation of $\GL_{m_i}(\A_E)$;
 \item $\Sym^{d_i-1}$ is the irreducible representation of $\SL_2(\C)$ of dimension $d_i$.
\end{itemize}
Then $\psi$ is called an elliptic $A$-parameter for $G$ if
\begin{itemize}
 \item $\sum_i m_i d_i = n$;
 \item if $d_i$ is odd, then $\pi_i$ is conjugate-self-dual with sign $(-1)^{n-1}$;
 \item if $d_i$ is even, then $\pi_i$ is conjugate-self-dual with sign $(-1)^n$;
 \item if $(\pi_i, d_i) = (\pi_j, d_j)$, then $i=j$.
\end{itemize}
We attach to $\psi$ an automorphic representation
\[
 \pi_\psi = \bigboxplus_i \left( \pi_i |\det|^{\frac{d_i-1}{2}} \boxplus \pi_i |\det|^{\frac{d_i-3}{2}} \boxplus \cdots \boxplus \pi_i |\det|^{-\frac{d_i-1}{2}} \right)
\]
of $\GL_n(\A_E)$, where $\boxplus$ denotes the isobaric sum.
Then the result of Kaletha-Minguez-Shin-White \cite[Theorem${}^*$ 1.7.1]{kmsw} (which is proved in \cite{kmsw} partially and will be completed in its sequels) says that 
\[
 L^2_{\disc}(G) = \bigoplus_\psi L^2_\psi(G),
\]
where $\psi$ runs over elliptic $A$-parameters for $G$ and $L^2_\psi(G)$ is the near equivalence class of irreducible subrepresentations $\pi$ of $L^2_\disc(G)$ such that for almost all places $v$ of $F$, the base change of $\pi_v$ to $\GL_n(E_v)$ is isomorphic to $\pi_{\psi,v}$.

We next describe this decomposition in terms of $L$-groups.
Recall that the $L$-group of $G = \U(\VV)$ is given by 
\[
 {}^L G = \GL_n(\C) \rtimes \Gal(\Qbar/F),
\]
where $\Gal(\Qbar/E)$ acts trivially on $\GL_n(\C)$ and the non-trivial element in $\Gal(E/F)$ acts as the automorphism $\theta_n$ defined by
\[
 \theta_n(g) = J_n \cdot {}^t g^{-1} \cdot J_n^{-1}, \qquad
 J_n = 
 \begin{pmatrix}
 & & & 1 \\
 & & -1 & \\
 & \varddots & & \\
 (-1)^{n-1} & & & 
 \end{pmatrix}.
\]
Put $\tilde{G} = \Res_{E/F} \GL_n$, so that its $L$-group is given by
\[
 {}^L \tilde{G} = \left( \GL_n(\C) \times \GL_n(\C) \right) \rtimes \Gal(\Qbar/F),
\]
where $\Gal(\Qbar/E)$ acts trivially on $\GL_n(\C) \times \GL_n(\C)$ and the non-trivial element in $\Gal(E/F)$ acts as the automorphism $(g_1,g_2) \mapsto (g_2,g_1)$.
Then the base change $L$-homomorphism $\BC:{}^L G \rightarrow {}^L \tilde{G}$ is given by 
\[
 \BC(g \rtimes \sigma) = (g, \theta_n(g)) \rtimes \sigma.
\]
Let $\psi$ be an elliptic $A$-parameter for $G$.
Then the local Langlands correspondence induces an (equivalence class of) $A$-parameter $\tilde{\psi}_v : \cL_{F_v} \times \SL_2(\C) \rightarrow {}^L \tilde{G}$ for any place $v$ of $F$, where 
\[
 \cL_{F_v} = 
 \begin{cases}
  \text{the Weil group of $F_v$} & \text{if $v$ is real;} \\
  \text{the Weil-Deligne group of $F_v$} & \text{if $v$ is finite.}
 \end{cases}
\]
Moreover, by \cite[Theorem 8.1]{ggp}, there exists a unique (equivalence class of) $A$-parameter $\psi_v : \cL_{F_v} \times \SL_2(\C) \rightarrow {}^L G$ such that $\tilde{\psi}_v = {\BC} \circ \psi_v$.
We associate to $\psi_v$ an $L$-parameter $\phi_{\psi_v}: \cL_{F_v} \rightarrow {}^L G$ by 
\[
 \phi_{\psi_v}(w) = \psi_v \! \left( w, \mat{|w|_v^{\frac{1}{2}}}{}{}{|w|_v^{-\frac{1}{2}}} \right).
\]
Then $L^2_\psi(G)$ consists of irreducible subrepresentations $\pi$ of $L^2_\disc(G)$ such that the $L$-parameter of $\pi_v$ is $\phi_{\psi_v}$ for almost all $v$.

For our applications, we will consider elliptic $A$-parameters with $n=4$ of the form
\begin{equation}
\label{eq:U(2,2)-A-param}
 \psi = \pi_E \boxtimes \Sym^1, 
\end{equation}
where 
\begin{itemize}
 \item $\pi$ is an irreducible cuspidal automorphic representation of $\GL_2(\A_F)$ with central character $\xi_E$;
 \item $\pi_E$ is the base change of $\pi$ to $\GL_2(\A_E)$.
\end{itemize}
We note the following:

\begin{lem}
\begin{enumerate}
\item If $\pi_E$ is cuspidal, then $\pi_E$ is conjugate-orthogonal.
\item If $\pi_E$ is not cuspidal, then $\pi_E = \chi \boxplus \chi^{-1}$ for some conjugate-orthogonal character $\chi$ of $\A_E^\times/E^\times$ such that $\chi^2 \ne 1$.
\end{enumerate}
\end{lem}

\begin{proof}
First assume that $\pi_E$ is not cuspidal.
Then $\pi$ is the automorphic induction of some character $\chi$ of $\A_E^\times/E^\times$, so that $\pi_E = \chi \boxplus \chi^\rho$.
Since $\pi$ is cuspidal, we have $\chi^\rho \ne \chi$.
Also, since the central character of $\pi$ is $\xi_E$, we have $\chi|_{\A_F^\times} = 1$, i.e.~$\chi$ is conjugate-orthogonal.
Hence the assertion follows.

Next assume that $\pi_E$ is cuspidal.
Put $H = \GL_2$ and $\tilde{H} = \Res_{E/F} \GL_2$, so that
\[
 {}^L H = \GL_2(\C) \times \Gal(\Qbar/F), \qquad
 {}^L \tilde{H} = \left( \GL_2(\C) \times \GL_2(\C) \right) \rtimes \Gal(\Qbar/F).
\]
Then the base change $L$-homomorphism $\BC:{}^L H \rightarrow {}^L \tilde{H}$ is given by 
\[
 \BC(h \times \sigma) = (h, h) \rtimes \sigma.
\]
Recall that $\mathrm{As}^+$ is the representation of ${}^L \tilde{H}$ on $\C^2 \otimes \C^2$ defined by
\begin{align*}
 \mathrm{As}^+((h_1, h_2) \rtimes 1)(x \otimes y) & = h_1 x \otimes h_2 y, \\
 \mathrm{As}^+((1, 1) \rtimes \sigma)(x \otimes y) & =
 \begin{cases}
  x \otimes y & \text{if $\sigma \in \Gal(\Qbar/E)$;} \\
  y \otimes x & \text{if $\sigma \notin \Gal(\Qbar/E)$.}
 \end{cases}
\end{align*}
Then we have
\[
 \mathrm{As}^+ \circ \BC \simeq \Sym^2 \oplus (\wedge^2 \otimes \xi_E)
\]
as representations of ${}^L H$.
Since the central character of $\pi$ is $\xi_E$, it follows that
\[
 L(s, \pi_E, \mathrm{As}^+) = L(s, \pi, \Sym^2) \cdot \zeta_F(s).
\]
This implies the assertion.
\end{proof}

\subsection{Local $A$-packets}
\label{ss:local-A-packets}

Let $\psi$ be an elliptic $A$-parameter for $G = \U(\VV)$ and $L^2_\psi(G)$ the near equivalence class associated to $\psi$.
Then the result of Kaletha-Minguez-Shin-White \cite[Theorem${}^*$ 1.7.1]{kmsw} also describes the local-global structure of $L^2_\psi(G)$ with a multiplicity formula.
In particular, if $\pi$ is an irreducible summand of $L^2_\psi(G)$, then for any place $v$ of $F$, the local component $\pi_v$ is an irreducible summand of some representation in $\Pi_{\psi_v}$.
Here $\Pi_{\psi_v}$ is the local $A$-packet associated to $\psi_v$ consisting of certain semisimple representations of $G_v$ of finite length.

Suppose that $v$ is real.
If $G_v$ is quasi-split and $\psi_v$ is ``cohomological'', then it follows from the result of Arancibia-M{\oe}glin-Renard \cite{amr} (the unitary group case had already been treated by Johnson \cite{johnson}) that $\Pi_{\psi_v}$ agrees with the packet constructed by Adams-Johnson \cite{aj}, which we recall below.
From now on, we suppress the subscript $v$ from the notation.

Let $\VV$ be an $n$-dimensional hermitian space over $\C$ of signature $(p,q)$.
Choosing a basis of $\VV$, we may identify the unitary group $G = \U(\VV)$ with
\[
 \U(p,q) = \{ g \in \GL_n(\C) \, | \, {}^t \bar{g} I_{p,q} g = I_{p,q} \},
 \qquad
 I_{p,q} = 
 \begin{pmatrix}
  \1_p & \\
  & - \1_q
 \end{pmatrix}.
\]
We define a Cartan involution $\theta$ of $G$ by $\theta(g) = {}^t \bar{g}^{-1}$.
Let $K \simeq \U(p) \times \U(q)$ be the maximal compact subgroup of $G$ with respect to $\theta$ and $T \simeq \U(1)^n$ the maximal torus of $G$ consisting of diagonal matrices.
Let $B_{G_\C}$ be the Borel subgroup of $G_\C \simeq \GL_n(\C)$ (which is not defined over $\R$) consisting of upper triangular matrices.

Let $\fg_0$, $\fk_0$, $\ft_0$ be the Lie algebras of $G$, $K$, $T$, respectively.
We have a Cartan decomposition $\fg_0 = \fk_0 \oplus \fp_0$, where $\fp_0$ is the $(-1)$-eigenspace of $\theta$.
Let $\fg$, $\fk$, $\fp$, $\ft$ be the complexifications of $\fg_0$, $\fk_0$, $\fp_0$, $\ft_0$, respectively.
Let $\fp^{\pm}$ be the $(\pm i)$-eigenspace of the complex structure on $\fp$ defined by 
\[
 X \longmapsto J X J^{-1}, \qquad 
 J = 
 \begin{pmatrix}
  e^{\frac{\pi i}{4}} \1_p & \\
  & e^{-\frac{\pi i}{4}} \1_q
 \end{pmatrix}.
\]
More explicitly, we have $\fg = \fk \oplus \fp^+ \oplus \fp^-$ with
\begin{align*}
 \fk & = \left\{ \left.
 \begin{pmatrix}
  A & 0 \\
  0 & D
 \end{pmatrix}
 \, \right| \, A \in \M_p(\C), \, D \in \M_q(\C) \right\}, \\
 \fp^+ & = \left\{ \left.
 \begin{pmatrix}
  0 & B \\
  0 & 0
 \end{pmatrix}
 \, \right| \, B \in \M_{p,q}(\C) \right\}, \\
 \fp^- & = \left\{ \left.
 \begin{pmatrix}
  0 & 0 \\
  C & 0
 \end{pmatrix}
 \, \right| \, C \in \M_{q,p}(\C) \right\}.
\end{align*}

The packet constructed by Adams-Johnson \cite{aj} consists of certain unitary representations $\pi$ such that $H^*(\fg, K; \pi \otimes F) \ne 0$ for some irreducible finite-dimensional representation $F$ of $G$.
Let $\lambda \in \ft^* \simeq \C^n$ be the highest weight of $F^*$ relative to $B_{G_\C}$.
We may write
\[
 \lambda = (\underbrace{\lambda_1, \dots, \lambda_1}_{n_1}, \underbrace{\lambda_2, \dots, \lambda_2}_{n_2}, \dots, \underbrace{\lambda_r, \dots, \lambda_r}_{n_r}) \in \Z^n
\]
with $\lambda_1 \ge \lambda_2 \ge \dots \ge \lambda_r$.
Then we consider the $A$-parameter $\psi : \cL_\R \times \SL_2(\C) \rightarrow {}^LG$ whose restriction to $\C^\times \times \SL_2(\C)$ is equal to
\[
 \left( \chi_{\lambda_1 + \rho_1} \boxtimes \Sym^{n_1-1} \right) \oplus \dots \oplus \left( \chi_{\lambda_r + \rho_r} \boxtimes \Sym^{n_r-1} \right),
\]
where
\begin{itemize}
 \item $\rho_i = \frac{1}{2}(- n_1 - \dots - n_{i-1} + n_{i+1} + \dots + n_r)$;
 \item $\chi_\kappa$ is the character of $\C^\times$ defined by $\chi_\kappa(z) = (z/\bar{z})^\kappa$;
 \item $\Sym^{d-1}$ is the irreducible $d$-dimensional representation of $\SL_2(\C)$.
\end{itemize}
This defines a parabolic subgroup $Q$ of $G_\C$ (which is not defined over $\R$) containing $B_{G_\C}$ with Levi component $L$ (which is defined over $\R$) such that
\[
 \psi = \xi \circ \psi_L,
\]
where $\xi : {}^L L \rightarrow {}^L G$ is the canonical embedding and $\psi_L :\cL_\R \times \SL_2(\C) \rightarrow {}^L L$ is an $A$-parameter such that the $A$-packet $\Pi_{\psi_L}$ consists of a single $1$-dimensional representation of $L$.
Note that 
\[
 L = G \cap \left( \GL_{n_1}(\C) \times \dots \times \GL_{n_r}(\C) \right).
\]
Put
\[
 S = W(L,T) \backslash W(G,T) / W_\R(G,T) \simeq \left( \mathfrak{S}_{n_1} \times \dots \times \mathfrak{S}_{n_r}\right) \backslash \mathfrak{S}_n / \left( \mathfrak{S}_p \times \mathfrak{S}_q \right),
\]
where $W$ and $W_\R$ denote the absolute and relative Weyl groups, respectively.
As a set of representatives for $S$, we can take the set of $w \in \mathfrak{S}_n$ such that 
\begin{itemize}
 \item $w^{-1}(i) < w^{-1}(j)$ for $n_1 + \dots + n_{k-1}+1 \le i < j \le n_1 + \dots + n_k$ for $1 \le k \le r$;
 \item $w(i) < w(j)$ for $1 \le i < j \le p$ and for $p+1 \le i < j \le n$.
\end{itemize}
For any $w \in S$, we have a $\theta$-stable parabolic subgroup $Q_w = w^{-1} Q w$ of $G_\C$ with Levi component $L_w = w^{-1} L w$.
Let $\fq_w$ be the Lie algebra of $Q_w$.
Then the Adams-Johnson packet $\Pi^{\mathrm{AJ}}_\psi$ is given by 
\[
 \Pi^{\mathrm{AJ}}_\psi = \left\{ A_{\fq_w}(w^{-1} \lambda) \, | \, w \in S \right\}.
\]

We now explicate the $A$-packet $\Pi_\psi$ when $G = \U(2,2)$ and $\psi$ is the localization of a global $A$-parameter as in \eqref{eq:U(2,2)-A-param}.
More precisely, we start with the discrete series representation of $\GL_2(\R)$ of weight $k+1$ with an even integer $k \ge 2$.
Since its base change to $\GL_2(\C)$ is the principal series representation $\Ind(\chi_{\frac{k}{2}} \otimes \chi_{-\frac{k}{2}})$, the associated $A$-parameter $\psi$ is given by 
\[
 \psi = \big( \chi_{\frac{k}{2}} \boxtimes \Sym^1 \big) \oplus \big( \chi_{-\frac{k}{2}} \boxtimes \Sym^1 \big),
\]
so that we need to take
\[
 \lambda = \big( \tfrac{k}{2}-1, \tfrac{k}{2}-1, -\tfrac{k}{2}+1, -\tfrac{k}{2}+1 \big).
\]
In this case, we have $S = \{ w_0, w_1, w_2 \}$ with
\[
 w_0 = 1, \qquad
 w_1 = (23), \qquad
 w_2 = (13)(24).
\]
Then
\[
 \Pi_\psi = \Pi^{\mathrm{AJ}}_\psi = \{ A_{\fq_i}(w_i^{-1} \lambda) \, | \, 0 \le i \le 2 \},
\]
where $\fq_i = \fq_{w_i}$.

\begin{prop}
\label{prop:Aq}
We have
\begin{align*}
 \dim H^{p,q}(\fg, K; A_{\fq_0}(w_0^{-1} \lambda) \otimes F) & = 
 \begin{cases}
  1 & \text{if $(p,q) = (4,0)$,}  \\
  0 & \text{otherwise,}
 \end{cases} \\
 \dim H^{p,q}(\fg, K; A_{\fq_1}(w_1^{-1} \lambda) \otimes F) & = 
 \begin{cases}
  1 & \text{if $(p,q) = (1,1), (3,3)$,}  \\
  2 & \text{if $(p,q) = (2,2)$,}  \\
  0 & \text{otherwise,}
 \end{cases} \\
 \dim H^{p,q}(\fg, K; A_{\fq_2}(w_2^{-1} \lambda) \otimes F) & = 
 \begin{cases}
  1 & \text{if $(p,q) = (0,4)$,}  \\
  0 & \text{otherwise,}
 \end{cases}
\end{align*}
where $F$ is the irreducible finite-dimensional representation of $G$ with highest weight $\lambda$.
\end{prop}

\begin{proof}
First, note that $F$ is self-dual, i.e.~$F^*$ has highest weight $\lambda$.
Also, if we write $\fq_i = \fl_i \oplus \fu_i$ with Levi component $\fl_i$ and unipotent radical $\fu_i$, then
\begin{align*}
 \fl_0 & = \left\{ 
 \begin{pmatrix}
  * & * & 0 & 0 \\
  * & * & 0 & 0 \\
  0 & 0 & * & * \\
  0 & 0 & * & * 
 \end{pmatrix}
 \right\}, &
 \fu_0 & = \left\{
 \begin{pmatrix}
  0 & 0 & * & * \\
  0 & 0 & * & * \\
  0 & 0 & 0 & 0 \\
  0 & 0 & 0 & 0 
 \end{pmatrix}
 \right\}, \\
 \fl_1 & = \left\{ 
 \begin{pmatrix}
  * & 0 & * & 0 \\
  0 & * & 0 & * \\
  * & 0 & * & 0 \\
  0 & * & 0 & * 
 \end{pmatrix}
 \right\}, &
 \fu_1 & = \left\{
 \begin{pmatrix}
  0 & * & 0 & * \\
  0 & 0 & 0 & 0 \\
  0 & * & 0 & * \\
  0 & 0 & 0 & 0 
 \end{pmatrix}
 \right\}, \\
 \fl_2 & = \left\{ 
 \begin{pmatrix}
  * & * & 0 & 0 \\
  * & * & 0 & 0 \\
  0 & 0 & * & * \\
  0 & 0 & * & * 
 \end{pmatrix}
 \right\}, &
 \fu_2 & = \left\{ 
 \begin{pmatrix}
  0 & 0 & 0 & 0 \\
  0 & 0 & 0 & 0 \\
  * & * & 0 & 0 \\
  * & * & 0 & 0 
 \end{pmatrix}
 \right\}.
\end{align*}
Hence the assertion follows from \cite[Proposition 6.19]{vz}.
\end{proof}

We also note the following:

\begin{lem}
\label{l:Aq-conj}
Let
\[
 \delta = 
 \begin{pmatrix}
  & \1_2 \\
  \1_2 & 
 \end{pmatrix}
 \in \GU(2,2).
\]
Then we have
\begin{align*}
 A_{\fq_0}(w_0^{-1} \lambda) \circ \Ad(\delta) & = A_{\fq_2}(w_2^{-1} \lambda), \\
 A_{\fq_1}(w_1^{-1} \lambda) \circ \Ad(\delta) & = A_{\fq_1}(w_1^{-1} \lambda).
\end{align*}
\end{lem}

\begin{proof}
The lemma immediately follows from the characterization of the cohomological representation as described in \S \ref{ss:coh-rep}.
\end{proof}

\subsection{The Hodge structure}
\label{sec:the-hodge-structure}

Suppose again that $F$ is a totally real number field and $E$ is a totally imaginary quadratic extension of $F$.
Let $\VV$ be a $4$-dimensional hermitian $E$-space.
We now change notation, and write $G = \GU(\VV)$ and $G' = \U(\VV)$ for the unitary similitude and unitary groups of $\VV$, respectively.
We assume that $\disc \VV \notin \N(E^\times)$ but $\disc \VV_v \in \N(E_v^\times)$ for all $v \in \Sigma_\infty$.
Then we may identify $G_v$ with the group
\[
 \{ g \in \GL_4(\C) \, | \, {}^t \bar{g} I_v g = \nu(g) \cdot I_v \},
\]
where
\[
 I_v = 
\begin{cases}
 \begin{pmatrix} \1_2 & \\ & -\1_2 \end{pmatrix} & \text{if $v \in \Sigma$;} \\
 \1_4 & \text{if $v \in \Sigma_\infty \smallsetminus \Sigma$}
\end{cases} 
\]
for some subset $\Sigma$ of $\Sigma_\infty$.
We further assume that $\Sigma \ne \Sigma_\infty$, so that $G'$ is anisotropic.
For $v \in \Sigma_\infty$, we define a maximal compact subgroup $K_v'$ of $G_v'$ by
\[
 K_v' = \left\{ \left. \begin{pmatrix} a & \\ & d \end{pmatrix} \, \right| \, a, d \in \U(2) \right\}
\]
if $v \in \Sigma$ and $K_v' = G_v'$ if $v \in \Sigma_\infty \smallsetminus \Sigma$.
Put $K_v = F_v^\times \cdot K_v'$, where we regard $F_v^\times \simeq \R^\times$ as a subgroup of $G_v$ via the map $z \mapsto z \1_4$.
Then $K_v$ is a maximal \emph{connected} compact modulo center subgroup of $G_v$.
Put
\begin{align*}
 G_\infty & = \prod_{v \in \Sigma_\infty} G_v, &
 K_\infty & = \prod_{v \in \Sigma_\infty} K_v, \\
 G_\infty' & = \prod_{v \in \Sigma_\infty} G'_v, &
 K_\infty' & = \prod_{v \in \Sigma_\infty} K_v'.
\end{align*}
Let $\fg$ and $\fg'$ be the complexified Lie algebras of $G_\infty$ and $G_\infty'$, respectively. 
Let $\Ss = \Res_{\C/\R} \mathbb{G}_m$ and $G_0 = \Res_{F/\Q} G$.
We define a homomorphism $h: \Ss \rightarrow G_{0,\R}$ by $h(z) = (h_v(z))_{v \in \Sigma_\infty}$ with 
\[
  h_v(z) = 
 \begin{cases}
  \begin{pmatrix} z \1_2 & \\ & \bar{z} \1_2 \end{pmatrix} & \text{if $v \in \Sigma$;} \\
  \1_4 & \text{if $v \in \Sigma_\infty \smallsetminus \Sigma$.}
 \end{cases} 
\]
Let $X$ be the $G_0(\R)$-conjugacy class of homomorphisms $\Ss \rightarrow G_{0,\R}$ containing $h$.
Then we have an identification
\[
 X = G_\infty / K_\infty.
\]
For any $v \in \Sigma_\infty$ and any even integer $k \ge 2$, let $(\rho_{v,k}, V_{v,k})$ be the irreducible algebraic representation of $G_v$ such that
\begin{itemize}
 \item $\rho_{v,k}$ has trivial central character;
 \item $\rho_{v,k}' = \rho_{v,k}|_{G_v'}$ is the irreducible finite-dimensional representation of $G_v'$ with highest weight
\[
 \lambda_v = \big( \tfrac{k}{2}-1, \tfrac{k}{2}-1, -\tfrac{k}{2}+1, -\tfrac{k}{2}+1 \big).
\]
\end{itemize}
Let $\uk = (k_v)_{v \in \Sigma_\infty}$ be a tuple of even integers $k_v \ge 2$ and put 
\[
 \rho_{\uk} = \bigotimes_{v \in \Sigma_\infty} \rho_{v,k_v}, \qquad 
 V_{\uk} = \bigotimes_{v \in \Sigma_\infty} V_{v, k_v}.
\]
For any open compact subgroup $\cK$ of $G(\A_{F,f})$ (where $\A_{F,f}$ denotes the ring of finite ad\`eles of $F$), let $\Sh_\cK$ be the Shimura variety associated to $(G_0, X, \cK)$:
\[
 \Sh_\cK = G(F) \backslash X \times G(\A_{F,f}) / \cK.
\]
Then $(\rho_{\uk}, V_{\uk})$ gives rise to a local system $\V_{\uk}$ on $\Sh_\cK$.
We have the Hodge decomposition
\[
 H^i(\Sh_\cK, \V_{\uk}) = \bigoplus_{p+q=i} H^{p,q}(\Sh_\cK, \V_{\uk}).
\]

In \S \ref{sec:constr-gl-ex-isom}, we have associated to $\VV$ a quaternion $F$-algebra $B$ and a $3$-dimensional skew-hermitian right $B$-space $\tilde{V}$ such that $\PGU_E(\VV) \simeq \PGU_B(\tilde{V})^0$.
By the above assumption on $\VV$, $B$ is division but $B_v$ is split for all $v \in \Sigma_\infty$.
Let $W$ be the $1$-dimensional hermitian left $B$-space as in \S \ref{ss:theta-setup}.
Then $\GU(W) \simeq B^\times$.
Let $\tau$ be an irreducible unitary automorphic representation of $\GU(W)(\A_F)^+$ with central character $\xi_E$ such that: 
\begin{itemize}
 \item $\tau_v$ is the anti-holomorphic discrete series representation of $\GL_2(\R)^+$ of weight $-k_v-1$ if $v \in \Sigma$;
 \item $\tau_v$ is the holomorphic discrete series representation of $\GL_2(\R)^+$ of weight $k_v+1$ if $v \in \Sigma_\infty \smallsetminus \Sigma$.
\end{itemize}
Let $\Pi = \theta(\tau)$ be the global theta lift of $\tau$ to $\GU(\tilde{V})(\A_F)$ relative to the standard additive character $\psi$ of $\A_F/F$ (i.e.~the additive character $\psi$ such that $\psi_v(x) = e^{2 \pi i x}$ for $v \in \Sigma_\infty$).
We assume that $\Pi$ is non-zero.
Then:
\begin{itemize}
 \item $\Pi$ is irreducible by Lemma \ref{lem:irred-theta-similitude};
 \item $\Pi$ has trivial central character by Proposition \ref{prop:spl-hodge}.
\end{itemize}
Hence we may regard $\Pi$ as an irreducible unitary automorphic representation of $G(\A_F)$ with trivial central character.
Let $S$ be a finite set of rational primes such that for all $p \notin S$ and all places $v$ of $F$ above $p$, 
\begin{itemize}
 \item $G_v$ is unramified over $F_v$;
 \item $\cK_v$ is a hyperspecial maximal compact subgroup of $G_v$;
 \item $\Pi_v$ has a non-zero $\cK_v$-fixed vector.
\end{itemize}
Let $\mathrsfs{H}^S = \mathrsfs{H}(G(\A_F^S), \cK^S)$ be the Hecke algebra of compactly supported $\cK^S$-bi-invariant functions on $G(\A_F^S)$, where $\A_F^S = {\prod'_{p \notin S}} \prod_{v \mid p} F_v$ and $\cK^S = \prod_{p \notin S} \prod_{v \mid p} \cK_v$.
Then $\mathrsfs{H}^S$ acts on $H^i(\Sh_\cK, \V_{\uk})$.
Put $\Pi^S = \bigotimes'_{p \notin S} \bigotimes_{v \mid p} \Pi_v$ and
\[
 H^i(\Sh_\cK, \V_{\uk})[\Pi^S] = \{ x \in H^i(\Sh_\cK, \V_{\uk}) \, | \, Tx = \chi(T) x \text{ for all } T \in \mathrsfs{H}^S \},
\]
where $\chi$ is the character of $\mathrsfs{H}^S$ associated to $\Pi^S$.
We define $H^{p,q}(\Sh_\cK, \V_{\uk})[\Pi^S]$ similarly.

\begin{prop}
\label{prop:hodge-only-(d,d)-survives}
We have 
\[
 H^{2d}(\Sh_\cK, \V_{\uk})[\Pi^S] = H^{d,d}(\Sh_\cK, \V_{\uk})[\Pi^S],
\]
where $d = |\Sigma|$.
\end{prop}

\begin{proof}
By Matsushima's formula \cite[VII.5.2]{borel-wallach}, we have
\[
 H^{2d}(\Sh_\cK, \V_{\uk}) \simeq \bigoplus_\pi m(\pi) H^{2d}(\fg, K_\infty; \pi_\infty \otimes \rho_{\uk}) \otimes \pi_f^\cK,
\]
where $\pi = \pi_\infty \otimes \pi_f$ runs over equivalence classes of irreducible admissible representations of $(\fg, K_\infty) \times G(\A_{F,f})$, $m(\pi)$ is the multiplicity of $\pi$ in the space of automorphic forms on $G(\A_F)$, and $\pi_f^\cK$ is the space of $\cK$-fixed vectors in $\pi_f$.
Since this isomorphism is compatible with the Hodge decompositions, it suffices to prove the following: if $m(\pi) > 0$ and $\pi_v \simeq \Pi_v$ for almost all $v$, then
\begin{equation}
\label{eq:hodge-only-(d,d)-survives}
 H^{2d}(\fg, K_\infty; \pi_\infty \otimes \rho_{\uk}) = H^{d,d}(\fg, K_\infty; \pi_\infty \otimes \rho_{\uk}).
\end{equation}

Under this assumption, $\pi_v$ has trivial central character for almost all $v$, and hence so is $\pi$ since it is automorphic.
Since $\pi_\infty$ and $\rho_{\uk}$ have trivial central character, we have
\begin{align*}
 H^{2d}(\fg, K_\infty; \pi_\infty \otimes \rho_{\uk}) & = H^{2d}(\fg', K'_\infty; \pi_\infty' \otimes \rho_{\uk}'), \\
 H^{d,d}(\fg, K_\infty; \pi_\infty \otimes \rho_{\uk}) & = H^{d,d}(\fg', K'_\infty; \pi_\infty' \otimes \rho_{\uk}'), 
\end{align*}
where $\pi_\infty' = \pi_\infty|_{(\fg', K'_\infty)}$ and $\rho'_{\uk} = \rho_{\uk}|_{G_\infty'}$.
Note that $\pi_\infty'$ and $\rho'_{\uk}$ remain irreducible.

Fix a realization $\mathrsfs{V}$ of $\pi$ in the space of automorphic forms on $G(\A_F)$.
Let $\mathrsfs{V}|_{G'(\A_F)}$ be the restriction of $\mathrsfs{V}$ to $G'(\A_F)$ as functions, so that $\mathrsfs{V}|_{G'(\A_F)}$ is a non-zero subspace of the space of automorphic forms on $G'(\A_F)$.
Fix an irreducible component $\pi'$ of $\mathrsfs{V}|_{G'(\A_F)}$.
Since the natural surjective map $\mathrsfs{V} \rightarrow \mathrsfs{V}|_{G'(\A_F)}$ is $(\fg', K'_\infty) \times G'(\A_{F,f})$-equivariant, $\pi'_v$ is an irreducible component of $\pi_v|_{G_v'}$ (resp.~$\pi'_v = \pi_v|_{(\fg_v', K_v')}$) if $v$ is finite (resp.~if $v$ is real).

We now compute the $A$-parameter $\psi$ of $\pi'$.
Choose an irreducible unitary cuspidal automorphic representation $\tilde{\tau}$ of $\GL_2(\A_F)$ so that $\tau$ is an irreducible component of $\tilde{\tau}^B|_{B^\times(\A_F)^+}$, where $\tilde{\tau}^B$ is the Jacquet-Langlands transfer of $\tilde{\tau}$ to $B^\times(\A_F)$.
Let $\tau_E$ be the base change of $\tilde{\tau}$ to $\GL_2(\A_E)$.
Note that $\tau_E$ does not depend on the choice of $\tilde{\tau}$.
For almost all $v$, $\tilde{\tau}_v$ is a principal series representation $\Ind(\chi_v \otimes \chi_v^{-1} \xi_{E_v})$ of $\GL_2(F_v)$ for some unramified character $\chi_v$ of $F_v^\times$.
Hence 
\[
 \tau_{E,v} = \Ind(\eta_v \otimes \eta_v^{-1})
\]
for almost all $v$, where $\eta_v = \chi_v \circ \N_{E_v/F_v}$.
On the other hand, by Lemmas \ref{l:isom-tori} and \ref{l:theta-unram}, $\pi_v$ is the unique irreducible unramified subquotient of
\[
 \Ind^{G_v}_{B_{G_v}}(\eta_v |\cdot|_{E_v}^{\frac{1}{2}} \otimes \eta_v |\cdot|_{E_v}^{-\frac{1}{2}} \otimes \chi_v^{-2})
\]
for almost all $v$, where $B_{G_v}$ is the standard Borel subgroup of $G_v$ containing the maximal torus $T_v \simeq (E_v^\times)^2 \times F_v^\times$ as in \S \ref{ss:theta-3}.
Hence $\pi'_v$ is the unique irreducible unramified subquotient of
\[
 \Ind^{G'_v}_{B_{G'_v}}(\eta_v |\cdot|_{E_v}^{\frac{1}{2}} \otimes \eta_v |\cdot|_{E_v}^{-\frac{1}{2}})
\]
for almost all $v$, where $B_{G'_v} = B_{G_v} \cap G_v'$ is the standard Borel subgroup of $G_v'$ containing the maximal torus $T_v \cap G_v' \simeq (E_v^\times)^2$.
Namely, we have
\[
 \psi = \tau_E \boxtimes \Sym^1.
\]

Thus, by the classification of automorphic representations of $G'(\A_F)$, $\pi_v'$ is an irreducible summand of some representation in the local $A$-packet $\Pi_{\psi_v}$ for all $v$.
In particular, if $v \in \Sigma$, then $\pi_v'$ is one of the representations $A_{\fq_i}(w_i^{-1} \lambda_v)$ as in \S \ref{ss:local-A-packets}, and hence we have
\[
 H^i(\fg_v',K'_v; \pi_v' \otimes \rho'_{v, k_v}) = 0
\]
for $i<2$ and
\[
 H^2(\fg_v',K'_v; \pi_v' \otimes \rho'_{v, k_v}) = H^{1,1}(\fg_v',K'_v; \pi_v' \otimes \rho'_{v, k_v})
\]
by Proposition \ref{prop:Aq}.
From this, we deduce that
\begin{align*}
 H^{2d}(\fg', K'_\infty; \pi_\infty' \otimes \rho_{\uk}')
 & = \bigotimes_{v \in \Sigma}
 H^2(\fg_v', K'_v; \pi_v' \otimes \rho_{v, k_v}')
 \bigotimes_{v \in \Sigma_\infty \smallsetminus \Sigma}
 H^0(\fg_v', K'_v; \pi_v' \otimes \rho_{v, k_v}') \\
 & = \bigotimes_{v \in \Sigma}
 H^{1,1}(\fg_v', K'_v; \pi_v' \otimes \rho_{v, k_v}')
 \bigotimes_{v \in \Sigma_\infty \smallsetminus \Sigma}
 H^0(\fg_v', K'_v; \pi_v' \otimes \rho_{v, k_v}') \\
 & = H^{d,d}(\fg', K'_\infty; \pi_\infty' \otimes \rho_{\uk}').
\end{align*}
This completes the proof.
\end{proof}

\begin{rem}
The classification of automorphic representations is used in the proof of Proposition \ref{prop:hodge-only-(d,d)-survives}, but in fact, we can avoid appealing to the result of Kaletha-Minguez-Shin-White \cite{kmsw} as follows.
Let $\pi$ be an irreducible automorphic representation of $G(\A_F)$ such that $\pi_v \simeq \Pi_v$ for almost all $v$.
Then, by Proposition \ref{p:near-equiv-theta}, we may write $\pi$ as a global theta lift.
Hence, if $\pi'_v$ is an irreducible component of $\pi_v|_{(\fg_v', K_v')}$ for $v \in \Sigma$, then it follows from the description of local theta lifts \cite{li-duke} (see also \S \ref{ss:local-theta} and Lemma \ref{l:Aq-conj}) that $\pi'_v = A_{\fq_i}(w_i^{-1} \lambda_v)$ for some $i$.
This implies \eqref{eq:hodge-only-(d,d)-survives}.
\end{rem}

\begin{prop}
\label{prop:hodge-is-of-type-(d,d)}
The Hodge structure $H^{2d}(\Sh_\cK, \V_{\uk})[\Pi^S]$ is purely of type $(d,d)$. 
\end{prop}

\begin{proof}
From the proof of Proposition \ref{prop:hodge-only-(d,d)-survives}, we need to compute the Hodge structure on 
\[
H^{d,d}(\fg', K'_\infty; \pi_\infty' \otimes \rho_{\uk}') \simeq  \bigotimes_{v \in \Sigma} H^{1,1}(\fg_v', K'_v; \pi_v' \otimes \rho_{v, k_v}')
 \bigotimes_{v \in \Sigma_\infty \smallsetminus \Sigma}
 H^0(\fg_v', K'_v; \pi_v' \otimes \rho_{v, k_v}').
 \]
 The term $H^0(\fg_v', K'_v; \pi_v' \otimes \rho_{v, k_v}')$ (for
 $v\in \Sigma_\infty \smallsetminus \Sigma$) is clearly of type $(0,0)$, so we are reduced to showing that 
 $H^{1,1}(\fg_v', K'_v; \pi_v' \otimes \rho_{v, k_v}')$ (for $v\in \Sigma$) is of type
 $(1,1)$. The Hodge type can be computed using \cite{zucker}, keeping
 in mind that {\it loc.~cit.}  gives the Hodge
 numbers in the standard normalization; they need to be 
twisted appropriately to get the Hodge type in the automorphic
(unitary) normalization that we are using. 
For ease
 of comparison with \cite{zucker}, we temporarily change notation to
 match that reference. (See also \cite{faltings-bgg}, \cite{harris-motivesvolume}.)

Fix $v \in \Sigma$ for the rest of the proof. 
Let $W \simeq \mathfrak{S}_4$ and $W_c \simeq \mathfrak{S}_2 \times \mathfrak{S}_2$ be the Weyl groups of $\fg'_v = \mathfrak{gl}(4,\C)$ and $\fk'_v = \mathfrak{gl}(2,\C) \oplus \mathfrak{gl}(2,\C)$, respectively. 
Let $W^1$ be the set of representatives for $W_c \backslash W$ given by 
\[
 \{ w \in W \, | \, w^{-1}(\Delta_c^+) \subset \Delta^+ \}, 
\]
where $\Delta_c^+$ and $\Delta^+$ are the sets of positive roots in $\fk'_v$ and $\fg'_v$, respectively. 
Put
\[
 W^1(p) = \{ w \in W^1 \, | \, \ell(w) = p \},
\]
where $\ell(w)$ is the length of $w$.
Then we can enumerate the elements in $W^1$ as follows:
\[
\begin{array}{c|cccccc}
 p & 0 & 1 & 2 & 2 & 3 & 4 \\ \hline
 w^{-1} & 1 & (23) & (243) & (123) & (1243) & (13)(24)
\end{array} 
\]
Recall that $k_v$ is a positive even integer and put 
\[
 \Lambda = \tfrac{1}{2}(k_v-2,k_v-2,-k_v+2,-k_v+2).
\]
Let $\rho$ be half the sum of positive roots in $\fg_v'$:
\[
 \rho = \tfrac{1}{2}(3,1,-1,-3).
\]

 Let $Z=\{t_\alpha:= (\alpha, \alpha^{-1} ) \in \U(1)\times \U(1) \} \subset \U(2)
 \times \U(2)$. 
As in \cite{zucker}, \S 1 and \S 4, let $\mu$ and $\lambda$ denote the
highest characters of $Z$ appearing in the adjoint representation
of $\U(2,2)$ and in the representation $\rho'_{v,k_v}$. Then $\mu$ is
just the action on $(\fp_v')^+$ and is explicitly given by
\[
\mu (t_\alpha) = \alpha^2.
\]
As for $\lambda$, it agrees with the action of $Z$ on the irreducible
representation of $\U(2) \times \U(2)$ with highest weight $\Lambda$;
since this representation is just $\det^{\frac{k_v}{2}-1} \boxtimes
\det^{-\frac{k_v}{2}+1}$, we see that
\[
\lambda(t_\alpha) = (\alpha^2)^{\frac{k_v}{2}-1}
(\alpha^{-2})^{-\frac{k_v}{2}+1} = \alpha^{2k_v-4}.
\]
Since $\rho'_{v,k_v}$ is self-dual, the lowest character of $Z$
appearing in $\rho'_{v,k_v}$ is simply $\lambda^{-1}$; thus $m=
2k_v-4$, where $m$ is defined as in {\it loc.~cit.} equation (4.8). 
Here $m$ is the total weight of the Hodge structure on the fiber of
the local system.
(To convert to our normalization, where $\rho_{v,k_v}$ has trivial central character and hence the total weight on the fiber is zero, we must therefore twist the Hodge numbers below by $(2-k_v,2-k_v)$.)

For completeness, we consider not just $H^{1,1}$ but all the non-zero
  $H^{p,q}(\fg_v', K'_v; \pi_v' \otimes \rho_{v, k_v}')$, where
  $\pi_v'$ is chosen such that $H^{p,q} \neq 0$. 
This space is then the sum of components of multi-degree $(p,q);(r,s)$
where $(r,s)$ with 
$r+s=m$ is the bidegree coming from the Hodge structure on the fiber. 
By \cite{zucker}, \S
  5,  the $(p,q);(k-p,m+p-k)$ component can only be non-zero if the
action $\tau_Z$ of 
$Z$ on the irreducible representation $\tau$ of $\U(2) \times \U(2)$ with 
highest weight $w(\Lambda + \rho) - \rho$ is $\lambda - k\mu$
 for some $w \in W^1 (p)$. 
Thus we just need to run through the different choices of $(p,q)$ and
$w\in W^1 (p)$. 

\begin{itemize}
\item If $(p,q)=(0,4)$ and $w^{-1}=1$, then 
\begin{align*}
 w(\Lambda+\rho) & = \tfrac{1}{2}(k_v+1,k_v-1,-k_v+1,-k_v-1), \\
 w(\Lambda+\rho) - \rho & = \tfrac{1}{2}(k_v-2,k_v-2,-k_v+2,-k_v+2), 
\end{align*}
so that
\begin{align*}
 \tau & = (\Sym^{0} \otimes {\det}^{\frac{k_v}{2}-1}) \boxtimes (\Sym^{0} \otimes {\det}^{-\frac{k_v}{2}+1}), \\
\tau_Z & : t_\alpha \mapsto (\alpha^2)^{\frac{k_v}{2}-1} (\alpha^{-2})^{-\frac{k_v}{2}+1}=\alpha^{2k_v-4}.
\end{align*}
Then $k=0$, so the Hodge type is 
\[
(0,4) + (0,2k_v-4) = (0,2k_v).
\]

\item If $(p,q)=(1,1)$ and $w^{-1}=(23)$, then 
\begin{align*}
 w(\Lambda+\rho) & = \tfrac{1}{2}(k_v+1,-k_v+1,k_v-1,-k_v-1), \\
 w(\Lambda+\rho) - \rho & = \tfrac{1}{2}(k_v-2,-k_v,k_v,-k_v+2), 
\end{align*}
so that
\begin{align*}
 \tau & = (\Sym^{k_v-1} \otimes {\det}^{-\frac{k_v}{2}}) \boxtimes (\Sym^{k_v-1} \otimes {\det}^{-\frac{k_v}{2}+1}), \\
 \tau_Z & : t_\alpha \mapsto \alpha^{k_v-1} (\alpha^2)^{-\frac{k_v}{2}}(\alpha^{-1})^{k_v-1} (\alpha^{-2})^{-\frac{k_v}{2}+1}=\alpha^{-2}. 
\end{align*}
Then $k=k_v-1$, so the Hodge type is 
\[
(1,1) + (k_v-1-1, 2k_v-4 +1-(k_v-1) = (k_v-1,k_v-1).
\]

\item If $(p,q)=(2,2)$ and $w^{-1}=(243)$, then 
\begin{align*}
 w(\Lambda+\rho) & = \tfrac{1}{2}(k_v+1,-k_v-1,k_v-1,-k_v+1), \\
 w(\Lambda+\rho) - \rho & = \tfrac{1}{2}(k_v-2,-k_v-2,k_v,-k_v+4), 
\end{align*}
so that
\begin{align*}
 \tau & = (\Sym^{k_v} \otimes {\det}^{-\frac{k_v}{2}-1}) \boxtimes (\Sym^{k_v-2} \otimes {\det}^{-\frac{k_v}{2}+2}), \\
\tau_Z & : t_\alpha \mapsto \alpha^{k_v} (\alpha^2)^{-\frac{k_v}{2}-1}(\alpha^{-1})^{k_v-2}(\alpha^{-2})^{-\frac{k_v}{2}+2}=\alpha^{-4}.
\end{align*}
Then $k=k_v$, so the Hodge type is
\[
(2,2) + (k_v-2, 2k_v-4 + 2-k_v) = (k_v, k_v).
\]

\item If $(p,q)=(2,2)$ and $w^{-1}=(123)$, then 
\begin{align*}
 w(\Lambda+\rho) & = \tfrac{1}{2}(k_v-1,-k_v+1,k_v+1,-k_v-1), \\
 w(\Lambda+\rho) - \rho & = \tfrac{1}{2}(k_v-4,-k_v,k_v+2,-k_v+2), 
\end{align*}
so that
\begin{align*}
 \tau & = (\Sym^{k_v-2} \otimes {\det}^{-\frac{k_v}{2}}) \boxtimes (\Sym^{k_v} \otimes {\det}^{-\frac{k_v}{2}+1}), \\
\tau_Z & : t_\alpha \mapsto \alpha^{k_v-2}
(\alpha^2)^{-\frac{k_v}{2}}(\alpha^{-1})^{k_v}(\alpha^{-2})^{-\frac{k_v}{2}+1}=\alpha^{-4}.
\end{align*}
Then $k=k_v$, so the Hodge type is
\[
(2,2) + (k_v-2, 2k_v-4 + 2-k_v) = (k_v, k_v).
\]

\item If $(p,q)=(3,3)$ and $w^{-1}=(1243)$, then 
\begin{align*}
 w(\Lambda+\rho) & = \tfrac{1}{2}(k_v-1,-k_v-1,k_v+1,-k_v+1), \\
 w(\Lambda+\rho) - \rho & = \tfrac{1}{2}(k_v-4,-k_v-2,k_v+2,-k_v+4), 
\end{align*}
so that
\begin{align*}
 \tau & = (\Sym^{k_v-1} \otimes {\det}^{-\frac{k_v}{2}-1}) \boxtimes (\Sym^{k_v-1} \otimes {\det}^{-\frac{k_v}{2}+2}), \\
\tau_Z & : t_\alpha \mapsto \alpha^{k_v-1}
(\alpha^2)^{-\frac{k_v}{2}-1}(\alpha^{-1})^{k_v-1}(\alpha^{-2})^{-\frac{k_v}{2}+2}=\alpha^{-6}.
\end{align*}
Then $k=k_v+1$, so the Hodge type is
\[
(3,3) + (k_v+1-3, 2k_v-4 + 3-(k_v+1)) = (k_v+1, k_v+1).
\]

\item If $(p,q)=(4,0)$ and $w^{-1}=(13)(24)$, then 
\begin{align*}
 w(\Lambda+\rho) & = \tfrac{1}{2}(-k_v+1,-k_v-1,k_v+1,k_v-1), \\
 w(\Lambda+\rho) - \rho & = \tfrac{1}{2}(-k_v-2,-k_v-2,k_v+2,k_v+2), 
\end{align*}
so that
\begin{align*}
 \tau & = (\Sym^{0} \otimes {\det}^{-\frac{k_v}{2}-1}) \boxtimes (\Sym^{0} \otimes {\det}^{\frac{k_v}{2}+1}), \\
\tau_Z & : t_\alpha \mapsto 
(\alpha^2)^{-\frac{k_v}{2}-1}(\alpha^{-2})^{\frac{k_v}{2}+1}=\alpha^{-2k_v-4}.
\end{align*}
Then $k=2k_v$, so the Hodge type is
\[
(4,0) + (2k_v-4, 2k_v-4 + 4-2k_v) = (2k_v, 0).
\]

\end{itemize}

The relevant case for us is the case $(p,q)=(1,1)$ in which case the overall
Hodge type $(k_v-1,k_v-1)$; twisting it by $(2-k_v, 2-k_v)$, we see
that $H^{1,1}(\fg_v', K'_v; \pi_v' \otimes \rho_{v, k_v}')$ has Hodge type $(1,1)$ as expected. 
\end{proof}

Finally, we note that by \S \ref{ss:aut-coh-ls}, we may regard $\Pi$ as an automorphic representation of any of the groups 
$\tilde{\mathrsfs{G}}(\A)$, $\mathrsfs{G}(\A)$, $\mathrsfs{G}_B (\A)$, $\tilde{\mathrsfs{G}}_B (\A)$. 
Let $S$ and $\cK$ be as in \S \ref{ss:aut-coh-ls} as well. Then we get

\begin{cor}
\label{cor:hodge-is-purely-of-type-(d,d)}
The Hodge structure on $H^{2d}(\Sh_{\tilde{\mathrsfs{G}}_{B,\cK}}, \V_{\uk})[\Pi^S]$ is purely of type $(d,d)$. 
\end{cor}

\subsection{Galois representations}
\label{sec:gal-rep-on-H2}
Finally we state the main result we need on Galois representations.
\begin{prop}
\label{prop:galrep-char}
Assume Kottwitz's conjecture for Shimura varieties attached to unitary similitude groups.
Then the action of $\Gal(\Qbar/F_\Sigma)$ on  
\[
H^{2d} (\Sh_{\tilde{\mathrsfs{G}},\cK}, \V_{\uk,\ell}) [\Pi^S] (d)
\]
is trivial.
\end{prop}

The proposition above encodes the expected relation between the automorphic form $\Pi$ and the cohomology of the Shimura variety 
$\Sh_{\tilde{\mathrsfs{G}},\cK}$, and is consistent with Corollary \ref{cor:hodge-is-purely-of-type-(d,d)}.
As such, it is an immediate consequence of the following special case of Kottwitz's conjecture \cite[\S 10]{kot}:

\begin{prop}
\label{prop:galrep-char-ss}
Assume Kottwitz's conjecture for Shimura varieties attached to unitary similitude groups.
Then the action of $\Gal(\Qbar/F_\Sigma)$ on the semisimplification of
\[
H^{2d} (\Sh_{\tilde{\mathrsfs{G}},\cK}, \V_{\uk,\ell}) [\Pi^S] (d)
\]
is trivial.
\end{prop}

\begin{rem}
Proposition \ref{prop:galrep-char} follows directly from Proposition \ref{prop:galrep-char-ss} by a standard argument using the finiteness of the class number of $F_\Sigma$ (as in \cite[\S 5.13]{nekovar}).
For convenience of the reader, we include the argument here.
Recall that any subquotient of the representation of $\Gal(\Qbar/F_\Sigma)$ on $H^{2d} (\Sh_{\tilde{\mathrsfs{G}},\cK}, \V_{\uk,\ell}) [\Pi^S] (d)$ (regarded as a $\Q_\ell$-vector space) is unramified at almost all places and is de Rham at all places dividing $\ell$.
Thus it suffices to show that $H^1_g(k,\Q_\ell) = 0$ for any number field $k$.
Here $H^1_g(k,\Q_\ell)$ is the Bloch-Kato Selmer group \cite{bloch-kato} consisting of elements $x \in H^1(k,\Q_\ell)$ such that 
\begin{itemize}
\item $x_v \in H^1_f(k_v, \Q_\ell)$ for almost all $v$;
\item $x_v \in H^1_g(k_v, \Q_\ell)$ for all $v$ dividing $\ell$, 
\end{itemize}
where for any finite place $v$ of $k$, $x_v$ denotes the restriction of $x$ to $H^1(k_v,\Q_\ell)$.
By class field theory, we identify 
\[
 H^1(k_v,\Q_\ell) = \Hom_{\mathrm{cont}}(\Gal(\overline{k_v}/k_v), \Q_\ell)
 = \Hom_{\mathrm{cont}}(k_v^\times, \Q_\ell).
\]
Under this identification, we have $H^1_f(k_v,\Q_\ell) = H^1(k_v,\Q_\ell) = \Hom(k_v^\times/\mathfrak{o}_v^\times, \Q_\ell)$ if $v$ does not divide $\ell$, and $H^1_g(k_v,\Q_\ell) = \Hom(k_v^\times/\mathfrak{o}_v^\times, \Q_\ell)$ if $v$ divides $\ell$ by \cite[Example 3.9]{bloch-kato}.
Here $\mathfrak{o}_v$ is the ring of integers of $k_v$.
From this, we deduce that 
\[
 H^1_g(k,\Q_\ell) = \Hom(\A_k^\times / k^\times k_\infty^\times \widehat{\mathfrak{o}}^\times, \Q_\ell)
\]
with $k_\infty = k \otimes_\Q \R$ and $\widehat{\mathfrak{o}} = \prod_v \mathfrak{o}_v$.
But since $\A_k^\times / k^\times k_\infty^\times \widehat{\mathfrak{o}}^\times$ is finite, we have $H^1_g(k,\Q_\ell) = 0$.
\end{rem}

\begin{rem}
We remark that Kottwitz's conjecture for $\Sh_{\tilde{\mathrsfs{G}},\cK}$ should follow from the stable trace formula for Shimura varieties of abelian type established by Kisin-Shin-Zhu \cite{ksz}, but is conditional on the classification of automorphic representations on unitary similitude groups and the equality \cite[(9.2.2.1)]{ksz} of certain stable distributions.  This is explained in more detail in Remark \ref{rem:characterization-of-galrep} in the introduction. 
\end{rem}

In the next section, we explain how to deduce Proposition \ref{prop:galrep-char-ss} from Kottwitz's conjecture.

\subsection{Kottwitz's conjecture}
\label{subsec:Kottwitz}

Put $\Gamma_k = \Gal(\Qbar/k)$ for a number field $k$.
Then $\Sigma_\infty$ (regarded as the set of embeddings of $F$ in $\C$) admits a natural action of $\Gamma_\Q$ induced by the inclusion $\Qbar \hookrightarrow \C$.
We identify $\Sigma_\infty$ with $\Gamma_\Q/\Gamma_F$ so that the fixed embedding $F \hookrightarrow \Qbar$ corresponds to the trivial coset $\Gamma_F$.
Choose a set of representatives $\{ \sigma_1, \dots, \sigma_n \}$ for $\Gamma_\Q/\Gamma_F$ so that $\Sigma = \{\sigma_1 \Gamma_F, \dots, \sigma_d \Gamma_F\}$, where $n = [F:\Q]$ and $d = |\Sigma|$.
Define an action of $\Gamma_\Q$ on $\{1, \dots, n\}$ so that
\[
 \gamma \sigma_i \Gamma_F = \sigma_{\gamma(i)} \Gamma_F
\]
for $\gamma \in \Gamma_\Q$.
We denote by $F_\Sigma$ the fixed field of the subgroup
\[
 \{\sigma \in \Gamma_\Q \, | \, \sigma \Sigma = \Sigma \}.
\]

Recall that $G = \GU(\VV)$ with 
\[
 G_v \simeq
 \begin{cases}
  \GU(2,2) & \text{if $v \in \Sigma$;} \\
  \GU(4) & \text{if $v \in \Sigma_\infty \smallsetminus \Sigma$}
 \end{cases}
\]
and $G_0 = \Res_{F/\Q} G$.
(Note that $G_0 = \tilde{\mathrsfs{G}}$ with the notation of \S \ref{ss:unitary-quat-unitary}.)
Then we have
\[
 {}^L G = \hat{G} \rtimes \Gamma_F, \qquad
 \hat{G} = \GL_4(\C) \times \C^\times,
\]
where $\Gamma_E$ acts trivially on $\hat{G}$ and the non-trivial element in $\Gal(E/F)$ acts as the automorphism $\hat{\theta}$ defined by
\[
 \hat{\theta}(g,\nu) = (\theta_4(g), \nu \cdot \det g).
\]
Also, by \cite[\S 5]{borel-corvallis}, we have
\[
 {}^LG_0 = \hat{G}_0 \rtimes \Gamma_\Q, \qquad 
 \hat{G}_0 = (\hat{G})^n,
\]
where $\gamma \in \Gamma_\Q$ acts on $\hat{G}_0$ as the automorphism
\[
 (g_1, \dots, g_n) \mapsto (\gamma_1 \cdot g_{\gamma^{-1}(1)}, \dots, \gamma_n \cdot g_{\gamma^{-1}(n)})
\]
with 
\[
 \gamma_i = \sigma_i^{-1} \gamma \sigma_{\gamma^{-1}(i)} \in \Gamma_F.
\]

To describe the Galois representation on the cohomology of the Shimura variety $\Sh_{\tilde{\mathrsfs{G}},\cK}$, we need to introduce some representation of the $L$-group.
Following \cite[\S 5.1]{blasius-rogawski}, we recall its definition.
Let $\mu: \mathbb{G}_{m,\C} \rightarrow \mathbb{S}_\C \rightarrow G_{0,\C}$ be the cocharacter associated to the homomorphism $h: \Ss \rightarrow G_{0,\R}$ as in \S \ref{sec:the-hodge-structure}.
More explicitly, we have $\mu(z) = (\mu_v(z))_{v \in \Sigma_\infty}$ with
\[
  \mu_v(z) = 
 \begin{cases}
  \begin{pmatrix} z \1_2 & \\ & \1_2 \end{pmatrix} \times z & \text{if $v \in \Sigma$;} \\
  \1_4 \times 1 & \text{if $v \in \Sigma_\infty \smallsetminus \Sigma$.}
 \end{cases} 
\]
From this, we see that the reflex field of the Shimura datum $(G_0,X)$ is $F_\Sigma$.
We also identify $\mu$ with a character of the standard maximal torus of $\hat{G}_0$.
Let $r_0$ be the irreducible algebraic representation of $\hat{G}_0$ with extreme weight $-\mu$, which can be explicated as follows.
Let $\wedge^2 \C^4$ be the exterior square of the standard representation of $\GL_4(\C)$ and regard it as a representation of $\hat{G}$ by letting $\nu \in \C^\times$ act as the scalar $\nu$.
We denote by $r$ its contragredient on $\cV = (\wedge^2 \C^4)^*$.
Then $r_0$ is the representation of $\hat{G}_0$ on $\cV_0 = \cV^{\otimes d}$ given by 
\[
 r_0(g_1, \dots, g_n) = r(g_1) \otimes \dots \otimes r(g_d).
\]
On the other hand, since $(\wedge^2 \C^4)^* \otimes \det \simeq \wedge^2 \C^4$ as representations of $\GL_4(\C)$, there exists a unique automorphism $A$ of $\cV$ such that
\[
 r(\hat{\theta}(g)) \circ A = A \circ r(g)
\]
for all $g \in \hat{G}$ and such that $A$ fixes the highest weight vector (unique up to scalars) in $\cV$ with respect to the standard Borel subgroup of $\hat{G}$.
Then we can extend $r$ to ${}^L G$ by setting
\[
 r(1 \rtimes \sigma) = 
 \begin{cases}
  \id & \text{if $\sigma \in \Gamma_E$;} \\
  A & \text{otherwise}
 \end{cases}
\]
and hence $r_0$ to $\hat{G}_0 \rtimes \Gamma_{F_\Sigma}$ by setting
\[
 r_0(1 \rtimes \gamma)(x_1 \otimes \dots \otimes x_d)
 = r(1 \rtimes \gamma_1) x_{\gamma^{-1}(1)} \otimes \dots \otimes r(1 \rtimes \gamma_d) x_{\gamma^{-1}(d)}.
\]

We also need to introduce the expected classification of automorphic representations of $G(\A_F)$.
Let $L^2_{\disc}(G)$ be the discrete spectrum of the unitary representation of $G(\A_F)$ on the Hilbert space $L^2(A_G(F_\infty)^0 G(F) \backslash G(\A_F))$, where $A_G$ is the split component of the center of $G$ and $F_\infty = F\otimes_\Q \R$.
We say that a pair $(\psi',\chi)$ is an elliptic $A$-parameter for $G$ if 
\begin{itemize}
\item $\psi'$ is an elliptic $A$-parameter for $G'$;
\item $\chi$ is a character of $\A_E^\times/E^\times$ such that $\chi^\rho/\chi$ is equal to the central character of $\pi_{\psi'}$.
\end{itemize}
Then one expects the decomposition 
\[
 L^2_{\disc}(G) = \bigoplus_\psi L^2_\psi(G),
\]
where $\psi = (\psi',\chi)$ runs over elliptic $A$-parameters for $G$ and $L^2_\psi(G)$ is the near equivalence class of irreducible subrepresentations $\pi$ of $L^2_\disc(G)$ such that for almost all places $v$ of $F$, the base change of $\pi_v$ to $\GL_4(E_v)  \times E_v^\times$ is isomorphic to $\pi_{\psi',v} \boxtimes \chi_v$.

To compute the Galois representation, it is convenient to introduce the hypothetical Langlands group $\cL_k$ of a number field $k$ equipped with a surjective homomorphism $\pr: \cL_k \twoheadrightarrow \Gamma_k$.
Let $\psi$ be an elliptic $A$-parameter for $G$ and regard it as an $L$-homomorphism $\psi : \cL_F \times \SL_2(\C) \rightarrow {}^L G$.
Let $\phi_\psi : \cL_F \rightarrow {}^L G$ be the $L$-parameter associated to $\psi$, i.e., 
\[
 \phi_{\psi}(w) = \psi \! \left( w, \mat{|w|^{\frac{1}{2}}}{}{}{|w|^{-\frac{1}{2}}} \right).
\]
Then we have a representation $r(\psi) = (r \circ \phi_{\psi}) \otimes |\cdot|^{-2}$ of $\cL_F$ on $\cV$ equipped with a decomposition
\[
 \cV = \bigoplus_i \cV^i,
\]
where 
\[
 \cV^i = \left\{ v \in \cV \, \left| \, (r \circ \psi) \left( 1, \smat{t}{}{}{t^{-1}} \right) v = t^{-i} v \text{ for all } t \in \C^\times \right. \right\}.
\]
Similarly, if $\psi_0: \cL_\Q \times \SL_2(\C) \rightarrow {}^L G_0$ is the $A$-parameter induced by $\psi$, then we have a representation $r(\psi_0) = (r_0 \circ \phi_{\psi_0}) \otimes |\cdot|^{-2d}$ of $\cL_{F_\Sigma}$ on $\cV_0$ equipped with a decomposition
\[
 \cV_0 = \bigoplus_i \cV^i_0.
\]
More explicitly, we have
\[
 \psi_0(\gamma, h) = (g(\gamma_1, h), \dots, g(\gamma_d, h)) \rtimes \pr(\gamma),
\]
where we write $\psi(\sigma, h) = g(\sigma, h) \rtimes \pr(\sigma)$ and $\gamma_i \in \cL_F$ is defined similarly as above, and 
\[
 \cV_0^i = \bigoplus_{i=i_1+\dots+i_d} \cV^{i_1} \otimes \dots \otimes \cV^{i_d}.
\]
We write $r^i_\ell(\psi_0)$ for the $\ell$-adic representation of $\Gamma_{F_\Sigma}$ which should correspond to the representation of $\cL_{F_\Sigma}$ on $\cV_0^i$.
Finally, let $\pi^S$ be the irreducible unramified representation of $G(\A_F^S)$ associated to $\psi$, where $S$ is a sufficiently large finite set of rational primes.
Then it follows from Kottwitz's conjecture \cite[\S 10]{kot} that the $\ell$-adic representation of $\Gamma_{F_\Sigma}$ on the semisimplification of 
\[
 H^i(\Sh_{\tilde{\mathrsfs{G}},\cK}, \V_{\uk,\ell})[\pi^S]
\]
is isomorphic to a subrepresentation of $r^{i-4d}_\ell(\psi_0)^{\oplus m}$ for some integer $m$.

To deduce Proposition \ref{prop:galrep-char-ss} from Kottwitz's conjecture, we now suppose that $\psi = (\psi',\chi)$ with $\psi' = \pi_E \boxtimes \Sym^1$ as in \eqref{eq:U(2,2)-A-param} and $\chi = 1$.
It suffices to show that $r_\ell^{-2d}(\psi_0)(d)$ is trivial.
Let $\rho$ be the $2$-dimensional representation of $\cL_F$ (conjecturally) associated to $\pi$ and put $\rho_E = \rho|_{\cL_E}$.
(We can justify the formal computation by using the $\ell$-adic representation of $\Gamma_F$ associated to $\pi$, but we omit the details.)
Since $\pi_E$ has trivial central character, $\rho_E$ is self-dual.
Let $\Wc$ be the $4$-dimensional representation of $\cL_E$ induced by $\phi_\psi|_{\cL_E}$, so that $\Wc = \Wc^1 \oplus \Wc^{-1}$ with $\Wc^{\pm 1} =  \rho_E \otimes |\cdot|^{\mp \frac{1}{2}}$.
Then, noting that $\Wc$ is self-dual, we have
\[
 \Vc = (\wedge^2 \Wc)^* \otimes |\cdot|^{-2} = \wedge^2 \Wc \otimes |\cdot|^{-2} = \Vc^2 \oplus \Vc^0 \oplus \Vc^{-2}
\]
as representations of $\cL_F$, where
\begin{align*}
 \Vc^2 & = \wedge^2 \Wc^1 \otimes |\cdot|^{-2} = |\cdot|^{-3}, \\
 \Vc^0 & = \Wc^1 \otimes \Wc^{-1} \otimes |\cdot|^{-2} = \mathrm{As}^+(\rho_E) \otimes |\cdot|^{-2} = |\cdot|^{-2} \oplus \left( \Sym^2(\rho) \otimes |\cdot|^{-2} \right), \\
 \Vc^{-2} & = \wedge^2 \Wc^{-1} \otimes |\cdot|^{-2} = |\cdot|^{-1}.
\end{align*}
(Note that in the context of \S \ref{sss:gal-rep-intro} when $F=\Q$ and $d=1$, $\Vc^i$ corresponds to the $\ell$-adic representation on $H^{i+4}$.)
Hence 
\[
 \Vc_0^{-2d} \otimes |\cdot|^d = (\Vc^{-2})^{\otimes d} \otimes |\cdot|^d
\]
is the trivial representation of $\cL_{F_\Sigma}$ as desired.

\section{Hodge-Tate classes and the proof of the main theorem}
\label{sec:proof-of-main-thm}

\subsection{Hodge-Tate classes}

We make the following definition. Recall that the category $\cM_k^L$ that is used below was defined in \S \ref{sec:realizations}. 

\begin{defn}
Let $(V,V_\ell, i_\ell)$ denote a (pure) object in $\cM_k^L$. A class $c\in V$ is said to be a Hodge-Tate class (HT in brief) if 
$c$ is a Hodge class in $V$ and $i_\ell (c)$ is a Tate class in $V_\ell$ for all $\ell$. Thus $c$ is required to lie in $V^{0,0}$ and $i_\ell(c)$ is 
required to be $G_k$-invariant for all $\ell$. 
\end{defn}

We let $\HT (V)$ denote the $L$-subspace of HT-classes in $V$ and $\HT(V)_\C$ its $\C$-span. (This notation 
is slightly ambiguous since it does not keep track of the isomorphisms $i_\ell$; this will typically not cause a problem since the 
maps $i_\ell$ will be understood from the context.)
Clearly, any morphism from $(V,V_\ell, i_\ell)$ to $(V', V'_\ell, i'_\ell)$ induces maps $\HT(V) \rightarrow \HT(V')$ and $\HT(V)_\C \rightarrow \HT(V')_\C$. 

If $L\subset L'\subset \C$, the natural functor $\cM_k^L \rightarrow \cM_k^{L'}$ carries 
$\HT(V)$ into $\HT(V_{L'})$, where we write $V_{L'}$ for $V\otimes_{L} L'$. 

\subsection{The construction of a cohomology class}

While some
aspects of the construction have been described previously at various
points in the paper, we now collect in a single place the entire
construction, which also makes clear the dependence on various
auxiliary choices.

\subsubsection {Spaces and groups} 
\label{constr:first-step} 
Choose a CM quadratic extension $E/F$ that embeds in both $B_1$ and $B_2$. (Later we will be more careful about the choice of $E$.)
Fix embeddings $E\hookrightarrow B_1$ and $E\hookrightarrow B_2$. 
Let $\VV_1=B_1$ and $\VV_2=B_2$, viewed as hermitian $E$-spaces with the
canonical hermitian form, as in
\cite[\S 2.2]{periods1}, and let $\VV=\VV_1 \oplus \VV_2$ be their
direct sum, viewed as a four-dimensional hermitian $E$-space. To the
space $\VV$, we can associate the skew-hermitian $B$-space $\tilde{V}$
as in
\S \ref{sec:constr-gl-ex-isom}. Further, as in
\S \ref{ss:sum-of-2-dim}, the decomposition $\VV=\VV_1 \oplus \VV_2$
of $E$-hermitian spaces induces a decomposition 
 $\tilde{V} = V \oplus V_0$ as the sum of
skew-hermitian $B$-spaces of dimensions two and one respectively. We
then get a collection of groups and maps between them as described in  \S \ref{sec:unitary-quat-unitary}, and the reader is referred especially to the diagrams of groups in
that section, which will be used often in the construction below.

\subsubsection {Shimura data}
\label{ss:shimura-data-for-all-groups}
Fix isomorphisms
\[
B_i \otimes_{F,\sigma_v} \R \simeq \M_2(\R) \ \  \text{for }v\in \Sigma; \quad
B_i \otimes_{F,\sigma_v} \R \simeq \H \ \  \text{for }v\in
\Sigma_\infty\smallsetminus \Sigma. 
\]
For concreteness, we can fix isomorphisms as follows:
\[
\i \mapsto \begin{pmatrix} 0 & 1 \\ \sigma_v(u) & 0 \end{pmatrix},
                                                  \quad \j_i
                                                  \mapsto \begin{pmatrix}
                                                    \sqrt{\sigma_v(J_i)}
                                                    & 0 \\ 0 &
                                                    -\sqrt{\sigma_v(J_i)}\end{pmatrix},
                                                               \quad
                                                               v\in
                                                               \Sigma_\infty,
\]
where for $v\in \Sigma_\infty\smallsetminus \Sigma$, the notation
$\sqrt{\sigma_v(J_i)}$ stands for $\sqrt{|\sigma_v(J_i)|}i$. We will
use these isomorphisms to identify 
\[
G_{B_i} (\R) \simeq \prod_{v\in\Sigma} \GL_2(\R) \times
\prod_{v\in \Sigma_\infty\smallsetminus \Sigma} \H^\times.
\]
Define Shimura data associated to $B_1$ and $B_2$ as in \S
\ref{sec:qsv-mt}. Namely, take the $G_{B_i}(\R)$-conjugacy class of
the homomorphisms
\[
h_i: \Ss \rightarrow G_{B_i} (\R), \quad (h_i(z))_v= \iota (z)
\text{ for }v\in \Sigma; \quad (h_i(z))_v =1 \text{ for }v\in \Sigma_\infty
\smallsetminus \Sigma,
\]
where $\iota: \C^\times \rightarrow \GL_2(\R)$ is defined as in \eqref{eqn:iota}. 
Note that for $v\in\Sigma$, $h_{i,v}$ is $(B_i\otimes_{F,\sigma_v}\R)^\times$-conjugate to the
embedding $\iota'_v:\C^\times \simeq (E\otimes_{F,\sigma_v} \R)^\times \subset
(B\otimes_{F,\sigma_v} \R)^\times $, where the first
of these isomorphisms is induced from the map $E\otimes_{F,\sigma_v}
\R \simeq \C$ sending $\i \mapsto \sqrt{|\sigma_v(u)|}i$. We denote
this latter isomorphism by $\sigma_v$ as well. 

Let $X_{B_1}$ and $X_{B_2}$ denote the associated Shimura
varieties. The Shimura data for the other groups are defined as
follows. For $(B_i^\times \times
E^\times)/F^\times \simeq \GU_E(\VV_i)$ with $(\beta,\alpha)\mapsto (x \mapsto \beta x \alpha^{-1})$, 
\[
(h_i(z))_v= (\iota (z),1)
\text{ for }v\in \Sigma; \quad (h_i(z))_v =1 \text{ for }v\in \Sigma_\infty
\smallsetminus \Sigma.
\]
In the basis $(1_{B_i},\j_i)$ of $\VV_i$, the map $(B_i^\times \times
E^\times)/F^\times \rightarrow \GU_E(\VV_i)$ is given by
\[
(\alpha + \beta \j_i , \delta) \mapsto \delta^{-1} \begin{pmatrix} \alpha & J_i \beta
  \\ \bar{\beta} & \bar{\alpha} \end{pmatrix}  \in \GL_2(E), \quad
\alpha, \beta, \delta \in E. 
\]
In the basis $(\w_{i1}, \w_{i2})= (1_{B_i},\frac{1}{\sqrt{|\sigma_v(J_i)|}}\j_i)$ of
$\VV_{i,v} :=\VV_i \otimes_{E,\sigma_v}
  \C$, the hermitian form is diagonal, given by the matrix
  $\begin{pmatrix} 1 & \\ & \pm 1 \end{pmatrix}$ with the sign
being $- 1$ (resp. $+1$) if $v \in \Sigma$ (resp. $v \in \Sigma_\infty \smallsetminus \Sigma$).

Let $v\in \Sigma$. The map $(B_i^\times \times
E^\times)/F^\times \rightarrow \GU_E(\VV_i)(\R)_v$ is given by 
\[
(\alpha + \beta \j_i , \delta) \mapsto \sigma_v(\delta)^{-1} \begin{pmatrix} \sigma_v(\alpha) &\sqrt{\sigma_v( J_i)} \sigma_v(\beta)
  \\ \sqrt{\sigma_v( J_i)} \sigma_v(\bar{\beta}) & \sigma_v(\bar{\alpha}) \end{pmatrix}  \in \GU(1,1), \quad
\alpha, \beta, \delta \in E. 
\]
In particular, the map $h_{i,v}: \C^\times \rightarrow \GU(1,1)$ is
$\GU(1,1)$-conjugate to the map
 $z\mapsto \begin{pmatrix} z & 0 \\ 0 & \bar{z} \end{pmatrix}$. 

For $\tilde{\mathcal{G}}=\G (\U_E(\VV_1) \times \U_E(\VV_2))$, let $h$ be defined
by $h(z) = (h_1(z), h_2(z))$. For $\tilde{\mathrsfs{G}}=\GU_E(\VV)$, let $h$ be
defined by composing the map $h$ for $\tilde{\mathcal{G}}$ with the inclusion
$i:\tilde{\mathcal{G} }\hookrightarrow \tilde{\mathrsfs{G}}$. In the
basis 
\[
(\w_{11}, \w_{21}, \w_{12}, \w_{22})= (1_{B_1}, 1_{B_2}, \frac{1}{\sqrt{\sigma_v(J_1)}}\j_1,
\frac{1}{\sqrt{\sigma_v(J_2)}}\j_2)
\]
of $\VV_v= \VV \otimes_{E,\sigma_v} \C$, the hermitian form is $\mathrm{diag}
(1,1,\pm 1,\pm 1)$ with the sign
being $- 1$ (resp. $+1$) if $v \in \Sigma$ (resp. $v \in 
\Sigma_\infty \smallsetminus \Sigma$). For $v\in \Sigma$, $h_v$ is $\tilde{\mathrsfs{G}}(\R)_v$-conjugate to the map 
\[
z \mapsto \begin{pmatrix} z \1_2 & \\ & \bar{z} \1_2 \end{pmatrix},
\] 
while for $v\in \Sigma_\infty\smallsetminus \Sigma$, $h_v$ is
trivial. 

For $\tilde{\mathcal{G}}_B=\G (\U_B(V) \times \U_B (V_0))=\G
((B_1^\times \times B_2^\times)/F^\times \times E^\times)$, we take 
\[
h(z)_v =((\iota(z),\iota(z)),z\bar{z})
\]
for $v \in \Sigma$ and $h(z)_v =1$ otherwise. 
For $\tilde{\mathrsfs{G}}_B$, we take $h$ to be the map obtained by
  composing $h$ for $\tilde{\mathcal{G}}_B$ with the inclusion
  $\tilde{\mathcal{G}}_B \hookrightarrow \tilde{\mathrsfs{G}}_B$. 
Thus for $v\in \Sigma$, the action of $h(z)_v$ on 
$\tilde{V}_v= V_v \oplus V_{0,v} = (\VV_{1,v} \otimes_{\C} \VV_{2,v}) \oplus (\wedge^2_{\C}
\VV_{1,v} \oplus \wedge^2_{\C} \VV_{2,v})$ is given by $\iota(z) \otimes \iota(z)
\oplus z\bar{z}$.  
To compute the $\tilde{\mathrsfs{G}}_B(\R)_v$-conjugacy class, we
may replace $\iota(z)$ by $\iota'_v (z)$. 
The action of $\iota'_v(z)$ on $\VV_{i,v}$ is
diagonalizable, given by $\begin{pmatrix} z & \\ & \bar{z}\end{pmatrix}$ in the
  basis $(\w_{i1},\w_{i2})$, i.e.,
\[
\iota'_v(z) \w_{i1} =\w_{i1} z, \quad \iota'_v(z)\w_{i2} =\w_{i2} \bar{z}.
\]
Now $(\w_{11} \otimes \w_{21}, \w_{12} \otimes \w_{21}, \w_{11}
\wedge \w_{12})$ is a $B\otimes_{F,\sigma_v} \R$-basis of
$\tilde{V}_v$ and in this basis, the action of $h(z)_v$ is diagonal,
given by 
$\mathrm{diag}(z^2, z\bar{z}, z \bar{z})$. 
From this description, it is clear that under the canonical
isomorphism
\[
\tilde{\mathrsfs{G}}/E^\times =\mathrsfs{G} \simeq \mathrsfs{G}_B = \tilde{\mathrsfs{G}}_B/F^\times,
\]
the chosen Shimura data are identified. 

\subsubsection {Local systems}
We consider a local system $\tilde{\V}_\rho$ on $\Sh_{\tilde{\mathrsfs{G}}_B}$  associated with a finite dimensional representation $\rho$ of $\tilde{\mathrsfs{G}}_B$. 
To construct this local system, we first fix isomorphisms $B\otimes_F F_v \simeq \M_2 (\R)$ for all infinite places $v$ of $F$. 
Then to each infinite place $v$, as in \S \ref{sec:local-exc} (and
\cite[\S C.2]{periods1}), we can associate orthogonal spaces $\tilde{V}^\dagger_v = V^\dagger_v \oplus V^\dagger_{0,v}$ such that
\[
\GU_{B_v}(\tilde{V}_v)^0 \simeq \GSO (\tilde{V}^\dagger_v), \quad \GU_{B_v}(V_v)^0 \simeq \GSO (V^\dagger_v), \quad \GU_{B_v}(V_{0,v})^0 \simeq \GSO (V^\dagger_{0,v}).
\] 
Recall that
\[
\tilde{\mathrsfs{G}}_B (\R) \simeq \prod_{v\in \Sigma} \GSO(4,2) \times \prod_{v\in\Sigma_\infty \smallsetminus \Sigma} \GSO(0,6),
\]
where for $p+q$ even, 
\[
 \GSO(p,q) = \{ g \in \GL_{p+q}(\R) \, | \, {}^t g I_{p,q} g = \nu(g) \cdot I_{p,q}, \, \det g = \nu(g)^{\frac{p+q}{2}} \}
\]
with 
\[
 I_{p,q} = 
 \begin{pmatrix}
  \1_p & \\
  & -\1_q
 \end{pmatrix}.
\]
The local system is then associated to the algebraic representation $\tilde{\V}_{\rho,\C}:=\otimes_v \Hs^{k_v-2} (\tilde{V}_v^\dagger)$ of $\tilde{\mathrsfs{G}}_B (\R)$,
where for $\ell$ even, we have 
\[
\Hs^{\ell}:= \ker(\Sym^\ell \rightarrow \Sym^{\ell-2}) \otimes \nu(\cdot)^{-\ell/2}.
\]
Note that by \S \ref{ss:fin-dim-rep-SO} (for $\ell=k-2$) the restriction of this representation to $\SO(4,2)$ is irreducible with highest weight $(k-2,0,0)$. Via the 
isomorphism given by Corollary \ref{c:isom-weights}, this corresponds (for $\ell=k-2$) to the irreducible representation of $\U(2,2)$ with highest weight 
\[
\left( \tfrac{k}{2}-1, \tfrac{k}{2}-1, -\tfrac{k}{2}+1,-\tfrac{k}{2}+1\right).
\]
A similar statement holds for the places $v\in \Sigma_\infty \smallsetminus \Sigma$, and for the pair $(\SO(0,6), \U(4,0))$. 
Thus the local system $\tilde{\V}_{\rho,\C}$ is isomorphic to the local system $\V_{\uk,\C}$ considered in \S \ref{sec:the-hodge-structure}. 

\begin{prop}
The $\C$-vector space $\tilde{\V}_{\rho,\C}$ contains a $\Q(\uk)$-subspace $\tilde{\V}_\rho$ that is stable by the action of $\tilde{\mathrsfs{G}}_B(\Q)$ and such that $\tilde{\V}_\rho \otimes_{\Q(\uk)} \C =\tilde{\V}_{\rho,\C}$.  Moreover, such a subspace is unique up to homothety. 
\end{prop}
\begin{proof}
Fix an infinite place $v$ of $F$. Then the representation $\Hs^{k_v-2} (\tilde{V}_v^\dagger)$ is defined over $\sigma_v(F)$ by 
\cite{tits-crelle}, Th\'{e}or\`{e}me 3.3, since the highest weight is both invariant by $\Gal(\Qbar/\sigma_v(F))$ and lies in the root lattice. 
Taking the tensor induction over all places $v\in \Sigma_\infty$ yields a $\Q(\uk)$-structure on $\tilde{\V}_{\rho,\C}$.
The uniqueness up to homothety follows from the irreducibility of $\tilde{\V}_{\rho,\C}$.
\end{proof}

\subsubsection {Auxiliary modular form} Let $\tilde{\tau}$ be an irreducible  automorphic representation of $B^\times (\A) $ corresponding to a 
holomorphic Hilbert modular form of weights $(\uk + \1,r)$ (with some
odd integer $r$) and central character $\xi_E$. 
Let $B^\times(\A)^+$ denote the subgroup of $B^\times (\A)$ consisting of elements with positive reduced norm at every infinite place. 
The restriction of $\tilde{\tau}$
to $B^\times (\A)^+$ splits up as a sum of $2^{[F:\Q]}$ representations, characterized by the local component at 
the $[F:\Q]$ infinite places being either holomorphic or anti-holomorphic discrete series. We let $\tau$ be the 
irreducible summand whose local component is {\it anti-holomorphic} for
$v\in \Sigma$ and {\it holomorphic} for $v\not \in \Sigma$, twisted by
a (half-integer power of) the norm character to make it {\it
  unitary}.

\subsubsection {Theta lift to $\tilde{\mathrsfs{G}}_B$} Let $\theta_{\tilde{\varphi}} (\phi)$ denote the element in $\mathrsfs{A}(\tilde{\mathrsfs{G}}_B) \otimes \wedge^{2d}  \tilde{\fp}^*  \otimes \tilde{\V}_{\rho,\C}$ 
constructed in \S \ref{ss:form-construction}, with an element $\phi$
in the space $\tau$ and a Schwartz form $\tilde{\varphi}$. (Note that
in that section, the group $\tilde{\mathrsfs{G}}_B$ is simply denoted by
$\tilde{G}$. Then $\theta_{\tilde{\varphi}} (\phi)$ 
 corresponds to a class 
\[
\xi_\tau \in H^{2d} (\Sh_{\tilde{\mathrsfs{G}}_B}, \tilde{\V}_{\rho,\C})
\]
via the isomorphism 
\[
H^{2d} (\Sh_{\tilde{\mathrsfs{G}}_B}, \tilde{\V}_{\rho,\C}) \simeq H^{2d} (\fg, K ; \mathrsfs{A}(\tilde{\mathrsfs{G}}_B) \otimes  \tilde{\V}_{\rho,\C}).
\]

\subsubsection {Pull-back to $\tilde{\cG}_B$}
 Pull back
  $\xi_\tau$ to a class $ i^*  \xi_\tau$ in $H^{2d}(\Sh_{\tilde{\cG}_B}, \tilde{\V}_{\rho,\C})$. 
 Decompose $\tilde{\V}_{\rho,\C}$ into a sum of irreducibles (as a representation of $\tilde{\cG}_B(\R)$) and project to the 
irreducible component $\V_{\rho,\C}:=\otimes_v \Hs^{k_v-2} (V_v)$, as in \eqref{eq:proj-Hs-ell}. Thus $ i^*  \xi_\tau$ can be viewed as an element of $H^{2d}(\Sh_{\tilde{\cG}_B}, \V_{\rho,\C})$. Note that the $\Q(\uk)$-rational structure on $\tilde{\V}_{\rho,\C}$ can be chosen such that 
the projection map carries it into the $\Q(\uk)$-rational structure on $\V_{\rho,\C}$.

\subsubsection {Auxiliary character} For a finite order Hecke character $\eta$ of $T_1(\A)$ (see \S \ref{sec:components}), we take the class $c_\eta \in H^0 (\Sh_{\tilde{\cG}_B}, \Q(\eta))$ and cup it with $ i^*  \xi_\tau$, to get
\[
\tilde{\xi}_{\tau,\eta} :=i^*\xi_\tau \cup c_\eta \in H^{2d}(\Sh_{\tilde{\cG}_B}, \V_{\rho,\C}).
\]

\subsubsection {Push-forward to $\Sh_{G}$ and $\cK$-invariant projection} Push forward the class $\tilde{\xi}_{\tau,\eta}$ to $\Sh_{G}$.
Pick an open compact $\cK$ of $Z(\A_f)\backslash G(\A_f)$ and take the $\cK$-invariant projection $\tilde{\xi}_{\tau,\eta,\cK}$. 

\subsubsection {Pull-back to $\Sh_{B_1} \times \Sh_{B_2}$}  Take an open compact subgroup $\cK_1 \times \cK_2 \subset B_1(\A_f) \times B_2 (\A_f)$ whose image under the natural map to $Z(\A_f)\backslash G(\A_f)$ is contained in $\cK$. Then pull back to $\Sh_{B_1,\cK_1} \times \Sh_{B_2,\cK_2}$ to get the class
\[
\xi_{\tau,\eta} := j^* p_{1,*} \tilde{\xi}_{\tau,\eta,\cK} \in H^{2d} (\Sh_{B_1,\cK_1} \times \Sh_{B_2,\cK_2}, \V_{\rho,\C}).
\]

\subsubsection {Project to $[\pi_1 , \pi_2] $-component} On $\Sh_{B_1} \times \Sh_{B_2}$, we have
\[
\V_{\rho} \simeq \V_{\uk} \boxtimes \V_{\uk}.
\]
Thus 
\[
H^{2d} (X_{B_1,\cK_1} \times X_{B_2,\cK_2}, \V_{\rho,\C}) = \bigoplus_{\pit_1, \pit_2} \cH_{\cK_1,\cK_2}^{2d} [\pit_1,\pit_2]
\]
where
\[
\cH_{\cK_1,\cK_2}^{2d} [\pit_1,\pit_2]=  ( \pit_{1,f}^{\cK_1}\otimes  \pit_{2,f}^{\cK_2} ) \otimes H^{2d} (X_{B_1,\cK_1} \times X_{B_2,\cK_2} , \V_{\uk,\C} \boxtimes \V_{\uk,\C} )_{\pit_1 \boxtimes \pit_2}
\]
and $\pit_1, \pit_2$ range over automorphic representations of $B_1^\times (\A)$ and $B_2^\times(\A)$ such that $\pit_1 \boxtimes \pit_2$ contributes to the cohomology of the local system $V_{\rho,\C}$. Then
\[
\epsilon_\pi (\xi_{\tau,\eta}) \in \cH_{\cK_1,\cK_2}^{2d} [\pi_1,\pi_2]
\]
is defined to be the projection to the $[\pi_1 , \pi_2]$-component. Note that 
\[
H^{2d} (X_{B_1,\cK_1} \times X_{B_2,\cK_2} , \V_{\uk,\C} \boxtimes \V_{\uk,\C} )_{\pi_1 \boxtimes \pi_2}
 = H^d (X_{B_1,\cK_1}, \V_{\uk,\C})_{\pi_1} \otimes H^d (X_{B_2,\cK_2}, \V_{\uk,\C})_{\pi_2}.
\]

\subsubsection {Contraction with $v_1 \otimes v_2$} \label{constr:last-step} Though we are assuming that $\pi^\vee \simeq \pi$, below we distinguish between them for purposes of clarity.
Pick non-zero elements $v_1 \in (\pi_1^{\vee})^{f,\cK_1}$,  $v_2 \in (\pi_2^\vee)^{f,\cK_2}$ such that $v_1 \otimes v_2$ is $\cK$-invariant. Then contracting $v_1 \otimes v_2$ with $\epsilon_\pi(\xi_\eta)$ gives an 
element
\[
\xi := \langle \epsilon_\pi(\xi_{\tau,\eta}), v_1 \otimes v_2 \rangle \in H^d (\Sh_{B_1} , \V_{\uk,\C})_{\pi_1}  \otimes H^d (\Sh_{B_2} , \V_{\uk,\C})_{\pi_2}.
\]

\subsection{The construction of a Hodge-Tate class}

Note that $\xi$ induces a map (for the moment of $\C$-vector spaces!) 
\[
I_\xi: H^d (\Sh_{B_1} , \V_{\uk,\C})_{\pi_1^\vee} \simeq  H^d (\Sh_{B_1} , \V_{\uk,\C})_{\pi_1}^\vee \xrightarrow{{} \cdot \xi} H^d (\Sh_{B_2} , \V_{\uk,\C})_{\pi_2}.
\]
Note also that $\xi$ depends on the choices of the following data:
\[
\Upsilon: = (E,\tilde{\tau}, \phi, \tilde{\varphi}, \eta, \cK, \cK_1, \cK_2, v_1 ,v_2 ).
\]

\begin{prop}
There exists a choice of data $\Upsilon$ such that $I_\xi$ is an isomorphism of $\C$-vector spaces:
\[
 H^d (\Sh_{B_1} , \V_{\uk,\C})_{\pi_1^\vee} \simeq   H^d (\Sh_{B_2} , \V_{\uk,\C})_{\pi_2}.
\]
\end{prop}

\begin{proof}
By Matsushima's formula, there are canonical isomorphisms:
\begin{align*}
H^d (\fg_1, K_1; \pi_{B_1,\infty}^\vee \otimes \rho_{\uk} ) &\simeq H^d (\Sh_{B_1} , \V_{\uk})_{\pi_1^\vee}, \\
H^d (\fg_2, K_2; \pi_{B_2,\infty} \otimes \rho_{\uk} ) &\simeq H^d (\Sh_{B_2} , \V_{\uk})_{\pi_2},
\end{align*}
and so we just need to check that the data $\Upsilon$ can be picked so that the induced map
\[
I_\xi: H^d (\fg_1, K_1; \pi_{B_1,\infty}^\vee \otimes \rho_{\uk} ) ^\vee \rightarrow H^d (\fg_2, K_2; \pi_{B_2,\infty} \otimes \rho_{\uk} )
\]
is an isomorphism. But this is exactly the content of Proposition \ref{p:form-nonvanishing}. The only point to note is that 
one can in fact pick a CM extension $E/F$ satisfying the conditions \eqref{item:E1} through \eqref{item:E3} in the statement of the proposition. 
But \eqref{item:E2} and \eqref{item:E3} hold for all but a finite number of finite places, so it is obvious that there exists $E$ satisfying the needed conditions. 
\end{proof}

We now come to the proof of the main theorem (Theorem \ref{thm:intro-main-full} of the introduction and its generalization, Theorem \ref{mainthm:generalcase}.)
\begin{thm}
\begin{enumerate}
\item 
There exists a Hodge class $\xi_0 \in V_{B_1,\pi_1} \otimes V_{B_2,\pi_2}(d) $ 
such that the induced map
\[
V_{B_1,\pi_1} \simeq V_{B_1,\pi_1}^\vee (-d) \xrightarrow{{} \cdot \xi_0} V_{B_2,\pi_2}
\]
is an isomorphism of $L$-Hodge structures. 

\item Assume Kottwitz's conjecture for Shimura varieties attached to unitary similitude groups. 
Then the Hodge class $\xi_0$ can be chosen such that it belongs to $\HT (V_{B_1,\pi_1} \otimes V_{B_2,\pi_2} (d))$, 
so that the induced map 
\[
V_{B_1,\pi_1} \otimes \Q_\ell \simeq (V_{B_1,\pi_1} \otimes \Q_\ell) ^\vee (-d) \xrightarrow{{} \cdot \xi_0} V_{B_2,\pi_2} \otimes \Q_\ell
\]
is an isomorphism of $G_{F_\Sigma}$-modules for all rational primes $\ell$.  
\end{enumerate}

\end{thm}

\begin{proof}
The construction outlined in \S \ref{constr:first-step} to \S \ref{constr:last-step} above gives a map (for any open compact subgroup $\tilde{\cK} $ of 
$\tilde{\mathrsfs{G}}_B (\A_f)$)
\[
\Res: H^{2d} (\Sh_{\tilde{\mathrsfs{G}}_B,\tilde{\cK}}, \tilde{\V}_\rho) (d) \rightarrow V_{B_1,\pi_1}\otimes V_{B_2,\pi_2} (d) \simeq \Hom( V_{B_1,\pi_1} , V_{B_2,\pi_2})
\]
such that $\Res_\C$ sends
\[
\xi_\tau (d) \mapsto \xi(d) \mapsto  I_{\xi (d)}.
\]
Let $\I$ be the kernel of the unramified part of the Hecke algebra (at level $\tilde{\cK}$) acting on $\Pi^{\tilde{\cK}}$ so that 
\[
H^{2d} (\Sh_{\tilde{\mathrsfs{G}}_B,\tilde{\cK}}, \tilde{\V}_{\rho,\C})[\I] (d)= \bigoplus_{\sigma}  H^{2d} (\Sh_{\tilde{\mathrsfs{G}}_B,\tilde{\cK}}, \tilde{\V}_{\rho,\C})[\Pi^\sigma] (d),
\] 
where 
$\sigma$ ranges over a set of automorphisms of $\C/\Q$ such that the $\Pi^\sigma$ are the distinct conjugates of $\Pi$. 
Since $\tilde{\V}_\rho$ is defined over $\Q(\uk)$, we may consider the $\Q(\uk)$-Hodge structure
\[
H^{2d} (\Sh_{\tilde{\mathrsfs{G}}_B,\tilde{\cK}}, \tilde{\V}_{\rho})[\I] (d)  \subset H^{2d} (\Sh_{\tilde{\mathrsfs{G}}_B,\tilde{\cK}}, \tilde{\V}_{\rho,\C})[\I] (d).
\]
Now $\xi_\tau \in H^{2d} (\Sh_{\tilde{\mathrsfs{G}}_B,\tilde{\cK}}, \tilde{\V}_{\rho,\C})[\I]$ and $\Res_\C (\xi_\tau(d))$ is an isomorphism, hence 
there exists an element $\Xi_\tau \in H^{2d} (\Sh_{\tilde{\mathrsfs{G}}_B,\tilde{\cK}}, \tilde{\V}_{\rho})[\I]$ such that $\Res(\Xi_\tau (d))$ is an isomorphism. 
Indeed, if we pick a $\Q(\uk)$-basis $(x_1, \ldots, x_r)$ for $H^{2d} (\Sh_{\tilde{\mathrsfs{G}}_B,\tilde{\cK}}, \tilde{\V}_{\rho})[\I]$, then this is also a $\C$-basis for 
$H^{2d} (\Sh_{\tilde{\mathrsfs{G}}_B,\tilde{\cK}}, \tilde{\V}_{\rho,\C})[\I]$. Expanding $\xi_\tau$ in this basis:
\[
\xi_\tau = a_1x_1 + \cdots + a_r x_r
\]
we see that the $(a_1, \ldots, a_r)\in \C^r$ satisfies $\det (I_{\sum_i a_i x_i} (d) ) \neq 0 $. Since this is a polynomial function in the $a_i$, it follows that there exist $b_i\in \Q(\uk)$ with $\det (I_{\sum_i b_i x_i} (d) ) \neq 0 $. Taking $\Xi_\tau = \sum_i b_i x_i$, we see that $\Res(\Xi_\tau (d))$ is an isomorphism.
By a similar argument using a determinant, we can replace the class $c_\eta$ in the original construction by some $\Q$-linear combination $c$ of 
the fundamental classes of the components of $\Sh_{\tilde{\mathcal{G}}_B}$. 
(We note that since the action of $\Gal (\Qbar/F_\Sigma)$ on the components of $\Sh_{\tilde{\mathcal{G}}_B}$ is trivial, the class $c$ is 
$\Gal (\Qbar/F_\Sigma)$-invariant.)
The class $\Res(\Xi_\tau (d))$ then has coefficients in $L$. (The only step where the coefficients might be enlarged is the projection to the $\pi_1 \boxtimes \pi_2$-component, and the coefficient field $L$ of $\pi$ contains $\Q(\uk)$.)
By Corollary \ref{cor:hodge-is-purely-of-type-(d,d)}, the class $\Xi_\tau (d)$ is a $\Q(\uk)$-rational Hodge class, hence $\xi_0 := \Res(\Xi_\tau (d))$ is an $L$-rational Hodge class.
Assuming Kottwitz's conjecture,  by Proposition \ref{prop:galrep-char}, the action of $\Gal (\Qbar/F_\Sigma)$ on $\Xi_\tau (d)$ is trivial and so the same is true for $\xi_0$. 
\end{proof}

\appendix

\section{Splittings}
\label{sec:weil-hodge}

\subsection{Setup}
\label{ss:weil-hodge-setup}

Let $F$ be a number field and $B$ a quaternion division algebra over $F$.
Let $E$ be a quadratic extension of $F$ which embeds into $B$.
Let $*$ be the main involution on $B$ and $\rho$ the non-trivial Galois automorphism of $E$ over $F$.
We write $E = F + F \i$ and $B = E + E \j$ for some trace zero elements $\i \in E^\times$ and $\j \in B^\times$.
Let $\pr : B \rightarrow E$ be the associated projection. 
Put $u = \i^2 \in F^\times$ and $J = \j^2 \in F^\times$.
Fix a non-trivial additive character $\psi$ of $\A/F$ and a character $\chi$ of $\A_E^\times / E^\times$ such that $\chi|_{\A^\times} = \xi_E$, where $\xi_E$ is the quadratic character of $\A^\times / F^\times$ associated to $E/F$ by class field theory.

We consider an $m$-dimensional right $B$-space $V$ equipped with a skew-hermitian form $\langle \cdot, \cdot \rangle : V \times V \rightarrow B$ given by
\begin{equation}
\label{eq:condition-skew-herm}
 \langle e_1 x_1 + \dots + e_m x_m, e_1 y_1 + \dots + e_m y_m \rangle
 = x_1^* \cdot \kappa_1 \i \cdot y_1 + \dots + x_m^* \cdot \kappa_m \i \cdot y_m
\end{equation}
for some basis $e_1,\dots,e_m$ of $V$ and some $\kappa_1, \dots, \kappa_m \in F^\times$.
We denote by $\GU(V)$ the unitary similitude group of $V$ and by $\nu : \GU(V) \rightarrow F^\times$ the similitude character:
\[
 \GU(V) = \{ g \in \GL(V) \, | \, \langle gv, gv' \rangle = \nu(g) \cdot \langle v, v' \rangle \text{ for all } v, v' \in V \},
\]
where $\GL(V)$ acts on $V$ on the left.
We have a natural embedding
\[
 E^\times \hookrightarrow \GU(V),
\]
where we may regard $\alpha \in E^\times$ as an element in $\GU(V)$ given by $e_i \mapsto e_i \alpha$ for all $i$.

Let $W = B$ be a left $B$-space equipped with a hermitian form $\langle \cdot, \cdot \rangle : W \times W \rightarrow B$ given by
\[
 \langle x, y \rangle = x \cdot y^*.
\]
We denote by $\GU(W)$ the unitary similitude group of $W$ and by $\nu : \GU(W) \rightarrow F^\times$ the similitude character:
\[
 \GU(W) = \{ h \in \GL(W) \, | \, \langle wh, w'h \rangle = \nu(h) \cdot \langle w, w' \rangle \text{ for all } w, w' \in W \},
\]
where $\GL(W)$ acts on $W$ on the right.
Then we have $\GU(W) \simeq B^\times$.

Let $\V = V \otimes_B W$ be a $4m$-dimensional $F$-space equipped with a symplectic form
\[
 \llangle \cdot, \cdot \rrangle := \frac{1}{2} \tr_{B/F} \left( \langle \cdot, \cdot \rangle \otimes \langle \cdot, \cdot \rangle^* \right).
\]
Then we have a natural homomorphism
\begin{equation}
\label{eq:hom-GU(V)-GU(W)}
 \G(\U(V) \times \U(W)) \longrightarrow \Sp(\V), 
\end{equation}
where 
\[
 \G(\U(V) \times \U(W)) = \{ (g, h) \in \GU(V) \times \GU(W) \, | \, \nu(g) = \nu(h) \}
\]
and $\GL(V) \times \GL(W)$ acts on $\V$ on the right:
\[
 (v \otimes w) \cdot (g, h) := g^{-1} v \otimes w h.
\]
Let $\cG$ be a subgroup of $\G(\U(V) \times \U(W))$ defined by 
\[
 \cG = \{ (g, h) \in \GU(V)^0 \times \GU(W) \, | \, \nu(g) = \nu(h) \in \N_{E/F}(E^\times) \}.
\]
We take a complete polarization $\V = \X \oplus \Y$ defined by
\begin{align*}
 \X & = F \cdot e_1 \otimes 1 + \dots + F \cdot e_m \otimes 1 + F \cdot e_1 \otimes \j + \dots + F \cdot e_m \otimes \j, \\
 \Y & = F \cdot e_1 \otimes \i + \dots + F \cdot e_m \otimes \i + F \cdot e_1 \otimes \i\j + \dots + F \cdot e_m \otimes \i\j.
\end{align*}

\subsection{Splitting over $\cG$}

For each place $v$ of $F$, let $\Mp(\V_v)$ be the metaplectic group over $F_v$:
\[
 1 \longrightarrow \C^1 \longrightarrow \Mp(\V_v) \longrightarrow \Sp(\V_v) \longrightarrow 1.
\]
Then $\Mp(\V_v)$ can be realized by a $2$-cocycle $z_{\Y_v}$ relative to $\Y_v$ and $\psi_v$ (see e.g.~\cite{rangarao}, \cite[\S 3.2.2]{periods1}).
For almost all $v$, there exists a map $s_{\Y_v}:K_v \rightarrow \C^1$, where $K_v$ is the standard maximal compact subgroup of $\Sp(\V_v)$, such that
\[
 z_{\Y_v}(k_1,k_2) = \frac{s_{\Y_v}(k_1 k_2)}{s_{\Y_v}(k_1) s_{\Y_v}(k_2)}
\]
for $k_1, k_2 \in K_v$ (see e.g.~\cite[\S 3.2.3]{periods1}).

\begin{prop}
\label{prop:spl-hodge}
For all $v$, there exists a map $s_v:\cG_v \rightarrow \C^1$ satisfying the following conditions:
\begin{enumerate}
\item For $\g_1, \g_2 \in \cG_v$, we have
\[
 z_{\Y_v}(\g_1,\g_2) = \frac{s_v(\g_1 \g_2)}{s_v(\g_1) s(\g_2)}.
\]
Here, by abuse of notation, we write $\g_i$ on the left-hand side for the image of $\g_i$ in $\Sp(\V_v)$ under \eqref{eq:hom-GU(V)-GU(W)}.
\item For $\z = (z, z)$ with $z \in F_v^\times$ and $\g \in \cG_v$, we have
\[
 s_v(\z \g) = \xi_{E_v}(z)^m \cdot s_v(\g).
\]
\item For almost all $v$, we have
\[
 s_v|_{\cG_v \cap K_v} = s_{\Y_v}|_{\cG_v \cap K_v}.
\]
\item For $\gamma \in \cG(F)$, we have
\[
 \prod_v s_v(\gamma) = 1.
\]
\end{enumerate}
\end{prop}

As in \cite[\S 3.3]{periods1}, Proposition \ref{prop:spl-hodge} enable us to define a Weil representation $\omega_\psi$ of $\cG(\A)$ on the Schwartz space $\SS(\X(\A))$.
Moreover, for any $\varphi \in \SS(\X(\A))$, the associated theta function
\[
 \Theta_{\varphi}(\g) := \sum_{x \in \X} \omega_{\psi}(\g) \varphi(x)
\]
on $\cG(\A)$ is left $\cG(F)$-invariant.

\begin{rem}
Suppose that $V$ is the $3$-dimensional skew-hermitian right $B$-space as in \S \ref{sec:constr-gl-ex-isom}.
Then $V$ satisfies the condition \eqref{eq:condition-skew-herm} and we may apply the above construction.
Note that $\nu(\GU(V_v)^0) = \N_{E_v/F_v}(E_v^\times)$ for all $v$, so that
\[
 \cG(\A) = \G(\U(V) \times \U(W))^0(\A).
\]
\end{rem}

The proof of Proposition \ref{prop:spl-hodge} will be given in \S \ref{ss:doubling-U(V)}--\S \ref{ss:proof-spl-hodge} below. 
From now on, we fix a place $v$ of $F$ and suppress the subscript $v$ from the notation.

\subsection{The doubling method for $\U(V)$}
\label{ss:doubling-U(V)}

We consider the doubled space $V^\square = V \oplus V$ equipped with a skew-hermitian form
\[
 \langle (v_1,v_2), (v_1',v_2') \rangle := \langle v_1, v_1' \rangle - \langle v_2, v_2' \rangle.
\]
Then we have a natural embedding
\[
 \iota : \G(\U(V) \times \U(V)) \longrightarrow \GU(V^\square).
\]
If $\V^\square = \V \oplus \V$ is the doubled space equipped with a symplectic form defined similarly as above, then we have a natural embedding
\[
 \iota : \Sp(\V) \times \Sp(\V) \longrightarrow \Sp(\V^\square)
\]
and an identification
\[
 \V^\square = V^\square \otimes_B W.
\]
We take a complete polarization $\V^\square = \V^\bigtriangledown \oplus \V^\triangle$ defined by
\[
 \V^{\bigtriangledown} = \{ (x,-x) \, | \, x \in \V \}, \qquad
 \V^{\triangle} = \{ (x,x) \, | \, x \in \V \}.
\]
Under the above identification, we have
\[
 \V^{\bigtriangledown} = V^{\bigtriangledown} \otimes_B W, \qquad
 \V^{\triangle} = V^{\triangle} \otimes_B W,
\]
where $V^\square = V^\bigtriangledown \oplus V^\triangle$ is the complete polarization over $B$ defined similarly as above.

Now we recall Kudla's splitting over $\U(V^\square)$, where we regard $\U(V^\square)$ as a subgroup of $\Sp(\V^\square)$ via the natural embedding.
As in \cite[\S C.3]{periods1}, we regard $V^\square$ as a left $B$-space and let $\GL(V^\square)$ act on $V^\square$ on the right:
\begin{align*}
 x \cdot v & := v \cdot x^*, & x & \in B, \\
 v \cdot g & := g^{-1} \cdot v, & g & \in \GL(V^\square).
\end{align*}
Similarly, we regard $W$ as a right $B$-space and let $\GL(W)$ act on $W$ on the left.
Then we have an identification
\[
 \V^\square = W \otimes_B V^\square.
\]
Put
\[
 \v_i = \frac{1}{2 \kappa_i \i} \cdot (e_i, -e_i), \qquad
 \v_i^* = (e_i, e_i),
\]
so that $\langle \v_i, \v_j^* \rangle = \delta_{ij}$.
Using a basis $\v_1, \dots, \v_m, \v_1^*, \dots, \v_m^*$ of $V^\square$, we identify $\U(V^\square)$ with
\[
 \left\{ g \in \GL_{2m}(B) \, \left| \,
 g \begin{pmatrix} & \1_m \\ -\1_m & \end{pmatrix} {}^t g^*
 = \begin{pmatrix} & \1_m \\ -\1_m & \end{pmatrix}
 \right. \right\}.
\]
Let $P_{V^\triangle}$ be the maximal parabolic subgroup of $\U(V^\square)$ stabilizing $V^\triangle$:
\[
 P_{V^\triangle} = \left\{ \left.
 \begin{pmatrix} a & * \\ & ({}^t a^*)^{-1} \end{pmatrix}
 \, \right|  \, a \in \GL_m(B) \right\}.
\]
We define a map
\[
 \hat{s}_1 : \U(V^\square) \longrightarrow \C^1
\]
as follows:
\begin{itemize}
\item If $B$ is split, then we set
\[
 \hat{s}_1(g) = 1
\]
for $g \in \U(V^\square)$.
\item If $B$ is ramified, then we set
\[
 \hat{s}_1(g) = (-1)^j
\]
for $g \in P_{V^\triangle} \tau_j P_{V^\triangle}$ with 
\[
 \tau_j = 
 \begin{pmatrix}
  \1_{m-j} & & & \\
  & & & -\1_j \\
  & & \1_{m-j} & \\
  & \1_j & & 
 \end{pmatrix}.
\]
\end{itemize}
By \cite[Theorem 3.1, cases $1_-$ and $2_+$]{kudla-splitting}, we have
\begin{equation}
\label{eq:spl-kudla-1} 
 z_{\V^\triangle}(g_1,g_2) = \frac{\hat{s}_1(g_1 g_2)}{\hat{s}_1(g_1) \hat{s}_1(g_2)}
\end{equation}
for $g_1, g_2 \in \U(V^\square)$.

\begin{lem}
\label{lem:s-hat-1-E-conj}
For $\alpha \in E^{\times}$ and $g \in \U(V^\square)$, we have
\[
 \hat{s}_1(\alpha g \alpha^{-1}) = \hat{s}_1(g).
\]
\end{lem}

\begin{proof}
Since $\alpha p \alpha^{-1} \in P_{V^\triangle}$ for $p \in P_{V^\triangle}$ and $\alpha \tau_j \alpha^{-1} = \tau_j$, the assertion follows.
\end{proof}

\begin{lem}
\label{lem:s-hat-1-E^1}
Let $\alpha \in E^1$.
Then we have
\[
 \hat{s}_1(\iota(\alpha, 1)) = 1
\]
if $B$ is split, and
\[
 \hat{s}_1(\iota(\alpha, 1)) =
 \begin{cases}
  1 & \text{if $\alpha = 1$,} \\
  (-1)^m & \text{if $\alpha \ne 1$}
 \end{cases}
\]
if $B$ is ramified.
\end{lem}

\begin{proof}
We may assume that $B$ is ramified and $\alpha \ne 1$.
Then we have
\[
 \begin{bmatrix}
  \v_i \cdot \iota(\alpha,1) \\
  \v_i^* \cdot \iota(\alpha,1)
 \end{bmatrix}
 = A \cdot
 \begin{bmatrix}
  \v_i \\
  \v_i^*
 \end{bmatrix},
\]
where
\begin{align*}
  A & = 
 \begin{pmatrix}
  \frac{1}{2}(\alpha+1) & \frac{1}{4 \kappa_i \i}(\alpha-1) \\
  \kappa_i \i (\alpha-1) & \frac{1}{2}(\alpha+1)  
 \end{pmatrix} \\
 & = 
 \begin{pmatrix}
  -\frac{1}{\kappa_i \i (\alpha^\rho - 1)} & \frac{1}{2} (\alpha+1) \\
  & \kappa_i \i (\alpha-1) \\
 \end{pmatrix} \cdot
 \begin{pmatrix}
  & -1 \\
  1 & 
 \end{pmatrix} \cdot 
 \begin{pmatrix}
  1 & \frac{\alpha+1}{2 \kappa_i \i (\alpha-1)} \\
  & 1
 \end{pmatrix}.
\end{align*}
This implies the assertion. 
\end{proof}

\subsection{The doubling method for $\U(\WW)$}
\label{ss:doubling-U(WW)}

We consider a $2$-dimensional left $E$-space $\WW = B$ equipped with a skew-hermitian form
\[
 (x, y) = -\i \cdot \pr (x \cdot y^*).
\]
Then we have a natural embedding
\[
 \GU(W) \hookrightarrow \GU(\WW)
\]
and an isomorphism $\GU(\WW) \simeq (B^{\times} \times E^{\times}) / F^{\times}$, where $B^{\times} \times E^{\times}$ acts on $\WW$ by
\[
 x \cdot (h, \alpha) = \alpha^{-1} \cdot x \cdot h.
\]
We write $[h,\alpha]$ for the image of $(h,\alpha)$ in $\GU(\WW)$.
Also, we consider the doubled space $\WW^\square = \WW \oplus \WW$ equipped with a skew-hermitian form
\[
 ((w_1,w_2), (w_1',w_2')) := (w_1, w_1') - (w_2, w_2').
\]
Then we have a natural embedding
\[
 \iota : \G(\U(\WW) \times \U(\WW)) \longrightarrow \GU(\WW^\square).
\]

Let $\VV = e_1 E + \dots + e_m E$ be an $m$-dimensional right $E$-space equipped with a hermitian form
\[
 (e_1 x_1 + \dots + e_m x_m, e_1 y_1 + \dots + e_m y_m)
 = x_1^\rho \cdot \kappa_1 \cdot y_1 + \dots + x_m^\rho \cdot \kappa_m \cdot y_m.
\]
Let $f: \VV \otimes_E \WW \rightarrow V \otimes_B W$ be the natural isomorphism.
Then we have
\begin{equation}
\label{eq:f-equivariance-U(WW)}
 f(v \otimes (w \cdot [h,\alpha])) = f(v \otimes w) \cdot (\alpha, h) 
\end{equation}
for $h \in B^\times$ and $\alpha \in E^\times$, and 
\[
 \llangle \cdot, \cdot \rrangle \circ (f \times f)
 = \frac{1}{2} \tr_{E/F} ( (\cdot, \cdot) \otimes (\cdot, \cdot)^{\rho} ).
\]
Hence we may identify $\VV \otimes_E \WW$ with $\V$ and omit $f$ from the notation.
Similarly, we identify $\VV \otimes_E \WW^\square$ with $\V^\square$.

Now we recall Kudla's splitting over $\U(\WW^\square)$, where we regard $\U(\WW^\square)$ as a subgroup of $\Sp(\V^\square)$ via the natural embedding.
We take a complete polarization $\WW^\square = \WW^\bigtriangledown \oplus \WW^\triangle$ over $E$ defined by
\[
 \WW^{\bigtriangledown} = \{ (w,-w) \, | \, w \in \WW \}, \qquad 
 \WW^{\triangle} = \{ (w,w) \, | \, w \in \WW \}.
\]
Put
\[
 \w_1 = - \frac{1}{2\i} \cdot (1, -1), \qquad 
 \w_2 = \frac{1}{2 J \i} \cdot (\j, -\j), \qquad
 \w_1^* = (1, 1), \qquad
 \w_2^* = (\j, \j), 
\]
so that $(\w_i, \w_j^*) = \delta_{ij}$.
Using a basis $\w_1, \w_2, \w_1^*, \w_2^*$ of $\WW^\square$, we identify $\U(\WW^\square)$ with
\[
 \left\{ h \in \GL_4(E) \, \left| \,
 h \begin{pmatrix} & \1_2 \\ -\1_2 & \end{pmatrix} {}^t h^\rho
 = \begin{pmatrix} & \1_2 \\ -\1_2 & \end{pmatrix}
 \right. \right\}.
\]
Let $P_{\WW^\triangle}$ be the maximal parabolic subgroup of $\U(\WW^\square)$ stabilizing $\WW^\triangle$:
\[
 P_{\WW^\triangle} = \left\{ \left.
 \begin{pmatrix} a & * \\ & ({}^t a^\rho)^{-1} \end{pmatrix}
 \, \right|  \, a \in \GL_2(E) \right\}.
\]
We define a map
\[
 \hat{s}_2 : \U(\WW^\square) \longrightarrow \C^1
\]
by setting
\[
 \hat{s}_2(h) = \chi(x(h))^m \cdot \gamma^{-j}
\]
for $h = p_1 \tau_j p_2$ with
\[
 p_i =
 \begin{pmatrix}
  a_i & * \\
  & ({}^t a_i^\rho)^{-1}
 \end{pmatrix}
 \in P_{\WW^\triangle}, \qquad
 \tau_j = 
 \begin{pmatrix}
  \1_{2-j} & & & \\
  & & & -\1_j \\
  & & \1_{2-j} & \\
  & \1_j & & 
 \end{pmatrix},
\]
where 
\[
 x(h) = \det(a_1 a_2) \mod \N_{E/F}(E^\times)
\]
and 
\[
 \gamma = (u, \det \VV)_F \cdot \gamma_F(-u, \tfrac{1}{2} \psi)^m \cdot \gamma_F(-1, \tfrac{1}{2} \psi)^{-m}.
\]
Here, $(\cdot, \cdot)_F$ is the quadratic Hilbert symbol of $F$ and $\gamma_F(\cdot,\frac{1}{2}\psi)$ is the Weil index as in \cite[Appendix]{rangarao}, \cite[\S 3.1.1]{periods1}.
By \cite[Theorem 3.1, cases $3_+$]{kudla-splitting}, we have
\begin{equation}
\label{eq:spl-kudla-2}
 z_{\V^\triangle}(h_1,h_2) = \frac{\hat{s}_2(h_1 h_2)}{\hat{s}_2(h_1) \hat{s}_2(h_2)} 
\end{equation}
for $h_1, h_2 \in \U(\WW^\square)$.

\begin{lem}
\label{lem:s-hat-2-E-conj}
For $\alpha \in E^{\times}$ and $h \in \U(\WW^\square)$, we have
\[
 \hat{s}_2(\iota([\alpha,1], [\alpha,1]) \cdot h \cdot \iota([\alpha,1], [\alpha,1])^{-1}) = \hat{s}_2(h).
\]
\end{lem}

\begin{proof}
Put $h_\alpha = \iota([\alpha,1], [\alpha,1])$.
Since 
\[
 \begin{bmatrix}
  \w_1 \cdot h_\alpha \\
  \w_2 \cdot h_\alpha \\
  \w_1^* \cdot h_\alpha \\
  \w_2^* \cdot h_\alpha
 \end{bmatrix}
 =
 \begin{pmatrix}
  \alpha & & & \\
  & \alpha^\rho & & \\
  & & \alpha & \\
  & & & \alpha^\rho
 \end{pmatrix}
 \cdot 
 \begin{bmatrix}
  \w_1 \\
  \w_2 \\
  \w_1^* \\
  \w_2^*
 \end{bmatrix},
\]
we have $x(h_\alpha p h_\alpha^{-1}) = x(p)$ for $p \in P_{\WW^\triangle}$ and $h_\alpha \tau_j h_\alpha^{-1} = \tau_j$.
Hence the assertion follows.
\end{proof}

\begin{lem}
\label{lem:s-hat-2-E-diag}
For $\alpha \in E^\times$, we have
\[
 \hat{s}_2(\iota([\alpha,\alpha], [\alpha,\alpha])) = \chi(\alpha)^{-2m}.
\]
\end{lem}

\begin{proof}
Put $h_\alpha = \iota([\alpha,\alpha], [\alpha,\alpha])$.
Since 
\[
 \begin{bmatrix}
  \w_1 \cdot h_\alpha \\
  \w_2 \cdot h_\alpha \\
  \w_1^* \cdot h_\alpha \\
  \w_2^* \cdot h_\alpha
 \end{bmatrix}
 =
 \begin{pmatrix}
  1 & & & \\
  & \alpha^{-1} \alpha^\rho & & \\
  & & 1 & \\
  & & & \alpha^{-1} \alpha^\rho
 \end{pmatrix}
 \cdot 
 \begin{bmatrix}
  \w_1 \\
  \w_2 \\
  \w_1^* \\
  \w_2^*
 \end{bmatrix}
\]
and $\chi(\alpha^\rho) = \chi(\alpha)^{-1}$, the assertion follows.
\end{proof}

\begin{lem}
\label{lem:s-hat-2-E^1}
Let $\alpha \in E^1$.
Then we have
\[
 \hat{s}_2(\iota([1,\alpha], 1)) = \chi(\alpha)^{-m}
\]
if $B$ is split, and
\[
 \hat{s}_2(\iota([1,\alpha], 1)) = \chi(\alpha)^{-m} \times 
 \begin{cases}
  1 & \text{if $\alpha = 1$,} \\
  (-1)^m & \text{if $\alpha \ne 1$}
 \end{cases}
\]
if $B$ is ramified.
\end{lem}

\begin{proof}
We may assume that $\alpha \ne 1$.
Then we have $\tr_{E/F}(\alpha) \ne 2$ and hence $\alpha -1 \in E^\times$.
As in the proof of Lemma \ref{lem:s-hat-1-E^1}, we have
\[
 \begin{bmatrix}
  \w_1 \cdot \iota([1,\alpha],1) \\
  \w_2 \cdot \iota([1,\alpha],1) \\
  \w_1^* \cdot \iota([1,\alpha],1) \\
  \w_2^* \cdot \iota([1,\alpha],1)
 \end{bmatrix}
 = A \cdot
 \begin{bmatrix}
  \w_1 \\
  \w_2 \\
  \w_1^* \\
  \w_2^*
 \end{bmatrix},
\]
where
\begin{align*}
 A & =  
\begin{pmatrix}
 \frac{1}{2}(\alpha^{-1}+1) & & - \frac{1}{4 \i}(\alpha^{-1}-1) & \\
 & \frac{1}{2}(\alpha^{-1}+1) & & \frac{1}{4 J \i}(\alpha^{-1}-1) \\
 - \i (\alpha^{-1}-1) & & \frac{1}{2}(\alpha^{-1}+1) & \\
 & J \i (\alpha^{-1}-1) & & \frac{1}{2}(\alpha^{-1}+1)  
\end{pmatrix} \\
 & = 
 \begin{pmatrix}
  \frac{1}{\i (\alpha - 1)} & & * & \\
  & -\frac{1}{J \i (\alpha - 1)} & & * \\
  & & - \i (\alpha^\rho-1) & \\
  & & & J \i (\alpha^\rho-1)
 \end{pmatrix}
 \cdot \tau_2 \cdot 
 \begin{pmatrix}
  1 & & * & \\
  & 1 & & * \\
  & & 1 & \\
  & & & 1
 \end{pmatrix}.
\end{align*}
Hence we have
\[
 x(\iota([1,\alpha],1)) = - \frac{1}{u J (\alpha-1)^2} \equiv 
\frac{J}{(\alpha-1)^2} \equiv - \frac{J}{\alpha} \mod \N_{E/F}(E^\times),
\]
so that
\begin{align*}
 \chi(x(\iota([1,\alpha],1))) & = \chi(\alpha)^{-1} \cdot \xi_E(-J) \\
 & = \chi(\alpha)^{-1} \cdot \xi_E(-1) \times 
 \begin{cases}
  1 & \text{if $B$ is split,} \\
  -1 & \text{if $B$ is ramified.}
 \end{cases}
\end{align*}
Also, we have
\[
 \gamma^2 = (-1, -u)_F^m \cdot (-1, -1)_F^m = (-1,u)_F^m = \xi_E(-1)^m.
\]
This implies the assertion.
\end{proof}

\subsection{Splitting over $\cG^\sharp$}

Let $\cG^\sharp$ be a subgroup of $\GU(V) \times \GU(W) \times \G\U(V) \times \GU(W)$ defined by
\[
 \cG^\sharp = \{ (g,h,\alpha,\alpha) \in \GU(V)^0 \times \GU(W) \times E^\times \times E^\times
 \, | \, \nu(g) = \nu(h) = \N_{E/F}(\alpha) \}.
\]
Then we have a natural homomorphism
\begin{equation}
\label{eq:hom-G^sharp}
 \cG^\sharp \subset \G(\U(V) \times \U(W)) \times \G(\U(V) \times \U(W))
 \longrightarrow \Sp(\V) \times \Sp(\V) \subset \Sp(\V^\square).
\end{equation}
We define a map
\[
 \hat{s}^\sharp: \cG^\sharp \longrightarrow \C^1
\]
by setting
\[
 \hat{s}^\sharp(g,h,\alpha,\alpha)
 = \chi(\alpha)^{-m} \cdot \hat{s}_1(\iota(g \alpha^{-1},1)) \cdot \hat{s}_2(\iota(h \alpha^{-1}, 1))
 \cdot z_{\V^\triangle}(\iota(g \alpha^{-1},1), \iota(h \alpha^{-1},1)).
\]

\begin{lem}
\label{lem:spl-s-hat-sharp}
For $\g_1, \g_2 \in \cG^\sharp$, we have
\[
 z_{\V^\triangle}(\g_1,\g_2) = \frac{\hat{s}^\sharp(\g_1 \g_2)}{\hat{s}^\sharp(\g_1) \hat{s}^\sharp(\g_2)}.
\] 
Here, by abuse of notation, we write $\g_i$ on the left-hand side for the image of $\g_i$ in $\Sp(\V^\square)$ under \eqref{eq:hom-G^sharp}.
\end{lem}

\begin{proof}
Write $\g_i = (g_i,h_i,\alpha_i,\alpha_i)$.
If $\alpha_1 = \alpha_2 = 1$, then the assertion follows from \eqref{eq:spl-kudla-1}, \eqref{eq:spl-kudla-2}, and \cite[Chapitre 2, II.5]{mvw}.
If $\alpha_1$ and $\alpha_2$ are arbitrary, put $\h_i = (g_i \alpha_i^{-1}, h_i \alpha_i^{-1}, 1,1)$ and $\ba_i = (\alpha_i,\alpha_i,\alpha_i,\alpha_i)$.
Let $P_{\V^\triangle}$ be the maximal parabolic subgroup of $\Sp(\V^\square)$ stabilizing $\V^\triangle$.
Then it follows from \cite[Theorem 4.1]{rangarao} that
\begin{equation}
\label{eq:rangarao-2-cocycle}
 z_{\V^\triangle}(p_1 \sigma p, p^{-1} \sigma' p_2) = z_{\V^\triangle}(\sigma, \sigma')
\end{equation}
for $p_1, p_2, p \in P_{\V^\triangle}$ and $\sigma, \sigma' \in \Sp(\V^\square)$ (see also \cite[\S 3.1.1]{periods1}).
Since $\g_1\g_2 = \h_1 \cdot \ba_1 \h_2 \ba_1^{-1} \cdot \ba_1 \ba_2$ and the image of $\ba_i$ in $\Sp(\V^\square)$ belongs to $P_{\V^\triangle}$, we have
\[
 z_{\V^\triangle}(\g_1, \g_2) = z_{\V^\triangle}(\h_1, \ba_1 \h_2 \ba_1^{-1})
 = \frac{\hat{s}^\sharp(\h_1 \ba_1 \h_2 \ba_1^{-1})}{\hat{s}^\sharp(\h_1) \hat{s}^\sharp(\ba_1\h_2\ba_1^{-1})}.
\]
On the other hand, by definition, we have
\[
 \frac{\hat{s}^\sharp(\g_1 \g_2)}{\hat{s}^\sharp(\g_1) \hat{s}^\sharp(\g_2)}
 = \frac{\hat{s}^\sharp(\g_1 \g_2 (\ba_1 \ba_2)^{-1})}{\hat{s}^\sharp(\g_1 \ba_1^{-1}) \hat{s}^\sharp(\g_2 \ba_2^{-1})}
 = \frac{\hat{s}^\sharp(\h_1 \ba_1 \h_2 \ba_1^{-1})}{\hat{s}^\sharp(\h_1) \hat{s}^\sharp(\h_2)}.
\]
It follows from Lemmas \ref{lem:s-hat-1-E-conj} and \ref{lem:s-hat-2-E-conj}, combined with \eqref{eq:rangarao-2-cocycle}, that
\[
 \hat{s}^\sharp(\ba_1 \h_2 \ba_1^{-1}) = \hat{s}^\sharp(\h_2).
\]
This completes the proof.
\end{proof}

\begin{lem}
\label{lem:compare-s-sharp-s-2}
For $\alpha \in E^1$, we have
\[
 \hat{s}^\sharp(1,1,\alpha,\alpha) = \hat{s}_2(\iota(1, [\alpha,\alpha])).
\]
\end{lem}

\begin{proof}
Put 
\begin{align*}
 g_\alpha & = \iota(\alpha, 1) \in \U(V^\square), \\ 
 h_\alpha & = \iota([\alpha, 1], 1) \in \U(\WW^\square), \\
 k_\alpha & = \iota([1, \alpha], 1) \in \U(\WW^\square), \\
 m_\alpha & = \iota([\alpha, \alpha], [\alpha, \alpha]) \in \U(\WW^\square).
\end{align*}
By definition, we have
\[
 \hat{s}^\sharp(1,1,\alpha,\alpha) 
 = \chi(\alpha)^{-m} \cdot \hat{s}_1(g_\alpha^{-1}) \cdot \hat{s}_2(h_\alpha^{-1})
 \cdot z_{\V^\triangle}(g_\alpha^{-1}, h_\alpha^{-1}).
\]
Since $m_\alpha = h_\alpha \cdot k_\alpha \cdot \iota(1, [\alpha,\alpha])$ and the image of $m_\alpha$ in $\Sp(\V^\square)$ belongs to $P_{\V^\triangle}$, it follows from \eqref{eq:spl-kudla-2}, \eqref{eq:rangarao-2-cocycle}, and Lemma \ref{lem:s-hat-2-E-diag} that 
\begin{align*}
 \hat{s}_2(\iota(1, [\alpha,\alpha]))
 & = \hat{s}_2(k_\alpha^{-1} h_\alpha^{-1}) \cdot \hat{s}_2(m_\alpha)
 \cdot z_{\V^\triangle}(k_\alpha^{-1} h_\alpha^{-1}, m_\alpha) \\
 & = \hat{s}_2(k_\alpha^{-1} h_\alpha^{-1}) \cdot \chi(\alpha)^{-2m} \\
 & = \chi(\alpha)^{-2m} \cdot \hat{s}_2(k_\alpha^{-1}) \cdot \hat{s}_2(h_\alpha^{-1})
 \cdot z_{\V^\triangle}(k_\alpha^{-1},  h_\alpha^{-1}).
\end{align*}
By Lemmas \ref{lem:s-hat-1-E^1} and \ref{lem:s-hat-2-E^1}, we have
\[
 \hat{s}_2(k_\alpha^{-1}) = \chi(\alpha)^m \cdot \hat{s}_1(g_\alpha^{-1}).
\]
By \eqref{eq:f-equivariance-U(WW)}, the image of $k_\alpha$ in $\Sp(\V^\square)$ agrees with that of $g_\alpha$, so that
\[
 z_{\V^\triangle}(k_\alpha^{-1}, h_\alpha^{-1})
 = z_{\V^\triangle}(g_\alpha^{-1}, h_\alpha^{-1}).
\]
This completes the proof.
\end{proof}

\subsection{Proof of Proposition \ref{prop:spl-hodge}}
\label{ss:proof-spl-hodge}

Now we take a complete polarization $\V^\square = \X^\square \oplus \Y^\square$ defined by
\[
 \X^\square = \X \oplus \X, \qquad
 \Y^\square = \Y \oplus \Y.
\]
As in \cite[\S D.3]{periods1}, we have
\[
 z_{\Y^\square}(\iota(\sigma_1,\sigma_2), \iota(\sigma_1',\sigma_2'))
 = z_\Y(\sigma_1, \sigma_1') \cdot z_\Y(\sigma_2, \sigma_2')^{-1}
\]
for $\sigma_i, \sigma_i' \in \Sp(\V)$.
Fix $\sigma_0 \in \Sp(\V^\square)$ such that $\V^\bigtriangledown = \X^\square \cdot \sigma_0$ and $\V^\triangle = \Y^\square \cdot \sigma_0$.
Put
\[
 \mu(\sigma)
 = z_{\Y^\square}(\sigma_0, \sigma)^{-1} \cdot z_{\Y^\square}(\sigma_0 \sigma \sigma_0^{-1}, \sigma_0)
\]
for $\sigma \in \Sp(\V^\square)$.
Note that $\mu$ does not depend on the choice of $\sigma_0$.
Then, by \cite[Lemma 4.2]{kudla-splitting}, we have
\begin{equation}
\label{eq:compare-zY-zV}
 z_{\Y^\square}(\sigma, \sigma') = z_{\V^\triangle}(\sigma, \sigma') 
 \cdot \frac{\mu(\sigma \sigma')}{\mu(\sigma) \mu(\sigma')}
\end{equation}
for $\sigma, \sigma' \in \Sp(\V^\square)$.

Put $s^\sharp = \hat{s}^\sharp \cdot \mu$ and $s_2 = \hat{s}_2 \cdot \mu$.
By Lemma \ref{lem:spl-s-hat-sharp} and \eqref{eq:spl-kudla-2}, we have
\[
 z_{\Y^\square}(\g_1, \g_2) = \frac{s^\sharp(\g_1 \g_2)}{s^\sharp(\g_1) s^\sharp(\g_2)}
\]
for $\g_1,\g_2 \in \cG^\sharp$ and
\[
 z_{\Y^\square}(h_1, h_2) = \frac{s_2(h_1 h_2)}{s_2(h_1) s_2(h_2)}
\]
for $h_1, h_2 \in \U(\WW^\square)$.
We define a map
\[
 s: \cG \longrightarrow \C^1
\]
by setting
\[
 s(g,h) = \frac{s^\sharp(g,h,\alpha,\alpha)}{s_2(\iota(1, [\alpha,\alpha]))},
\]
where we choose $\alpha \in E^\times$ such that $\nu(g) = \nu(h) = \N_{E/F}(\alpha)$.

\begin{lem}
\label{lem:s-well-def}
The map $s$ is well-defined, i.e., for $(g,h,\alpha,\alpha) \in \cG^\sharp$ and $\beta \in E^1$, we have
\[
 \frac{s^\sharp(g,h,\alpha \beta, \alpha \beta)}{s_2(\iota(1, [\alpha \beta, \alpha \beta]))}
 = \frac{s^\sharp(g,h,\alpha, \alpha)}{s_2(\iota(1, [\alpha, \alpha]))}.
\]
\end{lem}

\begin{proof}
First note that, by \eqref{eq:f-equivariance-U(WW)}, the image of $(\alpha, \alpha) \in \G(\U(V) \times \U(W))$ in $\Sp(\V)$ agrees with that of $[\alpha, \alpha] \in \U(\WW)$.
We have
\begin{align*}
 s^\sharp(g,h,\alpha \beta, \alpha \beta)
 & = s^\sharp(g,h,\alpha,\alpha) \cdot s^\sharp(1,1,\beta,\beta)
 \cdot z_{\Y^\square}((g,h,\alpha,\alpha), (1,1,\beta,\beta)) \\
 & = s^\sharp(g,h,\alpha,\alpha) \cdot s^\sharp(1,1,\beta,\beta)
 \cdot z_\Y((\alpha,\alpha), (\beta,\beta))^{-1}
\end{align*}
and
\begin{align*}
 s_2(\iota(1, [\alpha \beta, \alpha \beta])) 
 & = s_2(\iota(1, [\alpha, \alpha])) \cdot s_2(\iota(1, [\beta, \beta]))
 \cdot z_{\Y^\square}(\iota(1, [\alpha, \alpha]), \iota(1, [\beta, \beta])) \\
 & = s_2(\iota(1, [\alpha, \alpha])) \cdot s_2(\iota(1, [\beta, \beta]))
 \cdot z_\Y([\alpha, \alpha], [\beta, \beta])^{-1}.
\end{align*}
By Lemma \ref{lem:compare-s-sharp-s-2}, we have
\[
 s^\sharp(1,1,\beta,\beta) = s_2(\iota(1, [\beta, \beta])).
\]
This implies the assertion.
\end{proof}

\begin{lem}
\label{lem:s-proof(i)}
The map $s$ satisfies the condition (i) of Proposition \ref{prop:spl-hodge}, i.e., for $\g_1, \g_2 \in \cG$, we have
\[
 z_\Y(\g_1,\g_2) = \frac{s(\g_1\g_2)}{s(\g_1) s(\g_2)}.
\]
\end{lem}

\begin{proof}
Write $\g_i = (g_i, h_i)$ and choose $\alpha_i \in E^\times$ such that $\nu(g_i) = \nu(h_i) = \N_{E/F}(\alpha_i)$.
Then we have
\begin{align*}
 \frac{s(\g_1\g_2)}{s(\g_1) s(\g_2)}
 & = \frac{s^\sharp(g_1 g_2, h_1 h_2, \alpha_1 \alpha_2, \alpha_1 \alpha_2)}{s^\sharp(g_1, h_1, \alpha_1, \alpha_1) s^\sharp(g_2, h_2, \alpha_2, \alpha_2)}
 \cdot \frac{s_2(\iota(1, [\alpha_1, \alpha_1])) s_2(\iota(1, [\alpha_2, \alpha_2]))}{s_2(\iota(1, [\alpha_1 \alpha_2, \alpha_1 \alpha_2]))} \\
 & = z_{\Y^\square}((g_1, h_1, \alpha_1, \alpha_1), (g_2, h_2, \alpha_2, \alpha_2))
 \cdot z_{\Y^\square}(\iota(1, [\alpha_1, \alpha_1]), \iota(1, [\alpha_2, \alpha_2]))^{-1} \\
 & = z_\Y((g_1, h_1), (g_2, h_2)) \cdot z_\Y((\alpha_1, \alpha_1), (\alpha_2, \alpha_2))^{-1}
 \cdot z_{\Y}([\alpha_1, \alpha_1], [\alpha_2, \alpha_2]) \\
 & = z_\Y((g_1, h_1), (g_2, h_2)).
\end{align*}
This completes the proof.
\end{proof}

By definition, the map $s$ also satisfies the other conditions of Proposition \ref{prop:spl-hodge}.
This completes the proof of Proposition \ref{prop:spl-hodge}.

\subsection{Independence of the choice of $\chi$}

To define the map $s$, we have used the fixed character $\chi$ of $E^\times$ such that $\chi|_{F^\times} = \xi_E$.
However, we have:

\begin{lem}
The map $s$ defined as above does not depend on the choice of $\chi$.
\end{lem}

\begin{proof}
Let $\chi_1$ and $\chi_2$ be two characters of $E^\times$ such that $\chi_1|_{F^\times} = \chi_2|_{F^\times} = \xi_E$.
We will write $s = s_{\chi_i}$, etc., to indicate the dependence on $\chi_i$.
For $(g,h) \in \Gc$, choose $\alpha \in E^\times$ such that $\nu(g) = \nu(h) = \N_{E/F}(\alpha)$.
Then, by definition, we have
\begin{align*}
 \frac{s_{\chi_1}(g,h)}{s_{\chi_2}(g,h)}
 & = \frac{s^\sharp_{\chi_1}(g,h,\alpha,\alpha)}{s^\sharp_{\chi_2}(g,h,\alpha,\alpha)}
 \cdot \frac{s_{2,\chi_2}(\iota(1, [\alpha,\alpha]))}{s_{2,\chi_1}(\iota(1, [\alpha,\alpha]))} \\
 & = \frac{\hat{s}^\sharp_{\chi_1}(g,h,\alpha,\alpha)}{\hat{s}^\sharp_{\chi_2}(g,h,\alpha,\alpha)}
 \cdot \frac{\hat{s}_{2,\chi_2}(\iota(1, [\alpha,\alpha]))}{\hat{s}_{2,\chi_1}(\iota(1, [\alpha,\alpha]))} \\
 & = \eta(\alpha)^{-m}
 \cdot \frac{\hat{s}_{2,\chi_1}(\iota(h \alpha^{-1}, 1))}{\hat{s}_{2,\chi_2}(\iota(h \alpha^{-1}, 1))}
 \cdot \frac{\hat{s}_{2,\chi_2}(\iota(1, [\alpha,\alpha]))}{\hat{s}_{2,\chi_1}(\iota(1, [\alpha,\alpha]))},
\end{align*}
where $\eta = \chi_1/\chi_2$.
On the other hand, for $k \in \U(\WW^\square)$, we have
\[
 \frac{\hat{s}_{2,\chi_1}(k)}{\hat{s}_{2,\chi_2}(k)}
 = \eta(x(k))^m 
 = \tilde{\eta}(\det k)^m,
\]
where $\tilde{\eta}$ is the character of $E^1$ such that $\tilde{\eta}(x/x^\rho) = \eta(x)$ for $x \in E^\times$.
Since 
\begin{align*}
 \det \iota(h \alpha^{-1}, 1) & = \nu(h) \N_{E/F}(\alpha)^{-1} = 1, \\
 \det \iota(1, [\alpha,\alpha]) & = \N_{E/F}(\alpha) \alpha^{-2} = \alpha^{-1} \alpha^\rho,
\end{align*}
we have
\[
 \frac{s_{\chi_1}(g,h)}{s_{\chi_2}(g,h)}
 = \eta(\alpha)^{-m} \cdot \tilde{\eta}(\alpha^{-1} \alpha^\rho)^{-m}
 = 1.
\]
This completes the proof.
\end{proof}

\subsection{Compatibility with seesaws}

We write $V = V' \oplus V''$ as an orthogonal direct sum of skew-hermitian right $B$-spaces
\[
 V' = e_1 B \oplus \dots \oplus e_{m'} B, \qquad
 V'' = e_{m'+1} B \oplus \dots \oplus e_m B.
\]
Let $\V' = V' \otimes_B W$ and $\V'' = V'' \otimes_B W$ be the symplectic $F$-spaces as in \S \ref{ss:weil-hodge-setup}.
Then we have $\V = \V' \oplus \V''$, which gives rise to a seesaw diagram
\[
 \xymatrix{
  \GU(V) \ar@{-}[dr] \ar@{-}[d] & \G(\U(W) \times \U(W)) \ar@{-}[dl] \ar@{-}[d] \\
  \G(\U(V') \times \U(V'')) & \GU(W)}.
\]
Let $\cG'$ and $\cG''$ be the subgroups of $\G(\U(V') \times \U(W))$ and $\G(\U(V'') \times \U(W))$, respectively, as in \S \ref{ss:weil-hodge-setup}.
Put
\[
 \cG''' = \{ (g', g'', h) \in \GU(V')^0 \times \GU(V'')^0 \times \GU(W)
 \, | \, \nu(g') = \nu(g'') = \nu(h) \in \N_{E/F}(E^\times) \}.
\]
We regard $\cG'''$ as subgroups of $\cG$ and $\cG' \times \cG''$ via the above seesaw diagram.
We take the complete polarizations $\V' = \X' \oplus \Y'$ and $\V'' = \X'' \oplus \Y''$ as in \S \ref{ss:weil-hodge-setup}, so that 
\[
 \X = \X' \oplus \X'', \qquad 
 \Y = \Y' \oplus \Y''.
\]
Let $s': \cG' \rightarrow \C^1$ and $s'':\cG'' \rightarrow \C^1$ be the maps trivializing $z_{\Y'}$ and $z_{\Y''}$, respectively, defined similarly as above.
Then, by construction, we have
\[
 s = s' \otimes s''
\]
on $\cG'''$.

\subsection{Compatibility with \cite{periods2}}
\label{ss:compatibility-periods2}

In this section, we compare the splitting $s$ with the standard one for unitary dual pairs when $\dim V = 1$.
In this case, using the notation of \S \ref{ss:doubling-U(WW)}, we have a seesaw diagram
\[
 \xymatrix{
  E^\times \hspace{-1.5cm} & \simeq \hspace{-1cm} & \GU(V)^0 \ar@{-}[dr] \ar@{=}[d] & \GU(\WW) \ar@{-}[dl] \ar@{-}[d] & \hspace{-1cm} \simeq & \hspace{-1cm} (B^\times \times E^\times)/F^\times \\
  E^\times \hspace{-1.5cm} & \simeq \hspace{-1cm} & \GU(\VV) & \GU(W) & \hspace{-1cm} \simeq & \hspace{-1cm} B^\times}.
\]
We define a map 
\[
 s^\natural : \G(\U(V) \times \U(W))^0 \longrightarrow \C^1 
\]
by setting
\[
 s^\natural(\alpha, h) = s_2(\iota([h, \alpha], 1))
\]
for $(\alpha, h) \in \G(\U(V) \times \U(W))^0$, where $[h, \alpha] \in \U(\WW)$ and $s_2 = \hat{s}_2 \cdot \mu$.
Then $s^\natural$ trivializes $z_\Y$ by \eqref{eq:f-equivariance-U(WW)}.
This splitting will be used in \cite{periods2}.

\begin{lem}
We have
\[
 s^\natural(\alpha, h) = s(\alpha, h) \cdot \chi(\alpha)^{-1}.
\]
\end{lem}

\begin{proof}
Recall that
\[
 s(\alpha, h) = \frac{s^\sharp(\alpha, h, \alpha, \alpha)}{s_2(\iota(1, [\alpha, \alpha]))}.
\]
We have
\begin{align*}
 s^\sharp(\alpha, h, \alpha, \alpha)
 & = s^\sharp(1, h \alpha^{-1}, 1, 1)
 \cdot s^\sharp(\alpha, \alpha, \alpha, \alpha)
 \cdot z_{\Y^\square}((1, h \alpha^{-1}, 1, 1), (\alpha, \alpha, \alpha, \alpha)) \\
 & = s^\sharp(1, h \alpha^{-1}, 1, 1)
 \cdot s^\sharp(\alpha, \alpha, \alpha, \alpha)
 \cdot z_{\Y}((1, h \alpha^{-1}), (\alpha, \alpha)).
\end{align*}
By definition and Lemma \ref{lem:s-hat-2-E-diag}, we have
\[
 s^\sharp(1, h \alpha^{-1}, 1, 1) = s_2(\iota([h \alpha^{-1}, 1], 1))
\]
and 
\begin{align*}
 s^\sharp(\alpha, \alpha, \alpha, \alpha)
 & = \chi(\alpha) 
 \cdot s_2(\iota([\alpha, \alpha], [\alpha, \alpha])) \\
 & = \chi(\alpha)
 \cdot s_2(\iota([\alpha, \alpha], 1)) 
 \cdot s_2(\iota(1, [\alpha, \alpha]))
 \cdot z_{\Y^\square}(\iota([\alpha, \alpha], 1), \iota(1, [\alpha, \alpha])) \\
 & = \chi(\alpha)
 \cdot s_2(\iota([\alpha, \alpha], 1)) 
 \cdot s_2(\iota(1, [\alpha, \alpha])).
\end{align*}
Hence we have
\begin{align*}
 s(\alpha, h)
 & = \chi(\alpha)
 \cdot s_2(\iota([h \alpha^{-1}, 1], 1))
 \cdot s_2(\iota([\alpha, \alpha], 1)) 
 \cdot z_{\Y}([h \alpha^{-1}, 1], [\alpha, \alpha]) \\
 & = \chi(\alpha)
 \cdot s_2(\iota([h \alpha^{-1}, 1], 1))
 \cdot s_2(\iota([\alpha, \alpha], 1)) 
 \cdot z_{\Y^\square}(\iota([h \alpha^{-1}, 1], 1), \iota([\alpha, \alpha], 1)) \\
 & = \chi(\alpha) \cdot s_2(\iota([h, \alpha], 1)).
\end{align*}
\end{proof}

\subsection{Compatibility with \cite{periods1}}
\label{ss:compatibility-periods1}

In this section, we compare the splitting $s$ with the one defined in \cite[Appendix C]{periods1}.
Suppose again that $F$ is a number field.
Let $V = B_1 \otimes_E B_2$ be the $2$-dimensional skew-hermitian right $B$-space as in \cite[\S 2.2]{periods1},
where $B_1$ and $B_2$ are quaternion algebras over $F$ such that $E$ embeds into $B_1$ and $B_2$, and such that $B_1 \cdot B_2 = B$ in the Brauer group.
We write $B_i = E + E \j_i$ for some trace zero element $\j_i \in B_i^\times$ and put $J_i = \j_i^2 \in F^\times$.
We may assume that
\[
 J_1 \cdot J_2 = J.
\]
Then the skew-hermitian form on $V$ is given by \eqref{eq:condition-skew-herm} with 
\begin{align*}
 e_1 & = 1 \otimes 1, & \kappa_1 & = 1, \\
 e_2 & = \j_1 \otimes 1, & \kappa_2 & = -J_1.
\end{align*}
Recall the exact sequence
\[
 1 \longrightarrow F^\times \longrightarrow B_1^\times \times B_2^\times 
 \longrightarrow \GU(V)^0 \longrightarrow 1,
\]
where $F^\times$ embeds into $B_1^\times \times B_2^\times$ by $z \mapsto (z, z^{-1})$ and $B_1^\times \times B_2^\times$ acts on $V$ on the left by
\[
 (g_1, g_2) \cdot (x_1 \otimes x_2) = g_1 x_1 \otimes g_2 x_2.
\]
We write $[g_1,g_2]$ for the image of $(g_1,g_2)$ in $\GU(V)^0$.
If we put 
\[
 \tilde{\cG} = \{ (g_1, g_2, h) \in B_1^\times \times B_2^\times \times B^\times \, | \, \nu(g_1) \nu(g_2) = \nu(h) \in \N_{E/F}(E^\times) \},
\]
where $\nu$ denotes the reduced norm, then we have a natural surjective map $\tilde{\cG} \twoheadrightarrow \cG$.
We take the complete polarization $\V = \X \oplus \Y$ as in \S \ref{ss:weil-hodge-setup}, which agrees with the one given in \cite[\S 2.2]{periods1}.
For each place $v$ of $F$, let
\[
 \tilde{s}_v: \GU(V_v)^0 \times \GU(W_v) \longrightarrow \C^1
\]
be the map trivializing $z_{\Y_v}$ defined in \cite[Appendix C]{periods1}.
Since both $\tilde{s}_v$ and $s_v$ trivialize $z_{\Y_v}$,
there exists a \emph{continuous} character $\Chi$ of $\cG(\A)$ such that
\[
 \tilde{s}_v|_{\cG_v} = s_v \cdot \Chi_v
\]
for all $v$.
Since both $\tilde{s}_v$ and $s_v$ satisfy the product formula, $\Chi$ is trivial on $\cG(F)$.
We regard $\Chi$ as a character of $\tilde{\cG}(\A)$.

\begin{prop}
Assume that
\begin{itemize}
 \item $F$ is totally real;
 \item $E$ is totally imaginary;
 \item $B_{1,v}$ and $B_{2,v}$ are split for some real place $v$ of $F$.
\end{itemize}
Then, for $(g_1, g_2, h) \in \tilde{\cG}(\A)$, we have
\[
 \Chi(g_1, g_2, h) = 1.
\]
Namely, we have
\[
 \tilde{s}_v|_{\cG_v} = s_v 
\]
for all $v$.
\end{prop}

\begin{proof}
We define a homomorphism $\tilde{\nu}: \tilde{\cG}(\A) \rightarrow \A^\times$ by
\[
 \tilde{\nu}(g_1, g_2, h) = \nu(g_1).
\]
Then the image of $\tilde{\nu}$ consists of elements $a \in \A^\times$ with $a_v > 0$ for all infinite places $v$ such that $B_{1,v}$ or $B_{2,v}$ or $B_v$ is ramified.
Also, putting $\tilde{\cG}^{(1)} = B_1^{(1)} \times B_2^{(1)} \times B^{(1)}$, we have
\begin{align*}
 \ker \tilde{\nu} & = \tilde{\cG}^{(1)}(\A) \cdot \{ (1, \alpha,\alpha) \, | \, \alpha \in \A_E^\times \}, \\
 \tilde{\nu}^{-1}(\N_{E/F}(\A_E^\times)) & = \ker \tilde{\nu} \cdot \{ (\alpha,\alpha^{-1},1) \, | \, \alpha \in \A_E^\times \}, \\
 \tilde{\nu}^{-1}(F^\times) & = \ker \tilde{\nu} \cdot \tilde{\cG}(F),
\end{align*}
where we have used Eichler's norm theorem in the last equality.
Since $\tilde{\nu}^{-1}(F^\times \N_{E/F}(\A_E^\times))$ is the kernel of $\xi_E \circ \tilde{\nu}$, it is a subgroup of $\tilde{\cG}(\A)$ of index $2$ and does not contain any element $(g_{1,v}, g_{2,v}, h_v) \in \tilde{\cG}_v$ such that $\nu(g_{i,v}) \notin \N_{E_v/F_v}(E_v^\times)$.

Now we show that $\Chi$ is trivial.
Since $\Chi$ is automorphic, it is trivial on $\tilde{\cG}^{(1)}(\A)$.
Moreover, in \S \ref{ss:weil-hodge-computation} below, we will prove the following:
\begin{itemize}
\item For $\alpha \in \A_E^\times$, we have
\begin{equation}
\label{eq:weil-hodge-computation1}
 \Chi(1,\alpha,\alpha) = 1.
\end{equation}
\item For $\alpha \in \A_E^\times$, we have
\begin{equation}
\label{eq:weil-hodge-computation2}
 \Chi(\alpha, \alpha^{-1}, 1) = 1.
\end{equation}
\item Let $v$ be a real place of $F$ such that $B_{1,v}$ and $B_{2,v}$ are split.
Choose $t_{i,v} \in F_v^\times$ such that $J_i = t_{i,v}^2$.
Then we have
\begin{equation}
\label{eq:weil-hodge-computation3}
 \Chi_v(t_{1,v}^{-1} \cdot \j_1, t_{2,v}^{-1} \cdot \j_2, 1) = 1.
\end{equation}
Note that $\nu(t_{i,v}^{-1} \cdot \j_i) = -1 \notin \N_{E_v/F_v}(E_v^\times)$.
\end{itemize}
This implies the assertion. 
\end{proof}

\subsection{Computation of splittings}
\label{ss:weil-hodge-computation}

We retain the notation of \S \ref{ss:compatibility-periods1}.
We fix a place $v$ of $F$ and suppress the subscript $v$ from the notation.
Recall that $\Chi$ is a \emph{continuous} character of $\tilde{\cG}$ such that 
\[
 \Chi(g_1,g_2,h) = \frac{\tilde{s}(g_1,g_2,h)}{s(g_1,g_2,h)},
\]
where we regard $\tilde{s}$ and $s$ as maps on $\tilde{\cG}$.
To compute $\Chi$ explicitly, we need to introduce more notation.

\subsubsection{Notation}
\label{sss:notation}

We denote by $\GSp_{2n}(F)$ the symplectic similitude group and by $\nu : \GSp_{2n}(F) \rightarrow F^\times$ the similitude character:
\[
 \GSp_{2n}(F) = \left\{ \sigma \in \GL_{2n}(F) \, \left| \,
 \sigma \begin{pmatrix} & \1_n \\ -\1_n & \end{pmatrix} {}^t \sigma 
 = \nu(\sigma) \cdot \begin{pmatrix} & \1_n \\ -\1_n & \end{pmatrix}
 \right. \right\}.
\]
Let $\Sp_{2n}(F) = \ker \nu$ be the symplectic group and $P$ the standard maximal parabolic subgroup of $\Sp_{2n}(F)$:
\[
 P = \{ \m(\a) \n(\b) \, | \, \a \in \GL_n(F), \b \in \Sym_n(F) \},
\]
where 
\[
 \m(\a) = \begin{pmatrix} \a & \\ & {}^t \a^{-1} \end{pmatrix}, \qquad
 \n(\b) = \begin{pmatrix} \1_n & \b \\ & \1_n \end{pmatrix}.
\]
Put
\[
 d(\nu) = \begin{pmatrix} \1_n & \\ & \nu \cdot \1_n \end{pmatrix}, \qquad
 \tau_j =
 \begin{pmatrix}
  \1_{n-j} & & & \\
  & & & -\1_j \\
  & & \1_{n-j} & \\
  & \1_j & & 
 \end{pmatrix}.
\]
If $\sigma = p_1 \tau_j p_2 \in \Sp_{2n}(F)$ with $p_i = \m(\a_i) \n(\b_i) \in P$, put
\[
 x(\sigma) = \det(\a_1 \a_2) \mod (F^{\times})^2, \qquad j(\sigma) = j.
\]
Note that
\begin{equation}
\label{eq:xj-conj}
 x(d(\nu) \cdot \sigma \cdot d(\nu)^{-1}) = \nu^{j(\sigma)} \cdot x(\sigma),
 \qquad 
 j(d(\nu) \cdot \sigma \cdot d(\nu)^{-1}) = j(\sigma).
\end{equation}
We define a map
\[
 v: \Sp_{2n}(F) \times F^\times \longrightarrow \C^1
\]
by setting
\[
 v(\sigma, \nu) = (x(\sigma), \nu)_F \cdot \gamma_F(\nu, \tfrac{1}{2} \psi)^{-j(\sigma)},
\]
where $(\cdot, \cdot)_F$ is the quadratic Hilbert symbol of $F$ and $\gamma_F(\cdot,\frac{1}{2}\psi)$ is the Weil index as in \cite[Appendix]{rangarao}, \cite[\S 3.1.1]{periods1}.
Let $z$ be the $2$-cocycle on $\Sp_{2n}(F)$ realizing the metaplectic group (see e.g.~\cite{rangarao}, \cite[\S 3.2.2]{periods1}).
By \cite[Theorem 4.1 and Corollary 4.2]{rangarao}, we have:
\begin{itemize}
 \item $z(\sigma, \sigma^{-1}) = 1$ for $\sigma \in \Sp_{2n}(F)$;
 \item $z(p_1 \sigma p, p^{-1} \sigma' p_2) = z(\sigma, \sigma')$ for $p_1, p_2, p \in P$ and $\sigma, \sigma' \in \Sp_{2n}(F)$;
 \item $z(\tau_i, \tau_j) = 1$;
 \item $z(\tau_n, \n(\b) \tau_n) = \gamma_F(\tfrac{1}{2} \psi)^n \cdot \gamma_F(\det \b, \tfrac{1}{2} \psi) \cdot h_F(\b)$ for $\b \in \Sym_n(F) \cap \GL_n(F)$, where $h_F(\b)$ is the Hasse invariant of the non-degenerate symmetric bilinear form associated to $\b$.
\end{itemize}
We may extend $z$ to a $2$-cocycle on $\GSp_{2n}(F)$ (see e.g.~\cite[Appendix B]{periods1}).
Then, for $\sigma, \sigma' \in \GSp_{2n}(F)$ with $\nu(\sigma) = \nu$ and $\nu(\sigma') = \nu^{-1}$, we have
\[
 z(\sigma, \sigma') = z(\sigma \cdot d(\nu)^{-1}, d(\nu) \cdot \sigma') \cdot v(\sigma' \cdot d(\nu), \nu).
\]

Recall that $\V = V \otimes_B W$ is an $8$-dimensional symplectic $F$-space.
Let $\e_1, \dots, \e_4, \e_1^*, \dots, \e_4^*$ be the basis of $\V$ given in \cite[\S 2.2]{periods1}.
Then we have
\[
 \X = F \e_1 + \dots +  F \e_4, \qquad
 \Y = F \e_1^* + \dots +  F \e_4^*, \qquad
 \llangle \e_i, \e_j^* \rrangle = \delta_{ij}.
\]
Using this basis, we identify $\GSp(\V)$ with $\GSp_8(F)$.
Under this identification, we write $P_\Y$ and $z_\Y$ for $P$ and $z$, respectively.
We refer to \cite[\S C.1]{periods1} for an explicit description of the image of $B_1^\times \times B_2^\times \times B^\times$ in $\GSp(\V)$.
Also, using a basis
\[
 (\e_1,0), \dots, (\e_4, 0), (0,\e_1), \dots, (0,\e_4), 
 (\e_1^*,0), \dots, (\e_4^*, 0), (0,-\e_1^*), \dots, (0,-\e_4^*) 
\]
of $\V^{\square}$, we identify $\GSp(\V^{\square})$ with $\GSp_{16}(F)$.
For $1 \le i \le 4$, put
\begin{align*}
 \X_i & = F \e_i, &
 \X_i^{\square} & = \X_i \oplus \X_i, \\
 \Y_i & = F \e_i^*, & 
 \Y_i^{\square} & = \Y_i \oplus \Y_i, \\
 \V_i & = \X_i \oplus \Y_i, & 
 \V_i^{\square} & = \V_i \oplus \V_i.
\end{align*}
Then we have a natural embedding
\[
 \iota_i: \Sp(\V_i^\square) \longrightarrow \Sp(\V^\square).
\]
Using a basis $(\e_i,0), (0,\e_i), (\e_i^*,0), (0,-\e_i^*)$ of $\V^{\square}_i$, we identify $\GSp(\V^{\square}_i)$ with $\GSp_4(F)$.
Put
\[
 \sigma_0 = 
  \begin{pmatrix}
  \frac{1}{2} \1_4 & - \frac{1}{2} \1_4 & & \\
  & & \frac{1}{2} \1_4 & \frac{1}{2} \1_4 \\
  & & \1_4 & -\1_4 \\
  -\1_4 & -\1_4 & &
 \end{pmatrix}
 \in \Sp(\V^{\square}).
\]
Then we have
\[
 \begin{bmatrix}
  \frac{1}{2}(\vec{\e}, -\vec{\e}) \\
  \frac{1}{2}(\vec{\e}^*, -\vec{\e}^*) \\
  (\vec{\e}^*, \vec{\e}^*) \\
  (-\vec{\e}, -\vec{\e})
 \end{bmatrix}
 = \sigma_0 \cdot 
 \begin{bmatrix}
  (\vec{\e}, 0) \\
  (0, \vec{\e}) \\
  (\vec{\e}^*, 0) \\
  (0, -\vec{\e}^*)
 \end{bmatrix}, \qquad
 \vec{\e} = 
 \begin{bmatrix}
  \e_1 \\
  \e_2 \\
  \e_3 \\
  \e_4
 \end{bmatrix}, \qquad
 \vec{\e}^* = 
 \begin{bmatrix}
  \e_1^* \\
  \e_2^* \\
  \e_3^* \\
  \e_4^*
 \end{bmatrix},
\]
so that $\V^{\bigtriangledown} = \X^{\square} \cdot \sigma_0$ and $\V^{\triangle} = \Y^{\square} \cdot \sigma_0$.

\subsubsection{Proof of \eqref{eq:weil-hodge-computation1}}

In this section, we will show that
\[
 \Chi(1, \alpha, \alpha) = 1
\]
for $\alpha \in E^\times$.
We write $\alpha = a + b \i$ with $a, b \in F$ and put $\nu = a^2 - b^2 u$.
Since $\Chi$ is continuous, we may assume that 
\[
 a \ne 0, \qquad b \ne 0.
\]

\begin{lem}
\label{lem:weil-hodge-computation1-1}
We have
\[
 \tilde{s}(1, \alpha, \alpha) = \gamma_F(J_1, \tfrac{1}{2} \psi) \cdot (-2abJ_2, J_1)_F.
\]
\end{lem}

\begin{proof}
Put $g = [1, \alpha] \in \GU(V)^0$ and $h = \alpha \in \GU(W)$.
Then we have $\tilde{s}(1, \alpha, \alpha) = \tilde{s}(g) \cdot \tilde{s}(h) \cdot z_\Y(g, h)$.
By \cite[Proposition C.4.2]{periods1}, we have
\[
 \tilde{s}(g) = (- \nu J_2, J_1)_F, \qquad
 \tilde{s}(h) = (J_2, J_1)_F.
\]
It remains to compute $z_{\Y}(g, h)$.

Recall that
\[
 z_\Y(g, h) = z_\Y(g \cdot d(\nu), d(\nu)^{-1} \cdot h) \cdot v(h \cdot d(\nu)^{-1}, \nu^{-1}).
\]
We have
\[
 g = \nu^{-1} \cdot 
 \begin{pmatrix}
  a \cdot \1_4 & -bu \cdot \J_2 \\
  -b \cdot \J_2^{-1} & a \cdot \1_4
 \end{pmatrix}, \qquad
 h =
 \begin{pmatrix}
  a \cdot \1_4 & bu \cdot \J \\
  b \cdot \J^{-1} & a \cdot \1_4
 \end{pmatrix}
\]
in $\GSp(\V)$, where
\[
 \J_2 = 
 \begin{pmatrix}
  1 & & & \\
  & -J_1 & & \\
  & & J_2 & \\
  & & & -J
 \end{pmatrix}, \qquad
 \J = 
 \begin{pmatrix}
  1 & & & \\
  & -J_1 & & \\
  & & -J_2 & \\
  & & & J
 \end{pmatrix}.
\]
Since
\begin{align*}
 g & = 
 \begin{pmatrix}
 - b^{-1} \cdot \J_2 & \nu^{-1} a \cdot \1_4 \\
 & - \nu^{-1} b \cdot \J_2^{-1}
 \end{pmatrix}
 \cdot \tau_4 \cdot \n(- a b^{-1} \cdot \J_2), \\
 h & = \n(a b^{-1} \cdot \J) \cdot \tau_4 \cdot
 \begin{pmatrix}
  b \cdot \J^{-1} & a \cdot \1_4 \\
  & \nu b^{-1} \cdot \J
 \end{pmatrix},
\end{align*}
we have
\[
 z_\Y(g \cdot d(\nu), d(\nu)^{-1} \cdot h)
 = z_\Y(\tau_4 \cdot \n(- \nu a b^{-1} \cdot \J_2), \n(\nu a b^{-1} \cdot \J) \cdot \tau_4)
 = z_\Y(\tau_4, \n(\b) \cdot \tau_4),
\]
where
\[
 \b = - \nu a b^{-1} \cdot \J_2 +  \nu a b^{-1} \cdot \J
 = 2 \nu a b^{-1} \cdot
 \begin{pmatrix}
  0 & & & \\
  & 0 & & \\
  & & -J_2 & \\
  & & & J
 \end{pmatrix}.
\]
If we put
\[
 \b' = 
 2 \nu a b^{-1} \cdot 
 \begin{pmatrix}
  - J_2 & \\
  & J
 \end{pmatrix},
\]
then we have
\[
 z_\Y(\tau_4, \n(\b) \cdot \tau_4) = 
 \gamma_F(\tfrac{1}{2} \psi)^2 \cdot \gamma_F(\det \b', \tfrac{1}{2} \psi) \cdot h_F(\b').
\]
Hence, since $\det \b' \equiv -J_1 \bmod (F^{\times})^2$ and
\[
 h_F(\b') = (-2 \nu a b^{-1} J_2, 2 \nu a b^{-1} J)_F = (-2 \nu a b J_2, J_1)_F,
\]
we have
\begin{align*}
 z_\Y(g \cdot d(\nu), d(\nu)^{-1} \cdot h) 
 & = \gamma_F(-1, \tfrac{1}{2} \psi)^{-1} \cdot \gamma_F(-J_1, \tfrac{1}{2} \psi)
 \cdot (-2 \nu a b J_2, J_1)_F \\
 & = \gamma_F(J_1, \tfrac{1}{2} \psi) \cdot (2 \nu a b J_2, J_1)_F.
\end{align*}
On the other hand, since $x(h \cdot d(\nu)^{-1}) \equiv 1 \bmod (F^{\times})^2$ and $j(h \cdot d(\nu)^{-1}) = 4$, we have
\[
 v(h \cdot d(\nu)^{-1}, \nu^{-1}) = 1.
\]
Thus we obtain
\[
 z_\Y(g, h) = \gamma_F(J_1, \tfrac{1}{2} \psi) \cdot (2 \nu a b J_2, J_1)_F.
\]
This completes the proof.
\end{proof}

Now we compute $s(1, \alpha, \alpha)$.
Note that $[1, \alpha] \in \GU(V)^0$ is the image of $\alpha$ under the embedding $E^\times \hookrightarrow \GU(V)$ as in \S \ref{ss:weil-hodge-setup}.
Hence, by definition, we have
\[
 s(1,\alpha,\alpha) = \frac{s^\sharp(\alpha, \alpha, \alpha, \alpha)}{s_2(\iota(1, [\alpha, \alpha]))},
\]
where $[\alpha, \alpha] \in \U(\WW)$.
Since
\begin{align*}
 s_2(\iota([\alpha, \alpha], [\alpha, \alpha]))
 & = s_2(\iota([\alpha, \alpha], 1)) \cdot s_2(\iota(1, [\alpha, \alpha]))
 \cdot z_{\Y^\square}(\iota([\alpha, \alpha], 1), \iota(1, [\alpha, \alpha])) \\
 & = s_2(\iota([\alpha, \alpha], 1)) \cdot s_2(\iota(1, [\alpha, \alpha])),
\end{align*}
we have
\begin{align*}
 s(1,\alpha,\alpha)
 & = \frac{s^\sharp(\alpha, \alpha, \alpha, \alpha)}{s_2(\iota([\alpha, \alpha], [\alpha, \alpha]))}
 \cdot s_2(\iota([\alpha, \alpha], 1)) \\
 & = \frac{\hat{s}^\sharp(\alpha, \alpha, \alpha, \alpha)}{\hat{s}_2(\iota([\alpha, \alpha], [\alpha, \alpha]))}
 \cdot \hat{s}_2(\iota([\alpha, \alpha], 1)) \cdot \mu(\iota([\alpha, \alpha], 1)).
\end{align*}
Hence, by definition and Lemma \ref{lem:s-hat-2-E-diag}, we have
\[
 s(1, \alpha, \alpha) = \chi(\alpha)^2
 \cdot \hat{s}_2(\iota([\alpha, \alpha], 1)) \cdot \mu(\iota([\alpha, \alpha], 1)).
\]

\begin{lem}
\label{lem:weil-hodge-computation1-2}
We have
\[
 \hat{s}_2(\iota([\alpha, \alpha], 1)) = \chi(\alpha)^{-2} \cdot (u, J_1)_F.
\]
\end{lem}

\begin{proof}
Put $h = \iota([\alpha, \alpha], 1) \in \U(\WW^\square)$ and $\beta = \alpha^{-1} \alpha^\rho$, so that $\beta - 1 \in E^\times$.
As in the proof of Lemma \ref{lem:s-hat-1-E^1}, we have
\[
 \begin{bmatrix}
  \w_1 \cdot h \\
  \w_2 \cdot h \\
  \w_1^* \cdot h \\
  \w_2^* \cdot h
 \end{bmatrix}
 = A \cdot 
 \begin{bmatrix}
  \w_1 \\
  \w_2 \\
  \w_1^* \\
  \w_2^*
 \end{bmatrix}, 
\]
where 
\begin{align*}
 A & =
 \begin{pmatrix}
  1 & & & \\
  & \frac{1}{2} (\beta+1) & & \frac{1}{4 J \i} (\beta-1) \\
  & & 1 & \\
  & J \i (\beta-1) & & \frac{1}{2} (\beta+1)
 \end{pmatrix} \\
 & = 
 \begin{pmatrix}
  1 & & & \\
  & -\frac{1}{J \i (\beta^\rho - 1)} & & * \\
  & & 1 & \\
  & & & J \i (\beta-1)
 \end{pmatrix}
 \cdot \tau_1 \cdot 
 \begin{pmatrix}
  1 & & & \\
  & 1 & & * \\
  & & 1 & \\
  & & & 1
 \end{pmatrix}.
\end{align*}
Hence we have
\[
 \hat{s}_2(\iota([\alpha, \alpha], 1)) = \chi(J \i (\beta -1))^2 \cdot \gamma^{-1},
\]
where
\begin{align*}
 \gamma & = (u, \det \VV)_F \cdot \gamma_F(-u, \tfrac{1}{2} \psi)^2 \cdot \gamma_F(-1, \tfrac{1}{2} \psi)^{-2} \\
 & = (u, -J_1)_F \cdot (-1, -u)_F \cdot (-1, -1)_F \\
 & = (u, J_1)_F.
\end{align*}
Since $\beta-1 = \alpha^{-1}(\alpha^\rho - \alpha) = - 2b \i \alpha^{-1}$,
we have $\chi(J \i (\beta -1))^2 = \chi(- 2b uJ \alpha^{-1})^2 = \chi(\alpha)^{-2}$.
This completes the proof. 
\end{proof}

\begin{lem}
\label{lem:weil-hodge-computation1-3}
We have
\[
 \mu(\iota([\alpha, \alpha], 1)) = \gamma_F(J_1, \tfrac{1}{2} \psi) \cdot (- 2abuJ_2, J_1)_F.
\]
\end{lem}

\begin{proof}
We write $\alpha^{-1} \alpha^\rho = c + d \i$ with $c, d \in F$, so that
\[
 c = \frac{a^2 + b^2 u}{a^2 - b^2 u} \ne \pm 1, \qquad 
 d = - \frac{2ab}{a^2 - b^2 u} \ne 0.
\]
Recall that
\[
 \mu(\iota([\alpha, \alpha], 1))
 = z_{\Y^\square}(\sigma_0, \sigma)^{-1} \cdot z_{\Y^\square}(\sigma_0 \sigma \sigma_0^{-1}, \sigma_0),
\]
where $\sigma$ is the image of $\iota([\alpha, \alpha], 1)$ in $\Sp(\V^\square)$.
We have
\[
 \sigma_0 = \prod_{i=1}^4 \iota_i(\tau_1 \cdot \m(\a_1)), \qquad
 \sigma = \prod_{i=3}^4 \iota_i(\sigma_i),
\]
where
\[
 \a_1 = 
 \begin{pmatrix}
  \frac{1}{2} & - \frac{1}{2} \\
  - 1 & - 1
 \end{pmatrix}, \qquad
 \sigma_i =
 \begin{pmatrix}
  c & & d k_i u & \\
  & 1 & & \\
  \frac{d}{k_i} & & c & \\
  & & & 1
 \end{pmatrix}, \qquad
 k_i = 
 \begin{cases}
  J_2 & \text{if $i=3$,} \\
  -J & \text{if $i=4$.}
 \end{cases}
\]
Since $\tau_1 \cdot \m(\a_1) \in P_{\Y_i^\square} \cdot \tau_1 \cdot \m(\a_2)$ and $\sigma_i \in \n(\b_{1,i}) \cdot \tau \cdot P_{\Y_i^\square}$, where
\[
 \a_2 = 
 \begin{pmatrix}
  1 & \\
  1 & 1
 \end{pmatrix}, \qquad
 \b_{1,i} = \frac{c k_i}{d} \cdot
 \begin{pmatrix}
  1 & \\
  & 0
 \end{pmatrix}, \qquad
 \tau = 
 \begin{pmatrix}
  & & -1 & \\
  & 1 & & \\
  1 & & & \\
  & & & 1
 \end{pmatrix},
\]
we have
\begin{align*}
 z_{\Y^\square}(\sigma_0, \sigma) 
 & = \prod_{i=3}^4 z_{\Y_i^\square}(\tau_1 \cdot \m(\a_1), \sigma_i) \\
 & = \prod_{i=3}^4 z_{\Y_i^\square}(\tau_1, \m(\a_2) \cdot \n(\b_{1,i}) \cdot \tau).
\end{align*}
If we put
\[
 \b_{2,i} = \frac{c k_i}{d} \cdot
 \begin{pmatrix}
  1 & 1 \\
  1 & 
 \end{pmatrix}, \qquad
 \b_{3,i} = \frac{c k_i}{d} \cdot
 \begin{pmatrix}
  0 & \\
  & 1
 \end{pmatrix}, 
\]
then we have $\m(\a_2) \cdot \n(\b_{1,i}) = \n(\b_{2,i}) \cdot \n(\b_{3,i}) \cdot \m(\a_2)$ and hence
\[
 z_{\Y^\square}(\sigma_0, \sigma)
 = \prod_{i=3}^4 z_{\Y_i^\square}(\tau_1 \cdot \n(\b_{2,i}), \n(\b_{3,i}) \cdot \m(\a_2) \cdot \tau).
\]
Since $\tau_1 \cdot \n(\b_{2,i}) \in P_{\Y_i^\square} \cdot \tau_1$ and $\n(\b_{3,i}) \cdot \m(\a_2) \cdot \tau \in \tau \cdot P_{\Y_i^\square}$, we have
\[
 z_{\Y^\square}(\sigma_0, \sigma) = \prod_{i=3}^4 z_{\Y_i^\square}(\tau_1, \tau) = 1.
\]

On the other hand, we have
\[
 \sigma_0 \sigma \sigma_0^{-1} = \prod_{i=3}^4 \iota_i(\sigma_i'),
\]
where 
\[
 \sigma_i' = 
 \begin{pmatrix}
  \frac{1}{2} (c+1) & \frac{d k_i u}{2} &
  \frac{d k_i u}{4} & -\frac{1}{4} (c-1) \\
  \frac{d}{2 k_i} & \frac{1}{2} (c+1) &
  \frac{1}{4} (c-1) & - \frac{d}{4 k_i} \\
  \frac{d}{k_i} & c-1 & \frac{1}{2} (c+1) & -\frac{d}{2 k_i} \\
  -c+1 & - d k_i u & - \frac{d k_i u}{2} & \frac{1}{2} (c+1)
 \end{pmatrix}.
\]
Since $\sigma_i' \in P_{\Y_i^\square} \cdot \tau_2 \cdot \n(\b_{4,i}) \cdot \n(\b_{5,i})$, where
\[
 \b_{4,i} = - \frac{d}{2(c-1) k_i} \cdot
 \begin{pmatrix}
  0 & \\
  & 1
 \end{pmatrix}, \qquad
 \b_{5,i} = \frac{d k_i u}{2(c-1)} \cdot
 \begin{pmatrix}
  1 & \\
  & 0
 \end{pmatrix},
\]
we have
\begin{align*}
 z_{\Y^\square}(\sigma_0 \sigma \sigma_0^{-1}, \sigma_0)
 & = \prod_{i=3}^4 z_{\Y_i^\square}(\sigma_i', \tau_1 \cdot \m(\a_1)) \\
 & = \prod_{i=3}^4 z_{\Y_i^\square}(\tau_2, \n(\b_{4,i}) \cdot \n(\b_{5,i}) \cdot \tau_1).
\end{align*}
Hence, since $\n(\b_{5,i}) \cdot \tau_1 \in \tau_1 \cdot P_{\Y_i^{\square}}$ and
\[
 \frac{d}{c-1} = - \frac{a}{bu},
\]
we have
\begin{align*}
 z_{\Y^\square}(\sigma_0 \sigma \sigma_0^{-1}, \sigma_0)
 & = \prod_{i=3}^4 z_{\Y_i^\square}(\tau_2, \n(\b_{4,i}) \cdot \tau_1) \\
 & = \prod_{i=3}^4 \left[ \gamma_F(\tfrac{1}{2} \psi) \cdot \gamma_F(2ab k_i u, \tfrac{1}{2} \psi) \right] \\
 & = \gamma_F(-1, \tfrac{1}{2} \psi)^{-1}
 \cdot \gamma_F(k_3 k_4, \tfrac{1}{2} \psi) \cdot (2ab k_3 u, 2ab k_4 u)_F \\
 & = \gamma_F(-k_3 k_4, \tfrac{1}{2} \psi) \cdot (-2ab k_3 u, -k_3 k_4)_F.
\end{align*}
This completes the proof.
\end{proof}

By Lemmas \ref{lem:weil-hodge-computation1-2} and \ref{lem:weil-hodge-computation1-3}, we have
\[
 s(1, \alpha, \alpha) = \gamma_F(J_1, \tfrac{1}{2} \psi) \cdot (- 2abJ_2, J_1)_F.
\]
Now \eqref{eq:weil-hodge-computation1} follows from this and Lemma \ref{lem:weil-hodge-computation1-1}.

\subsubsection{Proof of \eqref{eq:weil-hodge-computation2}}

In this section, we will show that
\[
 \Chi(\alpha, \alpha^{-1}, 1) = 1
\]
for $\alpha \in E^\times$.
We write $\alpha = a + b \i$ with $a, b \in F$ and put $\nu = a^2 - b^2 u$.
Since $\Chi$ is continuous, we may assume that 
\[
 a \ne 0, \qquad b \ne 0.
\]

\begin{lem}
\label{lem:weil-hodge-computation2-1}
We have
\[
 \tilde{s}(\alpha, \alpha^{-1}, 1) = \gamma_F(J, \tfrac{1}{2} \psi) \cdot (-2abJ_1, J)_F.
\]
\end{lem}

\begin{proof}
Put $g_1 = [\alpha, 1] \in \GU(V)^0$ and $g_2 = [1, \alpha^{-1}] \in \GU(V)^0$.
Then we have $\tilde{s}(\alpha, \alpha^{-1}, 1) = \tilde{s}(g_1) \cdot \tilde{s}(g_2) \cdot z_\Y(g_1, g_2)$.
By \cite[Proposition C.4.2]{periods1}, we have
\[
 \tilde{s}(g_1) = (- \nu J_1, J_2)_F, \qquad
 \tilde{s}(g_2) = (- \nu J_2, J_1)_F.
\]
We have
\[
 g_1 = \nu^{-1} \cdot 
 \begin{pmatrix}
  a \cdot \1_4 & -bu \cdot \J_1 \\
  -b \cdot \J_1^{-1} & a \cdot \1_4
 \end{pmatrix}, \qquad
 g_2 =
 \begin{pmatrix}
  a \cdot \1_4 & bu \cdot \J_2 \\
  b \cdot \J_2^{-1} & a \cdot \1_4
 \end{pmatrix}
\]
in $\GSp(\V)$, where
\[
 \J_1 = 
 \begin{pmatrix}
  1 & & & \\
  & J_1 & & \\
  & & -J_2 & \\
  & & & -J
 \end{pmatrix}, \qquad
 \J_2 = 
 \begin{pmatrix}
  1 & & & \\
  & -J_1 & & \\
  & & J_2 & \\
  & & & -J
 \end{pmatrix}.
\]
Hence, as in the proof of Lemma \ref{lem:weil-hodge-computation1-1}, we have
\[
 z_{\Y}(g_1, g_2) = \gamma_F(J, \tfrac{1}{2} \psi) \cdot (2 \nu abJ_1, J)_F.
\]
This completes the proof.
\end{proof}

Now we compute $s(\alpha, \alpha^{-1}, 1)$.
By definition, we have
\begin{align*}
 s(\alpha,\alpha^{-1}, 1)
 & = s^\sharp([\alpha, \alpha^{-1}], 1, 1, 1) \\
 & = \hat{s}^\sharp([\alpha, \alpha^{-1}], 1, 1, 1) \cdot \mu(\iota([\alpha, \alpha^{-1}], 1)) \\
 & = \hat{s}_1(\iota([\alpha, \alpha^{-1}], 1)) \cdot \mu(\iota([\alpha, \alpha^{-1}], 1)), 
\end{align*}
where $[\alpha, \alpha^{-1}] \in \U(V)^0$.

\begin{lem}
\label{lem:weil-hodge-computation2-2}
We have
\[
 \hat{s}_1(\iota([\alpha, \alpha^{-1}], 1)) = (u, J)_F.
\]
\end{lem}

\begin{proof}
Put $g = \iota([\alpha, \alpha^{-1}], 1) \in \U(V^\square)$ and $\beta = \alpha^{-1} \alpha^\rho$, so that $\beta - 1 \in E^\times$.
As in the proof of Lemma \ref{lem:s-hat-1-E^1}, we have
\[
 \begin{bmatrix}
  \v_1 \cdot g \\
  \v_2 \cdot g \\
  \v_1^* \cdot g \\
  \v_2^* \cdot g
 \end{bmatrix}
 = A \cdot 
 \begin{bmatrix}
  \v_1 \\
  \v_2 \\
  \v_1^* \\
  \v_2^*
 \end{bmatrix},
\]
where
\begin{align*}
 A & = 
 \begin{pmatrix}
  1 & & & \\
  & \frac{1}{2} (\beta+1) & & - \frac{1}{4 J_1 \i} (\beta-1) \\
  & & 1 & \\
  & -J_1 \i (\beta-1) & & \frac{1}{2} (\beta+1)
 \end{pmatrix} \\
 & = 
 \begin{pmatrix}
  1 & & & \\
  & \frac{1}{J_1 \i (\beta^\rho - 1)} & & * \\
  & & 1 & \\
  & & & -J_1 \i (\beta-1)
 \end{pmatrix}
 \cdot \tau_1 \cdot 
 \begin{pmatrix}
  1 & & & \\
  & 1 & & * \\
  & & 1 & \\
  & & & 1
 \end{pmatrix}.
\end{align*}
Hence we have
\[
 \hat{s}_1(g) = 
 \begin{cases}
  1 & \text{if $B$ is split,} \\
  -1 & \text{if $B$ is ramified.}
 \end{cases}
\]
This completes the proof.
\end{proof}

\begin{lem}
\label{lem:weil-hodge-computation2-3}
We have
\[
 \mu(\iota([\alpha, \alpha^{-1}], 1))
 = \gamma_F(J, \tfrac{1}{2} \psi) \cdot (-2abu J_1, J)_F.
\]
\end{lem}

\begin{proof}
Recall that
\[
 \mu(\iota([\alpha, \alpha^{-1}], 1))
 = z_{\Y^\square}(\sigma_0, \sigma)^{-1} \cdot z_{\Y^\square}(\sigma_0 \sigma \sigma_0^{-1}, \sigma_0),
\]
where $\sigma$ is the image of $\iota([\alpha, \alpha^{-1}], 1)$ in $\Sp(\V^\square)$.
If we write $\alpha^{-1} \alpha^\rho = c + d \i$ with $c, d \in F$, then we have
\[
 \sigma = \prod_{i=2}^3 \iota_i(\sigma_i),
\]
where 
\[
 \sigma_i =
 \begin{pmatrix}
  c & &  d k_i u & \\
  & 1 & & \\
  \frac{d}{k_i} & & c & \\
  & & & 1
 \end{pmatrix}, \qquad
 k_i = 
 \begin{cases}
  J_1 & \text{if $i=2$,} \\
  -J_2 & \text{if $i=3$.}
 \end{cases}
\]
Hence, as in the proof of Lemma \ref{lem:weil-hodge-computation1-3}, we have $z_{\Y^\square}(\sigma_0, \sigma)  = 1$ and
\[
 z_{\Y^\square}(\sigma_0 \sigma \sigma_0^{-1}, \sigma_0)
 = \gamma_F(-k_2 k_3, \tfrac{1}{2} \psi) \cdot (-2ab k_2 u, -k_2 k_3)_F.
\]
This completes the proof.
\end{proof}

By Lemmas \ref{lem:weil-hodge-computation2-2} and \ref{lem:weil-hodge-computation2-3}, we have
\[
 s(\alpha, \alpha^{-1}, 1) = \gamma_F(J, \tfrac{1}{2} \psi) \cdot (-2ab J_1, J)_F.
\]
Now \eqref{eq:weil-hodge-computation2} follows from this and Lemma \ref{lem:weil-hodge-computation2-1}.

\subsection{Proof of \eqref{eq:weil-hodge-computation3}}

Assume that $J_1, J_2 \in (F^\times)^2$.
Choose $t_i \in F^\times$ such that $J_i = t_i^2$ and put $\j_i^\natural = t_i^{-1} \cdot \j_i$.
In this section, we will show that
\[
 \Chi(\j_1^\natural, \j_2^\natural, 1) = 1.
\]

\begin{lem}
\label{lem:weil-hodge-computation3-1}
We have
\[
 \tilde{s}(\j_1^\natural, \j_2^\natural, 1) = 1.
\]
\end{lem}

\begin{proof}
Put $g_1 = [\j_1^\natural, 1] \in \GU(V)^0$ and $g_2 = [1, \j_2^\natural] \in \GU(V)^0$.
Then we have $\tilde{s}(\j_1^\natural, \j_2^\natural, 1) = \tilde{s}(g_1) \cdot \tilde{s}(g_2) \cdot z_\Y(g_1, g_2)$.
By \cite[Proposition C.4.2]{periods1}, we have
\[
 \tilde{s}(g_1) = \tilde{s}(g_2) = 1.
\]
It remains to compute $z_{\Y}(g_1, g_2)$.

Noting that $\nu(g_1) = \nu(g_2) = -1$, we have
\[
 z_\Y(g_1, g_2) = z_\Y(g_1 \cdot d(-1), d(-1) \cdot g_2) \cdot v(g_2 \cdot d(-1), -1).
\]
We have
\[
 g_1 = \m(\a_1) \cdot d(-1), \qquad
 g_2 = \m(\a_2) \cdot d(-1)
\]
in $\GSp(\V)$, where
\[
 \a_1 = 
 \begin{pmatrix}
  & \frac{1}{t_1} & & \\
  t_1 & & & \\
  & & & \frac{1}{t_1} \\
  & & t_1 & 
 \end{pmatrix}, \qquad
 \a_2 = 
 \begin{pmatrix}
  & & \frac{1}{t_2} & \\
  & & & \frac{1}{t_2} \\
  t_2 & & & \\
  & t_2 & &
 \end{pmatrix}.
\]
Hence we have
\[
 z_\Y(g_1 \cdot d(-1), d(-1) \cdot g_2) = 1.
\]
On the other hand, since $x(g_2 \cdot d(-1)) = 1$ and $j(g_2 \cdot d(-1)) = 0$, we have
\[
 v(g_2 \cdot d(-1), -1) = 1.
\]
Thus we obtain $z_\Y(g_1, g_2) = 1$.
This completes the proof.
\end{proof}

Now we compute $s(\j_1^\natural, \j_2^\natural, 1)$.
By definition, we have
\begin{align*}
 s(\j_1^\natural, \j_2^\natural, 1)
 & = s^\sharp([\j_1^\natural, \j_2^\natural], 1, 1, 1) \\
 & = \hat{s}^\sharp([\j_1^\natural, \j_2^\natural], 1, 1, 1) \cdot \mu(\iota([\j_1^\natural, \j_2^\natural], 1)) \\
 & = \hat{s}_1(\iota([\j_1^\natural, \j_2^\natural], 1)) \cdot \mu(\iota([\j_1^\natural, \j_2^\natural], 1)),
\end{align*}
where $[\j_1^\natural, \j_2^\natural] \in \U(V)^0$.

\begin{lem}
\label{lem:weil-hodge-computation3-2}
We have
\[
 \hat{s}_1(\iota([\j_1^\natural, \j_2^\natural], 1)) = 1.
\]
\end{lem}

\begin{proof}
Since $B$ is split, the assertion follows.
\end{proof}

\begin{lem}
\label{lem:weil-hodge-computation3-3}
We have
\[
 \mu(\iota([\j_1^\natural, \j_2^\natural], 1)) = 1.
\]
\end{lem}

\begin{proof}
Recall that
\[
 \mu(\iota([\j_1^\natural, \j_2^\natural], 1))
 = z_{\Y^\square}(\sigma_0, \sigma)^{-1} \cdot z_{\Y^\square}(\sigma_0 \sigma \sigma_0^{-1}, \sigma_0),
\]
where $\sigma$ is the image of $\iota([\j_1^\natural, \j_2^\natural], 1)$ in $\Sp(\V^\square)$.
Since
\[
 \sigma = \m(\a_1),
\]
where
\[
 \a_1 = 
 \begin{pmatrix}
  \a & \\
  & \1_4
 \end{pmatrix}, \qquad
 \a = 
 \begin{pmatrix}
  & & & \frac{1}{t_1 t_2} \\
  & & \frac{t_1}{t_2} & \\
  & \frac{t_2}{t_1} & & \\
  t_1 t_2 & & & 
 \end{pmatrix},
\]
we have $z_{\Y^\square}(\sigma_0, \sigma) = 1$.
On the other hand, we have
\[
 \sigma_0 \sigma \sigma_0^{-1} = 
 \begin{pmatrix}
  \tfrac{1}{2} (\1_4 + \a) & & & \tfrac{1}{4} (\1_4 - \a) \\
  & \tfrac{1}{2} (\1_4 + {}^t \a) & - \tfrac{1}{4} (\1_4 - {}^t \a) & \\
  & - \1_4 + {}^t \a & \tfrac{1}{2} (\1_4 + {}^t \a) & \\
  \1_4 - \a & & & \tfrac{1}{2} (\1_4 + \a)
 \end{pmatrix}.
\]
Since $\sigma_0 \sigma \sigma_0^{-1} \in P_{\Y^\square} \cdot \tau \cdot \m(\a_2)$ and $\sigma_0 \in \tau_4 \cdot P_{\Y^\square}$, where
\[
 \a_2 =
 \begin{pmatrix}
  \a' & \\
  & \a''
 \end{pmatrix}, \qquad
 \a' = 
 \begin{pmatrix}
  1 & & & \\
  & 1 & & \\
  & -\frac{t_2}{t_1} & 1 & \\
  -t_1 t_2 & & & 1
 \end{pmatrix}, \qquad
 \a'' = 
 \begin{pmatrix}
  1 & & & \\
  & 1 & & \\
  & -\frac{t_1}{t_2} & 1 & \\
  -\frac{1}{t_1 t_2} & & & 1
 \end{pmatrix},
\]
and
\[
 \tau = 
 \begin{pmatrix}
  \1_2 & & & & & & & \\
  & \0_2 & & & & -\1_2 & & \\
  & & \1_2 & & & & & \\
  & & & \0_2 & & & & -\1_2 \\
  & & & & \1_2 & & & \\
  & \1_2 & & & & \0_2 & & \\
  & & & & & & \1_2 & \\
  & & & \1_2 & & & & \0_2 
 \end{pmatrix},
\]
we have
\[
 z_{\Y^\square}(\sigma_0 \sigma \sigma_0^{-1}, \sigma_0) 
 = z_{\Y^\square}(\tau, \m(\a_2) \cdot \tau_4).
\]
Hence, since $\m(\a_2) \cdot \tau_4 \in \tau_4 \cdot P_{\Y^\square}$, we have
\[
 z_{\Y^\square}(\sigma_0 \sigma \sigma_0^{-1}, \sigma_0)
 = z_{\Y^\square}(\tau, \tau_4) = 1.
\]
This completes the proof.
\end{proof}

By Lemmas \ref{lem:weil-hodge-computation3-2} and \ref{lem:weil-hodge-computation3-3}, we have
\[
 s(\j_1^\natural, \j_2^\natural, 1) = 1.
\]
Now \eqref{eq:weil-hodge-computation3} follows from this and Lemma \ref{lem:weil-hodge-computation3-1}.

\subsection{Compatibility with \cite{hk-duke}}
\label{ss:compatibility-hk-duke}

Suppose that $F$ is local.
In this section, we compare the splitting $s$ with the standard one for symplectic-orthogonal dual pairs when $B$ is split.
In this case, we have $J \in \N_{E/F}(E^\times)$, so that we may write $J = k^2 - l^2 u$ for some $k, l \in F$.
We define an isomorphism $\ii:B \rightarrow \M_2(F)$ of quaternion $F$-algebras by
\[
 \ii(a + b \i + c \j + d \i \j) =
 \begin{pmatrix}
  a & b \\
  bu  & a
 \end{pmatrix} 
 + \begin{pmatrix}
  c & d \\
  du  & c
 \end{pmatrix}
 \begin{pmatrix}
  k & -l \\
  lu  & -k
 \end{pmatrix}.
\]
Put
\begin{align*}
 e & = \frac{1}{2} + \frac{k}{2J} \j - \frac{l}{2J} \i \j, &
 e' & = \frac{1}{2} \i + \frac{lu}{2J} \j - \frac{k}{2J} \i \j, \\
 e'' & =  \frac{1}{2u} \i - \frac{l}{2J} \j + \frac{k}{2uJ} \i \j, &
 e^* & = \frac{1}{2} - \frac{k}{2J} \j + \frac{l}{2J} \i \j,
\end{align*}
so that
\[
 \ii(e) = \mat{1}{0}{0}{0}, \qquad
 \ii(e') = \mat{0}{1}{0}{0}, \qquad
 \ii(e'') = \mat{0}{0}{1}{0}, \qquad
 \ii(e^*) = \mat{0}{0}{0}{1}.
\]
In particular, we have
\[
 \begin{bmatrix}
  e \cdot x \\
  e' \cdot x
 \end{bmatrix}
 = \ii(x) \cdot 
 \begin{bmatrix}
  e \\
  e'
 \end{bmatrix}
\]
for $x \in B$.

Let $V$ be an $m$-dimensional skew-hermitian right $B$-space as in \eqref{eq:condition-skew-herm}.
We consider the $2m$-dimensional quadratic $F$-space $V^\dagger := V e$ given by Morita theory (see \cite[\S C.2]{periods1} for details).
With respect to a basis $e_1 e, e_1 e'', \dots, e_m e, e_m e''$ of $V^\dagger$, the symmetric bilinear form on $V^\dagger$ is associated to
\[
 \frac{1}{2} \cdot 
 \begin{pmatrix}
  \kappa_1 u & & & & \\
  & - \kappa_1& & & \\
  & & \ddots & & \\
  & & & \kappa_m u & \\
  & & & & - \kappa_m
 \end{pmatrix}.
\]
Similarly, we consider the $2$-dimensional symplectic $F$-space $W^\dagger := e W$.
Then, by \cite[Lemma C.2.2]{periods1}, we have an identification
\[
 \V = V^\dagger \otimes_F W^\dagger.
\]
We take a complete polarization $W^\dagger = X \oplus Y$ defined by 
\[
 X = F e, \qquad Y = F e'.
\]
This induces a complete polarization $\V = \X' \oplus \Y'$, where
\[
 \X' = V^\dagger \otimes_F X, \qquad 
 \Y' = V^\dagger \otimes_F Y.
\]
More explicitly, we have
\begin{align*}
 \X' & = F \cdot e_1 \otimes e + \dots + F \cdot e_m \otimes e + F \cdot e_1 \otimes e'' + \dots + F \cdot e_m \otimes e'', \\
 \Y' & = F \cdot e_1 \otimes e' + \dots + F \cdot e_m \otimes e' + F \cdot e_1 \otimes e^* + \dots + F \cdot e_m \otimes e^*.
\end{align*}

Now we recall the splitting defined in \cite[\S 5.1]{hk-duke}.
Using a basis $e, e'$ of $W^\dagger$, we identify $\GSp(W^\dagger)$ with $\GSp_2(F) = \GL_2(F)$.
We define a map
\[
 s^\dagger: \Sp(W^\dagger) \longrightarrow \C^1
\]
by setting
\[
 s^\dagger(h) = \xi_E(x(h))^m \cdot \gamma'^{-j(h)},
\]
where $x(h)$ and $j(h)$ are as in \S \ref{sss:notation}, and
\[
 \gamma' = \gamma_F(\tfrac{1}{2} \psi)^{2m} \cdot \gamma_F(\det V^\dagger, \tfrac{1}{2} \psi) \cdot h_F(V^\dagger).
\]
We extend $s^\dagger$ to a map
\[
 s^\dagger: \G(\O(V^\dagger) \times \Sp(W^\dagger)) \longrightarrow \C^1
\]
so that 
\[
 s^\dagger(\g) = s^\dagger(h \cdot d(\nu(h))^{-1})
\]
for $\g = (g, h) \in \G(\O(V^\dagger) \times \Sp(W^\dagger))$.

\begin{lem}
\label{lem:spl-s-dagger}
For $\g_1, \g_2 \in \G(\O(V^\dagger) \times \Sp(W^\dagger))$, we have
\[
 z_{\Y'}(\g_1, \g_2) = \frac{s^\dagger(\g_1 \g_2)}{s^\dagger(\g_1) s^\dagger(\g_2)}.
\]
\end{lem}

\begin{proof}
If $\g_1, \g_2 \in \Sp(W^\dagger)$, then the assertion follows from \cite[Theorem 3.1, cases $1_+$]{kudla-splitting}.
By \cite[\S 5.1]{hk-duke}, this implies the general case.
Nevertheless, we include a direct argument for the convenience of the reader.

Let $\g_i = (g_i, h_i) \in \G(\O(V^\dagger) \times \Sp(W^\dagger))$.
Recall that
\[
 z_{\Y'}(\g_1, \g_2) = \gamma_F(\tfrac{1}{2} \psi \circ q(\Y' \cdot \g_1^{-1}, \Y' \cdot \g_2^{-1} \g_1^{-1}, \Y')),
\]
where $q$ denotes the Leray invariant (see e.g.~\cite{rangarao}, \cite[\S 3.1.2]{periods1}).
Put 
\[
 \nu_i = \nu(h_i), \qquad
 h_i' = h_i \cdot d(\nu_i)^{-1}, \qquad
 h_2'' = d(\nu_1) \cdot h_2' \cdot d(\nu_1)^{-1},
\]
so that 
\[
 h_1 h_2 = h_1' \cdot d(\nu_1) \cdot h_2' \cdot d(\nu_2) = h_1' h_2'' \cdot d(\nu_1 \nu_2).
\]
Since $\Y' \cdot (g, d(\nu)) = \Y'$ for $g \in \GO(V^\dagger)$ and $\nu \in F^{\times}$, we have $\Y' \cdot \g_1^{-1} = \Y' \cdot h_1'^{-1}$ and $\Y' \cdot \g_2^{-1} \g_1^{-1} = \Y' \cdot h_2''^{-1} h_1'^{-1}$, so that
\[
 q(\Y' \cdot \g_1^{-1}, \Y' \cdot \g_2^{-1} \g_1^{-1}, \Y')
 = q(\Y' \cdot h_1'^{-1}, \Y' \cdot h_2''^{-1} h_1'^{-1}, \Y').
\]
Hence we have
\[
 z_{\Y'}(\g_1, \g_2) = z_{\Y'}(h_1', h_2'') = \frac{s^\dagger(h'_1 h''_2)}{s^\dagger(h'_1) s^\dagger(h''_2)}.
\]
On the other hand, by definition, we have $s^\dagger(\g_i) = s^\dagger(h'_i)$ and $s^\dagger(\g_1 \g_2) = s^\dagger(h_1' h_2'')$.
By \eqref{eq:xj-conj}, and noting that $\nu(\GO(V^\dagger)) = \N_{E/F}(E^\times)$ if $m$ is odd, we have
\[
 s^\dagger(h_2'') = \xi_E(\nu_1)^{j(h_2') m} \cdot s^\dagger(h_2') = s^\dagger(h_2').
\]
This completes the proof.
\end{proof}

Fix $\varsigma_0 \in \Sp(\V)$ such that $\X = \X' \cdot \varsigma_0$ and $\Y = \Y' \cdot \varsigma_0$.
Put
\[
 \mu_0(\sigma) = z_{\Y'}(\varsigma_0, \sigma) \cdot z_{\Y'}(\varsigma_0 \sigma \varsigma_0^{-1}, \varsigma_0)^{-1}
\]
for $\sigma \in \Sp(\V)$.
Note that $\mu_0$ does not depend on the choice of $\varsigma_0$.
Then, by \cite[Lemma 4.2]{kudla-splitting}, we have
\[
 z_\Y(\sigma, \sigma') = z_{\Y'}(\sigma, \sigma') 
 \cdot \frac{\mu_0(\sigma \sigma')}{\mu_0(\sigma) \mu_0(\sigma')}
\]
for $\sigma, \sigma' \in \Sp(\V)$.

Put $s_0 = s^\dagger \cdot \mu_0$.
Via the canonical isomorphisms $\GU(V) \simeq \GO(V^\dagger)$ and $\GU(W) \simeq \GSp(W^\dagger)$, we regard $s_0$ as a map 
\[
 s_0 : \G(\U(V) \times \U(W)) \longrightarrow \C^1.
\]
By Lemma \ref{lem:spl-s-dagger}, we have
\[
 z_\Y(\g_1, \g_2) = \frac{s_0(\g_1 \g_2)}{s_0(\g_1) s_0(\g_2)}
\]
for $\g_1, \g_2 \in \G(\U(V) \times \U(W))$.

\begin{prop}
\label{prop:compare-s-0-s}
We have
\[
 s_0|_\cG = s.
\]
\end{prop}

The rest of this section is devoted to the proof of Proposition \ref{prop:compare-s-0-s}.

As in \S \ref{ss:doubling-U(V)}, we define the doubled space $W^{\dagger \square} = W^\dagger \oplus W^\dagger$ and take the complete polarization $W^{\dagger \square} = W^{\dagger \bigtriangledown} \oplus W^{\dagger \triangle}$.
Then we have identifications
\[
 \V^\square = V^\dagger \otimes_F W^{\dagger \square}, \qquad 
 \V^\bigtriangledown = V^\dagger \otimes_F W^{\dagger \bigtriangledown}, \qquad
 \V^\triangle = V^\dagger \otimes_F W^{\dagger \triangle}.
\]
We also take complete polarizations $W^{\dagger \square} = X^\square \oplus Y^\square$ and $\V^\square = \X'^\square \oplus \Y'^\square$, where
\begin{align*}
 X^\square & = X \oplus X, &
 \X'^\square & = \X' \oplus \X' = V^\dagger \otimes_F X^\square, \\
 Y^\square & = Y \oplus Y, &
 \Y'^\square & = \Y' \oplus \Y' = V^\dagger \otimes_F Y^\square.
\end{align*}
As in \cite[\S D.3]{periods1}, we have
\[
 z_{\Y'^\square}(\iota(\sigma_1,\sigma_2), \iota(\sigma_1',\sigma_2'))
 = z_{\Y'}(\sigma_1, \sigma_1') \cdot z_{\Y'}(\sigma_2, \sigma_2')^{-1}
\]
for $\sigma_i, \sigma_i' \in \Sp(\V)$.
Using a basis $(e,0), (0,e), (e',0), (0,-e')$ of $W^{\dagger \square}$, we identify $\GSp(W^{\dagger \square})$ with $\GSp_4(F)$.
Put
\[
 h_0 = 
 \begin{pmatrix}
  \frac{1}{2} & -\frac{1}{2} & & \\
  & & \frac{1}{2} & \frac{1}{2} \\
  & & 1 & -1 \\
  -1 & -1 & &
 \end{pmatrix}
 \in \Sp(W^{\dagger \square}).
\]
Then we have
\[
 \begin{bmatrix}
  \frac{1}{2}(e, -e) \\
  \frac{1}{2}(e', -e') \\
  (e', e') \\
  (-e, -e)
 \end{bmatrix}
 = h_0 \cdot 
 \begin{bmatrix}
  (e, 0) \\
  (0, e) \\
  (e', 0) \\
  (0, -e')
 \end{bmatrix},
\]
so that $W^{\dagger \bigtriangledown} = X^\square \cdot h_0$ and $W^{\dagger \triangle} = Y^\square \cdot h_0$.
Put $\h_0 = \id \otimes h_0 \in \Sp(\V^\square)$ and
\[
 \mu'(\sigma) = z_{\Y'^\square}(\h_0, \sigma)^{-1} \cdot z_{\Y'^\square}(\h_0 \sigma \h_0^{-1}, \h_0)
\]
for $\sigma \in \Sp(\V^\square)$.
Since $\V^\bigtriangledown = \X'^\square \cdot \h_0$ and $\V^\triangle = \Y'^\square \cdot \h_0$, we have
\begin{equation}
\label{eq:compare-zY'-zV}
 z_{\Y'^\square}(\sigma, \sigma') = z_{\V^\triangle}(\sigma, \sigma') 
 \cdot \frac{\mu'(\sigma \sigma')}{\mu'(\sigma) \mu'(\sigma')}
\end{equation}
for $\sigma, \sigma' \in \Sp(\V^\square)$.

As in \S \ref{ss:proof-spl-hodge}, we put $s^{\sharp \prime} = \hat{s}^\sharp \cdot \mu'$ and $s_2' = \hat{s}_2 \cdot \mu'$, and define a map
\[
 s': \cG \longrightarrow \C^1
\]
by setting
\[
 s'(g,h) = \frac{s^{\sharp \prime}(g,h,\alpha,\alpha)}{s'_2(\iota(1, [\alpha,\alpha]))},
\]
where we choose $\alpha \in E^\times$ such that $\nu(g) = \nu(h) = \N_{E/F}(\alpha)$.
As in Lemma \ref{lem:s-well-def}, $s'$ is well-defined.
Moreover, as in Lemma \ref{lem:s-proof(i)}, we have
\[
 z_{\Y'}(\g_1,\g_2) = \frac{s'(\g_1\g_2)}{s'(\g_1) s'(\g_2)}
\]
for $\g_1, \g_2 \in \cG$.

\begin{lem}
We have
\[
 s = s' \cdot \mu_0.
\]
\end{lem}

\begin{proof}
Put $\varsigma_{00} = \iota(\varsigma_0, \varsigma_0) \in \Sp(\V^\square)$ and
\[
 \mu_{00}(\sigma) = z_{\Y'^\square}(\varsigma_{00}, \sigma) \cdot z_{\Y'^\square}(\varsigma_{00} \sigma \varsigma_{00}^{-1}, \varsigma_{00})^{-1}
\]
for $\sigma \in \Sp(\V^\square)$.
Since $\X^\square = \X'^\square \cdot \varsigma_{00}$ and $\Y^\square = \Y'^\square \cdot \varsigma_{00}$, we have
\begin{equation}
\label{eq:compare-zY-zY'}
 z_{\Y^\square}(\sigma, \sigma') = z_{\Y'^\square}(\sigma, \sigma') 
 \cdot \frac{\mu_{00}(\sigma \sigma')}{\mu_{00}(\sigma) \mu_{00}(\sigma')}
\end{equation}
for $\sigma, \sigma' \in \Sp(\V^\square)$.
Then it follows from \eqref{eq:compare-zY-zV}, \eqref{eq:compare-zY'-zV}, and \eqref{eq:compare-zY-zY'} that 
\[
 \frac{\mu_{00}(\sigma \sigma')}{\mu_{00}(\sigma) \mu_{00}(\sigma')}
 = \frac{\mu(\sigma \sigma')}{\mu(\sigma) \mu(\sigma')}
 \cdot \frac{\mu'(\sigma) \mu'(\sigma')}{\mu'(\sigma \sigma')}
\]
for $\sigma, \sigma' \in \Sp(\V^\square)$.
Namely, $\mu_{00} \cdot \mu' / \mu$ is a character of $\Sp(\V^\square)$.
Since $[\Sp(\V^\square), \Sp(\V^\square)] = \Sp(\V^\square)$, this character must be trivial and hence
\[
 \mu_{00} = \mu/\mu'.
\]
Since
\begin{align*}
 \mu_{00}(\iota(\sigma,1))
 & = z_{\Y'^\square}(\iota(\varsigma_0, \varsigma_0), \iota(\sigma, 1)) \cdot z_{\Y'^\square}(\iota(\varsigma_0 \sigma \varsigma_0^{-1}, 1), \iota(\varsigma_0, \varsigma_0))^{-1} \\
 & = z_{\Y'}(\varsigma_0, \sigma) \cdot z_{\Y'}(\varsigma_0 \sigma \varsigma_0^{-1}, \varsigma_0)^{-1} \\
 & = \mu_0(\sigma)
\end{align*}
for $\sigma \in \Sp(\V)$, it suffices to show that 
\[
 \frac{s(\g)}{\mu(\iota(\g, 1))} = \frac{s'(\g)}{\mu'(\iota(\g, 1))}
\]
for $\g \in \cG$.
Here, by abuse of notation, we write $\g$ in the denominator for the image of $\g$ in $\Sp(\V)$ under \eqref{eq:hom-GU(V)-GU(W)}, so that $\iota(\g, 1) \in \Sp(\V^\square)$.

For $\g = (g, h) \in \cG$, choose $\alpha \in E^\times$ such that $\nu(g) = \nu(h) = \N_{E/F}(\alpha)$.
Put $\g^\sharp = (g,h,\alpha,\alpha) \in \cG^\sharp$ and $\ba = [\alpha, \alpha] \in \U(\WW)$.
Note that the image of $\g^\sharp$ in $\Sp(\V^\square)$ agrees with $\iota(\g, 1) \cdot \iota(1,\ba)$.
By definition, we have
\[
 s(\g) = \frac{\hat{s}^\sharp(\g^\sharp)}{\hat{s}_2(\iota(1, \ba))} \cdot \frac{\mu(\g^\sharp)}{\mu(\iota(1, \ba))}, \qquad
 s'(\g) = \frac{\hat{s}^\sharp(\g^\sharp)}{\hat{s}_2(\iota(1, \ba))} \cdot \frac{\mu'(\g^\sharp)}{\mu'(\iota(1, \ba))}.
\]
Thus it remains to show that 
\[
 \frac{\mu(\g^\sharp)}{\mu(\iota(\g,1)) \cdot \mu(\iota(1, \ba))}
 = \frac{\mu'(\g^\sharp)}{\mu'(\iota(\g,1)) \cdot \mu'(\iota(1, \ba))}.
\]
But the left-hand side is equal to
\[
 \frac{z_{\Y^\square}(\iota(\g, 1), \iota(1, \ba))}{z_{\V^\triangle}(\iota(\g, 1), \iota(1, \ba))} = \frac{1}{z_{\V^\triangle}(\iota(\g, 1), \iota(1, \ba))}
\]
by \eqref{eq:compare-zY-zV}, whereas the right-hand side is equal to
\[
 \frac{z_{\Y'^\square}(\iota(\g, 1), \iota(1, \ba))}{z_{\V^\triangle}(\iota(\g, 1), \iota(1, \ba))} = \frac{1}{z_{\V^\triangle}(\iota(\g, 1), \iota(1, \ba))}
\]
by \eqref{eq:compare-zY'-zV}.
This completes the proof. 
\end{proof}

Thus, to finish the proof of Proposition \ref{prop:compare-s-0-s}, it remains to prove the following.

\begin{lem}
We have
\[
 s^\dagger|_\cG = s'.
\]
\end{lem}

\begin{proof}
Since both $s^\dagger$ and $s'$ trivialize $z_{\Y'}$, there exists a \emph{continuous} character $\Chi'$ of $\cG$ such that
\[
 s^\dagger|_\cG = s' \cdot \Chi'.
\]
We will show that $\Chi'$ is trivial.
Since $[\Sp(W^\dagger), \Sp(W^\dagger)] = \Sp(W^\dagger)$, $\Chi'$ is trivial on $\U(W) \simeq \Sp(W^\dagger)$.
Let $\g = (g,1) \in \cG$ with $g \in \U(V)^0 \simeq \SO(V^\dagger)$.
By definition, we have $s^\dagger(\g) = 1$ and 
\[
 s'(\g) = s^{\sharp \prime}(g,1,1,1) = \hat{s}^\sharp(g,1,1,1) \cdot \mu'(\iota(\g, 1)) = \mu'(\iota(\g, 1)).
\]
Since $\iota(\g, 1)$ belongs to $P_{\Y'^\square}$ and commutes with $\h_0$, we have
\[
 \mu'(\iota(\g, 1)) = z_{\Y'^\square}(\h_0, \iota(\g,1))^{-1} \cdot z_{\Y'^\square}(\iota(\g, 1), \h_0) = 1.
\]
Hence we have $s'(\g) = 1$, so that $\Chi'(\g) = 1$.

Thus it remains to show that
\[
 \Chi'(\g) = 1
\]
for $\g = (\alpha,\alpha) \in \cG$ with $\alpha \in E^\times$.
We write $\alpha = a + b \i$ with $a, b \in F$ and put $\nu = a^2 - b^2 u$.
Since $\Chi'$ is continuous, we may assume that
\[
 a \ne 0, \qquad b \ne 0.
\]
By definition, we have $s^\dagger(\g) = s^\dagger(h)$, where
\[
 h = 
 \begin{pmatrix}
  a & b \\
  bu & a
 \end{pmatrix}
 \cdot 
 \begin{pmatrix}
  1 & \\
  & \nu^{-1}
 \end{pmatrix}.
\]
Since 
\[
 h = 
 \begin{pmatrix}
  \frac{1}{bu} & a \\
  & bu
 \end{pmatrix}
 \cdot 
 \begin{pmatrix}
  & -1 \\
 1 & 
 \end{pmatrix}
 \cdot 
 \begin{pmatrix}
  1 & \frac{a}{\nu bu} \\
  & 1
 \end{pmatrix}, 
\]
we have
\[
 s^\dagger(h) = \xi_E(bu)^m \cdot \gamma'^{-1} = (-b, u)_F^m \cdot \gamma'^{-1}.
\]
Recall that 
\begin{align*}
 \gamma'^{-1} & = \gamma_F(\tfrac{1}{2} \psi)^{-2m} \cdot \gamma_F(\det V^\dagger, \tfrac{1}{2} \psi)^{-1} \cdot h_F(V^\dagger) \\
 & = \gamma_F(-1, \tfrac{1}{2} \psi)^m \cdot \gamma_F((-u)^m, \tfrac{1}{2} \psi)^{-1} \cdot h_F(V^\dagger) \\
 & = \gamma_F(-1, \tfrac{1}{2} \psi)^m \cdot \gamma_F(-u, \tfrac{1}{2} \psi)^{-m} \cdot (-u,-u)_F^{(m-1)m/2} \cdot h_F(V^\dagger) \\
 & = \gamma_F(u, \tfrac{1}{2} \psi)^m \cdot (-u,-u)_F^{(m-1)m/2} \cdot h_F(V^\dagger).
\end{align*}
If we put $V_i^\dagger = F e_i e + F e_i e''$, then we have
\begin{align*}
 h_F(V^\dagger) & = \prod_{i=1}^m h_F(V_i^\dagger) \cdot \prod_{1 \le i < j \le m} (\det V_i^\dagger, \det V_j^\dagger)_F \\
 & = \prod_{i=1}^m (- 2\kappa_i, u)_F \cdot (-u, -u)_F^{(m-1)m/2}.
\end{align*}
Hence we have
\[
 s^\dagger(\g) = \gamma_F(u, \tfrac{1}{2} \psi)^m \cdot \prod_{i=1}^m (2 b \kappa_i, u)_F.
\]
On the other hand, by definition, we have
\begin{align*}
 s'(\g) & = \frac{s^{\sharp \prime}(\alpha,\alpha,\alpha,\alpha)}{s_2'(\iota(1, \ba))} \\
 & = \frac{s^{\sharp \prime}(\alpha,\alpha,\alpha,\alpha)}{s_2'(\iota(\ba, \ba))} \cdot \frac{s_2'(\iota(\ba, \ba))}{s_2'(\iota(1, \ba))} \\
 & = \frac{s^{\sharp \prime}(\alpha,\alpha,\alpha,\alpha)}{s_2'(\iota(\ba, \ba))} \cdot s_2'(\iota(\ba, 1)) \cdot z_{\Y'^\square}(\iota(\ba, 1), \iota(1, \ba)) \\
 & = \frac{s^{\sharp \prime}(\alpha,\alpha,\alpha,\alpha)}{s_2'(\iota(\ba, \ba))} \cdot s_2'(\iota(\ba, 1)) \\
 & = \frac{\hat{s}^\sharp(\alpha,\alpha,\alpha,\alpha)}{\hat{s}_2(\iota(\ba, \ba))} \cdot \hat{s}_2(\iota(\ba, 1)) \cdot \mu'(\iota(\ba, 1)),
\end{align*}
where $\ba = [\alpha, \alpha] \in \U(\WW)$.
By definition and Lemma \ref{lem:s-hat-2-E-diag}, we have
\[
 \frac{\hat{s}^\sharp(\alpha,\alpha,\alpha,\alpha)}{\hat{s}_2(\iota(\ba, \ba))} = \chi(\alpha)^m.
\]
Also, as in Lemma \ref{lem:weil-hodge-computation1-2}, we have
\[
 \hat{s}_2(\iota(\ba, 1)) = \chi(J \i (\beta - 1))^m \cdot \gamma^{-1} = \chi(2 b \alpha^{-1})^m \cdot \gamma^{-1}, 
\]
where $\beta = \alpha^{-1} \alpha^\rho$ and
\begin{align*}
 \gamma^{-1} & = (\det \VV, u)_F \cdot \gamma_F(-u, \tfrac{1}{2} \psi)^{-m} \cdot \gamma_F(-1, \tfrac{1}{2} \psi)^m \\
 & = \gamma(u, \tfrac{1}{2} \psi)^m \cdot \prod_{i=1}^m (\kappa_i, u)_F.
\end{align*}
Hence, noting that the image of $\g$ in $\Sp(\V)$ agrees with that of $\ba$, we have
\[
 s'(\g) = \gamma(u, \tfrac{1}{2} \psi)^m \cdot \prod_{i=1}^m (2b \kappa_i, u)_F \cdot \mu'(\iota(\g, 1)).
\]

Thus we are reduced to showing that
\[
 \mu'(\iota(\g, 1)) = 1
\]
for $\g = (\alpha, \alpha) \in \cG$ with $\alpha = a + b \i \in E^\times$ such that $a \ne 0$ and $b \ne 0$.
This is further reduced to the case $\dim V = 1$.
Then we may identify $V^\dagger$ with the $F$-space $Fe + Fe''$ equipped with a symmetric bilinear form
\[
 \langle x_1 e + x_2 e'', y_1 e + y_2 e'' \rangle^\dagger
 = \kappa u \cdot x_1 y_1 - \kappa \cdot x_2 y_2, 
\]
where $\kappa = \kappa_1/2$.
We take a basis
\begin{align*}
 \x_1 & = e \otimes (e, 0), & 
 \y_1 & = \frac{1}{\kappa u} \cdot e \otimes (e', 0), \\
 \x_2 & = e'' \otimes (e, 0), & 
 \y_2 & = - \frac{1}{\kappa} \cdot e'' \otimes (e', 0), \\
 \x_3 & = e \otimes (0, e), & 
 \y_3 & = \frac{1}{\kappa u} \cdot e \otimes (0, -e'), \\
 \x_4 & = e'' \otimes (0, e), &
 \y_4 & = - \frac{1}{\kappa} \cdot e'' \otimes (0, -e')
\end{align*}
of $\V^\square = V^\dagger \otimes_F W^{\dagger \square}$, so that 
\[
 \X'^\square = F \x_1 + \dots + F \x_4, \qquad
 \Y'^\square = F \y_1 + \dots + F \y_4, \qquad
 \llangle \x_i, \y_j \rrangle = \delta_{ij}.
\]
Using this basis, we identify $\Sp(\V^\square)$ with $\Sp_8(F)$.
Then we have
\begin{align*}
 \h_0 & = 
 \begin{pmatrix}
  \frac{1}{2} & & -\frac{1}{2} & & & & & \\
  & \frac{1}{2} & & -\frac{1}{2} & & & & \\
  & & & & \frac{\kappa u}{2} & & \frac{\kappa u}{2} & \\
  & & & & & -\frac{\kappa}{2} & & -\frac{\kappa}{2} \\
  & & & & 1 & & -1 & \\
  & & & & & 1 & & -1 \\
  -\frac{1}{\kappa u} & & -\frac{1}{\kappa u} & & & & & \\
  & \frac{1}{\kappa} & & \frac{1}{\kappa} & & & &
 \end{pmatrix}, \\
 \iota(\g, 1) & =
 \begin{pmatrix}
  \frac{1}{2}(c+1) & \frac{du}{2} & & & -\frac{d\kappa u}{2} & \frac{\kappa}{2}(c-1) & & \\
  \frac{d}{2} & \frac{1}{2}(c+1) & & & - \frac{\kappa}{2}(c-1) & \frac{d\kappa}{2} & & \\
  & & 1 & & & & & \\
  & & & 1 & & & & \\
  -\frac{d}{2\kappa} & -\frac{1}{2\kappa}(c-1) & & & \frac{1}{2}(c+1) & -\frac{d}{2} & & \\
  \frac{1}{2\kappa}(c-1) & \frac{du}{2\kappa} & & & -\frac{du}{2} & \frac{1}{2}(c+1) & & \\
  & & & & & & 1 & \\
  & & & & & & & 1
 \end{pmatrix},
\end{align*}
where
\[
 c = \frac{a^2 + b^2 u}{a^2 - b^2 u} \ne \pm 1, \qquad 
 d = - \frac{2ab}{a^2 - b^2 u} \ne 0,
\]
so that $c^2 - d^2 u = 1$.
Recall that 
\[
 \mu'(\iota(\g, 1)) = z_{\Y'^\square}(\h_0, \iota(\g, 1))^{-1} \cdot z_{\Y'^\square}(\h_0 \cdot \iota(\g, 1) \cdot \h_0^{-1}, \h_0).
\]
Since $\h_0 \in P_{\Y'^\square} \cdot \tau_2 \cdot \m(\a_1)$ and $\iota(\g, 1) = \n(\b_1) \cdot \tau \cdot P_{\Y'^\square}$, where
\[
 \a_1 = 
 \begin{pmatrix}
  1 & & & \\
  & 1 & & \\
  1 & & 1 & \\
  & 1 & & 1 
 \end{pmatrix}, \qquad
 \b_1 = \frac{d \kappa}{c-1} \cdot 
 \begin{pmatrix}
  -u & & & \\
  & 1 & & \\
  & & 0 & \\
  & & & 0
 \end{pmatrix}, \qquad
 \tau = 
 \begin{pmatrix}
  & & -\1_2 & \\
  & \1_2 & & \\
  \1_2 & & & \\
  & & & \1_2
 \end{pmatrix},
\]
we have
\begin{align*}
 z_{\Y'^\square}(\h_0, \iota(\g, 1)) 
 & = z_{\Y'^\square}(\tau_2 \cdot \m(\a_1), \n(\b_1) \cdot \tau) \\
 & = z_{\Y'^\square}(\tau_2, \m(\a_1) \cdot \n(\b_1) \cdot \tau).
\end{align*}
If we put
\[
 \b_2 = \frac{d \kappa}{c-1} \cdot 
 \begin{pmatrix}
  -u & & -u & \\
  & 1 & & 1 \\
  -u & & & \\
  & 1 & & 
 \end{pmatrix}, \qquad
 \b_3 = \frac{d \kappa}{c-1} \cdot 
 \begin{pmatrix}
  0 & & & \\
  & 0 & & \\
  & & -u & \\
  & & & 1
 \end{pmatrix},
\]
then we have $\m(\a_1) \cdot \n(\b_1) = \n(\b_2) \cdot \n(\b_3) \cdot \m(\a_1)$ and hence
\[
 z_{\Y'^\square}(\h_0, \iota(\g, 1)) = z_{\Y'^\square}(\tau_2 \cdot \n(\b_2), \n(\b_3) \cdot \m(\a_1) \cdot \tau).
\]
Since $\tau_2 \cdot \n(\b_2) \in P_{\Y'^\square} \cdot \tau_2$ and $\n(\b_3) \cdot \m(\a_1) \cdot \tau \in \tau \cdot P_{\Y'^\square}$, we have
\[
 z_{\Y'^\square}(\h_0, \iota(\g, 1)) = z_{\Y'^\square}(\tau_2, \tau) = 1.
\]
On the other hand, we have $\h_0 \in \tau_2 \cdot P_{\Y'^\square}$ and
\begin{align*}
 & \h_0 \cdot \iota(\g, 1) \cdot \h_0^{-1} \\
 & =
\begin{pmatrix}
 \frac{1}{4}(c+3) & \frac{du}{4} & -\frac{d}{4} & -\frac{1}{4}(c-1) & -\frac{d \kappa u}{8} & \frac{\kappa}{8}(c-1) & -\frac{\kappa u}{8}(c-1) & \frac{d \kappa u}{8} \\
 \frac{d}{4} & \frac{1}{4}(c+3) & -\frac{1}{4u}(c-1) & -\frac{d}{4} & -\frac{\kappa}{8}(c-1) & \frac{d \kappa}{8} & -\frac{d \kappa u}{8} & \frac{\kappa}{8}(c-1) \\
 -\frac{du}{4} & -\frac{u}{4}(c-1) & \frac{1}{4}(c+3) & \frac{du}{4} & \frac{\kappa u}{8}(c-1) & -\frac{d \kappa u}{8} & \frac{d \kappa u^2}{8} & -\frac{\kappa u}{8}(c-1) \\
 -\frac{1}{4}(c-1) & -\frac{du}{4} & \frac{d}{4} & \frac{1}{4}(c+3) & \frac{d \kappa u}{8} & -\frac{\kappa}{8}(c-1) & \frac{\kappa u}{8}(c-1) & -\frac{d \kappa u}{8} \\
 -\frac{d}{2 \kappa} & -\frac{1}{2 \kappa}(c-1) & \frac{1}{2 \kappa u}(c-1) & \frac{d}{2 \kappa} & \frac{1}{4}(c+3) & -\frac{d}{4} & \frac{du}{4} & -\frac{1}{4}(c-1) \\
 \frac{1}{2 \kappa}(c-1) & \frac{du}{2 \kappa} & -\frac{d}{2 \kappa} & -\frac{1}{2 \kappa}(c-1) & -\frac{du}{4} & \frac{1}{4}(c+3) & -\frac{u}{4}(c-1) & \frac{du}{4} \\
 -\frac{1}{2 \kappa u}(c-1) & -\frac{d}{2 \kappa} & \frac{d}{2 \kappa u} & \frac{1}{2 \kappa u}(c-1) & \frac{d}{4} & -\frac{1}{4 u}(c-1) & \frac{1}{4}(c+3) & -\frac{d}{4} \\
 \frac{d}{2 \kappa} & \frac{1}{2 \kappa}(c-1) & -\frac{1}{2 \kappa u}(c-1) & -\frac{d}{2 \kappa} & -\frac{1}{4}(c-1) & \frac{d}{4} & -\frac{du}{4} & \frac{1}{4}(c+3)
\end{pmatrix} \\
 & \in P_{\Y'^\square} \cdot \tau \cdot \m(\a_2) \cdot \n(\b_4),
\end{align*}
where 
\[
 \a_2 = 
 \begin{pmatrix}
  1 & & & -1 \\
  & 1 & -\frac{1}{u} & \\
  & & 1 & \\
  & & & 1
 \end{pmatrix}, \qquad
 \b_4 = \frac{(c+1)\kappa}{d} \cdot 
 \begin{pmatrix}
  0 & & & \\
  & 0 & & \\
  & & u & \\
  & & & -1
 \end{pmatrix}.
\]
Since $\tau \cdot \n(\b_4) \in P_{\Y'^\square} \cdot \tau$ and $\n(\b_4)^{-1} \cdot \m(\a_2) \cdot \n(\b_4) \cdot \tau_2 \in \tau_2 \cdot P_{\Y'^\square}$, we have
\begin{align*}
 z_{\Y'^\square}(\h_0 \cdot \iota(\g, 1) \cdot \h_0^{-1}, \h_0)
 & = z_{\Y'^\square}(\tau \cdot \m(\a_2) \cdot \n(\b_4), \tau_2) \\
 & = z_{\Y'^\square}(\tau \cdot \n(\b_4), \n(\b_4)^{-1} \cdot \m(\a_2) \cdot \n(\b_4) \cdot \tau_2) \\
 & = z_{\Y'^\square}(\tau, \tau_2) \\
 & = 1.
\end{align*}
This completes the proof. 
\end{proof}


\begin{thebibliography}{99}

\bibitem{adams-theta-R}
Adams,~Jeffrey.
\emph{The theta correspondence over $\Bbb R$.}
Harmonic analysis, group representations, automorphic forms and invariant theory, 1--39, Lect. Notes Ser. Inst. Math. Sci. Natl. Univ. Singap., 12, World Sci. Publ., Hackensack, NJ, 2007.

\bibitem{aj}
Adams,~Jeffrey; Johnson,~Joseph F.
\emph{Endoscopic groups and packets of nontempered representations.}
Compositio Math. {\bf 64} (1987), no. 3, 271--309.

\bibitem{amr}
Arancibia,~Nicol{\'a}s; M{\oe}glin,~Colette; Renard,~David.
\emph{Paquets d'Arthur des groupes classiques et unitaires.}
Ann. Fac. Sci. Toulouse Math. (6) \textbf{27} (2018), no. 5, 1023--1105.

\bibitem{arthur}
Arthur,~James.
\emph{The endoscopic classification of representations. Orthogonal and symplectic groups.}
American Mathematical Society Colloquium Publications, 61. American Mathematical Society, Providence, RI, 2013.

\bibitem{bmm-unit}
Bergeron,~Nicolas; Millson,~John; M{\oe}glin,~Colette.
\emph{The Hodge conjecture and arithmetic quotients of complex balls. }
Acta Math.  \textbf{216} (2016), no. 1, 1--125. 

\bibitem{bmm-orth}
Bergeron,~Nicolas; Millson,~John; M{\oe}glin,~Colette.
\emph{Hodge type theorems for arithmetic manifolds associated to orthogonal groups.}
Int. Math. Res. Not. IMRN \textbf{2017}, no. 15, 4495--4624.

\bibitem{blasius-rogawski}
Blasius,~Don; Rogawski,~Jonathan D.
\emph{Zeta functions of Shimura varieties.}
Motives (Seattle, WA, 1991), 525--571, Proc. Sympos. Pure Math., 55, Part 2, Amer. Math. Soc., Providence, RI, 1994.

\bibitem{bloch-kato}
Bloch,~Spencer; Kato,~Kazuya.
\emph{$L$-functions and Tamagawa numbers of motives.}
The Grothendieck Festschrift, Vol. I, 333--400, Progr. Math., 86, Birkh{\"a}user Boston, Boston, MA, 1990.

\bibitem{borel-corvallis}
Borel,~A. 
\emph{Automorphic $L$-functions.}
Automorphic forms, representations and $L$-functions (Proc. Sympos. Pure Math., Oregon State Univ., Corvallis, Ore., 1977), Part 2, pp. 27--61, Proc. Sympos. Pure Math., XXXIII, Amer. Math. Soc., Providence, R.I., 1979. 

\bibitem{borel-wallach}
Borel,~A.; Wallach,~N.
\emph{Continuous cohomology, discrete subgroups, and representations of reductive groups.}
Second edition. Mathematical Surveys and Monographs, 67. American Mathematical Society, Providence, RI, 2000.



\bibitem{brylinski-labesse}
Brylinski, J.-L.; Labesse, J.-P.
\emph{Cohomologie d'intersection et fonctions L de certaines vari{\'e}t{\'e}s de Shimura.} (French) [Intersection cohomology and L-functions of some Shimura varieties] 
Ann. Sci. \'{E}cole Norm. Sup. (4) 17 (1984), no. 3, 361--412. 

\bibitem{buzzard-gee}
Buzzard,~Kevin; Gee,~Toby.
\emph{The conjectural connections between automorphic representations and Galois representations. }
Automorphic forms and Galois representations. Vol. 1, 135--187, London Math. Soc. Lecture Note Ser., 414, Cambridge Univ. Press, Cambridge, 2014. 
 
 \bibitem{carayol-hmf}
 Carayol,~Henri.
 \emph{Sur les repr\'{e}sentations $\ell$-adiques associ\'{e}es aux formes modulaires de Hilbert. (French) [On $\ell$-adic representations associated with Hilbert modular forms] }
 Ann. Sci. \'{E}cole Norm. Sup. \textbf{(4)} 19 (1986), no. 3, 409--468.

\bibitem{deligne-shimura}
Deligne,~Pierre.
\textit{Vari\'{e}t\'{e}s de Shimura: interpr\'{e}tation modulaire, et techniques de construction de mod\`{e}les canoniques. (French) Automorphic forms, representations and L-functions}
(Proc. Sympos. Pure Math., Oregon State Univ., Corvallis, Ore., 1977), Part 2, pp. 247--289, Proc. Sympos. Pure Math., XXXIII, Amer. Math. Soc., Providence, R.I., 1979. 

\bibitem{dmos}
Deligne,~Pierre; Milne,~James S.; Ogus,~Arthur; Shih,~Kuang-yen.
 \emph{Hodge cycles, motives, and Shimura varieties.} Lecture Notes in Mathematics, 900. Springer-Verlag, Berlin-New York, 1982. 

\bibitem{faltings-bgg}
Faltings,~Gerd.
\emph{On the cohomology of locally symmetric Hermitian spaces.}
 Paul Dubreil and Marie-Paule Malliavin algebra seminar, 35th year (Paris, 1982), 55--98, Lecture Notes in Math., 1029, Springer, Berlin, 1983.

\bibitem{faltings-mordell}
Faltings,~Gerd.
\emph{Endlichkeitss\"{a}tze f\"{u}r abelsche Variet\"{a}ten \"{u}ber Zahlk\"{o}rpern. (German) [Finiteness theorems for abelian varieties over number fields] }
Invent. Math. {\bf 73} (1983), no. 3, 349--366. 


\bibitem{fulton-harris}
Fulton,~William; Harris,~Joe.
\emph{Representation theory. A first course.}
Graduate Texts in Mathematics \textbf{129}.
Readings in Mathematics. Springer-Verlag, New York, 1991. xvi+551 pp.



\bibitem{funke-millson}
Funke,~Jens; Millson,~John.
\emph{Cycles with local coefficients for orthogonal groups and vector-valued Siegel modular forms.}
Amer. J. Math. \textbf{128} (2006), no. 4, 899--948.

\bibitem{ggp}
Gan,~Wee Teck; Gross,~Benedict H.; Prasad,~Dipendra.
\emph{Symplectic local root numbers, central critical $L$ values, and restriction problems in the representation theory of classical groups.}
Sur les conjectures de Gross et Prasad. I.
Ast{\'e}risque No. 346 (2012), 1--109.

\bibitem{gqt}
Gan,~Wee Teck; Qiu,~Yannan; Takeda,~Shuichiro.
\emph{The regularized Siegel-Weil formula (the second term identity) and the Rallis inner product formula.}
Invent. Math. \textbf{198} (2014), no. 3, 739--831.


\bibitem{gan-sun}
Gan,~Wee Teck; Sun,~Binyong.
\emph{The Howe duality conjecture: quaternionic case.}
Representation theory, number theory, and invariant theory, 175--192, Progr. Math., 323, Birkh{\"a}user/Springer, Cham, 2017.

\bibitem{gan-takeda}
Gan,~Wee Teck; Takeda,~Shuichiro.
\emph{A proof of the Howe duality conjecture.}
J. Amer. Math. Soc. \textbf{29} (2016), no. 2, 473--493.


\bibitem{gan-tantono}
Gan,~Wee Teck; Tantono,~Welly.
\emph{The local Langlands conjecture for $\rm GSp(4)$, II: The case of inner forms.}
Amer. J. Math. \textbf{136} (2014), no. 3, 761--805.

\bibitem{harris-motivesvolume}
Harris,~Michael.
\emph{Hodge-de Rham structures and periods of automorphic forms.}
 Motives (Seattle, WA, 1991), 573--624, Proc. Sympos. Pure Math., \textbf{55}, Part 2, Amer. Math. Soc., Providence, RI, 1994. 

\bibitem{hk-duke}
Harris,~Michael; Kudla,~Stephen S.
\emph{Arithmetic automorphic forms for the nonholomorphic discrete series of ${\rm GSp}(2)$.}
Duke Math. J. \textbf{66} (1992), no. 1, 59--121.

\bibitem{hks}
Harris,~Michael; Kudla,~Stephen S.; Sweet,~William J.
\emph{Theta dichotomy for unitary groups.}
J. Amer. Math. Soc. \textbf{9} (1996), no. 4, 941--1004.

\bibitem{howe-jams}
Howe,~Roger.
\emph{Transcending classical invariant theory.}
J. Amer. Math. Soc. \textbf{2} (1989), no. 3, 535--552.

\bibitem{ichino-mathZ}
Ichino,~Atsushi.
\emph{On the Siegel-Weil formula for unitary groups.}
Math. Z. \textbf{255} (2007), no. 4, 721--729.

\bibitem{periods1}
Ichino,~Atsushi; Prasanna,~Kartik.
\emph{Periods of quaternionic Shimura varieties. I.}
Contemporary Mathematics, 762. American Mathematical Society, Providence, RI, 2021.

\bibitem{periods2}
Ichino,~Atsushi; Prasanna,~Kartik.
\emph{Periods of quaternionic Shimura varieties. II.}
In preparation. 

\bibitem{johnson}
Johnson,~Joseph F.
\emph{Stable base change $\mathbb{C} / \mathbb{R}$ of certain derived functor modules.}
Math. Ann. \textbf{287} (1990), no.~3, 467--493.


\bibitem{kakuhama}
Kakuhama,~Hirotaka.
\emph{On the local factors of irreducible representations of quaternionic unitary groups.}
Manuscripta Math. \textbf{163} (2020), no.~1-2, 57--86.

\bibitem{kmsw}
Kaletha,~Tasho; Minguez,~Alberto; Shin,~Sug Woo; White,~Paul-James.
\emph{Endoscopic classification of representations: inner forms of unitary groups.}
arXiv:1409.3731.

\bibitem{ksz}
Kisin,~Mark; Shin,~Sug Woo; Zhu,~Yihang.
\emph{The stable trace formula for Shimura varieties of abelian type.}
arXiv:2110.05381.

\bibitem{kv}
Knapp,~Anthony W.; Vogan,~David A.,~Jr.
\emph{Cohomological induction and unitary representations.}
Princeton Mathematical Series \textbf{45}.
Princeton University Press, Princeton, NJ, 1995.

\bibitem{kot}
Kottwitz,~Robert E.
\emph{Shimura varieties and $\lambda$-adic representations.}
Automorphic forms, Shimura varieties, and $L$-functions, Vol. I (Ann Arbor, MI, 1988), 161--209, Perspect. Math., 10, Academic Press, Boston, MA, 1990.

\bibitem{kot-jams}
Kottwitz, Robert E.
\emph{Points on some Shimura varieties over finite fields.}
J. Amer. Math. Soc. \textbf{5} (1992), no.~2, 373--444.

\bibitem{kot-invent}
Kottwitz, Robert E.
\emph{On the $\lambda$-adic representations associated to some simple Shimura varieties.}
Invent. Math. \textbf{108} (1992), no.~3, 653--665.

\bibitem{kudla-splitting}
Kudla,~Stephen S.
\emph{Splitting metaplectic covers of dual reductive pairs.}
Israel J. Math. {\bf 87} (1994), no. 1-3, 361--401.

\bibitem{km1}
Kudla,~Stephen S.; Millson,~John J.
\emph{The theta correspondence and harmonic forms. I.}
Math. Ann. \textbf{274} (1986), no. 3, 353--378.

\bibitem{km-ihes}
Kudla,~Stephen S.; Millson,~John J.
\emph{Intersection numbers of cycles on locally symmetric spaces and Fourier coefficients of holomorphic modular forms in several complex variables.}
Inst. Hautes \'Etudes Sci. Publ. Math. No. 71 (1990), 121--172.

\bibitem{kr90}
Kudla,~Stephen; Rallis,~Stephen.
\emph{Poles of Eisenstein series and $L$-functions.}
Festschrift in honor of I. I. Piatetski-Shapiro on the occasion of his sixtieth birthday, Part II (Ramat Aviv, 1989), 81--110, Israel Math. Conf. Proc., {\bf 3}, Weizmann, Jerusalem, 1990.

\bibitem{kr94}
Kudla,~Stephen S.; Rallis,~Stephen.
\emph{A regularized Siegel-Weil formula: the first term identity.}
Ann. of Math. (2) \textbf{140} (1994), no. 1, 1--80.

\bibitem{langlands72}
Langlands,~R. P. 
\emph{Modular forms and $\ell$-adic representations. Modular functions of one variable, II.} (Proc. Internat. Summer School, Univ. Antwerp, Antwerp, 1972), pp. 361--500. Lecture Notes in Math., Vol. 349, Springer, Berlin, 1973.

\bibitem{lapid-rallis}
Lapid,~Erez M.; Rallis,~Stephen.
\emph{On the local factors of representations of classical groups.}
Automorphic representations, $L$-functions and applications: progress and prospects, 309--359, Ohio State Univ. Math. Res. Inst. Publ., {\bf 11}, de Gruyter, Berlin, 2005.
 
\bibitem{li-invent}
Li,~Jian-Shu.
\emph{Singular unitary representations of classical groups.}
Invent. Math. \textbf{97} (1989), no. 2, 237--255.

\bibitem{li-duke}
Li,~Jian-Shu.
\emph{Theta lifting for unitary representations with nonzero cohomology.}
Duke Math. J. \textbf{61} (1990), no. 3, 913--937.

\bibitem{li-crelle}
Li,~Jian-Shu.
\emph{Nonvanishing theorems for the cohomology of certain arithmetic quotients.}
J. Reine Angew. Math. \textbf{428} (1992), 177--217.

\bibitem{loke}
Loke,~Hung Yean.
\emph{Howe quotients of unitary characters and unitary lowest weight modules. With an appendix by Soo Teck Lee.}
Represent. Theory \textbf{10} (2006), 21--47.

\bibitem{mvw}
M{\oe}glin,~Colette; Vign{\'e}ras,~Marie-France; Waldspurger,~Jean-Loup.
\emph{Correspondances de Howe sur un corps $p$-adique.}
Lecture Notes in Mathematics, 1291. Springer-Verlag, Berlin, 1987.



\bibitem{mok}
Mok,~Chung Pang.
\emph{Endoscopic classification of representations of quasi-split unitary groups.}
Mem. Amer. Math. Soc. \textbf{235} (2015), no. 1108.

\bibitem{mur-ram}
Murty,~V. Kumar; Ramakrishnan,~Dinakar.
\emph{Period relations and the Tate conjecture for Hilbert modular surfaces.}
 Invent. Math. \textbf{89} (1987), no. 2, 319--345.

\bibitem{nekovar}
\Nekovar,~Jan.
\emph{Eichler-Shimura relations and semisimplicity of {\'e}tale cohomology of quaternionic Shimura varieties.}
Ann. Sci. {\'E}c. Norm. Sup{\'e}r. (4) 51 (2018), no. 5, 1179--1252.

\bibitem{oda-book}
Oda,~Takayuki.
\emph{Periods of Hilbert modular surfaces.} 
Progress in Mathematics, \textbf{19}. Birkh\"{a}user, Boston, Mass., 1982. 
 
 \bibitem{oda-hs}
 Oda,~Takayuki.
\emph{Hodge structures of Shimura varieties attached to the unit groups of quaternion algebras.}
 Galois groups and their representations (Nagoya, 1981), 15--36, Adv. Stud. Pure Math., \textbf{2}, North-Holland, Amsterdam, 1983.
 
\bibitem{psr}
Piatetski-Shapiro,~I.I.; Rallis,~Stephen.
\emph{L-functions for the classical groups.}
Explicit constructions of automorphic $L$-functions.
Lecture Notes in Mathematics {\bf 1254}, Springer-Verlag, 1987, pp. 1--52.
 

\bibitem{rangarao}
Ranga Rao,~R.
\emph{On some explicit formulas in the theory of Weil representation.}
Pacific J. Math. {\bf 157} (1993), no. 2, 335--371.

\bibitem{reimann}
Reimann, Harry.
\emph{The semi-simple zeta function of quaternionic Shimura varieties.} Lecture Notes in Mathematics, 1657. Springer-Verlag, Berlin, 1997.

\bibitem{repka}
Repka,~Joe.
\emph{Tensor products of unitary representations of ${\rm SL}\sb{2}({\bf R})$.}
Amer. J. Math. \textbf{100} (1978), no. 4, 747--774.


\bibitem{roberts}
Roberts,~Brooks. 
\emph{The theta correspondence for similitudes.}
Israel J. Math. \textbf{94} (1996), 285--317.

\bibitem{shin-templier}
Shin,~Sug Woo; Templier,~Nicolas.
\emph{On fields of rationality for automorphic representations.}
 Compos. Math. \textbf{150} (2014), no. 12, 2003--2053.

\bibitem{sun-zhu}
Sun,~Binyong; Zhu,~Chen-Bo.
\emph{Conservation relations for local theta correspondence.}
J. Amer. Math. Soc. \textbf{28} (2015), no. 4, 939--983.

\bibitem{tits-crelle}
Tits,~J.
\emph{Repr\'{e}sentations lin\'{e}aires irr\'{e}ductibles d'un groupe r\'{e}ductif sur un corps quelconque.}
(French) J. Reine Angew. Math. \textbf{247} (1971), 196--220.

\bibitem{vz}
Vogan,~David A.,~Jr.; Zuckerman,~Gregg J.
\emph{Unitary representations with nonzero cohomology.}
Compositio Math. {\bf 53} (1984), no. 1, 51--90.

\bibitem{waldspurger-rationality}
Waldspurger, J.-L. 
\emph{Quelques propri\'{e}t\'{e}s arithm\'{e}tiques de certaines formes automorphes sur GL(2).} (French) [Some arithmetical properties of certain automorphic forms on GL(2)] Compositio Math. {\bf 54} (1985), no. 2, 121--171.

\bibitem{wald-periods85}
Waldspurger,~J.-L. 
\emph{Sur les valeurs de certaines fonctions $L$ automorphes en leur centre de sym\'etrie.}
Compositio Math. \textbf{54} (1985), no. 2, 173--242.

\bibitem{waldspurger-howe-duality}
Waldspurger,~J.-L.
\emph{D\'emonstration d'une conjecture de dualit\'e de Howe dans le cas $p$-adique, $p\neq 2$.}
Festschrift in honor of I. I. Piatetski-Shapiro on the occasion of his sixtieth birthday, Part I (Ramat Aviv, 1989), 267--324, Israel Math. Conf. Proc., 2, Weizmann, Jerusalem, 1990.

\bibitem{yamana:dps}
Yamana,~Shunsuke.
\emph{Degenerate principal series representations for quaternionic unitary groups.}
Israel J. Math. \textbf{185} (2011), 77--124.

\bibitem{yamana:sw}
Yamana,~Shunsuke.
\emph{On the Siegel-Weil formula for quaternionic unitary groups.}
Amer. J. Math. \textbf{135} (2013), no. 5, 1383--1432.

\bibitem{yamana}
Yamana,~Shunsuke.
\emph{L-functions and theta correspondence for classical groups.}
Invent. Math. \textbf{196} (2014), no. 3, 651--732.

\bibitem{zucker}
Zucker,~Steven.
\emph{Locally homogeneous variations of Hodge structure.} 
Enseign. Math. (2) \textbf{27} (1981), no. 3-4, 243--276 (1982). 


\bibitem{SGA4-3}
\emph{Th\'{e}orie des topos et cohomologie \'{e}tale des sch\'{e}mas. Tome 3.} (French)
S\'{e}minaire de G\'{e}om\'{e}trie Alg\'{e}brique du Bois-Marie 1963–1964 (SGA 4). 
Dirig\'{e} par M.~Artin, A.~Grothendieck et J.~L.~Verdier. Avec la collaboration de P.~Deligne et B.~Saint-Donat.
Lecture Notes in Mathematics, Vol. 305. Springer-Verlag, Berlin-New York, 1973.























\end{thebibliography}
\end{document}